\documentclass[12pt]{book}
\usepackage[margin=0.8in]{geometry}

\usepackage{background}
\backgroundsetup{
  position=current page.east,
  angle=-90,
  vshift=-3mm,
  hshift=0mm,
  opacity=1,
  scale=3,
  contents= {\rm --~$\overline{\underline{\mbox{DRAFT}}}$ -- {\tiny(contains {\bf errors}, if you find them please let us know) --- {\today}}}
}

\usepackage[htt]{hyphenat}
\usepackage{url}

\newcommand{\R}{\mathbb{R}}
\newcommand{\Z}{\mathbb{Z}}
\newcommand{\N}{\mathbb{N}}

\newcommand{\C}{\mathbb{C}}
\newcommand{\cF}{\mathcal{F}}

\renewcommand{\epsilon}{\varepsilon}
\newcommand{\eps}{\varepsilon}
\newcommand{\e}{\varepsilon}

\usepackage[english]{babel}
\usepackage[alphabetic]{amsrefs}
\usepackage{amsmath,amssymb,amsfonts,amsthm,enumerate}
\usepackage{mathtools}
\usepackage{hyperref}
\usepackage{graphicx,epstopdf,color}
\usepackage{newtxtext}
\usepackage[varvw]{newtxmath}
\usepackage{calligra}
\usepackage[T1]{fontenc}
\usepackage{multirow}

\numberwithin{equation}{chapter}
\numberwithin{figure}{chapter}

\newtheorem{theorem}{Theorem}[chapter]
\newtheorem{lemma}[theorem]{Lemma}

\newtheorem{proposition}[theorem]{Proposition}
\newtheorem{corollary}[theorem]{Corollary}

\theoremstyle{definition}
\newtheorem{definition}[theorem]{Definition}

\theoremstyle{remark}
\newtheorem{remark}[theorem]{Remark}

\renewcommand{\leq}{\leqslant}
\renewcommand{\le}{\leqslant}
\renewcommand{\geq}{\geqslant}
\renewcommand{\ge}{\geqslant}

\DeclareMathOperator{\DIV}{div}
\DeclareMathOperator{\pv}{p.v.}

\renewcommand{\div}{\DIV}

\newcommand{\hf}{\,{}_2F_1}
\allowdisplaybreaks

\usepackage{imakeidx}
\makeindex

\begin{document}

\author{Nicola Abatangelo,\thanks{
NA: Universit\`a di Bologna,
Piazza di Porta S. Donato 5, I-40126 Bologna, Italy. {\tt nicola.abatangelo@unibo.it}}
Serena Dipierro \& Enrico Valdinoci\thanks{
SD \& EV:
University of Western Australia,
35 Stirling Highway,
Crawley WA 6009, Australia. {\tt serena.dipierro@uwa.edu.au}, 
{\tt enrico.valdinoci@uwa.edu.au}}}

\title{A gentle invitation to the fractional world}
\maketitle

\tableofcontents

\chapter*{Preface}

This book is intended as a self-contained introduction to selected topics in the fractional world, focusing particularly on aspects that arise in the study of equations driven by the fractional Laplacian.

The scope of this work is not intended to be exhaustive or all-encompassing. The literature on nonlocal equations is vast, given their relevance across various domains in mathematics and the sciences. Thus, creating a definitive compendium that fully captures the complexity and breadth of fractional operators would be practically unfeasible. Instead, we have chosen topics that we believe will appeal to readers embarking on their journey into fractional analysis. We hope that this book can further inspire and motivate interested readers, offering a perspective on the nonlocal landscape and sparking a curiosity for new fractional adventures.

The book requires only fundamental calculus and a basic understanding of measure theory. In Chapter~\ref{CHAPT-EQUIV}, we introduce the primary object of study, the fractional Laplacian. This operator appears in diverse contexts, prompting multiple definitions and viewpoints, many of which we explore, along with some key identities.

A notable distinction between local and nonlocal analysis is that in the latter, explicit calculations are often impractical or impossible. While traditional differentiation may be achieved with diligence and patience, computing the fractional Laplacian for a specific function is typically out of reach. There are anyway some fortunate exceptions, in which explicit calculations can be performed. Some of these exceptions are gathered in Chapter~\ref{CH:EXAMPLES}, providing also useful and instructive examples.

A similar feature between the classical and the nonlocal world is however that a building block in the analysis and classifications of the solutions often relies on some Liouville-type result. The literature on this topic is quite extensive and we do not aim at presenting here all possible results in their full generality, however Chapter~\ref{LIOUV:CHAP} presents an introduction to this important aspect.

A large portion of this book is devoted to the regularity theory of solutions. This can be approached through various methodologies, including partial differential equations, harmonic analysis, pseudodifferential operators, functional analysis, etc. While an exhaustive treatment of the vast literature on fractional regularity theory is beyond our scope, we have selected key topics that can be approached with minimal prerequisites and extended to more complex situations. Our focus is on regularity in Lebesgue spaces. Chapter~\ref{CHAP6} examines global solutions using Riesz and Bessel potential analysis, which happen to be well-suited to integral operators, capturing the impact of both low and high frequencies on smoothness, decay, and oscillations. These spaces are also flexible enough to provide, as a byproduct, a solid regularity theory in the more commonly\footnote{Here, we assume that the reader is already familiar with the basics of
Sobolev spaces, for which we refer e.g. to~\cite{MR2944369, MR3726909} and the references therein.} used fractional Sobolev spaces
(however, it must be recalled that
regularity theory for pseudodifferential operators often
sits better into Bessel potential and Besov settings rather
than into the Sobolev space framework, which sometimes presents
a wealth of exceptional cases with less sharpness).

From this theory, in Chapter~\ref{CHAP:INTER-REG-LEB} we derive, through suitable cutoffs and localizations, the corresponding interior regularity theory for solutions within a bounded domain using appropriate cutoffs and localization techniques. Additionally, technical appendices include auxiliary results used in key proofs.

In terms of functional analysis, one of the main objective here is to present a regularity theory
in Sobolev spaces, since these are a class of functions which have received much attention in the literature; however, in the approach that we highlight here, this kind of results is just a byproduct of the theory developed in various scales of spaces.

As a notational remark, we acknowledge the lack of a universally adopted notation for the functional spaces discussed in this book. To assist readers in comparing our notation with that of other classical references, we provide a comparative table to map equivalent notations across sources.

\begin{center}
\begin{tabular}{ |p{4cm}||p{2cm}|p{2cm}|p{2cm}|  }
 \hline
& This book &\cite{MR0290095}&\cite{MR2884718}\\
\hline
 \hline
Sobolev space &$ W^{k,p}(\R^n) $&$L^p_k(\R^n)$&$ W^{k}_p(\R^n)$\\
\hline
Bessel potential space & ${\mathcal{L}}^p_{2s}(\R^n) $ &${\mathcal{L}}^p_{2s}(\R^n)$&  $H^{2s}_p(\R^n)$\\
\hline
Besov space& $B^{s,p,q}(\R^n)$ & $\Lambda^{p,q}_\alpha(\R^n)$  &$B^s_{pq}(\R^n)$\\
 \hline
\end{tabular}\end{center}

We have strived to present a coherent and accessible resource, but perhaps human beings (and certainly ourselves) were not meant to be perfect. Therefore, if you spot inconsistencies or errors, or you would like to provide criticisms or feedback, please feel free to contact us.

\chapter{Equivalent definitions of the fractional Laplacian}\label{CHAPT-EQUIV}

The fractional Laplacian\footnote{We do not delve into the vast realm of fractional derivatives and fractional vector calculus here. For an in-depth discussion of these topics, we refer the reader to~\cite{MR30102, MR52018, MR2218073}.

In this vein, it is worth noting that numerous variations of fractional operators have been explored in the literature, resulting in a variety of so-called ``fractional Laplacians''. Thus, what we study here may not be regarded as ``the fractional Laplacian'' but rather ``a fractional Laplacian''.

Our choice of focus is primarily guided by personal preference, though alternative notions of the fractional Laplacian are equally fascinating, offering both a rich mathematical framework and practical applications. It is important to keep in mind, however, that different fractional operators exhibit distinct mathematical properties. For a description of their similarities and differences, see, for example,\cite[Sections~2.1--2.4 and~4.2--4.5]{getting}.}
can be introduced in a number of different ways, 
each one having its own advantages and drawbacks. 
In this chapter we list some of these 
and we try to highlight the different ``domains'' on which the definitions act,
i.e., the families of functions on which the definitions make sense.
Nevertheless, all these different definitions are equivalent
on a suitable common core of functions.

We take as a reference definition the one which is most times 
used as primary definition in the literature,
which is the one as an integral operator 
with singular\footnote{Sometimes ``hypersingular'' is preferred, 
to stress that the kernel is not locally integrable.} kernel.

\begin{definition}
Let~$s\in(0,1)$, $r>0$, $\sigma>s$, $x\in\R^n$ and~$u\in C^{2\sigma}(B_r(x))$ be such that
\begin{align}\label{w-cond-infty}
\int_{\R^n}\frac{|u(y)|}{{(1+|y|)}^{n+2s}}\;dy<+\infty.
\end{align}
The fractional Laplacian of~$u$ at the point~$x$ is defined as
\begin{align}
(-\Delta)^s u(x) &:= c_{n,s}\lim_{\eps\searrow 0}
\int_{\R^n\setminus B_\eps(x)}\frac{u(x)-u(y)}{{|x-y|}^{n+2s}}\;dy, \label{pv-def0} \\
{\mbox{where }}\qquad c_{n,s} &:= -\frac{2^{2s}\Gamma(\frac{n+2s}2)}{\pi^{n/2}\Gamma(-s)}>0. \label{cns}
\end{align}
\end{definition}

Other possible representations of the fractional Laplacian are listed in the following\footnote{We skip for the moment on some interesting details such as the value of the different normalizing constants involved in the formulas, limiting ourselves to just highlighting their dependences on~$n$,~$s$ and possibly other quantities. Yet, the constant denoted by~$c_{n,s}$ is always the one defined in~\eqref{cns}.

Also, in this book we are using the
``ordinary frequency'', ``unitary''
Fourier Transform $$\widehat u(\xi):=
\int_{\R^n}u(x)\,e^{-2\pi ix\cdot\xi}\,dx.$$
This convention has several advantages with respect to others, such as:
\begin{itemize}
\item The inverse Fourier Transform is the same, up to a minus sign,
\item No scaling factor appears in the integration,
\item The coordinates in Fourier space maintain the physical notion of frequencies rather than rescaling to angular frequencies,
\item The convention is in line with several classical textbooks,
such as~\cite{MR0209834}, and also~\cite{MR0290095} up to a minus sign in the complex exponential.\end{itemize}

An unpleasant drawback of this choice is that some powers of~$2\pi$ appear in the representation of pseudo-differential operators. The fact that~$2\pi\ne1$ seems to be an  unavoidable nuisance when dealing with Fourier analysis,
and other conventions can be taken to
shift the appearance to this annoying factors to other formulas.

A commonly used alternative is for instance to define
$$\widehat u(\xi):=
\int_{\R^n}u(x)\,e^{-ix\cdot\xi}\,dx,$$
which provides the benefit of not introducing powers of~$2\pi$
when considering (fractional) derivatives.
The downside with this convention is however that the map is
not unitary, that a factor is needed in the inverse Fourier Transform, and that the Fourier space
refers to angular frequencies.

In any case, our suggestion to the person entering in the theory
is to ``forget'' factors of~$2\pi$ whenever this makes the reading smoother, temporarily suspending disbelief in front of formulas such as~$2\pi=1$.

By the way, the fact that in 1897 the General Assembly of the State of Indiana wanted to pass a ``bill for an act introducing a new mathematical truth'' according to which~$\pi=3.2$
would not solve the annoyance of the factor~$2\pi$
in this book (the fact that the law, in spite of being presented as ``free of cost'', did not pass thanks to the intervention
of the mathematician Clarence Abiathar Waldo, suggests that in~1897 mathematicians did possess some influence on politicians, at least on matters of mathematical truth).} result.

\begin{theorem}\label{CARATT}
Let~$s\in(0,1)$,~$u\in C^\infty_c(\R^n)$ and~$x\in\R^n$.
Then, the fractional Laplacian of~$u$ at~$x$ can be computed as:
\begin{align}
& (-\Delta)^s u(x) \nonumber \\
&=
c_{n,s}\int_{\R^n}\frac{u(x)-u(y)}{{|x-y|}^{n+2s}}\;dy
&\text{for }s\in\left(0,\frac12\right) \label{SENZAPV}\\
&=
c_{n,s}\int_{\R^n}\frac{u(x)-u(y)-\nabla u(x)\cdot(x-y)}{{|x-y|}^{n+2s}}\;dy
&\text{for }s\in\left(\frac12,1\right) \label{taylor-def} \\
&=
c_{n,s}\int_{\R^n}\frac{u(x)-u(y)-\nabla u(x)\cdot(x-y)\chi_{B_1}(x-y)}{{|x-y|}^{n+2s}}\;dy
& \label{taylor-def-cutoff} \\
&=
\frac{c_{n,s}}2\int_{\R^n}\frac{2u(x)-u(x+y)-u(x-y)}{{|y|}^{n+2s}}\;dy \label{symmetric-def} \\
&=
\frac{c_{n,m,s}}{2}\int_{\R^n}\left[\sum_{k=-m}^m (-1)^k { \binom{2m}{m-k}} u(x+ky)\right]\;\frac{dy}{|y|^{n+2s}}
&\text{for }m\in\N\setminus\{0\} \label{HI6} \\
&=
c_{n,s-1}\int_{\R^n}\frac{\Delta u(y)}{{|x-y|}^{n-2+2s}}\;dy 
&\text{for }n\neq 2s\label{RIE}\\
&=
\int_{\R^n}(2\pi|\xi|)^{2s}\widehat{u}(\xi)e^{2\pi i x\cdot\xi}\;d\xi 
&\text{where }\widehat{u}\text{ is the Fourier Transform of }u \label{FORi} \\
&=
\frac{s}{\Gamma(1-s)}\int_0^{+\infty}\frac{u(x)-\big[e^{t\Delta}u\big](x)}{t^{1+s}}\;dt
&\text{where } \{e^{t\Delta}\}_{t>0} \text{ is the heat semigroup} \label{HESE}\\
&=
-\div^s\big(\nabla^s u\big)(x) \label{RADCVE}\\
&=
-\div^{2s-1}\big(\nabla u)(x) 
& \text{for }s\in\left(\frac12,1\right) \label{RADCVE2}\\
&=
-\div\big(\nabla^{2s-1}u\big)(x)
& \text{for }s\in\left(\frac12,1\right) \label{RADCVE3}
\end{align}  
where~$\div^s$ and~$\nabla^s$ denote respectively the fractional divergence and gradient
(see Section~\ref{sec:nonlocal-grad}). 
\end{theorem}

Here above and in the rest of this monograph, $\Gamma$ stands for the Euler Gamma Function, for which we refer to Appendix~\ref{sec:gamma}.

The proof of Theorem~\ref{CARATT} is provided here below in Sections~\ref{fractsection:1}, \ref{fractsection:2}, \ref{fractsection:3},
\ref{fractsection:4} and~\ref{sec:nonlocal-grad},
where also more general approaches will be presented
(see also~\cite{MR3613319} for a detailed investigation
of different equivalent definitions of the fractional Laplacian). The relevant notation will be introduced when needed in the following sections.

Unless otherwise stated, in all the following~$s$ will denote a real number in the interval~$(0,1)$.

\section[The fractional Laplacian as an integral operator]{The fractional Laplacian as an integral operator:
proof of \eqref{SENZAPV}, \eqref{taylor-def},
\eqref{taylor-def-cutoff} and~\eqref{symmetric-def}}\label{fractsection:1}

In the following we indicate
\begin{align}\label{w-l1-space}
L^1_s(\R^n):=\left\{u\in L^1_{{\rm{loc}}}(\R^n)\text{ s.t. }\int_{\R^n}\frac{|u(y)|}{{(1+|y|)}^{n+2s}}\;dy<+\infty\right\}.
\end{align}
Namely, $L^1_s(\R^n)$ is the space of functions in~$L^1_{{\rm{loc}}}(\R^n)$
for which the assumption in~\eqref{w-cond-infty} is satisfied.

As introduced in~\eqref{pv-def0}, the fractional Laplacian operator can be
defined as an integral operator  with singular kernel:
\begin{align}
(-\Delta)^s u(x) &= c_{n,s}\pv\int_{\R^n}\frac{u(x)-u(y)}{{|x-y|}^{n+2s}}\;dy, \label{pv-def}
\end{align}
where~$c_{n,s}$ is the constant defined in~\eqref{cns}.
The~``$\pv$'' specification stands for the integral taken in its principal value sense, i.e.,
with the limit as in~\eqref{pv-def0}: this specification is not necessary when~$s\in(0,1/2)$,
in which case some regularity on~$u$ is enough to make~\eqref{pv-def} meaningful
in the Lebesgue sense.
This is made clear by the following observation:

\begin{lemma}\label{SENZAPVL}
Let~$s\in(0,1/2)$.
If there exist~$r>0$ and~$\sigma\in(s,1/2)$ such that~$u\in C^{2\sigma}(B_r(x))\cap L^1_s(\R^n)$, then
\begin{align*}
(-\Delta)^s u(x) &= c_{n,s}\int_{\R^n}\frac{u(x)-u(y)}{{|x-y|}^{n+2s}}\;dy, 
\end{align*}
with the integral taken in Lebesgue sense.
\end{lemma}

\begin{proof}
The claim follows by a direct estimate:
\begin{align*}
& \int_{\R^n}\frac{|u(x)-u(y)|}{{|x-y|}^{n+2s}}\;dy =
\int_{B_r(x)}\frac{|u(x)-u(y)|}{{|x-y|}^{n+2s}}\;dy
+\int_{\R^n\setminus B_r(x)}\frac{|u(x)-u(y)|}{{|x-y|}^{n+2s}}\;dy \\
&\leq
[u]_{C^{2\sigma}(B_r(x))}\int_{B_r(x)}\frac{dy}{{|x-y|}^{n+2s-2\sigma}}
+|u(x)|\int_{\R^n\setminus B_r(x)}\frac{dy}{{|x-y|}^{n+2s}}
+\int_{\R^n\setminus B_r(x)}\frac{|u(y)|}{{|x-y|}^{n+2s}}\;dy \\
&=
[u]_{C^{2\sigma}(B_r(x))}\frac{|\mathbb{S}^{n-1}|}{2\sigma-2s}r^{2\sigma-2s}
+|u(x)|\frac{|\mathbb{S}^{n-1}|}{2s}r^{-2s}
+\int_{B_{2|x|+1}\setminus B_r(x)}\frac{|u(y)|}{{|x-y|}^{n+2s}}\;dy\\
&\qquad
+\int_{\R^n\setminus B_{2|x|+1}}\frac{|u(y)|}{{|x-y|}^{n+2s}}\;dy \\
&\leq
[u]_{C^{2\sigma}(B_r(x))}\frac{|\mathbb{S}^{n-1}|}{2\sigma-2s}r^{2\sigma-2s}
+|u(x)|\frac{|\mathbb{S}^{n-1}|}{2s}r^{-2s}
+r^{-n-2s}\int_{B_{2|x|+1}\setminus B_r(x)}|u(y)|\;dy \\
&\qquad +2^{n+2s}\int_{\R^n\setminus B_{2|x|+1}}\frac{|u(y)|}{{(1+|y|)}^{n+2s}}\;dy ,
\end{align*}
where we have used that~$|y|\geq2|x|+1$ implies~$2|y|-2|x|\geq1+|y|$ and therefore~$2|x-y|>1+|y|$.
\end{proof}

The claim in~\eqref{SENZAPV} is thus a consequence of Lemma~\ref{SENZAPVL}.

In order to avoid the use of the principal value in~\eqref{pv-def} for~$s\in(1/2,1)$
one could alternatively take inspiration from a Taylor expansion of~$u$ centred at~$x$, from which~\eqref{taylor-def} follows
(the condition~$s\in(1/2,1)$ being needed to ensure integrability at infinity).

To avoid both the principal value and the restriction~$s\in(1/2,1)$,
it is possible to cut off the linear term in~\eqref{taylor-def} 
as it is done in~\eqref{taylor-def-cutoff}: 
in this way no integrability issues arise at infinity 
and the singularity of the kernel is still compensated by the expansion of~$u$.

Also, one can exploit the symmetries of the integral in~\eqref{pv-def}, thus obtaining~\eqref{symmetric-def}.

Similarly as above one could verify that these expressions make sense 
when there exist~$r>0$ and~$\sigma\in(s-1/2,1)$ such that~$u\in C^{2\sigma}(B_r(x))\cap L^1_s(\R^n)$.

\subsection[Higher-order representations]{Higher-order representations: proof of~\eqref{HI6}}

We point out that
the expressions for~$(-\Delta)^s$ discussed so far
do not allow for values~$s\geq1$.
This is because the expression in~\eqref{taylor-def}
allows only for operators up to order~$2$
acting on a smooth function~$u$.

Still, it is possible to include larger values of~$s$ at the price of 
increasing the degree of the Taylor polynomial in~\eqref{taylor-def}. Namely,
one can write 
\begin{align}\label{higher-taylor-def}
(-\Delta)^s u(x) = 
-c_{n,s}\int_{\R^n}\left(u(x+y)-\sum_{|\alpha|\leq 2s}\frac{D^\alpha u(x)}{\alpha!}\,y^\alpha\right)\;\frac{dy}{{|y|}^{n+2s}}.
\end{align}
Here~$\alpha=(\alpha_1,\ldots,\alpha_n)\in\N^n$ denotes a multi-index,
$|\alpha|=\alpha_1+\ldots+\alpha_n$ the length of~$\alpha$, 
$\alpha!=\alpha_1!\cdots\alpha_n!$, $y^\alpha=y_1^{\alpha_1}\cdots y_n^{\alpha_n}$ and~$D^\alpha=\partial_{x_1}^{\alpha_1}\cdots\partial_{x_n}^{\alpha_n}$. Also, we will denote by~$\lfloor{2s}\rfloor$ the lower integer part of~$2s$.

Representation~\eqref{higher-taylor-def} does not directly apply to~$s\in\N$, 
indeed the constant~$c_{n,s}$ degenerates there (as a byproduct of
the singularity of the~$\Gamma$ functions at negative integers). 
If one is willing to obtain a ``smoother'' representation in~$s$, 
valid for all~$s>0$ and more in the spirit of~\eqref{symmetric-def}, 
an equivalent expression for~\eqref{higher-taylor-def} is
\begin{eqnarray}
&&
(-\Delta)^su(x) = 
\frac{c_{n,m,s}}{2}\int_{\R^n} \frac{\delta_m u(x,y)}{|y|^{n+2s}}\;dy\qquad
\text{for some }m\in\N,\ m>s, \label{higher-symmetric-def} \\ 
&& \text{with}\quad 
\delta_m u(x,y) = 
\sum_{k=-m}^m (-1)^k { \binom{2m}{m-k}} u(x+ky) \label{higher-symmetric-def2} .\end{eqnarray}
Also, the constant~$c_{n,m,s}$ can be written as
\begin{equation}c_{n,m,s} = 
-c_{n,s}\left(\sum_{k=1}^{m}(-1)^{k}{\binom{2m}{m-k}} k^{2s}\right)^{-1} 
\quad
\text{for } s\in(0,m)\setminus\N \label{cnms} \end{equation}
and
\begin{equation}c_{n,m,s} = 
\frac{2^{2s-1}\,\Gamma(\frac{n}2+s)\,s!}{\pi^{n/2}}\left(\sum_{k=2}^{m}(-1)^{k-s+1}{\binom{2m}{m-k}} k^{2s}\ln(k)\right)^{-1} 
\quad
\text{for } s\in(0,m)\cap\N. \label{cnms2}
\end{equation}
See~\cite{MR3809107} for more details on this definition and related references.

The validity of~\eqref{higher-taylor-def} and~\eqref{higher-symmetric-def} 
is useful when one is interested in analysing 
the behaviour of~${(-\Delta)}^s$ when~$s\to1$ or, more philosophically,
to place the fractional and the classical Laplacian under a common roof:
it might not be evident at first glance, but~\eqref{higher-taylor-def} 
and~\eqref{higher-symmetric-def} give back a classical Laplacian upon setting~$s=1$,
and~\eqref{higher-symmetric-def} is merely a more convoluted version of~\eqref{symmetric-def}
for~$s\in(0,1)$. We give here below an explanation of these facts.

\begin{lemma}
Let~$r>0$, $2s\in(0,+\infty)\setminus\N$, $\beta\in(2s,\lfloor{2s}\rfloor+1]$, $x\in\R^n$ and~$u\in C^\beta(B_r(x))\cap L^1_s(\R^n)$. Then, \eqref{higher-taylor-def} and~\eqref{higher-symmetric-def} coincide.
\end{lemma}

\begin{proof}
We define, for every~$y\in\R^n$,
\begin{eqnarray*}
R(x,y)&:=&u(x+y)- \sum_{|\alpha|\leq {2s}}\frac{D^\alpha u(x)}{\alpha!}y^\alpha\\&=&
u(x+y)- \sum_{|\alpha|\leq \lfloor{2s}\rfloor}\frac{D^\alpha u(x)}{\alpha!}y^\alpha.\end{eqnarray*}
Thanks to the regularity of~$u$, we have that
\begin{align*}
\big|R(x,y)\big|&\leq \|u\|_{C^{\beta}(B_r(x))}{|y|}^{\beta}
\qquad\text{for every }y\in B_r(x).
\end{align*}
Therefore, one can see that
\begin{align*}
\int_{\R^n}\frac{\big|R(x,ky)\big|}{{|y|}^{n+2s}}\;dy<+\infty
\qquad\text{for }k\in\Z.
\end{align*}

Also, whenever~$|\alpha|\leq 2s<2m$, from~\eqref{app-combinissima} we deduce that
\begin{align*}
\sum_{k=-m}^m(-1)^k\binom{2m}{m-k}k^{|\alpha|}=0.
\end{align*}
As a consequence,
\begin{align*}
& \frac{c_{n,m,s}}{2}\int_{\R^n}\frac{\delta_mu(x,y)}{|y|^{n+2s}}\,dy
\\
&=\frac{c_{n,m,s}}{2}\int_{\R^n}\sum_{k=-m}^m(-1)^k\binom{2m}{m-k}u(x+ky)\;\frac{dy}{{|y|}^{n+2s}} \\
&=\frac{c_{n,m,s}}{2}\int_{\R^n}\sum_{k=-m}^m(-1)^k\binom{2m}{m-k} \left(R(x,ky)+ \sum_{|\alpha|\leq {2s} }\frac{D^\alpha u(x)}{\alpha!}k^{|\alpha|} y^\alpha \right)
\;\frac{dy}{{|y|}^{n+2s}}\\
& =
\frac{c_{n,m,s}}{2}\int_{\R^n}\sum_{k=-m}^m(-1)^k\binom{2m}{m-k}R(x,ky)\;\frac{dy}{{|y|}^{n+2s}} \\
& =
\frac{c_{n,m,s}}{2}\sum_{k=-m}^m(-1)^k\binom{2m}{m-k}
\int_{\R^n}R(x,ky)\;\frac{dy}{{|y|}^{n+2s}} \\
& =
\frac{c_{n,m,s}}{2}\Bigg(\sum_{k=-m}^m(-1)^k\binom{2m}{m-k}|k|^{2s}\Bigg)
\int_{\R^n}R(x,z)\;\frac{dy}{{|z|}^{n+2s}} \\
& =
-c_{n,s}
\int_{\R^n}R(x,z)\;\frac{dy}{{|z|}^{n+2s}} 
\\&=
-c_{n,s}
\int_{\R^n}\left(u(x+z)-\sum_{|\alpha|\leq {2s} }\frac{D^\alpha u(x)}{\alpha!}{z}^\alpha\right)\;\frac{dy}{{|z|}^{n+2s}},
\end{align*}as desired.
\end{proof}

\begin{lemma}\label{lem:sym-vs-high}
Let~$s\in(0,1)$,~$r>0$,~$\beta\in(2s,\lfloor 2s\rfloor+1]$, $x\in\R^n$ and~$u\in C^\beta(B_r(x))\cap L^1_s(\R^n)$. Then, \eqref{symmetric-def} and~\eqref{higher-symmetric-def} coincide.
\end{lemma}

\begin{proof}
Thanks to the combinatorial identity in~\eqref{foot-combino}, we have that
\begin{align}\label{combino}
\sum_{k=-m}^m(-1)^k\binom{2m}{m-k}=0,
\end{align}
and therefore we can suppose that~$u(x)=0$ in~\eqref{higher-symmetric-def},
up to substituting~$u(x\pm ky)$ with~$u(x\pm ky)-u(x)$.

We start from~\eqref{higher-symmetric-def} and write
\begin{align*}
\int_{\R^n}\frac{\delta_m u(x,y)}{{|y|}^{n+2s}}\;dy=\lim_{\eps\searrow 0}
\sum_{k=-m}^m(-1)^k\binom{2m}{m-k}\int_{\R^n\setminus B_\eps}\frac{u(x+ky)}{{|y|}^{n+2s}}\;dy,
\end{align*}
where, for~$k>0$,
\begin{align*}
\int_{\R^n\setminus B_\eps}\frac{u(x+ky)}{{|y|}^{n+2s}}\;dy
&=
k^{2s}\int_{\R^n\setminus B_{\eps k}}\frac{u(x+y)}{{|y|}^{n+2s}}\;dy \\
&=
k^{2s}\int_{\R^n\setminus B_{\eps}}\frac{u(x+y)}{{|y|}^{n+2s}}\;dy
-k^{2s}\int_{B_{\eps k}\setminus B_\eps}\frac{u(x+y)}{{|y|}^{n+2s}}\;dy,
\end{align*}
and, for~$k<0$,
\begin{align*}
\int_{\R^n\setminus B_\eps}\frac{u(x+ky)}{{|y|}^{n+2s}}\;dy
&=
|k|^{2s}\int_{\R^n\setminus B_{\eps |k|}}\frac{u(x-y)}{{|y|}^{n+2s}}\;dy \\ 
&=
|k|^{2s}\int_{\R^n\setminus B_{\eps}}\frac{u(x-y)}{{|y|}^{n+2s}}\;dy
-|k|^{2s}\int_{B_{\eps |k|}\setminus B_\eps}\frac{u(x-y)}{{|y|}^{n+2s}}\;dy.
\end{align*}

Notice also that, for all~$k\in\{1,\dots,m\}$,
$$ \binom{2m}{m-k}=\binom{2m}{m+k}.$$
Gathering these pieces of information, we obtain that
\begin{equation}\label{657483gfhdsged624utru}\begin{split}
&\sum_{k=-m}^m(-1)^k\binom{2m}{m-k}\int_{\R^n\setminus B_\eps}\frac{u(x+ky)}{{|y|}^{n+2s}}\;dy \\
&=
\sum_{k=1}^m(-1)^k\binom{2m}{m-k}k^{2s}\int_{\R^n\setminus B_{\eps k}}\frac{u(x+y)}{{|y|}^{n+2s}}\;dy 
\\&\qquad\qquad+\sum_{k=-m}^{-1}(-1)^k\binom{2m}{m-k}{|k|}^{2s}\int_{\R^n\setminus B_{\eps |k|}}\frac{u(x-y)}{{|y|}^{n+2s}}\;dy \\
&=
\sum_{k=1}^m(-1)^k\binom{2m}{m-k}k^{2s}\int_{\R^n\setminus B_{\eps k}}\frac{u(x+y)}{{|y|}^{n+2s}}\;dy 
\\&\qquad\qquad+\sum_{k=1}^{m}(-1)^k\binom{2m}{m+k}{k}^{2s}\int_{\R^n\setminus B_{\eps k}}\frac{u(x-y)}{{|y|}^{n+2s}}\;dy \\
&=
\sum_{k=1}^m(-1)^k\binom{2m}{m-k}k^{2s}\int_{\R^n\setminus B_{\eps}}\frac{u(x+y)+u(x-y)}{{|y|}^{n+2s}}\;dy \\
&\qquad\qquad 
-\sum_{k=1}^m(-1)^k\binom{2m}{m-k}k^{2s}\int_{B_{\eps k}\setminus B_\eps}\frac{u(x+y)+u(x-y)}{{|y|}^{n+2s}}\;dy .
\end{split}\end{equation}

We now check that
\begin{equation}\label{feiw75658490203}
\lim_{\eps\searrow 0}
\sum_{k=1}^m(-1)^k\binom{2m}{m-k}k^{2s}\int_{B_{\eps k}\setminus B_\eps}\frac{u(x+y)+u(x-y)}{{|y|}^{n+2s}}\;dy=0.
\end{equation}
Indeed, for every~$k\in\{1,\ldots,m\}$, recalling that we are supposing~$u(x)=0$
and~$u\in C^\beta(B_r(x))$ with~$\beta>2s$,
\begin{eqnarray*}
\left|\int_{B_{\eps k}\setminus B_\eps}\frac{u(x+y)+u(x-y)}{{|y|}^{n+2s}}\;dy\right|
&\leq&\|u\|_{C^{\beta}(B_r(x))}\int_{B_{\eps m}}\frac{dy}{{|y|}^{n+2s-\beta}}
\\&=&\|u\|_{C^{\beta}(B_r(x))}|\mathbb{S}^{n-1}|\frac{\eps^{\beta-2s} m^{\beta-2s}}{\beta-2s},
\end{eqnarray*}
which gives~\eqref{feiw75658490203}.

{F}rom~\eqref{657483gfhdsged624utru}
and~\eqref{feiw75658490203}, and recalling~\eqref{cnms}, we thereby see that~\eqref{symmetric-def} and~\eqref{higher-symmetric-def} coincide.
\end{proof}

The claim in~\eqref{HI6} now follows from~\eqref{higher-symmetric-def} and~\eqref{higher-symmetric-def2} in combination with
Lemma~\ref{lem:sym-vs-high}.

Although~\eqref{higher-symmetric-def} with~\eqref{cnms} does not apply for~$s\in\N$
because of the inherent degeneracy in the constant~$c_{n,s}$,
we point out that a suitable change of the value of 
the normalizing constant, as performed in~\eqref{cnms2}, 
allows for an extension of~\eqref{symmetric-def} also to~$s\in\N$. 
This accounts for an interesting nonlocal representation\footnote{Results such as the one
in Lemma~\ref{qkjdcm.203weodlc} are perhaps a bit puzzling, since they reveal that it
is sometimes difficult to detect the ``locality'' of an operator given one of its possible representation
and a natural question is thus how to recognize that a given operator is ``local''. 
For this, see~\cite[Th\'eor\`em~2]{MR124611}
or~\cite[Theorem 3.3.11]{MR346855},
which give that a pseudo-differential operator is local if and only if its symbol is a polynomial (see also~\cite[Definition~3.3.1 on page~172]{MR346855} for the definition of linear differential operators and~\cite[Theorem 3.3.3 on page~174]{MR346855}
which locally identifies the local operators with differential operators).

For related questions, see also~\cite{MR850715, MR968206, MR1873235, MR1917230, MR2158336}.}
of a local operator, namely, a polylaplacian.
We show this fact in the particular case~$s=1$ by giving the following:

\begin{lemma}\label{qkjdcm.203weodlc}
Let~$r>0$, $\beta>2$, $x\in\R^n$ and~$u\in C^\beta(B_r(x))$ with
\begin{align*}
\int_{\R^n}\frac{|u(y)|}{{(1+|y|)}^{n+2}}\;dy<+\infty.
\end{align*}
Then,
\begin{align*}
\Delta u(x)=
\frac{\Gamma(\frac{n}2+1)}{\pi^{n/2}}
\left(\sum_{k=1}^m(-1)^{k+1}\binom{2m}{m-k}k^2\,\ln k\right)^{-1}
\int_{\R^n} \frac{\delta_m u(x,y)}{|y|^{n+2}}\;dy
\qquad\text{for }m\in\N\setminus\{0,1\}.
\end{align*}
\end{lemma}

\begin{proof}
We proceed here in a somewhat similar way as that of the proof of Lemma~\ref{lem:sym-vs-high}.
Thanks to~\eqref{combino}, without loss of generality, we suppose that~$u(x)=0$ in~\eqref{higher-symmetric-def}. 

We start from~\eqref{higher-symmetric-def} with~$s=1$ and write
\begin{equation}\label{65748o1qsxc56534yoiky4}
\int_{\R^n}\frac{\delta_m u(x,y)}{{|y|}^{n+2}}\;dy=\lim_{\eps\searrow 0}
\sum_{k=-m}^m(-1)^k\binom{2m}{m-k}\int_{\R^n\setminus B_\eps}\frac{u(x+ky)}{{|y|}^{n+2}}\;dy.
\end{equation}

Thanks to the combinatorial identity in~\eqref{app-combinissima}, we find that
\begin{align*}
\sum_{k=1}^m(-1)^k\binom{2m}{m-k}k^2=0
\qquad\text{for }m>1.
\end{align*}
Therefore, using also the computations in~\eqref{657483gfhdsged624utru} with~$s=1$,
\begin{align*}
& \sum_{k=-m}^m(-1)^k\binom{2m}{m-k}\int_{\R^n\setminus B_\eps}\frac{u(x+ky)}{{|y|}^{n+2}}\;dy\\
& =
\sum_{k=1}^m
(-1)^k\binom{2m}{m-k}k^2\int_{\R^n\setminus B_{\eps}}\frac{u(x+y)+u(x-y)}{{|y|}^{n+2}}\;dy \\
& \qquad\qquad 
-\sum_{k=1}^m (-1)^k\binom{2m}{m-k}k^2\int_{B_{\eps k}\setminus B_\eps}\frac{u(x+y)+u(x-y)}{{|y|}^{n+2}}\;dy\\
& =
-\sum_{k=1}^m (-1)^k\binom{2m}{m-k}k^2\int_{B_{\eps k}\setminus B_\eps}\frac{u(x+y)+u(x-y)}{{|y|}^{n+2}}\;dy.
\end{align*}
Plugging this information into~\eqref{65748o1qsxc56534yoiky4} we thus find that
\begin{eqnarray*}
\int_{\R^n}\frac{\delta_m u(x,y)}{{|y|}^{n+2}}\;dy
&=&\lim_{\eps\searrow 0}
\sum_{k=1}^m(-1)^{k+1}\binom{2m}{m-k}k^2\int_{B_{\eps k}\setminus B_\eps}\frac{u(x+y)+u(x-y)}{{|y|}^{n+2}}\;dy\\
&=&2\lim_{\eps\searrow 0}
\sum_{k=1}^m(-1)^{k+1}\binom{2m}{m-k}k^2\int_{B_{\eps k}\setminus B_\eps}\frac{u(x+y)}{{|y|}^{n+2}}\;dy
.
\end{eqnarray*}

Now, exploiting the symmetries and using polar coordinates, we see that
\begin{align}
& \int_{B_{\eps k}\setminus B_\eps}\frac{u(x+y)}{{|y|}^{n+2}}\;dy=
\int_{B_{\eps k}\setminus B_\eps}\frac{u(x+y)-\nabla u(x)\cdot y}{{|y|}^{n+2}}\;dy \nonumber \\
&=
\int^{\eps k}_\eps\frac{1}{\rho^3}\int_{\mathbb{S}^{n-1}}\big(u(x+\rho\theta)-\rho\nabla u(x)\cdot\theta\big)\;d\theta\;d\rho \nonumber \\
&=
\int^{\eps k}_\eps\frac{1}{\rho^3}\int_{\mathbb{S}^{n-1}}\left(
u(x+\rho\theta)-\rho\nabla u(x)\cdot\theta-\frac{\rho^2}2 D^2u(x)\theta\cdot\theta\right)\;d\theta\;d\rho \label{fi80qyewiohflnkc} \\
&\qquad 
+\frac12\int^{\eps k}_\eps\frac{1}{\rho}\int_{\mathbb{S}^{n-1}} D^2u(x)\theta\cdot\theta\;d\theta\;d\rho.
\label{jf90we7jwihoe}
\end{align}
By the regularity of~$u$, if~$\rho$ is sufficiently small,
\begin{align*}
\left|u(x+\rho\theta)-\rho\nabla u(x)\cdot\theta-\frac{\rho^2}2 D^2u(x)\theta\cdot\theta\right|
\leq \|u\|_{C^\beta(B_r(x))}\rho^\beta,
\end{align*}
and therefore the integral in~\eqref{fi80qyewiohflnkc} can be estimated by 
\begin{align*}
\|u\|_{C^\beta(B_r(x))}\big|\mathbb{S}^{n-1}\big|
\int^{\eps k}_\eps\rho^{\beta-3}\;d\rho
=\|u\|_{C^\beta(B_r(x))}\big|\mathbb{S}^{n-1}\big|
\big(k^{\beta-2}-1\big)\frac{\eps^{\beta-2}}{\beta-2},
\end{align*}
which is infinitesimal as~$\epsilon\searrow0$ (recall that~$\beta>2$).

Moreover, by symmetry, we have that
\begin{align*}
\int_{\mathbb{S}^{n-1}}\partial_{ij}u(x)\theta_i \theta_j\;d\theta=0
\qquad\text{whenever }i,j\in\{1,\ldots,n\},\ i\neq j,
\end{align*}
and
\begin{align*}
\big|\mathbb{S}^{n-1}\big|=\int_{\mathbb{S}^{n-1}}|\theta|^2\;d\theta
=\sum_{i=1}^n\int_{\mathbb{S}^{n-1}}\theta_i^2\;d\theta=n\int_{\mathbb{S}^{n-1}}\theta_j^2\;d\theta \qquad\text{for all }j\in\{1,\ldots,n\}.
\end{align*}
Accordingly,
\begin{align}\label{o9fpewFJKEòJFkvnsdm}
\int_{\mathbb{S}^{n-1}}D^2u(x)\theta\cdot\theta\;d\theta=
\sum_{j=1}^n\partial_{jj}^2u(x)\int_{\mathbb{S}^{n-1}}\theta_j^2\;d\theta=
\sum_{j=1}^n\partial_{jj}^2u(x)\,\frac{\big|\mathbb{S}^{n-1}\big|}n=\frac{\big|\mathbb{S}^{n-1}\big|}n\,\Delta u(x).
\end{align}
Using this, we have that the integral in~\eqref{jf90we7jwihoe} equals
\begin{align*}
\frac12\ln k\int_{\mathbb{S}^{n-1}} D^2u(x)\theta\cdot\theta\;d\theta=\frac{\ln k}{2n}\big|\mathbb{S}^{n-1}\big|\,\Delta u(x).
\end{align*}

Hence, recalling also
the identity~\eqref{measure-n-sphere} for the measure of the sphere
and~\eqref{gamma-recursive},
we conclude that
\begin{align*}
\int_{\R^n}\frac{\delta_m u(x,y)}{{|y|}^{n+2}}\;dy
&=\frac{\big|\mathbb{S}^{n-1}\big|}{n}
\sum_{k=1}^m(-1)^{k+1}\binom{2m}{m-k}k^2\,\ln k
\,\Delta u(x)\\
&= \frac{2\pi^{n/2}}{n\Gamma(\frac{n}2)}
\sum_{k=1}^m(-1)^{k+1}\binom{2m}{m-k}k^2\,\ln k
\,\Delta u(x)\\&= \frac{\pi^{n/2}}{\Gamma(\frac{n}2+1)}
\sum_{k=1}^m(-1)^{k+1}\binom{2m}{m-k}k^2\,\ln k
\,\Delta u(x)
\end{align*}
and the proof is complete.
\end{proof}

\section[The fractional Laplacian as a Riesz potential]{The fractional Laplacian as a Riesz potential: proof of~\eqref{RIE}}\label{fractsection:2}

The fractional Laplacian can be obtained as the composition of a classical Laplacian
with a Riesz kernel, the one providing 
the fundamental solution of the fractional Laplacian of order~$2(1-s)$ (see Corollary~\ref{FURIEZ}): 
formally, such composition will give back an operator of order~$2-2(1-s)=2s$, 
which is consistent with the order of the fractional Laplacian.

More precisely we have\footnote{We point out that there is an abuse of notation
when writing~$c_{n,s-1}$ in~\eqref{cnwkel09j3o32mf}. Indeed the constant~$c_{n,s}$ in~\eqref{cns} was defined for~$s\in(0,1)$
and here we are implicitly computing this constant at a negative value in~$(-1,0)$
for its second entry.
\label{treyuwidbsna6437urfgsdjwktr36uwyfjhsgoyriwek}

While this notation is not going to create any confusion, it is useful to remark that a sign change occurs. Namely, while~$c_{n,s}>0$, we have that~$c_{n,s-1}<0$
for all~$s\in(0,1)$.} the following:

\begin{lemma}\label{RIEL}
Let~$s\in(0,1)$ (additionally, $s\neq 1/2$ if $n=1$) and~$u\in C^2(\R^n)\cap L^1_s(\R^n)$ be such that
\begin{equation}\label{RIEP}
\int_{\R^n}\frac{|\nabla u(y)|}{(1+|y|)^{n-1+2s}}\;dy<+\infty
\qquad{\mbox{and}}\qquad
\int_{\R^n}\frac{|\Delta u(y)|}{(1+|y|)^{n-2+2s}}\;dy<+\infty.
\end{equation}
Then,
\begin{align}\label{cnwkel09j3o32mf}
(-\Delta)^s u(x)=c_{n,s-1}\int_{\R^n}\frac{\Delta u(y)}{{|x-y|}^{n-2+2s}}\;dy.
\end{align}
\end{lemma}

\begin{proof}We stress that the right-hand side of~\eqref{cnwkel09j3o32mf}
is finite, due to~\eqref{RIEP}.

The proof of~\eqref{cnwkel09j3o32mf} relies on a double integration by parts,
but we need first to write 
\begin{align}\label{126BIS00}
\int_{\R^n}\frac{-\Delta u(y)}{{|x-y|}^{n-2+2s}}\;dy
=\frac12\int_{\R^n}\frac{\Delta_y\big(2u(x)-u(x+y)-u(x-y)\big)}{{|y|}^{n-2+2s}}\;dy
\end{align}
where by~$\Delta_y$ we mean that derivatives are taken with respect to the~$y$ variable.

In this notation, we have that
\begin{align*}
\nabla_y{|y|}^{2-n-2s}=(2-n+2s) {|y|}^{-n-2s}y
\end{align*}
and
\begin{align*}
\Delta_y{|y|}^{2-n-2s}=(2-n-2s)\div_y\big({|y|}^{-n-2s}y\big)=-2s(2-n-2s){|y|}^{-n-2s}.
\end{align*}

Also, if~$\rho\ge|x|$ and~$y\in \R^n\setminus B_\rho$,
$$ |x+y|\le |x|+|y|\le\rho+|y|\le2|y|$$
and accordingly, by~\eqref{w-cond-infty} and~\eqref{RIEP}, 
\begin{eqnarray*}&&
\int_{\R^n\setminus B_\rho}\left( \frac{|u(x+y)|}{|y|^{n+2s}}+\frac{|\nabla u(x+y)|}{|y|^{n-1+2s}}\right)\,dy\le C\int_{\R^n\setminus B_\rho}\left( \frac{|u(x+y)|}{(1+|y|)^{n+2s}}+\frac{|\nabla u(x+y)|}{(1+|y|)^{n-1+2s}}\right)\,dy
\\&&\qquad\le
C\int_{\R^n}\left( \frac{|u(x+y)|}{(1+|x+y|)^{n+2s}}+\frac{|\nabla u(x+y)|}{(1+|x+y|)^{n-1+2s}}\right)\,dy
\\&&\qquad=C\int_{\R^n}\left( \frac{|u(y)|}{(1+|y|)^{n+2s}}+\frac{|\nabla u(y)|}{(1+|y|)^{n-1+2s}}\right)\,dy<+\infty,
\end{eqnarray*}
for some~$C>0$ possibly depending on~$\rho$ and varying from line to line.

Hence, by polar coordinates,
\begin{equation}\label{MEEXPo-0.1}
\int_\rho^{+\infty}
\left[
\int_{\partial B_R}
\left( \frac{|u(x+y)|}{|y|^{n+2s}}+\frac{|\nabla u(x+y)|}{|y|^{n-1+2s}}\right)
\,d{\mathcal{H}}^{n-1}_y
\right]\,dR<+\infty.\end{equation}

Now we claim that there exists a diverging sequence~$R_k$ for which
\begin{equation}\label{MEEXPo-0.2}
\lim_{k\to+\infty}\int_{\partial B_{R_k}}
\left( \frac{|u(x+y)|}{|y|^{n-1+2s}}+\frac{|\nabla u(x+y)|}{|y|^{n-2+2s}}\right)
\,d{\mathcal{H}}^{n-1}_y=0.\end{equation}
Indeed, suppose not, then there exist~$R_0>\rho$ and~$c>0$ such that, for all~$R\ge R_0$,
$$\int_{\partial B_{R}}
\left( \frac{|u(x+y)|}{|y|^{n-1+2s}}+\frac{|\nabla u(x+y)|}{|y|^{n-2+2s}}\right)
\,d{\mathcal{H}}^{n-1}_y\ge c.$$
Consequently,
\begin{eqnarray*}&&
\int_\rho^{+\infty}
\left[
\int_{\partial B_R}
\left( \frac{|u(x+y)|}{|y|^{n+2s}}+\frac{|\nabla u(x+y)|}{|y|^{n-1+2s}}\right)
\,d{\mathcal{H}}^{n-1}_y
\right]\,dR\\&&\qquad\ge
\int_{R_0}^{+\infty}
\left[
\int_{\partial B_R}
\left( \frac{|u(x+y)|}{|y|^{n-1+2s}}+\frac{|\nabla u(x+y)|}{|y|^{n-2+2s}}\right)
\,d{\mathcal{H}}^{n-1}_y
\right]\,\frac{dR}R\ge c\int_{R_0}^{+\infty}\frac{dR}R=+\infty,
\end{eqnarray*}
in contradiction with~\eqref{MEEXPo-0.1}, thus proving~\eqref{MEEXPo-0.2}.

We are now in position to integrate by parts and use~\eqref{MEEXPo-0.2} to
obtain that 
\begin{eqnarray*} &&
\int_{\R^n}\frac{\Delta_y\big(2u(x)-u(x+y)-u(x-y)\big)}{{|y|}^{n-2+2s}}\;dy\\&&\qquad
=\lim_{k\to+\infty}\int_{B_{R_k}}\frac{\Delta_y\big(2u(x)-u(x+y)-u(x-y)\big)}{{|y|}^{n-2+2s}}\;dy
\\&&\qquad=\lim_{k\to+\infty}
\int_{B_{R_k}}\big(2u(x)-u(x+y)-u(x-y)\big)\,\Delta_y{|y|}^{2-n-2s}\;dy\\&&\qquad\qquad
+\int_{\partial B_{R_k}}
\nu(y)\cdot\nabla_y\big(2u(x)-u(x+y)-u(x-y)\big)|y|^{2-n-2s}\,d{\mathcal{H}}^{n-1}\\&&\qquad\qquad
-\int_{\partial B_{R_k}}\nu(y)\cdot\nabla_y|y|^{2-n-2s}
\big(2u(x)-u(x+y)-u(x-y)\big) 
\,d{\mathcal{H}}^{n-1}\\&&\qquad
=-2s(2-n-2s)\int_{\R^n}\frac{2u(x)-u(x+y)-u(x-y)}{|y|^{n+2s}}\,dy.
\end{eqnarray*}
Using this into~\eqref{126BIS00}, and recalling the representation
in~\eqref{symmetric-def}, we conclude that
\begin{equation}\label{thiscomprye74ur}\begin{split}
\int_{\R^n}\frac{\Delta u(y)}{{|x-y|}^{n-2+2s}}\;dy
&=s(2-n-2s)\int_{\R^n}\frac{2u(x)-u(x+y)-u(x-y)}{|y|^{n+2s}}\,dy\\&=
\frac{2s(2-n-2s)}{c_{n,s}}(-\Delta)^s u(x).\end{split}
\end{equation}

We use the property of the Gamma function in~\eqref{gamma-recursive}
and we recall footnote~\ref{treyuwidbsna6437urfgsdjwktr36uwyfjhsgoyriwek}
to compute
\begin{eqnarray*}
&& \frac{c_{n,s}}{2s(2-n-2s)}=
-\frac{2^{2s}\Gamma(\frac{n+2s}2)}{2s(2-n-2s)\pi^{n/2}\Gamma(-s)}\\&&\qquad
= \frac{2^{2s}s(n+2s-2)\Gamma(\frac{n-2+2s}2)}{4s(2-n-2s)\pi^{n/2}\Gamma(1-s)}=-\frac{2^{2s-2}\Gamma(\frac{n-2+2s}2)}{\pi^{n/2}\Gamma(1-s)}=c_{n,s-1}.
\end{eqnarray*}
This computation and~\eqref{thiscomprye74ur}
show that the representation in~\eqref{cnwkel09j3o32mf} holds true,
as desired.
\end{proof}

The claim in~\eqref{RIE} is now a consequence of Lemma~\ref{RIEL}.

\section[The fractional Laplacian as a Fourier multiplier]{The fractional Laplacian as a Fourier multiplier: proof of~\eqref{FORi}}\label{fractsection:3}

In this section, we consider the Fourier Transform approach to the fractional Laplacian. 
More precisely, as claimed in~\eqref{FORi},
the fractional Laplacian can be considered as a Fourier multiplier in frequency space with symbol~$(2\pi|\xi|)^{2s}$,
namely
\begin{align*}
(-\Delta)^s u(x) = \int_{\R^n}(2\pi|\xi|)^{2s}\widehat{u}(\xi)e^{2\pi i x\cdot\xi}\;d\xi
\end{align*}
for any $u\in C^\infty_c(\R^n)$.

For an introduction to the topic of pseudodifferential operators, we refer the reader to Appendix~\ref{ojdlfwenSTRFGbdollDeltafRn}.

\begin{lemma}
Let~$s\in(0,1)$, $u\in C^\infty_c(\R^n)$ and~$x\in\R^n$.
Then, \eqref{FORi} holds true.
\end{lemma}

\begin{proof}
Let us start from the representation of the fractional Laplacian given in~\eqref{cnwkel09j3o32mf}. In this way, we write
\begin{eqnarray*}
(-\Delta)^s u(x)=
c_{n,s-1}\int_{\R^n}\frac{\Delta u(y)}{{|x-y|}^{n-2+2s}}\;dy
=c_{n,s-1}\int_{\R^n}\frac{\Delta u(x+y)}{|y|^{n-2+2s}}\;dy .
\end{eqnarray*}

If~$n\ge2$, we exploit Lemma~\ref{CAL:GA:099}
with~$k:=0$, $P\equiv 1$ and~$\alpha:=n-2+2s$. Notice that with this choice~$\alpha\in(0,n)$. As a result, we have that
\begin{align*}&
\frac{\pi^{-\frac{n}2+2-2s}\,\Gamma\big(\frac{n-2+2s}2\big)}{\Gamma(1-s)}\int_{\R^n}{|y|}^{-n+2-2s}\,\Delta u(x+y)\;dy \\
&\qquad\qquad=
\int_{\R^n}{|\xi|}^{-2+2s}\,\mathcal{F}\big[\Delta u(x+\cdot)\big](\xi)\;d\xi \\
&\qquad\qquad=
\int_{\R^n}{|\xi|}^{-2+2s}\,\mathcal{F}\big[\Delta u\big](\xi)\,e^{2\pi ix\cdot\xi}\;d\xi \\
&\qquad\qquad=-
\int_{\R^n}{|\xi|}^{-2+2s} (2\pi|\xi|)^{2}\widehat{u}(\xi)\,e^{2\pi ix\cdot\xi}\;d\xi \\
&\qquad\qquad=-
2^{2-2s} \pi^{2-2s}\int_{\R^n}{(2\pi|\xi|)}^{2s}\widehat{u}(\xi)\,e^{2\pi ix\cdot\xi}\;d\xi. 
\end{align*}
Using~\eqref{cns}, we see that
$$ c_{n,s-1}= -\frac{2^{2s-2}\Gamma(\frac{n-2+2s}2)}{\pi^{n/2}\Gamma(1-s)}$$
and therefore
\begin{eqnarray*}
\int_{\R^n}{(2\pi|\xi|)}^{2s}\widehat{u}(\xi)\,e^{2\pi ix\cdot\xi}\;d\xi
&=& -\frac{2^{2s-2}\pi^{-\frac{n}2}\,\Gamma\big(\frac{n-2+2s}2\big)}{\Gamma(1-s)}\int_{\R^n}{|y|}^{-n+2-2s}\,\Delta u(x+y)\;dy\\&=&
c_{n,s-1}\int_{\R^n}{|y|}^{-n+2-2s}\,\Delta u(x+y)\;dy\\&=&
(-\Delta)^s u(x),
\end{eqnarray*}
which proves~\eqref{FORi} in this case.

Actually, the same approach would work also for~$n=1$ and~$s\in\left(\frac12,1\right)$,
but we prefer to treat separately the whole case~$n=1$ as follows.

Starting from~\eqref{symmetric-def} 
and representing~$u$ in terms of its Fourier Transform~$\widehat{u}$, 
we have that
\begin{align*}
{(-\Delta)}^su(x)
&=
\frac{c_{1,s}}2
\int_\R\frac{2u(x)-u(x+y)-u(x-y)}{{|y|}^{1+2s}}\;dy \\
&=
-\frac{2^{2s}\Gamma(\frac12+s)}{\sqrt\pi\,\Gamma(-s)}
\int_0^{+\infty}\frac{2u(x)-u(x+y)-u(x-y)}{y^{1+2s}}\;dy \\
&=
-\frac{2^{2s}\Gamma(\frac12+s)}{\sqrt\pi\,\Gamma(-s)}
\int_0^{+\infty}
\int_\R\left(2-e^{2\pi iy\xi}-e^{-2\pi iy\xi}\right)\widehat{u}(\xi)e^{2\pi ix\xi}\;d\xi\;\frac{dy}{y^{1+2s}} \\
&=
-\frac{2^{2s}\Gamma(\frac12+s)}{\sqrt\pi\,\Gamma(-s)}
\int_\R\int_0^{+\infty}
\frac{2-e^{2\pi iy\xi}-e^{-2\pi iy\xi}}{y^{1+2s}}\;dy\;\widehat{u}(\xi)e^{2\pi ix\xi}\;d\xi \\
&=
-\frac{2^{2s+1}\Gamma(\frac12+s)}{\sqrt\pi\,\Gamma(-s)}
\int_\R\int_0^{+\infty}
\frac{1-\cos(2\pi\xi y)}{y^{1+2s}}\;dy\;\widehat{u}(\xi)e^{2\pi ix\xi}\;d\xi.
\end{align*}
With the change of variables~$z:=2\pi|\xi|y$, we obtain that
\[
\int_0^{+\infty}
\frac{1-\cos(2\pi\xi y)}{y^{1+2s}}\;dy
=
(2\pi|\xi|)^{2s}\int_0^{+\infty}
\frac{1-\cos\big(\frac{\xi}{|\xi|} z\big)}{z^{1+2s}}\;dz
\]
and therefore, using the fact that~$\cos$ is an even function,
\[
\int_0^{+\infty}
\frac{1-\cos(2\pi\xi y)}{y^{1+2s}}\;dy
=
(2\pi|\xi|)^{2s}\int_0^{+\infty}
\frac{1-\cos y}{y^{1+2s}}\;dy.
\]

Moreover, it holds from Lemma~\ref{lem:coseno} that 
\[
\int_0^{+\infty}
\frac{1-\cos y}{y^{1+2s}}\;dy=\frac{\Gamma(s)\Gamma(1-s)}{2\Gamma(1+2s)}
\]
from which we deduce that
\[
{(-\Delta)}^su(x)=
-\frac{2^{2s}\Gamma(\frac12+s)}{\sqrt\pi\,\Gamma(-s)}\frac{\Gamma(s)\Gamma(1-s)}{\Gamma(1+2s)}
\int_\R(2\pi|\xi|)^{2s}\widehat{u}(\xi)e^{2\pi ix\xi}\;d\xi.
\]

Now, using the properties of $\Gamma$, in particular~\eqref{gamma-recursive}
and~\eqref{gamma-dupli}, we remark that 
\[
-\frac{2^{2s}\Gamma(\frac12+s)}{\sqrt\pi\,\Gamma(-s)}\frac{\Gamma(s)\Gamma(1-s)}{\Gamma(1+2s)}
=\frac{2^{2s}\Gamma(\frac12+s)}{\sqrt\pi}\frac{\Gamma(s)s}{\Gamma(1+2s)}
=\frac{2^{2s-1}\Gamma(\frac12+s)}{\sqrt\pi}\frac{\Gamma(s)}{\Gamma(2s)}
=1
\]
and this concludes the proof.
\end{proof}

\section[The fractional Laplacian as a Bochner integral]{The fractional Laplacian as a Bochner integral: proof of~\eqref{HESE}}\label{fractsection:4}

As claimed in~\eqref{HESE},
the fractional Laplacian can also be introduced as 
\begin{align}\label{bochner-def}
(-\Delta)^s u(x)=\frac{s}{\Gamma(1-s)}\int_0^{+\infty}\frac{u(x)-[e^{t\Delta}u](x)}{t^{1+s}}\;dt
\end{align}
where~$(e^{t\Delta})_{t>0}$ stands for the heat semigroup, i.e.,
\begin{align}\label{bochner-def00}
e^{t\Delta}u(x)=\frac1{{(4\pi t)}^{n/2}}\int_{\R^n}e^{-|x-y|^2/(4t)}u(y)\;dy.
\end{align}
We refer the reader to~\cite{MR3916700, MR3965397} for a thorough introduction
to the semigroup approach to the fractional Laplacian.

The factor~${(4\pi t)}^{-n/2}$ in~\eqref{bochner-def00}
plays the role of a normalizing constant,
according to the following observation:

\begin{lemma}\label{lemmazero0000}
We have that
$$  \int_{\R^n}e^{-|x-y|^2/(4t)}\;dy={(4\pi t)}^{n/2}.$$
\end{lemma}

\begin{proof}
The change of variable~$z:=(y-x)/\sqrt{4t}$ gives that
\begin{align*}
\int_{\R^n}e^{-|x-y|^2/(4t)}\;dy={(4t)}^{n/2}\int_{\R^n}e^{-|z|^2}\;dz.
\end{align*}
Thus, passing to polar coordinates,
\begin{align*}
\int_{\R^n}e^{-|x-y|^2/(4t)}\;dy &=
{(4t)}^{n/2}\big|\mathbb{S}^{n-1}\big|\int_{\R^n}e^{-r^2}r^{n-1}\;dz \\
&=
{(4t)}^{n/2}\frac{\big|\mathbb{S}^{n-1}\big|}2\int_{\R^n}e^{-\rho}\rho^{n/2-1}\;dz
\\&=
{(4t)}^{n/2}\frac{|\mathbb{S}^{n-1}|}2\Gamma\left(\frac{n}2\right)\\&=
{(4\pi t)}^{n/2},
\end{align*}
see~\eqref{gamma-def} for the definition of~$\Gamma$ and
\eqref{measure-n-sphere} for further details on the measure of the sphere.
\end{proof}

Now, to establish~\eqref{HESE}, we provide the following result:

\begin{lemma}
Let~$s\in(0,1)$, $r>0$, $\sigma>s$, $x\in\R^n$ and~$u\in C^{2\sigma}(B_r(x))\cap L^1_s(\R^n)$.

Then, the definitions in~\eqref{symmetric-def} and~\eqref{bochner-def}
coincide. 
\end{lemma}

\begin{proof}
Using Lemma~\ref{lemmazero0000} we write 
\begin{align*}&
\int_0^{+\infty}\frac{u(x)-[e^{t\Delta}u](x)}{t^{1+s}}\;dt\\
 &\qquad=
\int_0^{+\infty}\left(u(x)-\frac1{{(4\pi t)}^{n/2}}\int_{\R^n}e^{-|y|^2/(4t)}u(x+y)\;dy\right)\;\frac{dt}{t^{1+s}} \\
 &\qquad=
\frac12\int_0^{+\infty}\left(2u(x)-\frac1{{(4\pi t)}^{n/2}}\int_{\R^n}e^{-|y|^2/(4t)}\big(u(x+y)+u(x-y)\big)\;dy\right)\;\frac{dt}{t^{1+s}} \\
 &\qquad=
\frac12\int_0^{+\infty}\frac1{{(4\pi t)}^{n/2}}\int_{\R^n}e^{-|y|^2/(4t)}\big(2u(x)-u(x+y)-u(x-y)\big)\;dy\;\frac{dt}{t^{1+s}} \\
 &\qquad=
\frac1{2^{n+1}\pi^{n/2}}\int_{\R^n}\left(
\big(2u(x)-u(x+y)-u(x-y)\big)\int_0^{+\infty}e^{-|y|^2/(4t)}\;\frac{dt}{t^{n/2+1+s}}\right)\;dy.
\end{align*}

We now compute the integral in the~$t$ variable. With the change of variable~$\tau:=|y|^2/(4t)$ we have
\begin{align*}
\int_0^{+\infty}e^{-|y|^2/(4t)}\;\frac{dt}{t^{n/2+1+s}}
=4^{n/2+s}|y|^{-n-2s}\int_0^{+\infty}e^{-\tau}\tau^{n/2-1+s}\;d\tau
=2^{n+2s}|y|^{-n-2s}\,\Gamma\Big(\frac{n+2s}2\Big)
\end{align*}
where we have used the definition of the Gamma function~$\Gamma$, see~\eqref{gamma-def}.

With this, and recalling the definition of~$c_{n,s}$ in~\eqref{cns},
we obtain that 
\begin{align*}
\int_0^{+\infty}\frac{u(x)-[e^{t\Delta}u](x)}{t^{1+s}}\;dt
&=\frac{2^{2s}\,\Gamma\big(\frac{n+2s}2\big)}{\pi^{n/2}}\frac12
\int_{\R^n}\frac{2u(x)-u(x+y)-u(x-y)}{{|y|}^{n+2s}}\;dy\\&
=-\frac{c_{n,s}\Gamma(-s)}2
\int_{\R^n}\frac{2u(x)-u(x+y)-u(x-y)}{{|y|}^{n+2s}}\;dy\\&=
\frac{c_{n,s}\Gamma(1-s)}{2s}
\int_{\R^n}\frac{2u(x)-u(x+y)-u(x-y)}{{|y|}^{n+2s}}\;dy,
\end{align*}
where the property of the Gamma function in~\eqref{gamma-recursive}
has been used in the last line.

This shows that
the definitions in~\eqref{symmetric-def} and~\eqref{bochner-def} coincide,
as desired.
\end{proof}

\section{Nonlocal gradients and divergences}\label{sec:nonlocal-grad}

An interesting property of the classical Laplacian is that it can be written as the divergence of the gradient.
To recover this feature in the nonlocal setting, one needs to introduce a suitable notion of fractional divergence and gradient.

\begin{definition}
Let~$s\in(0,1)$, $\sigma>s$, $r>0$, $x\in\R^n$
and~$u\in C^\sigma(B_r(x))\cap L^1_{s/2}(\R^n)$, where the notation in~\eqref{w-l1-space} has been used.

The nonlocal gradient of~$u$ at~$x$ is defined as
\begin{align}\label{NONGRAD}
\begin{split}
\nabla^su(x)&:=d_{n,s}\int_{\R^n}\frac{u(y)-u(x)}{{|y-x|}^{n+s}}\,\frac{y-x}{|y-x|}\;dy \\
&=d_{n,s}\,\pv\int_{\R^n}\frac{u(y)}{{|y-x|}^{n+s}}\,\frac{y-x}{|y-x|}\;dy 
\end{split}
\end{align}
where
\begin{align}\label{dns}
d_{n,s}:=
\frac{2^s\Gamma(\frac{n+s+1}2)}{\pi^{n/2}\Gamma(\frac{1-s}2)}
=\frac1{1+s}\,c_{n,\frac{1+s}2}.
\end{align}
\end{definition}

\begin{definition}
Let~$s\in(0,1)$, $\sigma>s$, $r>0$, $x\in\R^n$ and~$F\in C^\sigma(B_r(x),\R^n)\cap L^1_{s/2}(\R^n)$. 

The nonlocal divergence of~$F$ at~$x$ is defined as
\begin{align*}
\div^sF(x)&:=d_{n,s}\int_{\R^n}\frac{F(y)-F(x)}{{|y-x|}^{n+s}}\cdot
\frac{y-x}{|y-x|}\;dy \\
&=d_{n,s}\,\pv\int_{\R^n}\frac{F(y)}{{|y-x|}^{n+s}}\cdot\frac{y-x}{|y-x|}\;dy 
\end{align*}
where~$d_{n,s}$ is defined as in~\eqref{dns}.
\end{definition}

According to~\cite[Section~1]{silhavi},
the first appearance of the nonlocal gradient dates back to the works~\cite{MR107788, MR500133}.
Nonlocal gradient and divergence have been studied in the recent literature
(see e.g.~\cite{MR3420498, shieh, MR3615452, MR3714833, shiehtwo, comi, piola, mora}), though a complete understanding of their
rather complex behavior is still under development. 

The constant~$d_{n,s}$ appearing in~\eqref{NONGRAD} is such that the fractional gradient~$\nabla^su$ coincides with the classical gradient~$\nabla u$
when~$s=1$, see e.g.~\cite[Proposition~2.6]{depas}. 
We refer the reader to~\cite[Section~2.1]{depas} for a thorough introduction
to nonlocal gradients.
\medskip

In what follows, we will denote by~$\cF u$ the Fourier Transform of~$u$.

\begin{lemma}\label{lem:fourier-symbol-of-nonlocal-grad}
If~$u$ belongs to the Schwartz space of smooth and rapidly decreasing functions, then, for any~$\xi\in\R^n$,
\begin{align}\label{rmk:fourier-symbol-of-nonlocal-div00}
\cF\big[\nabla^s u\big](\xi)=i{(2\pi|\xi|)}^{s-1}2\pi\xi\,\cF{u}(\xi).
\end{align}

Furthermore, if~$F=(F_1,\dots, F_n)$ and~$F_j$ belongs to the Schwartz space of smooth and rapidly decreasing functions for all~$j\in\{1,\dots,n\}$,
then, for any~$\xi\in\R^n$,
\begin{align}\label{rmk:fourier-symbol-of-nonlocal-div}
\cF[\div^s F](\xi)=i({(2\pi|\xi|)}^{s-1}\,2\pi\xi\cdot \cF F(\xi).
\end{align}
\end{lemma}

\begin{proof}
We will show~\eqref{rmk:fourier-symbol-of-nonlocal-div00}, as the proof
of~\eqref{rmk:fourier-symbol-of-nonlocal-div} is similar and therefore is omitted.

We rewrite
\[
\nabla^su(x)=d_{n,s}\int_{\R^n}\frac{u(x+z)-u(x)}{{|z|}^{n+s}}\,\frac{z}{|z|}\;dz.
\]
In this way, we see that
\begin{equation}\label{star647432bstarr43thgfbs067}\begin{split}
\cF\left[\int_{\R^n}\frac{u(\cdot+z)-u}{{|z|}^{n+s}}\,\frac{z}{|z|}\;dz\right](\xi)
&=
\int_{\R^n}\frac{e^{2\pi iz\cdot\xi}\cF u(\xi)-\cF u(\xi)}{{|z|}^{n+s}}\,\frac{z}{|z|}\;dz\\&=
\cF u(\xi)\int_{\R^n}\frac{e^{2\pi iz\cdot\xi}-1}{{|z|}^{n+s}}\,\frac{z}{|z|}\;dz \\
&={(2\pi|\xi|)}^s\cF u(\xi)\int_{\R^n}\frac{e^{iz\cdot\frac{\xi}{|\xi|}}-1}{{|z|}^{n+s}}\,\frac{z}{|z|}\;dz\\&=
i{(2\pi|\xi|)}^s\cF u(\xi)\int_{\R^n}\frac{\sin\left(z\cdot\frac{\xi}{|\xi|}\right)}{{|z|}^{n+s}}\,\frac{z}{|z|}\;dz.
\end{split}\end{equation}
Without loss of generality, by rotational symmetry we can suppose that~$\xi=|\xi|e_1$. Accordingly,
\begin{align*}
\int_{\R^n}\frac{\sin\left(z\cdot\frac{\xi}{|\xi|}\right)}{{|z|}^{n+s}}\,\frac{z_j}{|z|}\;dz=\int_{\R^n}\frac{\sin(z_1)}{{|z|}^{n+s}}\,\frac{z_j}{|z|}\;dz
=0
\qquad\text{for any }j\in\{2,\ldots,n\}
\end{align*}
and
\begin{equation}\label{03ur2j09ghg83w4hgw389hgoi2904gh2309hgfivowhgeuogb}
\begin{split}
\int_{\R^n}\frac{\sin\left(z\cdot\frac{\xi}{|\xi|}\right)}{{|z|}^{n+s}}\,\frac{z_1}{|z|}\;dz &=
\int_{\R^n}\frac{z_1\,\sin z_1}{{|z|}^{n+s+1}}\;dz\\&=
\int_\R \left(z_1\,\sin z_1\int_{\R^{n-1}}\frac{dz'}{\big(z_1^2+|z'|^2\big)^{(n+s+1)/2}}\right)\;dz_1 \\
&=
\int_\R\left(\frac{z_1\,\sin z_1 }{{|z_1|}^{2+s}}\int_{\R^{n-1}}\frac{dz'}{\big(1+{|z'|}^2\big)^{(n+s+1)/2}}\right)\;dz_1 \\
&= 2\big|\mathbb{S}^{n-2}\big|
\int_0^{+\infty}\frac{\sin z_1 }{z_1^{1+s}}\;dz_1 
\int_0^{+\infty}\frac{t^{n-2}}{\big(1+t^2\big)^{(n+s+1)/2}}\;dt.
\end{split}
\end{equation}

Now, if~$n\ge2$, we can perform
the change of variable~$\tau:=1/(1+t^2)$ and,
using~\eqref{beta-identity} with~$a:=1+s/2$ and~$b:=(n-1)/2$,
we obtain that
\begin{align}\label{932urpoj2390rfu23j90f23jwighjiog8ho8g2}
\int_0^{+\infty}\frac{t^{n-2}}{\big(1+t^2\big)^{(n+s+1)/2}}\;dt
=\frac12\int_0^1{(1-\tau)}^{\frac{n-3}2}\tau^{\frac{s}2}\;d\tau
=\frac{\Gamma(\frac{n-1}2)\,\Gamma(1+\frac{s}2)}{2\Gamma(\frac{n+s+1}2)}
.\end{align}

Given~$a>0$, integrating by parts twice, we find that
\begin{align*}
\int_0^{+\infty}e^{-a\rho}\sin\rho\;d\rho
=
1-a\int_0^{+\infty}e^{-a\rho}\cos\rho\;d\rho
=
1-a^2\int_0^{+\infty}e^{-a\rho}\sin\rho\;d\rho
\end{align*}
and therefore
\begin{align*}
\int_0^{+\infty}e^{-a\rho}\sin\rho\;d\rho=\frac1{1+a^2}.
\end{align*}
That is, for all~$\tau>0$,
\begin{align}\label{laplasine}
\int_0^{+\infty}e^{-\frac\rho\tau}\sin\rho\;d\rho
=
\frac{\tau^2}{1+\tau^2}.
\end{align}

We also claim that
\begin{align}\label{9u23rijfp1jfc1i230dcjkwpo0}
\int_0^{+\infty}\frac{\sin z_1}{z_1^{1+s}}\;dz_1
=\frac{\sqrt\pi}{2^{1+s}}\frac{\Gamma(\frac{1-s}2)}{\Gamma(1+\frac{s}2)}.
\end{align}
To prove this, we argue as follows:
by the definition of~$\Gamma$ in~\eqref{gamma-def}, we see that
\begin{align*}
\Gamma(1+s)\int_0^{+\infty}\frac{\sin z_1}{z_1^{1+s}}\;dz_1
=\int_0^{+\infty}\int_0^{+\infty}e^{-r}r^s\,\frac{\sin z_1}{z_1^{1+s}}\;dz_1\;dr.
\end{align*}
We use the changes of variable~$z_1:=r\tau$ and~$\rho:=r\tau$ and we obtain that
\begin{align*}
\int_0^{+\infty}\int_0^{+\infty}e^{-r}r^s\,\frac{\sin z_1}{z_1^{1+s}}\;dz_1\;dr
&=
\int_0^{+\infty}\int_0^{+\infty}e^{-r}\,\frac{\sin(r\tau)}{\tau^{1+s}}\;d\tau\;dr\\&=
\int_0^{+\infty}\int_0^{+\infty}e^{-\frac\rho\tau}\,\frac{\sin\rho}{\tau^{2+s}}\;d\tau\;d\rho.
\end{align*}
Recalling~\eqref{laplasine}, we have that
\begin{align*}
\int_0^{+\infty}\int_0^{+\infty}e^{-\frac\rho\tau}\,\frac{\sin\rho}{\tau^{2+s}}\;d\tau\;d\rho
=
\int_0^{+\infty}\frac{d\tau}{\tau^{s}(1+\tau^2)}.
\end{align*}
Moreover, the change of variable~$\eta:=1/(1+\tau^2)$ gives
\begin{eqnarray*}&&
\int_0^{+\infty}\frac{d\tau}{\tau^{s}(1+\tau^2)}
=
\frac12\int_0^1\eta^{\frac{s-1}2}(1-\eta)^{-\frac{1+s}{2}}\;d\eta
=
\frac{\Gamma(\frac{1+s}2)\,\Gamma(\frac{1-s}2)}{2}\\&&\qquad
=\frac{\sqrt\pi\, \Gamma(s)\Gamma(\frac{1-s}2)}{2^s\Gamma(\frac{s}2)}
= \frac{\sqrt\pi \,\Gamma(1+s)\Gamma(\frac{1-s}2)}{2^{1+s}\Gamma(\frac{1+s}2)}\end{eqnarray*}
due to~\eqref{gamma-recursive}, \eqref{beta-identity} with~$a:=(1+s)/2$ and~$b:=(1-s)/2$
and~\eqref{gamma-dupli} with~$a:=s/2$.

Gathering together these considerations proves~\eqref{9u23rijfp1jfc1i230dcjkwpo0}, as desired.

Plugging~\eqref{932urpoj2390rfu23j90f23jwighjiog8ho8g2} and~\eqref{9u23rijfp1jfc1i230dcjkwpo0}
into~\eqref{03ur2j09ghg83w4hgw389hgoi2904gh2309hgfivowhgeuogb}, we find that
$$
\int_{\R^n}\frac{\sin\left(z\cdot\frac{\xi}{|\xi|}\right)}{{|z|}^{n+s}}\,\frac{z_1}{|z|}\;dz
=\frac{\pi^{n/2}}{2^{s}}\,
\frac{\Gamma(\frac{1-s}2)}{\Gamma(\frac{n+s+1}2)}.
$$
Using this information into~\eqref{star647432bstarr43thgfbs067}
and recalling the definition of~$d_{n,s}$ in~\eqref{dns}
we complete the proof of~\eqref{rmk:fourier-symbol-of-nonlocal-div00}.
\end{proof}

With this preliminary work, in the next subsections we will focus on the proof of~\eqref{RADCVE}, \eqref{RADCVE2} and~\eqref{RADCVE3}.

\subsection[The fractional Laplacian as the nonlocal divergence of the nonlocal gradient]{The fractional Laplacian as the nonlocal divergence of the nonlocal gradient: proof of~\eqref{RADCVE}}

Here we establish~\eqref{RADCVE},
namely we show the following:

\begin{lemma}\label{divgrad-def2-lemma}
Let~$u$ be in the Schwartz space of rapidly decreasing functions. Then,
\begin{align}\label{divgrad-def}
(-\Delta)^s u(x)=-\div^s\big(\nabla^s u(x)\big).
\end{align}
\end{lemma}

\begin{proof}
Identity~\eqref{divgrad-def} follows by computing the Fourier Transform on both sides.
Indeed, exploiting Lemma~\ref{lem:fourier-symbol-of-nonlocal-grad},
we have that
\begin{align*}
\cF\big[\div^s(\nabla^s u)\big](\xi)
&=i{(2\pi|\xi|)}^{s-1}2\pi\xi\cdot\cF[\nabla^s u](\xi)\\ 
&=i{(2\pi|\xi|)}^{s-1}2\pi\xi\cdot\big(i{(2\pi|\xi|)}^{s-1}2\pi\xi\cF u(\xi)\big)
\\&=-{(2\pi|\xi|)}^{2s}\cF u(\xi) \\
&=-\cF\big[(-\Delta)^s u\big](\xi),
\end{align*} as desired.
\end{proof}

\subsection[The fractional Laplacian as the nonlocal divergence of the local gradient]{The fractional Laplacian as the nonlocal divergence of the local gradient: proof of~\eqref{RADCVE2}}

When~$s\in(1/2,1)$, the fractional Laplacian can be also represented
as in~\eqref{RADCVE2}. To prove this claim, one can use a Fourier Transform
approach, as in Lemma~\ref{divgrad-def2-lemma},
but we follow here a different idea which relies on an integration by parts.

Indeed, note that
\begin{align}\label{cwcjeklwecmlkwecmlwkecmlwenmc}
\div\big(|x|^{-n-2s}x\big)=-(n+2s){|x|}^{-n-2s}+n{|x|}^{-n-2s}=-2s{|x|}^{-n-2s}
\qquad \text{for }x\in\R^n\setminus\{0\}.
\end{align}
With this information, we will be able to perform the desired integration by part
trick and show~\eqref{RADCVE2}, as follows.

\begin{lemma}
Let~$s\in(1/2,1)$, $\sigma>2s$, $r>0$, $x\in\R^n$  and~$u\in C^\sigma(B_r(x))\cap C^1(\R^n)\cap L^1_s(\R^n)$ be
such that 
\[
\int_{\R^n}\frac{|\nabla u(y)|}{(1+|y|)^{n+2s-1}}\;dy<+\infty.
\]
Then, \eqref{RADCVE2} holds true.
\end{lemma}

\begin{proof}
By~\eqref{cwcjeklwecmlkwecmlwkecmlwenmc} and the Divergence Theorem,
\begin{align*}
\pv\int_{\R^n}\frac{u(x)-u(y)}{{|x-y|}^{n+2s}}\;dy
&=-\frac1{2s}\,\pv\int_{\R^n}\big(u(x)-u(y)\big)\div_y\big({|x-y|}^{-n-2s}(x-y)\big)\;dy \\
&=\frac1{2s}\,\pv\int_{\R^n}
\nabla_y\big(u(x)-u(y)\big)\cdot{|x-y|}^{-n-2s}(x-y) \;dy \\
&=-\frac1{2s}\,\pv\int_{\R^n}\frac{\nabla u(y)}{{|x-y|}^{n+2s-1}}\cdot
\frac{x-y}{|x-y|}\;dy\\&=
\frac1{2s d_{n,2s-1}}\div^{2s-1}(\nabla u)(x).
\end{align*}
The desired claim then follows from~\eqref{dns}.
\end{proof}

\subsection[The fractional Laplacian as the local divergence of the nonlocal gradient]{The fractional Laplacian as the local divergence of the nonlocal gradient: proof of~\eqref{RADCVE3}}

When~$s\in(1/2,1)$, one can also express the fractional
Laplacian as in~\eqref{RADCVE3}. To check this, one can rely on the Fourier
Transform approach as in Lemma~\ref{divgrad-def2-lemma}. Alternatively,
one can proceed in a more direct way, as follows.

\begin{lemma}
Let~$s\in(1/2,1)$, $\sigma>2s$, $r>0$, $x\in\R^n$  and~$u\in C^\sigma(B_r(x))\cap C^1(\R^n)\cap L^1_s(\R^n)$ be
such that 
\[
\int_{\R^n}\frac{|u(y)|+|\nabla u(y)|}{(1+|y|)^{n+2s-1}}\;dy<+\infty.
\]
Then, \eqref{RADCVE3} holds true.
\end{lemma}

\begin{proof}
First let us notice that
\begin{align*}
\int_{\R^n}\frac{u(x)-u(y)}{{|x-y|}^{n+2s-1}}\frac{x-y}{|x-y|}\;dy
=
\int_{\R^n}\frac{u(x)-u(x-z)}{{|z|}^{n+2s-1}}\frac{z}{|z|}\;dz 
=
-\int_{\R^n}\frac{u(x)-u(x+z)}{{|z|}^{n+2s-1}}\frac{z}{|z|}\;dz
\end{align*}
and therefore
\begin{align*}
\int_{\R^n}\frac{u(x)-u(y)}{{|x-y|}^{n+2s-1}}\frac{x-y}{|x-y|}\;dy
&=
\frac12\int_{\R^n}\frac{u(x)-u(x-z)}{{|z|}^{n+2s-1}}\frac{z}{|z|}\;dz
-\frac12\int_{\R^n}\frac{u(x)-u(x+z)}{{|z|}^{n+2s-1}}\frac{z}{|z|}\;dz \\
&=
\frac12\int_{\R^n}\frac{u(x+z)-u(x-z)}{{|z|}^{n+2s-1}}\frac{z}{|z|}\;dz.
\end{align*}
Then, by using an integration by parts,
\begin{align*}
\div_x\left(\int_{\R^n}\frac{u(x)-u(y)}{{|x-y|}^{n+2s-1}}\frac{x-y}{|x-y|}\;dy\right)
&=
\frac12\div_x\left(\int_{\R^n}\frac{u(x+z)-u(x-z)}{{|z|}^{n+2s-1}}\frac{z}{|z|}\;dz\right) \\
&=
\frac12\int_{\R^n}\frac{\nabla_x\big(u(x+z)-u(x-z)\big)}{{|z|}^{n+2s-1}}\cdot\frac{z}{|z|}\;dz \\
&=
\frac12\int_{\R^n}\frac{\nabla_z\big(u(x+z)+u(x-z)\big)}{{|z|}^{n+2s-1}}\cdot\frac{z}{|z|}\;dz \\
&=
\frac12\int_{\R^n}\nabla_z\big(u(x+z)+u(x-z)-2u(x)\big)\cdot\frac{z}{{|z|}^{n+2s}}\;dz \\
&=
-\frac12\int_{\R^n}\big(u(x+z)+u(x-z)-2u(x)\big)\div_z\left(\frac{z}{{|z|}^{n+2s}}\right)\;dz .
\end{align*}
Thus, taking advantage of~\eqref{cwcjeklwecmlkwecmlwkecmlwenmc}, we finally get
\begin{align*}
\div_x\left(\int_{\R^n}\frac{u(x)-u(y)}{{|x-y|}^{n+2s-1}}\frac{x-y}{|x-y|}\;dy\right)
&=
s\int_{\R^n}\frac{u(x+z)+u(x-z)-2u(x)}{{|z|}^{n+2s}}\;dz \\
&=
-s\int_{\R^n}\frac{2u(x)-u(x+z)-u(x-z)}{{|z|}^{n+2s}}\;dz.
\end{align*}
In conclusion, \eqref{RADCVE3} follows by using~\eqref{symmetric-def} and~\eqref{dns}.
\end{proof}

\section{General properties of the fractional Laplacian}

We have thus far dealt with the very definition of the fractional Laplacian.
A number of nice properties of the operator can be directly derived from 
its representation in~\eqref{pv-def0} or the other equivalent ones given in Theorem~\ref{CARATT}.
We list here below some of these.

\subsection{Translation and rotation invariances, homogeneity}

We recall that the space~$L^1_s(\R^n)$ has been introduced in~\eqref{w-l1-space}.

The first one that we present here is the translation invariance. 

\begin{lemma}\label{lem:translation}
Let~$r>0$,~$\sigma>s$,~$x\in\R^n$ and~$u\in C^{2\sigma}(B_r(x))\cap L^1_s(\R^n)$.
Let~$h\in\R^n$  
and define the translated function 
\begin{align*}
u_h(y):=u(y+h)\qquad\text{for any }y\in\R^n.
\end{align*}
Then,
$$
{(-\Delta)}^s u_h(x)={(-\Delta)}^s u(x+h).
$$
\end{lemma}

\begin{proof}
The claim directly follows from a translation ($z:=y+h$) in the integral defining the fractional Laplacian:
\begin{align*}
{(-\Delta)}^s u_h(x) &=
c_{n,s}\pv\int_{\R^n}\frac{u_h(x)-u_h(y)}{{|x-y|}^{n+2s}}\;dy \\
&= 
c_{n,s}\pv\int_{\R^n}\frac{u(x+h)-u(y+h)}{{|x-y|}^{n+2s}}\;dy \\
&= 
c_{n,s}\pv\int_{\R^n}\frac{u(x+h)-u(z)}{{|x+h-z|}^{n+2s}}\;dy \\&
=
{(-\Delta)}^s u(x+h),
\end{align*}as desired.
\end{proof}

The second property that we mention is the rotation invariance.

\begin{lemma}\label{lem:rotation}
Let~$r>0$,~$\sigma>s$,~$x\in\R^n$ and~$u\in C^{2\sigma}(B_r(x))\cap L^1_s(\R^n)$. Let~$R\in\R^{n\times n}$ be an orthogonal matrix
and define the rotated function 
\begin{align*}
u_R(y):=u(Ry)\qquad \text{for any }y\in\R^n.
\end{align*}
Then,
$$
{(-\Delta)}^s u_R(x)={(-\Delta)}^s u(Rx).
$$
\end{lemma}

\begin{proof}
The claim directly follows from a rotation ($z:=Ry$) in the integral defining the fractional Laplacian:
\begin{align*}
{(-\Delta)}^s u_R(x) &=
c_{n,s}\pv\int_{\R^n}\frac{u_R(x)-u_R(y)}{{|x-y|}^{n+2s}}\;dy \\
&=
c_{n,s}\pv\int_{\R^n}\frac{u(Rx)-u(Ry)}{{|Rx-Ry|}^{n+2s}}\;dy \\
&= 
c_{n,s}\pv\int_{\R^n}\frac{u(Rx)-u(z)}{{|Rx-z|}^{n+2s}}\;dz \\&
=
{(-\Delta)}^s u(Rx),
\end{align*}as desired.
\end{proof}

We also remark that the fractional Laplace operator is homogeneous, according to the following result:

\begin{lemma}\label{lem:homogeneity}
Let~$r>0$,~$\sigma>s$,~$x\in\R^n$ and~$u\in C^{2\sigma}(B_r(x))\cap L^1_s(\R^n)$.
Given~$\lambda>0$, define the dilated function 
\begin{align*}
u_\lambda(y):=u(\lambda y)\qquad \text{for any }y\in\R^n.
\end{align*}
Then,
$$
{(-\Delta)}^s u_\lambda(x) = \lambda^{2s} {(-\Delta)}^s u(\lambda x).
$$
\end{lemma}

\begin{proof}
The claim directly follows from a dilation ($z:=\lambda y$) in the integral defining the fractional Laplacian:
\begin{align*}
{(-\Delta)}^s u_\lambda(x) &=
c_{n,s}\pv\int_{\R^n}\frac{u_\lambda(x)-u_\lambda(y)}{{|x-y|}^{n+2s}}\;dy \\
&=
\lambda^{n+2s} c_{n,s}\pv\int_{\R^n}\frac{u(\lambda x)-u(\lambda y)}{{|\lambda x-\lambda y|}^{n+2s}}\;dy \\
&=
\lambda^{2s} c_{n,s}\pv\int_{\R^n}\frac{u(\lambda x)-u(z)}{{|\lambda x-z|}^{n+2s}}\;dz \\&
=
\lambda^{2s} {(-\Delta)}^s u(\lambda x),
\end{align*}as desired.
\end{proof}

\subsection{Asymptotic behaviour as~\texorpdfstring{$s\searrow0$}{s to 0} and~\texorpdfstring{$s\nearrow1$}{s to 1}}

In this section we discuss the asymptotic behaviour of the fractional Laplacian
as~$s\searrow0$ and~$s\nearrow1$.
The Fourier Transform representation in~\eqref{FORi} provides a hint
on what these limits could give and here we aim at making this intuition rigorous.

For this, we remark that the normalizing constant~$c_{n,s}$ 
defined in~\eqref{cns} can be written in an equivalent, but perhaps more transparent way by making use of the properties of the~$\Gamma$ function.
In particular, formulas~\eqref{gamma-recursive} and~\eqref{gamma-recursive2}
give that
\begin{align*}
\Gamma(-s)=\frac{\Gamma(1-s)}{-s}=\frac{\Gamma(2-s)}{-s(1-s)},
\end{align*}
and therefore we can write 
\begin{align}\label{cns2}
c_{n,s}=\frac{2^{2s}\Gamma(\frac{n+2s}2)}{\pi^{n/2}\Gamma(2-s)}\,s(1-s).
\end{align}
Representation~\eqref{cns2} makes it clear that
\begin{align}\label{cns-asymp}
\lim_{s\searrow 0} c_{n,s} =0
\qquad\text{and}\qquad
\lim_{s\nearrow 1} c_{n,s} =0,
\end{align}
and, more precisely,
\begin{align}\label{cns-asymp-precise}
\lim_{s\searrow 0} \frac{c_{n,s}}{s} =\frac{\Gamma(\frac{n}2)}{\pi^{n/2}}=\frac{2}{\big|\mathbb{S}^{n-1}\big|}
\qquad\text{and}\qquad
\lim_{s\nearrow 1} \frac{c_{n,s}}{1-s} =\frac{2n\Gamma(\frac{n}2)}{\pi^{n/2}}=\frac{4n}{\big|\mathbb{S}^{n-1}\big|},
\end{align}
see~\eqref{measure-n-sphere}.

These limit properties play a pivotal role in showing that
the family of operators~${\big\{(-\Delta)^s\big\}}_{s\in(0,1)}$ interpolates 
the identity and the Laplacian, according to the following statement:

\begin{lemma}
Let~$r$, $\sigma>0$ and~$x\in\R^n$. 

For every~$u\in C^\sigma(B_r(x))$ with
\begin{align}\label{weighto}
\int_{\R^n}\frac{|u(y)|}{{(1+|y|)}^n}\;dy<+\infty
\end{align}
it holds that
\begin{align}\label{ds-sto0} 
\lim_{s\searrow 0}(-\Delta)^s u(x) = u(x).
\end{align} 

Moreover, for every~$u\in C^{2+\sigma}(B_r(x))$ with
\begin{align}\label{weightoo}
\int_{\R^n}\frac{|u(y)|}{{(1+|y|)}^{n+2-\sigma}}\;dy<+\infty
\end{align}
it holds that
\begin{align}\label{ds-sto1}
\lim_{s\nearrow 1}(-\Delta)^s u(x) = -\Delta u(x).
\end{align}
\end{lemma}

\begin{proof}
We start by proving~\eqref{ds-sto0}. We have that
\begin{align*}
(-\Delta)^s u(x)=c_{n,s}\int_{\R^n}\frac{u(x)-u(y)}{{|x-y|}^{n+2s}}\;dy.
\end{align*}
We point out that we do not need to use the principal value notation
because we can assume~$s\in(0,1/2)$.

Using that~$u\in C^\sigma(B_r(x))$ we obtain that
\begin{align*}
|u(x)-u(y)|\leq\|u\|_{C^\sigma(B_r(x))}|x-y|^\alpha
\qquad\text{for every }y\in B_r(x)
\end{align*}
and therefore
\begin{align*}
&\int_{B_r(x)}\frac{|u(x)-u(y)|}{{|x-y|}^{n+2s}}\;dy \leq 
\|u\|_{C^\sigma(B_r(x))}\int_{B_r(x)}{|x-y|}^{\sigma-n-2s}\;dy \\&\qquad
=
\|u\|_{C^\sigma(B_r(x))}\big|\mathbb{S}^{n-1}\big|\int_0^r{t}^{\sigma-1-2s}\;dy=
\|u\|_{C^\sigma(B_r(x))}\big|\mathbb{S}^{n-1}\big|\,\frac{r^{\sigma-2s}}{\sigma-2s},
\end{align*}
where we have used that, as we are interested in~$s\searrow 0$, 
we can suppose with no harm that~$\sigma-2s>0$. 

This and~\eqref{cns-asymp} then give that
\begin{align}\label{lllll1}
\lim_{s\searrow 0}(-\Delta)^s u(x) = 
\lim_{s\searrow 0} c_{n,s}\int_{\R^n\setminus B_r(x)}\frac{u(x)-u(y)}{{|x-y|}^{n+2s}}\;dy.
\end{align}

Now, we point out that there exists~$C>0$, possibly depending on~$x$, such that, for all~$y\in\R^n\setminus B_r(x)$,
\begin{equation}\label{fghebdnfhbedwjny4u3i854678ytigefsdjhj54y45reTG}
|x-y|\ge C(1+|y|).
\end{equation}
To check this, we observe that
$$ \lim_{|y|\to+\infty}\frac{|x-y|}{1+|y|}=1$$
and thus there exists~$R>1$, possibly depending on~$x$, such that~$|x-y|\ge (1+|y|)/2$ for all~$y\in\R^n\setminus B_R$.

Also, if~$y\in B_R\setminus B_r(x)$, then
$$ \frac{|x-y|}{1+|y|}\ge \frac{r}{1+R}.$$
Gathering these observations, we obtain~\eqref{fghebdnfhbedwjny4u3i854678ytigefsdjhj54y45reTG}.

As a consequence of~\eqref{fghebdnfhbedwjny4u3i854678ytigefsdjhj54y45reTG},
we see that
\begin{align*}
\int_{\R^n\setminus B_r(x)}\frac{|u(y)|}{{|x-y|}^{n+2s}}\;dz\leq 
\frac1{C^{n+2s}}\int_{\R^n\setminus B_r(x)}\frac{|u(y)|}{{(1+|y|)}^{n+2s}}\;dy\leq \frac1{C^{n+2s}}
\int_{\R^n}\frac{|u(y)|}{{(1+|y|)}^n}\;dy,
\end{align*}
which is finite, thanks to~\eqref{weighto}.

Therefore, by~\eqref{cns-asymp},
\begin{align*}
\lim_{s\searrow 0} c_{n,s}\int_{\R^n\setminus B_r(x)}\frac{u(y)}{{|x-y|}^{n+2s}}\;dy = 0.
\end{align*}
Thus, \eqref{lllll1} reduces to
\begin{equation}\label{fbdnsr435627qiwfsdafcxgjhas34562}
\lim_{s\searrow 0}(-\Delta)^s u(x) = 
\lim_{s\searrow 0} c_{n,s}\int_{\R^n\setminus B_r(x)}\frac{u(x)}{{|x-y|}^{n+2s}}\;dy.
\end{equation}

Now, we have that
\begin{eqnarray*}&&
\int_{\R^n\setminus B_r(x)}\frac{dy}{{|x-y|}^{n+2s}}
=|\mathbb{S}^{n-1}|\int_r^{+\infty}t^{-1-2s}\;dy
=\frac{ |\mathbb{S}^{n-1}|}{2s r^{2s}}
.
\end{eqnarray*}
Therefore, by the first limit in~\eqref{cns-asymp-precise},
$$ \lim_{s\searrow0} c_{n,s}\int_{\R^n\setminus B_r(x)}\frac{dy}{{|x-y|}^{n+2s}}
=\lim_{s\searrow0} \frac{c_{n,s} |\mathbb{S}^{n-1}|}{2s r^{2s}}
=1.
$$
This fact, together with~\eqref{fbdnsr435627qiwfsdafcxgjhas34562},
establishes~\eqref{ds-sto0}.

We now show~\eqref{ds-sto1}. 
We remark that, by the translation invariance 
of the operator in Lemma~\ref{lem:translation},
we can suppose without loss of generality that~$x=0$ and~$u(0)=0$.
In this setting, we write, using~\eqref{taylor-def},
\begin{align*}
(-\Delta)^s u(0)=-c_{n,s}\int_{\R^n}\frac{u(y)-\nabla u(0)\cdot y}{{|y|}^{n+2s}}\;dy.
\end{align*}
Exploiting symmetries and using~\eqref{fghebdnfhbedwjny4u3i854678ytigefsdjhj54y45reTG}, we deduce that, whenever~$2s>2-\sigma$,
\begin{align*}
\left|\int_{\R^n\setminus B_r}\frac{u(y)-\nabla u(0)\cdot y}{{|y|}^{n+2s}}\;dy\right|
&=
\left|\int_{\R^n\setminus B_r}\frac{u(y)}{{|y|}^{n+2s}}\;dy\right| \\
&\leq
\frac{1}{C^{n+2s}}\int_{\R^n\setminus B_r}\frac{|u(y)|}{{(1+|y|)}^{n+2s}}\;dy \\
&\leq
\frac{1}{C^{n+2s}}\int_{\R^n}\frac{|u(y)|}{{(1+|y|)}^{n+2-\sigma}}\;dy
\end{align*}
which is finite, thanks to~\eqref{weightoo}. 

As a result, recalling~\eqref{cns-asymp}, we find at
\begin{align}\label{nwvlkwovnklwvnkl0}
\lim_{s\nearrow1}c_{n,s}\int_{\R^n\setminus B_r}\frac{u(y)-\nabla u(0)\cdot y}{{|y|}^{n+2s}}\;dy
=0 .
\end{align}

Now, we write
\begin{equation}\label{nwvlkwovnklwvnkl}\begin{split}&
\int_{B_r}\frac{u(y)-\nabla u(0)\cdot y}{{|y|}^{n+2s}}\;dy\\&\qquad=
\frac12\int_{B_r}\frac{D^2u(0)y\cdot y}{{|y|}^{n+2s}}\;dy+
\int_{B_r}\frac{u(y)-\nabla u(0)\cdot y-\frac12D^2u(0)y\cdot y}{{|y|}^{n+2s}}\;dy.\end{split}
\end{equation}
Using polar coordinates and recalling the computation
in~\eqref{o9fpewFJKEòJFkvnsdm}, we have that
the first integral in the right-hand side of~\eqref{nwvlkwovnklwvnkl} satisfies
\begin{equation*}\begin{split}
\int_{B_r}\frac{D^2u(0)y\cdot y}{{|y|}^{n+2s}}\;dy&=
\left(\int_0^r t^{1-2s}\;dt\right)
\left(\int_{\mathbb{S}^{n-1}}D^2u(0)\theta\cdot\theta\;d\theta\right)
\\&=\frac{r^{2-2s}}{2-2s}
\int_{\mathbb{S}^{n-1}}D^2u(0)\theta\cdot\theta\;d\theta\\
&=\frac{r^{2-2s}}{2-2s}\frac{\big|\mathbb{S}^{n-1}\big|}n\,\Delta u(0).
\end{split}\end{equation*}

Therefore, by~\eqref{cns-asymp-precise},
\begin{align*}
\lim_{s\nearrow 1}\frac{c_{n,s}}2\int_{B_r}\frac{D^2u(0)y\cdot y}{{|y|}^{n+2s}}\;dy=\lim_{s\nearrow 1}
\frac{\big|\mathbb{S}^{n-1}\big|}{4n}\,\frac{c_{n,s}}{1-s}\,r^{2-2s}\,\Delta u(0)
=\Delta u(0).
\end{align*}
{F}rom this fact, \eqref{nwvlkwovnklwvnkl0} and~\eqref{nwvlkwovnklwvnkl},
we thus conclude that
\begin{equation}\label{1234567qwertyuisdfwertdfg9507yhgekj}
\lim_{s\nearrow1}(-\Delta)^su(0)=-\Delta u(0)
- \lim_{s\nearrow1}c_{n,s}\int_{B_r}\frac{u(y)-\nabla u(0)\cdot y-\frac12D^2u(0)y\cdot y}{{|y|}^{n+2s}}\;dy.
\end{equation}

Now, the regularity of~$u$ yields that
$$
\left|u(y)-\nabla u(0)\cdot y-\frac12D^2u(0)y\cdot y\right| \leq \omega(r)|y|^2
\qquad\text{for every }y\in B_r,$$
where
\begin{align}\lim_{r\searrow 0}\omega(r) = 0. \label{9ru23jpoidm}
\end{align}
Thus, we estimate
\begin{align*}
&\left|\int_{B_r}\frac{u(y)-\nabla u(0)\cdot y-\frac12D^2u(0)y\cdot y}{{|y|}^{n+2s}}\;dy\right| \leq
\omega(r)\int_{B_r}{|y|}^{2-n-2s}\;dy \\
&\qquad =\omega(r) \big|\mathbb{S}^{n-1}\big|\int_0^r t^{1-2s}\;dt
=\omega(r)\big|\mathbb{S}^{n-1}\big|\frac{r^{2-2s}}{2-2s}.
\end{align*}
Therefore, recalling~\eqref{cns2} and~\eqref{measure-n-sphere},
\begin{eqnarray*}&&
\limsup_{s\nearrow 1}
c_{n,s}\left|\int_{B_r}\frac{u(y)-\nabla u(0)\cdot y-\frac12D^2u(0)y\cdot y}{{|y|}^{n+2s}}\;dy\right|
\\&&\qquad \leq\limsup_{s\nearrow 1}
\frac{2^{2s}\Gamma(\frac{n+2s}2)}{\pi^{n/2}\Gamma(2-s)}\,s(1-s)
\omega(r)\big|\mathbb{S}^{n-1}\big|\frac{r^{2-2s}}{2-2s}
\\&&\qquad = \limsup_{s\nearrow 1}
\frac{2^{2s}\Gamma(\frac{n+2s}2)}{\Gamma(2-s)\Gamma(\frac{n}2)}
\omega(r) r^{2-2s}
= \frac{4\Gamma(\frac{n+2}2)}{\Gamma(\frac{n}2)}
\omega(r)
= 2n\omega(r).
\end{eqnarray*}
This and~\eqref{9ru23jpoidm} give that
$$ \lim_{r\searrow0}
\limsup_{s\nearrow 1}
c_{n,s}\left|\int_{B_r}\frac{u(y)-\nabla u(0)\cdot y-\frac12D^2u(0)y\cdot y}{{|y|}^{n+2s}}\;dy\right|=0.$$
Using this information together with~\eqref{1234567qwertyuisdfwertdfg9507yhgekj},
we obtain the desired limit in~\eqref{ds-sto1}.
\end{proof}

\subsection{Point inversions for the fractional Laplacian}

In this section, we discuss 
a useful transformation that allows to reduce the study of
(non)local equations in exterior domains to that of interior ones.

The simplest example of point inversion is the so-called Kelvin Transform, that is defined as a map~${\mathcal{K}}:\R^n\setminus\{0\}\to\R^n\setminus\{0\}$ given by
\begin{align}\label{kelvin-0}
\mathcal{K}(x):= \frac{x}{{|x|}^2}.
\end{align}
This transformation
fixes all points on~$\partial B_1$ and maps the interior of~$B_1$ to the exterior and vice versa. 

See e.g.~\cite[Section~2.6]{2021arXiv210107941D} and the references therein for
the properties of the Kelvin Transform and historical remarks.

Here we focus our attention on how the Kelvin Transform interacts with the fractional Laplacian.
For this, recall the definition of the space~$L^1_s(\R^n)$ given in~\eqref{w-l1-space}.
For a function~$u\in L^1_s(\R^n)$, define
\begin{align}\label{kelvin-00}
u_{{\mathcal{K}}}(y):=|y|^{2s-n}\,u\left({\mathcal{K}}(y)\right)
\qquad\text{for any }y\in\R^n\setminus\{0\}.
\end{align}

\begin{lemma}\label{lem:weight-inversion}
If~$u\in L^1_s(\R^n)$, then~$u_{{\mathcal{K}}}\in L^1_s(\R^n)$.
\end{lemma}

\begin{proof}
We first check that~$u_{{\mathcal{K}}}\in L^1_{{\rm{loc}}}(\R^n)$. For this, let~$R>0$ and notice that 
\begin{eqnarray*}
&&\int_{B_R} |u_{{\mathcal{K}}}(y)|\,dy =
\int_{B_R} |y|^{2s-n}\,| u({\mathcal{K}}(y) )|\,dy=
\int_{B_R} |y|^{2s-n}\,\left|
u\left(\frac{y}{|y|^2}\right)\right|\,dy.
\end{eqnarray*}
We now change variable~$z:={\mathcal{K}}(y)=y/|y|^2$.
In this way,
we have that~$dy=dz/|z|^{2n}$ (see formula~(2.6.7) and footnote~9 in~\cite{2021arXiv210107941D}). Also, if~$y\in B_R$ then
$$ |z|=\left|\frac{y}{|y|^2}\right| =\frac1{|y|}>\frac1R,$$
and thus~${\mathcal{K}}(B_R)=\R^n\setminus B_{1/R}$.
As a result,
$$ \int_{B_R} |u_{{\mathcal{K}}}(y)|\,dy=
\int_{\R^n\setminus B_{1/R}} \frac{|u(z)|}{|z|^{n+2s}}\,dz
\le C\int_{\R^n}\frac{|u(z)|}{{(1+|z|)}^{n+2s}}\;dz<+\infty
,$$
for some~$C>0$, depending on~$n$, $s$ and~$R$.

We now check that
$$ \int_{\R^n}\frac{|u_{{\mathcal{K}}}(y)|}{{(1+|y|)}^{n+2s}}\;dy<+\infty.$$
For this, we perform the change of variable~$z:=y/|y|^2$ and we see that
\begin{align*}
& \int_{\R^n}\frac{|u_{{\mathcal{K}}}(y)|}{{(1+|y|)}^{n+2s}}\;dy=
\int_{\R^n}\frac{\left|u\left(\frac{y}{|y|^2}\right)\right|}{{(1+|y|)}^{n+2s}}\;\frac{dy}{{|y|}^{n-2s}}\\
& \qquad=\int_{\R^n}\frac{|u(z)|}{\big(1+|z|^{-1}\big)^{n+2s}}\;{|z|}^{n-2s-2n}\;dz=
\int_{\R^n}\frac{|u(z)|}{{(1+|z|)}^{n+2s}}\;dz<+\infty,
\end{align*} as desired.
\end{proof}

Whenever~$x\neq 0$, the function~$u_{{\mathcal{K}}}$ has in a neighborhood of~$x$ the same regularity 
as the one that~$u$ has in a neighborhood of~$x/|x|^2$: this observation, 
together with Lemma~\ref{lem:weight-inversion}, indicates that it will make sense 
to compute~$(-\Delta)^s u_{{\mathcal{K}}}$ at those points where~$(-\Delta)^s u({\mathcal{K}}(x))$ is well-defined. This is the content of the following
statement.

\begin{proposition}\label{prop:inversion}
Let~$r>0$, $x\in\R^n\setminus\{0\}$ and~$u\in C^2(B_r({\mathcal{K}}(x) ))\cap L^1_s(\R^n)$.

Then, $u_{{\mathcal{K}}}\in C^2(B_r(x))\cap L^1_s(\R^n)$ and
\begin{align}\label{fl-inversion}
{(-\Delta)}^s u_{{\mathcal{K}}}(x)=|x|^{-n-2s} {(-\Delta)}^s u({\mathcal{K}}(x)).
\end{align}
Moreover,
\begin{align}\label{fundsol}
{(-\Delta)}^s |x|^{2s-n}=0 \qquad\text{in }\R^n\setminus\{0\}.
\end{align}
\end{proposition}

\begin{proof}
We remark that the fact that~$u\in C^2(B_r({\mathcal{K}}(x)))$ implies
that~$ u_{{\mathcal{K}}}\in C^2(B_r(x))$. Moreover, we know that~$u_{{\mathcal{K}}}\in L^1_s(\R^n)$, thanks to Lemma~\ref{lem:weight-inversion}.

We can use therefore use~\eqref{pv-def} and write
$$ {(-\Delta)}^s u_{{\mathcal{K}}}(x) =
c_{n,s}\pv\int_{\R^n}\frac{u_{{\mathcal{K}}}(x)-u_{{\mathcal{K}}}(y)}{{|x-y|}^{n+2s}}\;dy.$$
We now perform the change of variable~$z:={\mathcal{K}}(y)=y/|y|^2$,
which gives~$dy=dz/|z|^{2n}$ (see formula~(2.6.7) and footnote~9 in~\cite{2021arXiv210107941D}), thus finding that
\begin{equation}\label{asdfghjpoiuytremnbvcxz098765432}\begin{split}
{(-\Delta)}^s u_{{\mathcal{K}}}(x)
&=c_{n,s}\pv\int_{\R^n}\frac{|x|^{2s-n}u({\mathcal{K}}(x))-|y|^{2s-n}u({\mathcal{K}}(y))}{{|x-y|}^{n+2s}}\;dy  \\
&=c_{n,s}\pv\int_{\R^n}\frac{|x|^{2s-n}u({\mathcal{K}}(x))
-|z|^{n-2s}u(z)}{|x-{\mathcal{K}}(z)|^{n+2s}}\;\frac{dz}{{|z|}^{2n}}.
\end{split}\end{equation}

We recall that, for all~$y$, $z\in\R^n\setminus \{0\}$,
$$ |{\mathcal{K}}(y)-{\mathcal{K}}(z)|=\frac{|y-z|}{|y|\,|z|},$$
see formula~(2.6.5) in~\cite{2021arXiv210107941D}.

We exploit this formula with~$y:={\mathcal{K}}(x)$, which gives that~${\mathcal{K}}(y)= {\mathcal{K}}({\mathcal{K}}(x))=x$, obtaining that
\begin{align}\label{inverted-distance}
|x-{\mathcal{K}}(z)|=\frac{|{\mathcal{K}}(x)-z|}{|{\mathcal{K}}(x)|\,|z|}.
\end{align}
Using this into~\eqref{asdfghjpoiuytremnbvcxz098765432} we find that
\begin{equation}\label{1269-8219}\begin{split}
{(-\Delta)}^s u_{{\mathcal{K}}}(x)
&=
c_{n,s}\pv\int_{\R^n}\frac{|{\mathcal{K}}(x)|^{n-2s}u({\mathcal{K}}(x))-|z|^{n-2s}u(z)}{{|{\mathcal{K}}(x)-z|}^{n+2s}}\;
\frac{dz}{{|z|}^{n-2s}}\;{|{\mathcal{K}}(x)|}^{n+2s} \\
&= 
c_{n,s}\pv\int_{\R^n}\frac{|z|^{n-2s}u({\mathcal{K}}(x))-|z|^{n-2s}u(z)}{{|{\mathcal{K}}(x)-z|}^{n+2s}}\;\frac{dz}{{|z|}^{n-2s}}\;{|{\mathcal{K}}(x)|}^{n+2s} \\
& \qquad +c_{n,s}\pv\int_{\R^n}\frac{|{\mathcal{K}}(x)|^{n-2s}u({\mathcal{K}}(x))-|z|^{n-2s}u({\mathcal{K}}(x))}{{|{\mathcal{K}}(x)-z|}^{n+2s}}\;\frac{dz}{{|z|}^{n-2s}}\;{|{\mathcal{K}}(x)|}^{n+2s} \\
&= |{\mathcal{K}}(x)|^{n+2s} {(-\Delta)}^s u({\mathcal{K}}(x))
\\&\qquad+c_{n,s}\pv\int_{\R^n}\frac{|{\mathcal{K}}(x)|^{2s-n}-|z|^{2s-n}}{{|{\mathcal{K}}(x)-z|}^{n+2s}}\;dz\;{|{\mathcal{K}}(x)|}^{n+2s}u({\mathcal{K}}(x)). 
\end{split}\end{equation}

We now use the notation~$u_{2s-n}(x):=|x|^{2s-n}$ and we see that,
for any~$x\in\R^n\setminus\{0\}$,
\begin{eqnarray*}
(u_{2s-n})_{{\mathcal{K}}}(x)=
|x|^{2s-n}\,u_{2s-n}\left({\mathcal{K}}(x)\right)
=|x|^{2s-n}\,|{\mathcal{K}}(x)|^{2s-n}=1
\end{eqnarray*}
and therefore
\begin{align*}
{(-\Delta)}^s (u_{2s-n})_{{\mathcal{K}}}(x) = 0. 
\end{align*}
Hence, employing~\eqref{1269-8219} with~$u:=u_{2s-n}$ gives
\begin{align*}
0 &= |{\mathcal{K}}(x)|^{n+2s} {(-\Delta)}^s u_{2s-n}({\mathcal{K}}(x))\\&\qquad+c_{n,s}\pv\int_{\R^n}\frac{u_{2s-n}({\mathcal{K}}(x))
-u_{2s-n}(z)}{{|{\mathcal{K}}(x)-z|}^{n+2s}}\;dz\;{|{\mathcal{K}}(x)|}^{n+2s}u_{2s-n}({\mathcal{K}}(x)).
\end{align*}
Namely,
\begin{align*}
{(-\Delta)}^s u_{2s-n}({\mathcal{K}}(x))
=-u_{2s-n}({\mathcal{K}}(x)){(-\Delta)}^s u_{2s-n}({\mathcal{K}}(x))
\qquad\text{for any } {\mathcal{K}}(x)\in\R^n\setminus\{0\}.
\end{align*}
In particular, 
\begin{align*}
{(-\Delta)}^s u_{2s-n}({\mathcal{K}}(x))=0
\qquad\text{for any }{\mathcal{K}}(x)\in\partial B_1
\end{align*}
{F}rom this and the homogeneity of the operator (see Lemma~\ref{lem:homogeneity}), we deduce~\eqref{fundsol}.

Having proved~\eqref{fundsol}, we insert it in~\eqref{1269-8219} which therefore simplifies to
\begin{align*}
{(-\Delta)}^s u_{{\mathcal{K}}}(x)=|{\mathcal{K}}(x)|^{n+2s} {(-\Delta)}^s u({\mathcal{K}}(x))
\end{align*}
and the proof of~\eqref{fl-inversion} is concluded by observing that~$|{\mathcal{K}}(x)|=1/|x|$.
\end{proof}

\begin{remark}
Identity~\eqref{fundsol} is tightly connected to the notion of fundamental solution,
see Corollary~\ref{FURIEZ}.
\end{remark}

More generally, for~$R>0$ and~$x_0\in\R^n$, point inversions are mappings of the form~${\mathcal{K}}_{R,x_0}:\R^n\setminus\{x_0\} \to \R^n\setminus\{x_0\} $ defined as
\begin{align}\label{isinforceBIS}
{\mathcal{K}}_{R,x_0}(x):= R^2\frac{x-x_0}{{|x-x_0|}^2}+x_0.
\end{align}
They fix points on~$\partial B_R(x_0)$ and map the interior of~$B_R(x_0)$ to the exterior and vice versa.

As this general form is obtained by simply composing the one with~$R=1$ and~$x_0=0$ with dilations and translations, we get the following generalization of formula~\eqref{fl-inversion}.

\begin{proposition}\label{prop:inversion-general}
Let~$R$, $r>0$, $x_0\in\R^n$, $x\in\R^n\setminus\{x_0\}$ and~$u\in C^2(B_r(\mathcal{K}_{R,x_0}(x)))\cap L^1_s(\R^n)$.

Let
\begin{align}\label{isinforceBIS2}
u_{{\mathcal{K}}_{R,x_0}}(y):=|y-x_0|^{2s-n}u({\mathcal{K}}_{R,x_0}(y))\qquad\text{for any }y\in\R^n\setminus\{x_0\}.
\end{align}
Then, $u_{{\mathcal{K}}_{R,x_0}} \in C^2(B_r(x))\cap L^1_s(\R^n)$ and
\begin{align}\label{fl-inversion-general}
{(-\Delta)}^s u_{{\mathcal{K}}_{R,x_0}}(x)=R^{4s}|x-x_0|^{-n-2s}{(-\Delta)}^s u(\mathcal{K}_{R,x_0}(x)).
\end{align}
\end{proposition}

\subsection{The fractional Laplacian for functions depending
only on a subset of variables}

If a function depends only on a subset of the Euclidean coordinates, its full Laplacian coincides with the Laplacian computed on
these coordinates. The same holds true for the fractional Laplacian (provided that normalizing constants are chosen appropriately).

To check this claim, we consider~$k\in\{1,\ldots,n-1\}$ and write~$x\in\R^n$ as~$x=(x',x'')$ 
with~$x'\in\R^k$ and~$x''\in\R^{n-k}$. Also, we will make it explicit that the fractional Laplacian is computed in~$\R^n$ by writing~$(-\Delta)^s_{\R^n}$.

\begin{lemma}\label{lem:mute}
Let~$r>0$,~$\sigma>s$,~$x\in\R^n$ and~$u\in C^{2\sigma}(B_r(x))\cap L^1_s(\R^n)$. Suppose that
there exists~$\overline{u}:\R^k\to\R$ such that~$u(y)=\overline{u}(y')$
for any~$y\in\R^n$.

Then,
\begin{align}\label{fl-mute}
(-\Delta)^s_{\R^n} u(x)=(-\Delta)^s_{\R^k}\overline{u}(x').
\end{align}
\end{lemma}

Note that in the above identity~\eqref{fl-mute} 
the fractional Laplacian on the left-hand side 
and the one on the right-hand side 
are not, strictly speaking, the same operator: they are respectively the 
$n$-dimensional and the~$k$-dimensional versions of~\eqref{pv-def}.
Note also that~\eqref{fl-mute} is more precise than just a translation invariance of the operator:
being~$u$ invariant for translations in the~$x''$ directions, 
we know that its fractional will also be invariant along these directions; 
formula~\eqref{fl-mute} precisely states what value is attained by the operator.

\begin{proof}[Proof of Lemma~\ref{lem:mute}]
To prove identity~\eqref{fl-mute}, we make use of~\eqref{symmetric-def} 
and proceed via a direct calculation:
\begin{align*}
(-\Delta)^s_{\R^n} u(x) &= \frac{c_{n,s}}{2}\int_{\R^n}\frac{2u(x)-u(x+y)-u(x-y)}{{|y|}^{n+2s}}\;dy \\
&=
\frac{c_{n,s}}{2}\int_{\R^k}\Big(2\overline{u}(x'))-\overline{u}(x'+y')-\overline{u}(x'-y')\Big)\left(\int_{\R^{n-k}}\frac{dy''}{{(|y'|^2+|y''|^2)}^{n/2+s}}\right)\;dy'.
\end{align*}
Now, by the change of variables~$z:=y''/|y'|$, one has that
\begin{align*}
\int_{\R^{n-k}}\frac{dy''}{{(|y'|^2+|y''|^2)}^{n/2+s}}=
\frac1{{|y'|}^{k+2s}}\int_{\R^{n-k}}\frac{dz}{{(1+|z|^2)}^{n/2+s}}
\end{align*}
which entails that
\begin{equation}\label{scuiore8563od3seperchegu544i}\begin{split}
(-\Delta)^s_{\R^n} u(x)&=\frac{c_{n,s}}{2}\left(\int_{\R^k}\frac{2\overline{u}(x'))-\overline{u}(x'+y')-\overline{u}(x'-y')}{{|y'|}^{k+2s}}\;dy'\right)\left(\int_{\R^{n-k}}\frac{dz}{{(1+|z|^2)}^{n/2+s}}\right)\\&
=\frac{c_{n,s}}{c_{k,s}}(-\Delta)^s_{\R^k}\overline{u}(x')\int_{\R^{n-k}}\frac{dz}{{(1+|z|^2)}^{n/2+s}}.
\end{split}\end{equation}

Now, passing to polar coordinates, we see that
\begin{align*}
\int_{\R^{n-k}}\frac{dz}{{(1+|z|^2)}^{n/2+s}}=
\big|\mathbb{S}^{n-k-1}\big|\int_0^{+\infty}\frac{\rho^{n-k-1}}{{(1+\rho^2)}^{n/2+s}}\;d\rho
\end{align*}
which, after the change of variable~$t:=(1+\rho^2)^{-1}$, gives (see~\eqref{beta-identity})
\begin{align*}
\int_0^{+\infty}\frac{\rho^{n-k-1}}{{(1+\rho^2)}^{n/2+s}}\;d\rho=
\frac12\int_0^1 t^{k/2+s-1}{(1-t)}^{n/2-k/2-1}\;dt=\frac{\Gamma(\frac{k}2+s)\,\Gamma(\frac{n-k}2)}{2\,\Gamma(\frac{n}2+s)}.
\end{align*}
Hence (see~\eqref{measure-n-sphere} for the measure of the sphere) we have that 
\begin{align*}
\frac{c_{n,s}}{c_{k,s}}\int_{\R^{n-k}}\frac{dz}{{(1+|z|^2)}^{n/2+s}}=
\frac{2^{2s}\Gamma(\frac{n+2s}2)s}{\pi^{n/2}\Gamma(1-s)}
\frac{\pi^{k/2}\Gamma(1-s)}{2^{2s}\Gamma(\frac{k+2s}2)s}
\frac{2\pi^{n/2-k/2}}{\Gamma(\frac{n-k}2)}
\frac{\Gamma(\frac{k}2+s)\,\Gamma(\frac{n-k}2)}{2\,\Gamma(\frac{n}2+s)}=
1.
\end{align*}
Combining this information with~\eqref{scuiore8563od3seperchegu544i}
completes the proof.
\end{proof}

\subsection{Exchanging the order of derivatives and the fractional Laplacian}

In this section, we investigate 
how the fractional Laplacian 
interacts with the usual derivatives.

Let us first start from a simple case.

\begin{lemma}
Let~$u\in C^\infty_c(\R^n)$. Then, for any~$j\in\{1,\ldots,n\}$,
\begin{align*}
\partial_{x_j}(-\Delta)^s u(x)=(-\Delta)^s\big(\partial_{x_j}u\big)(x)
\qquad\text{for any }x\in\R^n.
\end{align*}
\end{lemma}

\begin{proof}
By the representation of the fractional Laplacian in~\eqref{FORi},
\begin{align*}
\partial_{x_j}(-\Delta)^s u(x)
&=
\partial_{x_j}\int_{\R^n}{(2\pi|\xi|)}^{2s}\widehat{u}(\xi)e^{2\pi ix\cdot\xi}\;d\xi \\
&=
\int_{\R^n}{(2\pi|\xi|)}^{2s}(2\pi i\xi_j)\widehat{u}(\xi)e^{2\pi ix\cdot\xi}\;d\xi \\
&=
\int_{\R^n}{(2\pi|\xi|)}^{2s}\,\widehat{\partial_{x_j}u}(\xi)\,e^{2\pi ix\cdot\xi}\;d\xi \\&
=
(-\Delta)^s\big(\partial_{x_j}u\big)(x),
\end{align*} as desired.
\end{proof}

More generally, we have the following:

\begin{proposition}\label{prop:exchange}
Let~$r>0$,~$\sigma>s$,~$x\in\R^n$ and~$u\in C^{1+2\sigma}(B_r(x))\cap L^1_s(\R^n)$ with~$\nabla u\in L^1_s(\R^n)$.

Then, for any~$j\in\{1,\ldots,n\}$,
\begin{align*}
\partial_{x_j}(-\Delta)^s u(x)=(-\Delta)^s\big(\partial_{x_j}u\big)(x).
\end{align*}
\end{proposition}

\begin{proof}
Take a bump function~$\varphi_0\in C^\infty_c(B_r, [0,1])$
with~$\varphi=1$ in~$B_{r/2}$.
Recalling~\eqref{symmetric-def}, we find that
\begin{align}
(-\Delta)^s u(x)
&=
\frac{c_{n,s}}{2}\int_{\R^n}\frac{2u(x)-u(x+y)-u(x-y)}{{|y|}^{n+2s}}\;dy \notag \\
&=
\frac{c_{n,s}}{2}\int_{B_r}\frac{2u(x)-u(x+y)-u(x-y)}{{|y|}^{n+2s}}\,\varphi_0(y)\;dy
\label{8ryh2oryhf9823yf98f32h98cnw} \\
&\qquad+
\frac{c_{n,s}}{2}\int_{\R^n}\frac{2u(x)-u(x+y)-u(x-y)}{{|y|}^{n+2s}}(1-\varphi_0(y))\;dy.
\label{d8y932hd9823hd8y3hywgcg67} 
\end{align}
The term in~\eqref{8ryh2oryhf9823yf98f32h98cnw} can be differentiated under the integral sign 
after we notice that
\begin{align*}
\big|2\partial_{x_j}u(x)-\partial_{x_j}u(x+y)-\partial_{x_j}u(x-y)\big|
\leq\|u\|_{C^{1+2\sigma}(B_r(x))}|y|^{2\sigma}
\qquad\text{for any }y\in B_r.
\end{align*}
To show the differentiability under the integral sign also of the term in~\eqref{d8y932hd9823hd8y3hywgcg67},
we split it into three addends and we simply show the differentiability under the integral sign of 
\begin{align*}
\int_{\R^n}\frac{u(x-y)}{{|y|}^{n+2s}}(1-\varphi_0(y))\;dy.
\end{align*}
To do so, we first write 
\begin{align*}
\int_{\R^n}\frac{u(x-y)}{{|y|}^{n+2s}}(1-\varphi_0(y))\;dy=
\int_{\R^n}\frac{u(y)}{{|x-y|}^{n+2s}}(1-\varphi_0(x-y))\;dy
\end{align*}
and we integrate by part as follows:
\begin{align*}
\partial_{x_j}\int_{\R^n}\frac{u(y)}{{|x-y|}^{n+2s}}(1-\varphi_0(x-y))\;dy
&=
\int_{\R^n}u(y)\partial_{x_j}\left(\frac{1-\varphi_0(x-y)}{{|x-y|}^{n+2s}}\right)\;dy \\
&=
-\int_{\R^n}u(y)\partial_{y_j}\left(\frac{1-\varphi_0(x-y)}{{|x-y|}^{n+2s}}\right)\;dy \\
&= 
\int_{\R^n}\frac{\partial_{y_j}u(y)}{{|x-y|}^{n+2s}}(1-\varphi_0(x-y))\;dy \\
&=
\int_{\R^n}\frac{\partial_{x_j}u(x-y)}{{|y|}^{n+2s}}(1-\varphi_0(y))\;dy,
\end{align*}thus establishing the desired result.
\end{proof}

\section{The fractional Laplacian of periodic functions}

Here we compute the fractional Laplacian of a periodic function, by reducing the calculation to its Fourier modes.
For this, we pick~$a=(a_1,\dots,a_n)\in\R^n$ with positive components, i.e.,~$a_j>0$ for every~$j=1,\ldots,n$.
We consider the lattices
\begin{eqnarray*} &&
L_a:=\big\{(k_1a_1,\ldots,k_na_n)\;{\mbox{ with }}\; k=(k_1,\ldots,k_n)\in\Z^n\big\}\\ {\mbox{and }} &&
L_a':=\left\{\left(\frac{k_1}{a_1},\ldots,\frac{k_n}{a_n}\right)\;{\mbox{ with }}\;k=(k_1,\ldots,k_n)\in\Z^n\right\}.
\end{eqnarray*}
We say that~$u:\R^n\to\R$ is~$a$-periodic if
\begin{align*}
u(x+h)=u(x) \qquad \text{for any~$x\in\R^n$ and any~$ h\in L_a$}.
\end{align*}

If~$u$ is~$a$-periodic, we can describe it via the Fourier series
\begin{equation}\label{168BIS56748329hgjfkd}\begin{split}
&u(x)=\sum_{k\in L_a'}u_k e^{2\pi ik\cdot x},
\\ 
{\mbox{where }}\quad & u_k:=\left(\displaystyle\prod_{j=1}^n a_j\right)^{-1}\int_{[0,a_1]\times\cdots\times[0,a_n]} u(x) e^{-2\pi ik\cdot x}\;dx.
\end{split}\end{equation}
The Fourier coefficients~$u_k$ determine the fractional Laplacian of~$u$, as stated in the following result:

\begin{lemma}
Let~$r>0$, $\sigma>s$, $x\in\R^n$ and~$u\in C^{2\sigma}(B_r(x))\cap L^1_s(\R^n)$.
If~$u$ is~$a$-periodic, then
\begin{align}\label{fl-period}
{(-\Delta)}^s u(x)=\sum_{k\in L_a'}|2\pi k|^{2s}u_k e^{2\pi ik\cdot x}.
\end{align}
\end{lemma}

\begin{proof}
By~\eqref{168BIS56748329hgjfkd} and the change of variable~$z:=|k|y$, we have that
\begin{equation}\label{168BIS56748329hgjfkd2}\begin{split}&
\int_{\R^n}\frac{2u(x)-u(x+y)-u(x-y)}{{|y|}^{n+2s}}\;dy\\
&=
\sum_{k\in L_a'}\int_{\R^n}\frac{2u_k e^{2\pi ik\cdot x}-u_k e^{2\pi ik\cdot (x+y)}-u_k e^{2\pi ik\cdot (x-y)}}{{|y|}^{n+2s}}\;dy \\
&=
\sum_{k\in L_a'}u_k e^{2\pi ik\cdot x}\int_{\R^n}\frac{2-e^{2\pi ik\cdot y}-e^{-2\pi ik\cdot y}}{{|y|}^{n+2s}}\;dy \\
&=
\sum_{k\in L_a'}{|k|}^{2s} u_k e^{2\pi ik\cdot x}\int_{\R^n}\frac{2-e^{2\pi i\frac{k}{|k|}\cdot z}-e^{-2\pi i\frac{k}{|k|}\cdot z}}{{|z|}^{n+2s}}\;dz.
\end{split}\end{equation}
Up to a rotation, the latter integral can be rewritten as
\begin{align*}
\int_{\R^n}\frac{2-e^{2\pi iz_1}-e^{-2\pi iz_1}}{{|z|}^{n+2s}}\;dz,
\end{align*}
which in particular does not depend on~$k\in L_a'$ and produces a multiplicative constant that we will now compute. First, split the integration as 
\begin{equation}\label{902fujioewnubgeifj8io23}\begin{split}
\int_{\R^n}\frac{2-e^{2\pi iz_1}-e^{-2\pi iz_1}}{{|z|}^{n+2s}}\;dz
&=
\int_{\R}\big(2-e^{2\pi iz_1}-e^{-2\pi iz_1}\big)\int_{\R^{n-1}}\frac{dz'}{\big(z_1^2+|z'|^2\big)^{n/2+s}}\;dz_1  \\
&=
\int_{\R}\frac{2-e^{2\pi iz_1}-e^{-2\pi iz_1}}{|z_1|^{1+2s}}\;dz_1
\int_{\R^{n-1}}\frac{d\eta}{\big(1+|\eta|^2\big)^{n/2+s}}. 
\end{split}\end{equation}
The two integrals obtained are now decoupled, so we can compute them separately. 

We start with the one in the~$\eta$ variable: by passing to polar coordinates
and applying the change of variables~$\rho:=(1+r^2)^{-1}$, we find that
\begin{align*}
\int_{\R^{n-1}}\frac{d\eta}{\big(1+|\eta|^2\big)^{n/2+s}}&=
|\mathbb{S}^{n-2}|\int_0^{+\infty}\frac{r^{n-2}}{\big(1+r^2\big)^{n/2+s}}\;dr\\
&=
\frac12|\mathbb{S}^{n-2}|\int_0^1\rho^{(n-3)/2}{(1-\rho)}^{s-1/2}\;d\rho \\
&=
\frac12
\frac{2\pi^{(n-1)/2}}{\Gamma\big(\frac{n-1}2\big)}
\frac{\Gamma\big(\frac{n-1}2\big)\,\Gamma\big(s+\frac12\big)}{\Gamma\big(\frac{n}2+s\big)}\\&
=
\pi^{(n-1)/2}
\frac{\Gamma\big(s+\frac12\big)}{\Gamma\big(\frac{n}2+s\big)}, 
\end{align*}
see~\eqref{beta-identity} for the integral in~$\rho$ and~\eqref{measure-n-sphere} for the explicit expression of~$|\mathbb{S}^{n-2}|$. 

The expression of the above formula can be modified with~\eqref{gamma-dupli} into 
\begin{align}\label{r9230uru23jifofhf3209329fj}
\int_{\R^{n-1}}\frac{d\eta}{\big(1+|\eta|^2\big)^{n/2+s}}=
\pi^{n/2}
\frac{2^{-2s}\Gamma(2s+1)}{\Gamma(s+1)\Gamma\big(\frac{n}2+s\big)}.
\end{align}

We turn now to the computation of the integral in~$z_1$ in~\eqref{902fujioewnubgeifj8io23}.
By rescaling and exploiting the parity of the integrand, we see that
\begin{equation}\label{xczvxfuyt8thebdnt5u6iykjgn}
\int_{\R}\frac{2-e^{2\pi iz_1}-e^{-2\pi iz_1}}{|z_1|^{1+2s}}\;dz_1
=2(2\pi)^{2s}\int_0^{+\infty}\frac{2-e^{it}-e^{-it}}{t^{1+2s}}\;dt.
\end{equation}
Now, by using the definition of the~$\Gamma$ function~\eqref{gamma-def}, we notice that
\begin{align*}
\Gamma(1+2s)\int_0^{+\infty}\frac{2-e^{it}-e^{-it}}{t^{1+2s}}\;dt
=\int_0^{+\infty}\int_0^{+\infty}e^{-w}w^{2s}\frac{2-e^{it}-e^{-it}}{t^{1+2s}}\;dt\;dw
\end{align*}
and, applying the change of variable~$t:=w\tau$, 
\begin{align*}
\Gamma(1+2s)\int_0^{+\infty}\frac{2-e^{it}-e^{-it}}{t^{1+2s}}\;dt
&=
\int_0^{+\infty}\frac1{\tau^{1+2s}}\int_0^{+\infty}e^{-w}\big(2-e^{iw\tau}-e^{-iw\tau}\big)\;dw\;d\tau \\
&=
\int_0^{+\infty}\frac1{\tau^{1+2s}}\left(2+\frac1{i\tau-1}-\frac1{i\tau+1}\right)\;d\tau \\
&=
\int_0^{+\infty}\frac{\tau^{1-2s}}{\tau^2+1}\;d\tau.
\end{align*}
We now use the change of variable~$v:=(\tau^2+1)^{-1}$ and deduce that
\begin{align*}
\Gamma(1+2s)\int_0^{+\infty}\frac{2-e^{it}-e^{-it}}{t^{1+2s}}\;dt
=\int_0^1 v^{s-1}(1-v)^{-s}\;dv=\Gamma(1-s)\Gamma(s),
\end{align*}
see~\eqref{beta-identity}.

{F}rom this and~\eqref{xczvxfuyt8thebdnt5u6iykjgn} we thus obtain that
\begin{align}\label{r8y23horir238ryh23ohir}
\int_{\R}\frac{2-e^{2\pi iz_1}-e^{-2\pi iz_1}}{|z_1|^{1+2s}}\;dz_1
=2(2\pi)^{2s}\frac{\Gamma(1-s)\Gamma(s)}{\Gamma(1+2s)}.
\end{align}
Inserting~\eqref{r9230uru23jifofhf3209329fj} and~\eqref{r8y23horir238ryh23ohir} into~\eqref{902fujioewnubgeifj8io23}, we conclude that
\begin{align*}
&\int_{\R^n}\frac{2-e^{2\pi iz_1}-e^{-2\pi iz_1}}{{|z|}^{n+2s}}\;dz
=
\pi^{n/2+2s}
\frac{2}{\Gamma(s+1)\Gamma\big(\frac{n}2+s\big)}
\Gamma(1-s)\Gamma(s) \\
&\qquad\qquad =
(2\pi)^{2s}\frac{2\pi^{n/2}\Gamma(1-s)}{2^{2s}s\Gamma\big(\frac{n}2+s\big)}
=(2\pi)^{2s}\frac{2}{c_{n,s}},
\end{align*}
which, together with~\eqref{168BIS56748329hgjfkd2}, completes the proof of~\eqref{fl-period}.
\end{proof}

As a straightforward application of~\eqref{fl-period}, we have that
if~$u:\R\to\R$ is~$a$-periodic for some~$a>0$ then
\begin{align}\label{fgdhsjtyreuw5647382qwsaxdertgbhy789i09876543}
{(-\Delta)}^s u(x)=\left(\frac{2\pi}{a}\right)^{2s}\sum_{k\in\Z}|k|^{2s}u_ke^{\frac{2\pi}{a} ikx}
\qquad\text{for any }x\in\R.
\end{align}

In particular, for example:

\begin{lemma} For any~$x\in\R$, we have that
\begin{align*}
{(-\Delta)}^s \sin (x)= \sin x 
\qquad\text{and}\qquad
{(-\Delta)}^s \cos(x) = \cos x.
\end{align*}
\end{lemma}

\begin{proof} We prove the desired formula for~$\sin x$, being the one for~$\cos x$ analogous.

One can check that
\begin{eqnarray*} u_k&=& \frac1{2\pi}\int_0^{2\pi} \sin y\, e^{-iky}\,dy=
\begin{cases}
-\frac1{2i} &{\mbox{ if }} k=1,\\
\frac1{2i} &{\mbox{ if }} k=-1,\\
0 &{\mbox{ if }} k\neq 
\pm1.
\end{cases} \end{eqnarray*}
Therefore, from~\eqref{fgdhsjtyreuw5647382qwsaxdertgbhy789i09876543} we see that
\begin{eqnarray*}
{(-\Delta)}^s \sin (x)=\sum_{k\in\Z}|k|^{2s}u_ke^{ikx}=-\frac1{2i}e^{-ix}+\frac1{2i} e^{ix}=\sin x,
\end{eqnarray*}
as desired.
\end{proof}

\section{Other identities for the fractional Laplacian}

We prove here below a couple of other formulas which may come in handy at times.

We will use the definition of nonlocal gradient in~\eqref{NONGRAD}. Moreover,
we will denote by~$e_k$ the~$k$-th element of the Euclidean basis.

\begin{lemma}\label{fkoencalmd}
Let~$s\in(1/2,1)$, $r>0$, $\sigma>s$, $x\in\R^n$ and~$u\in C^{2\sigma}(B_r(x))$ with
\begin{align*}
\int_{\R^n}\frac{|u(y)|}{{(1+|y|)}^{n+2s-1}}\;dy<+\infty.
\end{align*}
Then,
\begin{align*} 
(-\Delta)^s (x_ku)(x) &= x_k\,(-\Delta)^s u(x)-2s\,e_k\cdot\nabla^{2s-1} u(x).
\end{align*}
\end{lemma}

\begin{proof}
Recalling~\eqref{NONGRAD} and~\eqref{dns}, 
\begin{align*}
(-\Delta)^s (x_ku)(x) &=
c_{n,s}\pv\int_{\R^n}\frac{x_k u(x)-y_ku(y)}{{|x-y|}^{n+2s}}\;dy \\
&=
x_k\,c_{n,s}\pv\int_{\R^n}\frac{u(x)-u(y)}{{|x-y|}^{n+2s}}\;dy
+c_{n,s}\pv\int_{\R^n}\frac{x_k-y_k}{{|x-y|}^{n+2s}}\,u(y)\;dy \\
&= x_k\,(-\Delta)^s u(x)+2s\,d_{n,s}e_k\cdot\pv\int_{\R^n}\frac{u(y)}{{|x-y|}^{n+2s-1}}\frac{x-y}{|x-y|}\;dy \\
&= x_k\,{(-\Delta)}^s u(x)-2s\,e_k\cdot\nabla^{2s-1}u(x),
\end{align*}as desired.
\end{proof}

\begin{proposition}
Let~$s\in(1/2,1)$, $r>0$, $\sigma>s$, $x\in\R^n$ and~$u\in C^{1+2\sigma}(B_r(x))\cap C^1(\R^n)\cap L^1_s(\R^n)$ with
\begin{align}\label{rewfydu647312qwsdx4rfg7uj9ol}
\int_{\R^n}\frac{|\nabla u(y)|}{{(1+|y|)}^{n+2s-1}}\;dy<+\infty.
\end{align}
Then,
\begin{align*}
(-\Delta)^s (x\cdot\nabla u)(x) = x\cdot\nabla(-\Delta)^su(x)+2s(-\Delta)^su(x).
\end{align*}
\end{proposition}

\begin{proof}
We point out that the assumptions of Lemma~\ref{fkoencalmd}
are satisfied with~$u$ replaced by~$\partial_j u$ for all~$j\in\{1,\dots,n\}$,
thanks to the regularity of~$u$ and~\eqref{rewfydu647312qwsdx4rfg7uj9ol}.
Therefore, exploiting Lemma~\ref{fkoencalmd}, we obtain that
\begin{eqnarray*}
(-\Delta)^s (x\cdot\nabla u)(x) &=&
\sum_{j=1}^n(-\Delta)^s(x_j\partial_j u)(x)\\&=&
\sum_{j=1}^n x_j(-\Delta)^s(\partial_j u)(x)-2s\sum_{j=1}^ne_j\cdot\nabla^{2s-1}(\partial_j u)(x) .\end{eqnarray*}
Furthermore, Proposition~\ref{prop:exchange}
and formula~\eqref{RADCVE2} give that
\begin{align*}
(-\Delta)^s (x\cdot\nabla u)(x) &=\sum_{j=1}^n x_j\partial_j(-\Delta)^su(x)-2s \sum_{j=1}^ne_j\cdot\nabla^{2s-1} (\partial_ju)(x)
\\&=
x\cdot\nabla(-\Delta)^su(x)-2s\div^{2s-1}\big(\nabla u\big)(x)
\\&
=x\cdot\nabla(-\Delta)^su(x)+2s(-\Delta)^su(x),
\end{align*} yielding the desired result.
\end{proof}

\chapter{Some explicit calculations of fractional Laplacians}\label{CH:EXAMPLES}

Differently from the classical case of the Laplacian, 
whose explicit calculations only involve derivatives,
the fractional Laplacian typically involve integral computations 
which do not produce simple results in terms of elementary functions.
There are however a few exceptions in which exact results can be obtained in a rather explicit way. 
We describe some
of these cases here below (see also~\cite{MR2974318, MR3469920, MR3640641} and~\cite[Lemma~A.2]{MR3596708}
for related examples and methodologies).

\section{Functions of one real variable}

\subsection{Power functions on the half-line}

Let~$\alpha\in(-1,2s)$ and
\begin{align}\label{accadialfa}
h_\alpha(x):=x_+^\alpha:=\left\lbrace\begin{aligned}
& x^\alpha && \text{for }x>0, \\
& 0 && \text{for }x\leq 0.
\end{aligned}\right.
\end{align}

\begin{lemma}\label{lem:halfline}
There exists~$\kappa(\alpha)\in\R$ such that
\begin{align}\label{powers-half-line}
(-\Delta)^s h_\alpha(x)=\left\lbrace\begin{aligned}
& c_{1,s}\,\kappa(\alpha)\,x^{\alpha-2s} && \text{for }x>0, \\
& -c_{1,s}\,\frac{\Gamma(\alpha+1)\,\Gamma(2s-\alpha)}{\Gamma(2s+1)}\,|x|^{\alpha-2s} && \text{for }x<0.
\end{aligned}\right.
\end{align}
The constant~$\kappa(\alpha)$ satisfies
\begin{align}\label{kappalfa}
\kappa(\alpha)\ \left\lbrace\begin{aligned}
& <0 && \text{for }\alpha\in(s,2s), \\
& =0 && \text{for }\alpha=s, \\
& >0 && \text{for }\alpha\in(s-1,s), \\
& =0 && \text{for }\alpha=s-1, \\
& <0 && \text{for }\alpha\in(-1,s-1).
\end{aligned}\right.
\end{align}
\end{lemma}

\begin{proof}
Let us notice that~$h_\alpha$ is homogeneous of degree~$\alpha$, i.e.,~$h_\alpha(\lambda x)=\lambda^\alpha h_\alpha(x)$ for every~$\lambda>0$ and~$x\in\R$. 

Hence, by the homogeneity of the fractional Laplacian (see
Lemma~\ref{lem:homogeneity}), we have that
\begin{eqnarray*}&& {(-\Delta)}^s h_\alpha(x)= |x|^\alpha {(-\Delta)}^s\left( h_\alpha\left(\frac{ x}{|x|}\right)\right)=
|x|^{\alpha-2s} {(-\Delta)}^s h_\alpha\left(\frac{ x}{|x|}\right)
\\&&\qquad =
\left\lbrace\begin{aligned}
& x^{\alpha-2s}{(-\Delta)}^s h_\alpha(1) && \text{for }x>0, \\
& (-x)^{\alpha-2s}{(-\Delta)}^s h_\alpha(-1) && \text{for }x<0. \\
\end{aligned}\right.
\end{eqnarray*}

Now, let us first deal with the case~$x<0$. In this range, we see that
\begin{align*}
(-\Delta)^s h_\alpha(-1)
=
c_{1,s}\pv\int_\R\frac{-h_\alpha(y)}{{|y+1|}^{1+2s}}\;dy
=-c_{1,s}\int_0^{+\infty}\frac{y^\alpha}{{(y+1)}^{1+2s}}\;dy 
\end{align*}
which, after the change of variable~$z:=y/(1+y)$, results in
\begin{align*}
(-\Delta)^s h_\alpha(-1)=-c_{1,s}\int_0^1 t^\alpha(1-t)^{2s-\alpha-1}\;dt
=-c_{1,s}\,\frac{\Gamma(\alpha+1)\,\Gamma(2s-\alpha)}{\Gamma(2s+1)}
\end{align*}
where we have used~\eqref{beta-identity} (with~$a:=\alpha+1$ and~$b:= 2s-\alpha$) to evaluate the last integral. Formula~\eqref{powers-half-line}
when~$x<0$ is thereby established.

If instead~$x>0$, we split our computation into different cases. 

Let us first consider the case~$\alpha\in(0,2s)$.
By using~\eqref{pv-def} and integrating by parts, we obtain that
\begin{align*}
{(-\Delta)}^s h_\alpha(1) &=
c_{1,s}\int_{-\infty}^0\frac{dy}{{|1-y|}^{1+2s}}+
c_{1,s}\lim_{\eps\searrow 0}\left[
\int_0^{1-\eps}\frac{1-y^\alpha}{{(1-y)}^{1+2s}}\;dy
+\int_{1+\eps}^{+\infty}\frac{1-y^\alpha}{{(y-1)}^{1+2s}}\;dy
\right] \\
&=
c_{1,s}\int_0^{+\infty}\frac{dy}{{(1+y)}^{1+2s}}+
c_{1,s}\lim_{\eps\searrow 0}\left[
\frac{1-y^\alpha}{2s(1-y)^{2s} }\Big|_{y=0}^{1-\eps}
+\frac\alpha{2s}\int_0^{1-\eps}\frac{y^{\alpha-1}}{{(1-y)}^{2s}}\;dy \right. \\
& \qquad\qquad\left.
-\frac{1-y^\alpha}{2s {(y-1)}^{2s}}\Big|_{y=1+\eps}^{+\infty}
-\frac\alpha{2s}\int_{1+\eps}^{+\infty}\frac{y^{\alpha-1}}{{(y-1)}^{2s}}\;dy
\right] \\
&=
\frac{c_{1,s}}{2s}+
c_{1,s}\lim_{\eps\searrow 0}\left[\frac{1-(1-\eps)^\alpha}{2s \eps^{2s}}-\frac1{2s}
+\frac\alpha{2s}\int_0^{1-\eps}\frac{y^{\alpha-1}}{{(1-y)}^{2s}}\;dy \right. \\
& \qquad\qquad\left.
+\frac{1-(1+\eps)^\alpha}{2s\eps^{2s}}
-\frac\alpha{2s}\int_{1+\eps}^{+\infty}\frac{y^{\alpha-1}}{{(y-1)}^{2s}}\;dy
\right] \\
&=
\frac{c_{1,s}}{2s}\lim_{\eps\searrow 0}\frac{2-(1-\eps)^\alpha-(1+\eps)^\alpha}{\eps^{2s}} \\
& \qquad\qquad
+\frac{\alpha c_{1,s}}{2s}\lim_{\eps\searrow 0}\left[
\int_0^{1-\eps}\frac{y^{\alpha-1}}{{(1-y)}^{2s}}\;dy 
-\int_{1+\eps}^{+\infty}\frac{y^{\alpha-1}}{{(y-1)}^{2s}}\;dy
\right] \\
&=
\frac{c_{1,s}}{2s}\lim_{\eps\searrow 0}\eps^{-2s}\left(
2-1+\alpha\eps -\frac12(\alpha-1)\alpha\eps^2-1-\alpha\eps-\frac12(\alpha-1)\alpha\eps^2+o(\eps^2)
\right) \\
&  \qquad\qquad
+\frac{\alpha c_{1,s}}{2s}\lim_{\eps\searrow 0}\left[
\int_0^{1-\eps}\frac{y^{\alpha-1}}{{(1-y)}^{2s}}\;dy 
-\int_{1+\eps}^{+\infty}\frac{y^{\alpha-1}}{{(y-1)}^{2s}}\;dy
\right] \\
&=
\frac{\alpha c_{1,s}}{2s}\lim_{\eps\searrow 0}\left[
\int_0^{1-\eps}\frac{y^{\alpha-1}}{{(1-y)}^{2s}}\;dy 
-\int_{1+\eps}^{+\infty}\frac{y^{\alpha-1}}{{(y-1)}^{2s}}\;dy
\right] .
\end{align*}
We transform the two integrals above via the change of variables respectively~$t:=1-y$ and~$t:=(y-1)/y$ and we find that
\begin{equation}\label{usingthis547593hgvddkjf859}
{(-\Delta)}^s h_\alpha(1)=
\frac{\alpha c_{1,s}}{2s}\lim_{\eps\searrow 0}\left[\int_\eps^1\frac{(1-t)^{\alpha-1}}{t^{2s}}\;dt-\int_{\eps/(1+\eps)}^1\frac{{(1-t)}^{2s-\alpha-1}}{t^{2s}}\;dt\right].\end{equation}

We notice that~$(1-t)^{2s-\alpha-1}$ remains bounded for~$t$ close to~$0$ and therefore, for some~$C>0$,
\begin{align*}
\left|\int_{\eps/(1+\eps)}^\eps\frac{{(1-t)}^{2s-\alpha-1}}{t^{2s}}\;dt\right|
\leq C\int_{\eps/(1+\eps)}^\eps\frac{dt}{t^{2s}}=\frac{C\eps^{1-2s}}{2s-1}\Big((1+\eps)^{2s-1}-1\Big).
\end{align*}
This implies that
$$\lim_{\epsilon\searrow0}\left|\int_{\eps/(1+\eps)}^\eps\frac{{(1-t)}^{2s-\alpha-1}}{t^{2s}}\;dt\right|\le
\lim_{\epsilon\searrow0}\frac{C\eps^{1-2s}}{2s-1}\Big((1+\eps)^{2s-1}-1\Big)=0.$$
Plugging this information into~\eqref{usingthis547593hgvddkjf859} we thus find that
\begin{eqnarray*}
{(-\Delta)}^s h_\alpha(1)&=&\frac{\alpha c_{1,s}}{2s}\lim_{\eps\searrow 0}\left[\int_\eps^1\frac{(1-t)^{\alpha-1}}{t^{2s}}\;dt-\int_{\eps}^1\frac{{(1-t)}^{2s-\alpha-1}}{t^{2s}}\;dt\right]\\
&=&
\frac{\alpha c_{1,s}}{2s}\int_0^1\frac{(1-t)^{\alpha-1}-(1-t)^{2s-\alpha-1}}{t^{2s}}\;dt.
\end{eqnarray*}

The value of~$\kappa(\alpha)$ in~\eqref{powers-half-line} is therefore given by
\begin{align}\label{thrjekwfbdnsfgdhsjt54y3ut54y3564738}
\kappa(\alpha)=\frac{\alpha}{2s}\int_0^1\frac{(1-t)^{\alpha-1}-(1-t)^{2s-\alpha-1}}{t^{2s}}\;dt
\qquad\text{for~$\alpha\in(0,2s)$}. 
\end{align}
Notice that, in this case,
\begin{align}\label{thrjekwfbdnsfgdhsjt54y3ut54y35647382}
\kappa(\alpha)\ \left\lbrace\begin{aligned}
& <0 && \text{for }\alpha\in(s,2s), \\
& =0 && \text{for }\alpha=s, \\
& >0 && \text{for }\alpha\in(0,s),
\end{aligned}\right.
\end{align}
which is in agreement with~\eqref{kappalfa}.

For~$\alpha\in(-1,2s-1)$ we use the point inversion in~\eqref{kelvin-00}.
Notice indeed that, for any~$x\in\R\setminus\{0\}$,
\begin{align*}
\big(h_\alpha\big)_{\mathcal{K}}(x)=|x|^{2s-1} h_{\alpha}\left(\frac{x}{|x|^2}\right)=
h_{2s-\alpha-1}(x).
\end{align*}
Moreover, using~\eqref{fl-inversion}, we see that
\begin{align*}
{(-\Delta)}^s h_\alpha(1)={(-\Delta)}^s 
\big(h_\alpha\big)_{\mathcal{K}}(1)={(-\Delta)}^s h_{2s-\alpha-1}(1)=
c_{1,s}\kappa(2s-\alpha-1).
\end{align*}
Notice also that~$2s-\alpha-1\in(0,2s)$.
As a consequence, we have that
\begin{equation}\label{ytrueiobfnd1qaszxcder456tyhnu89o}
\kappa(\alpha)=\kappa(2s-\alpha-1)\qquad{\mbox{for all~$\alpha\in(-1,2s-1)\cup(0,2s)$}}.\end{equation}

Now, if~$s\in\left(\frac12,1\right)$, this identity holds true for all~$\alpha\in(-1,2s)$, and therefore the value
of~$\kappa(\alpha)$ and the claim in~\eqref{kappalfa} follow from~\eqref{thrjekwfbdnsfgdhsjt54y3ut54y3564738}
and~\eqref{thrjekwfbdnsfgdhsjt54y3ut54y35647382}.

If~$s\in\left(0,\frac12\right]$, the value
of~$\kappa(\alpha)$ and the claim in~\eqref{kappalfa} follow
from~\eqref{thrjekwfbdnsfgdhsjt54y3ut54y3564738}, \eqref{thrjekwfbdnsfgdhsjt54y3ut54y35647382}
and~\eqref{ytrueiobfnd1qaszxcder456tyhnu89o}
when~$\alpha\in(-1,2s-1)\cup(0,2s)$. 

So, we are left with the case~$s\in\left(0,\frac12\right]$
and~$\alpha\in[2s-1,0]$. In this range, 
the particular case~$\alpha=0$ is easy, because
\begin{align*}
{(-\Delta)}^s h_0(1)=c_{1,s}\int_{-\infty}^0\frac{dy}{{(1-y)}^{1+2s}}=\frac{c_{1,s}}{2s}>0.
\end{align*}
Furthermore, when~$\alpha\in[2s-1,0)$, we have
\begin{align*}
{(-\Delta)}^s h_\alpha(1)
&=
c_{1,s}\int_0^{+\infty}\frac{1-y^\alpha}{{|1-y|}^{1+2s}}\;dy
+c_{1,s}\int_{-\infty}^0\frac{dy}{{(1-y)}^{1+2s}} \\
&=
c_{1,s}\int_0^1\frac{1-y^\alpha}{{(1-y)}^{1+2s}}\;dy
+c_{1,s}\int_1^{+\infty}\frac{1-y^\alpha}{{(y-1)}^{1+2s}}\;dy
+\frac{c_{1,s}}{2s}.
\end{align*}
We now transform the integral over~$(1,+\infty)$ via
the change of variable~$t:=1/y$, thus obtaining that
\begin{align*}
{(-\Delta)}^s h_\alpha(1)
&=
c_{1,s}\int_0^1\frac{1-y^\alpha}{{(1-y)}^{1+2s}}\;dy
+c_{1,s}\int_0^1\frac{1-t^{-\alpha}}{{(1-t)}^{1+2s}}\,t^{2s-1}\;dt
+\frac{c_{1,s}}{2s} \\
&=
-c_{1,s}\int_0^1\frac{y^\alpha-1}{{(1-y)}^{1+2s}}\;dy
+c_{1,s}\int_0^1\frac{t^\alpha-1}{{(1-t)}^{1+2s}}\,t^{2s-\alpha-1}\;dt
+\frac{c_{1,s}}{2s} \\
&=
c_{1,s}\int_0^1\frac{\big(y^\alpha-1\big)\big(y^{2s-\alpha-1}-1\big)}{{(1-y)}^{1+2s}}\;dy
+\frac{c_{1,s}}{2s}.
\end{align*}
Hence, in this case,
$$ \kappa(\alpha)=c_{1,s}\int_0^1\frac{\big(y^\alpha-1\big)\big(y^{2s-\alpha-1}-1\big)}{{(1-y)}^{1+2s}}\;dy
+\frac{c_{1,s}}{2s},$$
which is positive, since~$\alpha<0$ and~$2s-\alpha-1\leq 0$.
This also establishes the claim in~\eqref{kappalfa}.
\end{proof}

We point out that, by symmetry, a similar statement to Lemma~\ref{lem:halfline}
holds true for the function~$x_-^\alpha$.

Let now~$\alpha\in(-1,2s)$ and
\begin{align*}
r_\alpha(x):=\left\lbrace\begin{aligned}
& |x|^\alpha && \text{for }x\neq 0, \\
& 0 && \text{for }x=0.
\end{aligned}\right.
\end{align*}
As a direct consequence of Lemma~\ref{lem:halfline}, we obtain the following result.

\begin{proposition}\label{prop:fullline}
We have that
\begin{align}\label{powers-full-line}
(-\Delta)^s r_\alpha(x)=
c_{1,s}\left(\kappa(\alpha)-\frac{\Gamma(\alpha+1)\,\Gamma(2s-\alpha)}{\Gamma(2s+1)}\right)
|x|^{\alpha-2s} \qquad\text{for }x\neq 0,
\end{align}
where~$\kappa(\alpha)$ is the one given in Lemma~\ref{lem:halfline}.
\end{proposition}

\begin{proof}
We observe that, for all~$x\in\R\setminus\{0\}$,
$$ r_\alpha(x)= |x|^\alpha= x_+^\alpha+x_-^\alpha $$
and therefore we deduce from Lemma~\ref{lem:halfline} that
\begin{eqnarray*}
(-\Delta)^s r_\alpha(x)=(-\Delta)^s  x_+^\alpha+ (-\Delta)^s x_-^\alpha
= c_{1,s}\kappa(\alpha)|x|^{\alpha-2s} -c_{1,s}\,\frac{\Gamma(\alpha+1)\,\Gamma(2s-\alpha)}{\Gamma(2s+1)}\,|x|^{\alpha-2s},
\end{eqnarray*} 
as desired.
\end{proof}

Recall that, from Proposition~\ref{prop:inversion}, 
and in particular from~\eqref{fundsol},
we already know that
\begin{align*}
(-\Delta)^s r_{2s-1}(x)=0
\qquad\text{for }x\neq 0.
\end{align*}
This fact and Proposition~\ref{prop:fullline}
give the particular equality
\begin{align*}
\kappa(2s-1)=\frac{\Gamma(2s)\,\Gamma(1)}{\Gamma(2s+1)}=\frac1{2s}.
\end{align*}

\subsection{Power functions on intervals}

For~$\alpha>-1$, we consider the function
\begin{align}\label{ytruieow54738564783574839fgdhsja}
u_\alpha(x):=\left\lbrace\begin{aligned}
& \big(1-x^2\big)^\alpha && \text{if }|x|< 1 \\
& 0 && \text{if } |x|\geq 1.
\end{aligned}\right.
\end{align}
In this case one has the following result:

\begin{lemma}\label{lemma:p405475ryfrg}
We have that
\begin{align}\label{fl-ball-1d}
\begin{split}
& {(-\Delta)}^s u_\alpha(x) \\
&=\left\lbrace\begin{aligned}
& -c_{1,s}\frac{\Gamma(\alpha+1)\,\Gamma(-s)}{\Gamma(\alpha-s+1)}
{(1-x^2)}^{\alpha-2s}
\hf\left(-s,\alpha-s+\frac12;\frac12\,\bigg|\,x^2\right) 
\qquad \text{if }|x|< 1 \\
\\
& -c_{1,s}\,2^{2\alpha+1}\frac{\Gamma(\alpha+1)^2}{\Gamma(2\alpha+2)}
{(x^2-1)}^{\alpha-2s}{(|x|+1)}^{2s-2\alpha-1}\hf\left(2\alpha-2s+1,\alpha+1;2\alpha+2\,\bigg|\,\frac2{|x|+1}\right) \\
& \hspace*{.7\linewidth} \text{if }|x|>1.
\end{aligned}\right.
\end{split}
\end{align}
\end{lemma}

Here above and in what follows, $\hf\big(a,b;c|x\big)$ denotes the
hypergeometric function
defined in~\eqref{def:hyp}.

\begin{proof}[Proof of Lemma~\ref{lemma:p405475ryfrg}]
Let us start from the easier part, which is the case~$|x|>1$ when~$u_\alpha(x)=0$. In this situation, we have that
\begin{align*}
{(-\Delta)}^s u_\alpha(x)=c_{1,s}\int_\R\frac{-u_\alpha(y)}{{|x-y|}^{1+2s}}\;dy
=-c_{1,s}\int_{-1}^1\frac{{(1-y^2)}^\alpha}{{|x-y|}^{1+2s}}\;dy.
\end{align*}
To compute the above integral, we apply the change of variable~$t:=(1-|x|y)/(|x|-y)$
which gives
\begin{align*}
y=\frac{1-|x|t}{|x|-t},\qquad 
1-y^2=\frac{(x^2-1)(1-t^2)}{{(|x|-t)}^2}, \qquad
dy=\frac{x^2-1}{(|x|-t)^2}\,dt,
\end{align*}
and therefore
\begin{align*}
\int_{-1}^1\frac{{(1-y^2)}^\alpha}{{|x-y|}^{1+2s}}\;dy=
{(x^2-1)}^{\alpha-2s}\int_{-1}^1{(1-t^2)}^\alpha{(|x|-t)}^{2s-2\alpha-1}\;dt.
\end{align*}
We now apply a second change of variable, namely~$\tau:=(t+1)/2$, and get
\begin{align*}&
\int_{-1}^1\frac{{(1-y^2)}^\alpha}{{|x-y|}^{1+2s}}\;dy \\&=
2^{2\alpha+1}{(x^2-1)}^{\alpha-2s}{(|x|+1)}^{2s-2\alpha-1}
\int_0^1\tau^\alpha{(1-\tau)}^\alpha\left(1-\frac{2\tau}{|x|+1}\right)^{2s-2\alpha-1}\;d\tau \\
&=
2^{2\alpha+1}\frac{\Gamma(\alpha+1)^2}{\Gamma(2\alpha+2)}
{(x^2-1)}^{\alpha-2s}{(|x|+1)}^{2s-2\alpha-1} \\
& \qquad\qquad 
\times\hf\left(2\alpha-2s+1,\alpha+1;2\alpha+2\,\bigg|\,\frac2{|x|+1}\right),
\end{align*}
where we have used~\eqref{hyp:int1} with~$a:=2\alpha-2s+1$,
$b:=\alpha+1$ and~$c:=2\alpha+2$.
This completes the proof of~\eqref{fl-ball-1d} when~$|x|>1$.

The case~$|x|<1$ is more involved. 
The computations presented here are based on~\cites{zbMATH05836420,MR2974318}.
We note that
\begin{equation}\label{jfwoiehfnwfklnczxlkh}\begin{split}
{(-\Delta)}^su_\alpha(x) &=
c_{1,s}\pv\int_\R\frac{u_\alpha(x)-u_\alpha(y)}{{|x-y|}^{1+2s}}\;dy \\
&=c_{1,s}\int_{-\infty}^{-1}\frac{{(1-x^2)}^\alpha}{{(x-y)}^{1+2s}}\;dy
+c_{1,s}\pv\int_{-1}^1\frac{{(1-x^2)}^\alpha-{(1-y^2)}^\alpha}{{|x-y|}^{1+2s}}\;dy
\\&\qquad\qquad+c_{1,s}\int_1^{+\infty}\frac{{(1-x^2)}^\alpha}{{(y-x)}^{1+2s}}\;dy. 
\end{split}\end{equation}
The first and third integral can be directly computed. Indeed,
\begin{align}\label{jdjdfjfjfjgnvndjdmxmcj}
\int_1^{+\infty}\frac{{(1-x^2)}^\alpha}{{(y-x)}^{1+2s}}\;dy
={(1-x^2)}^\alpha\frac{-1}{2s}{(y-x)}^{-2s}\bigg|_{y=1}^{+\infty}
=\frac{1}{2s}{(1-x^2)}^\alpha{(1-x)}^{-2s}
\end{align}
and similarly
\begin{align}\label{jdjdfjfjfjgnvndjdmxmcj22}
\int_{-\infty}^{-1}\frac{{(1-x^2)}^\alpha}{{(x-y)}^{1+2s}}\;dy
=\frac{1}{2s}{(1-x^2)}^\alpha{(1+x)}^{-2s}.
\end{align}

Accordingly, we are left with the evaluation of the middle term in~\eqref{jfwoiehfnwfklnczxlkh}, which is quite elaborate. We start off with the change of variable~$t:=(x-y)/(1-xy)$, which gives
\begin{align*}
y=\frac{x-t}{1-xt},\qquad 1-y^2=\frac{(1-x^2)(1-t^2)}{(1-xt)^2},
\qquad dy=\frac{1-x^2}{(1-xt)^2}\,dt,
\end{align*}
and so
\begin{eqnarray*}&&
\pv\int_{-1}^1\frac{{(1-x^2)}^\alpha-{(1-y^2)}^\alpha}{{|x-y|}^{1+2s}}\;dy
\\&&\qquad={(1-x^2)}^{\alpha-2s}\pv\int_{-1}^1\frac{1-{(1-t^2)}^\alpha{(1-xt)}^{-2\alpha}}{{|t|}^{1+2s}}\,{(1-xt)}^{2s-1}\;dt.
\end{eqnarray*}
We write
\begin{eqnarray*}&&
\big[1-{(1-t^2)}^\alpha{(1-xt)}^{-2\alpha}\big]{(1-xt)}^{2s-1}\\&&\qquad
={(1-xt)}^{2s-1}-1+1-{(1-t^2)}^\alpha+{(1-t^2)}^\alpha\big[1-(1-xt)^{2s-1-2\alpha}\big]
.\end{eqnarray*}
As a result,\begin{align}
\pv\int_{-1}^1\frac{{(1-x^2)}^\alpha-{(1-y^2)}^\alpha}{{|x-y|}^{1+2s}}\;dy
=
& {(1-x^2)}^{\alpha-2s} \notag \\
&\qquad \times\bigg[\pv\int_{-1}^1\frac{(1-xt)^{2s-1}-1}{{|t|}^{1+2s}}\;dt \label{84u2roi}\\
&\qquad+\pv\int_{-1}^1\frac{1-{(1-t^2)}^\alpha}{{|t|}^{1+2s}}\;dt \label{8924h3oinf}\\
&\qquad+\pv\int_{-1}^1\frac{1-(1-xt)^{2s-1-2\alpha}}{{|t|}^{1+2s}}\,{(1-t^2)}^\alpha\;dt\bigg]. \label{riu2o4ejfwn}
\end{align}
Now we proceed with the computations of~\eqref{84u2roi}, \eqref{8924h3oinf} and~\eqref{riu2o4ejfwn} separately.

We first deal with the term in~\eqref{84u2roi}. Changing variable~$\eta:=-t$, we have that
\begin{equation*}\begin{split}
& \pv\int_{-1}^1\frac{(1-xt)^{2s-1}-1}{{|t|}^{1+2s}}\;dt
\\&=\lim_{\eps\searrow 0} \int_{\epsilon}^1\frac{(1-xt)^{2s-1}-1}{{|t|}^{1+2s}}\;dt+
\int_{-1}^{-\epsilon}\frac{(1-xt)^{2s-1}-1}{{|t|}^{1+2s}}\;dt\\
&=\lim_{\eps\searrow 0} \int_{\epsilon}^1\frac{(1-xt)^{2s-1}-1}{{|t|}^{1+2s}}\;dt+
\int_{\epsilon}^1\frac{(1+x\eta)^{2s-1}-1}{{|\eta|}^{1+2s}}\;d\eta \\
&=
\lim_{\eps\searrow 0}\int_\eps^1\frac{(1-xt)^{2s-1}+(1+xt)^{2s-1}-2}{t^{1+2s}}\;dt.
\end{split}\end{equation*}
We now split the three integrals in the last line above and we see that
\begin{equation} \label{84u2roi2222}\begin{split}
& \pv\int_{-1}^1\frac{(1-xt)^{2s-1}-1}{{|t|}^{1+2s}}\;dt\\&=
\lim_{\eps\searrow 0}\left[\int_\eps^1\left(\frac1t-x\right)^{2s-1}\frac{dt}{t^2}
+\int_\eps^1\left(\frac1t+x\right)^{2s-1}\frac{dt}{t^2}
-2\int_{\eps}^1\frac{dt}{t^{1+2s}}
\right]  \\
&=
\lim_{\eps\searrow 0}\left[
-\frac1{2s}(1-x)^{2s}
+\frac1{2s}\left(\frac1{\eps}-x\right)^{2s}
-\frac1{2s}(1+x)^{2s}
-\frac1{2s}\left(\frac1{\eps}+x\right)^{2s}
+\frac{1-\eps^{-2s}}{s}\right]   \\
&=
\lim_{\eps\searrow 0}\frac1{2s}\left[2-(1-x)^{2s}-(1+x)^{2s}
+\eps^{-2s}\left(
(1-\eps x)^{2s}+(1+\eps x)^{2s}-2\right)\right]  \\
&=\frac1{2s}\left[2-(1-x)^{2s}-(1+x)^{2s}\right].
\end{split}\end{equation}

To compute the term in~\eqref{8924h3oinf}, we notice that
\begin{equation}\label{mjhy54tgdfrgvcf3wesf43547547}\begin{split}
\pv\int_{-1}^1\frac{1-{(1-t^2)}^\alpha}{{|t|}^{1+2s}}\;dt
&=
\lim_{\eps\searrow 0}2\int_\eps^1\frac{1-{(1-t^2)}^\alpha}{t^{1+2s}}\;dt\\&=
\lim_{\eps\searrow 0}2\int_\eps^1\frac{dt}{t^{1+2s}}-
2\int_\eps^1\frac{{(1-t^2)}^\alpha}{t^{1+2s}}\;dt\\&=
\lim_{\eps\searrow 0}\frac{\eps^{-2s}-1}{s}-
2\int_\eps^1\frac{{(1-t^2)}^\alpha}{t^{1+2s}}\;dt
.
\end{split}\end{equation}
Now, changing variable~$\tau:=t^2$, we have that
$$ 2\int_\eps^1\frac{{(1-t^2)}^\alpha}{t^{1+2s}}\;dt
=\int_{\eps^2}^1\frac{{(1-\tau)}^\alpha}{\tau^{1+s}}\;d\tau.$$
Also, we observe that
\begin{eqnarray*}
&&\frac{{(1-\tau)}^\alpha}{\tau^{1+s}}=\frac{{(1-\tau)}^\alpha(1-\tau+\tau)}{\tau^{1+s}}
=\frac{{(1-\tau)}^{\alpha+1}}{\tau^{1+s}}+\frac{{(1-\tau)}^\alpha}{\tau^{s}}
\end{eqnarray*}
and therefore
$$ 2\int_\eps^1\frac{{(1-t^2)}^\alpha}{t^{1+2s}}\;dt
=\int_{\eps^2}^1\frac{{(1-\tau)}^{\alpha+1}}{\tau^{1+s}}\;d\tau
+\int_{\eps^2}^1\frac{{(1-\tau)}^\alpha}{\tau^{s}}\;d\tau
.$$
We now integrate by parts the first integral in the right-hand side above
and find that
\begin{eqnarray*}
&&2\int_\eps^1\frac{{(1-t^2)}^\alpha}{t^{1+2s}}\;dt
\\&&=
\frac1s\left[{\big(1-\eps^2\big)}^{\alpha+1}\eps^{-2s}
-(\alpha+1)\int_{\eps^2}^1(1-\tau)^\alpha\tau^{-s}\;d\tau
+s\int_{\eps^2}^1(1-\tau)^\alpha\tau^{-s}\;d\tau\right] \\
&&=
\frac1s\left[{\big(1-\eps^2\big)}^{\alpha+1}\eps^{-2s}
-(\alpha+1-s)\int_{\eps^2}^1(1-\tau)^\alpha\tau^{-s}\;d\tau\right]   .\end{eqnarray*}

Plugging this information into~\eqref{mjhy54tgdfrgvcf3wesf43547547}, we thereby obtain that
\begin{equation*}\begin{split}
&\pv\int_{-1}^1\frac{1-{(1-t^2)}^\alpha}{{|t|}^{1+2s}}\;dt\\
&=\frac1{s}
\lim_{\eps\searrow 0}\left[ \eps^{-2s}-1-
{\big(1-\eps^2\big)}^{\alpha+1}\eps^{-2s}
+(\alpha+1-s)\int_{\eps^2}^1(1-\tau)^\alpha\tau^{-s}\;d\tau\right]\\&=\frac1{s}
\left[
(\alpha+1-s)\int_{0}^1(1-\tau)^\alpha\tau^{-s}\;d\tau -1\right]
.
\end{split}\end{equation*}
Also, using~\eqref{beta-identity} with~$a:=\alpha+1$ and~$b:=1-s$, and then~\eqref{gamma-recursive} and~\eqref{gamma-recursive2},
\begin{equation}\label{8924h3oinf2222}\begin{split}
\pv\int_{-1}^1\frac{1-{(1-t^2)}^\alpha}{{|t|}^{1+2s}}\;dt&=\frac1{s}
\left[
(\alpha+1-s)\frac{\Gamma(\alpha+1)\,\Gamma(1-s)}{\Gamma(\alpha+2-s)} -1\right]\\&=
-\frac{\Gamma(\alpha+1)\,\Gamma(-s)}{\Gamma(\alpha+1-s)}-\frac1s
.
\end{split}\end{equation}

We finally compute the term in~\eqref{riu2o4ejfwn}. For this, we use the series expansion
\begin{align*}
(1-xt)^{2s-1-2\alpha}-1=\sum_{j=1}^{+\infty}\binom{2s-1-2\alpha}{j}(-xt)^j
=\sum_{j=1}^{+\infty}(-1)^j\frac{\Gamma(2s-2\alpha)}{\Gamma(2s-2\alpha-j)\,j!}x^jt^j.
\end{align*}
As the integral in~\eqref{riu2o4ejfwn} is on a symmetric interval, we can discard odd powers in the above expansion. This entails that
\begin{eqnarray*}&&
\pv\int_{-1}^1\frac{1-(1-xt)^{2s-1-2\alpha}}{{|t|}^{1+2s}}\,{(1-t^2)}^\alpha\;dt
\\&&=-\sum_{k=1}^{+\infty}\frac{\Gamma(2s-2\alpha)}{\Gamma(2s-2\alpha-2k)\,(2k)!}x^{2k}
\int_{-1}^1{|t|}^{2k-1-2s}\,{(1-t^2)}^\alpha\;dt.
\end{eqnarray*}
By the change of variable~$\tau:=t^2$ and~\eqref{beta-identity} with~$a:=k-s$ and~$b:=\alpha+1$, we see that
\begin{eqnarray*}&&
\int_{-1}^1{|t|}^{2k-1-2s}\,{(1-t^2)}^\alpha\;dt=
2\int_0^1 t^{2k-1-2s}\,{(1-t^2)}^\alpha\;dt\\&&\qquad=
\int_{-1}^1\tau^{k-1-s}\,(1-\tau)^\alpha\;d\tau
=\frac{\Gamma(k-s)\,\Gamma(\alpha+1)}{\Gamma(k-s+\alpha+1)}.
\end{eqnarray*}
As a result,
\begin{equation}\label{758493wsdahgtxzxcfvg}\begin{split}
&\pv\int_{-1}^1\frac{1-(1-xt)^{2s-1-2\alpha}}{{|t|}^{1+2s}}\,{(1-t^2)}^\alpha\;dt
\\&=
-\sum_{k=1}^{+\infty}\frac{\Gamma(2s-2\alpha)}{\Gamma(2s-2\alpha-2k)\,(2k)!}
\frac{\Gamma(k-s)\,\Gamma(\alpha+1)}{\Gamma(k-s+\alpha+1)}\,x^{2k}.
\end{split}\end{equation}

This series expansion can be transformed using some identities on the Gamma Function 
(see~\eqref{gamma-integer}, \eqref{gamma-dupli} and~\eqref{euler-reflection}), according to which
\begin{align*}
\Gamma(2s-2\alpha) &=
\frac{\pi}{\sin\big(\pi(2s-2\alpha)\big)}
\frac1{\Gamma(2\alpha-2s+1)}\\&=
\frac{\pi}{\sin\big(\pi(2s-2\alpha)\big)}\frac{2^{2s-2\alpha}\sqrt\pi}{\Gamma\big(\alpha-s+\frac12\big)\,\Gamma(\alpha-s+1)},
\end{align*}
$$
(2k)!=\Gamma(2k+1) = \frac{2^{2k}}{\sqrt\pi}\Gamma\left(
k+\frac12\right)\,\Gamma(k+1)=\frac{2^{2k}}{\sqrt\pi}\Gamma\left(k+\frac12\right)\,k! $$
and
\begin{align*}
\Gamma(k-s+\alpha+1) &= 2^{2s-2\alpha-2k}\sqrt\pi\,\frac{\Gamma(2k-2s+2\alpha+1)}{\Gamma\big(k-s+\alpha+\frac12\big)} \\
&=2^{2s-2\alpha-2k}\frac{\pi^{3/2}}{\sin\big(\pi(2s-2\alpha-2k)\big)}\frac1{\Gamma(2s-2\alpha-2k)\,\Gamma\big(k-s+\alpha+\frac12\big)}.
\end{align*}
In this way, after some cancellations,
\begin{align*}&
\frac{\Gamma(2s-2\alpha)}{\Gamma(2s-2\alpha-2k)\,(2k)!}
\frac1{\Gamma(k-s+\alpha+1)}\\&\qquad=
\frac{\sqrt\pi\,\Gamma\big(k-s+\alpha+\frac12\big)}{\Gamma\big(\alpha-s+\frac12\big)\,\Gamma(\alpha-s+1)\,\Gamma\big(k+\frac12\big)\,k!}
.
\end{align*}

Therefore, from this and~\eqref{758493wsdahgtxzxcfvg} we get that
\begin{equation}\label{riu2o4ejfwn2222}\begin{split}
& \pv\int_{-1}^1\frac{1-(1-xt)^{2s-1-2\alpha}}{{|t|}^{1+2s}}\,{(1-t^2)}^\alpha\;dt  \\
&=
\frac{-\sqrt\pi\,\Gamma(\alpha+1)}{\Gamma\big(\alpha-s+\frac12\big)\,\Gamma(\alpha-s+1)}\sum_{k=1}^{+\infty}\frac{\Gamma(k-s)\,\Gamma\big(k-s+\alpha+\frac12\big)}{\Gamma\big(k+\frac12\big)}\frac{x^{2k}}{k!}  \\
&=
-\frac{\Gamma(\alpha+1)\,\Gamma(-s)}{\Gamma(\alpha-s+1)}
\left[\hf\left(-s,\alpha-s+\frac12;\frac12\,\bigg|\,x^2\right)
-1\right]. 
\end{split}\end{equation}

Collecting the new representations~\eqref{84u2roi2222},
\eqref{8924h3oinf2222} and~\eqref{riu2o4ejfwn2222} for~\eqref{84u2roi},
\eqref{8924h3oinf} and~\eqref{riu2o4ejfwn} respectively, 
we obtain that 
\begin{eqnarray*}&&
\pv\int_{-1}^1\frac{{(1-x^2)}^\alpha-{(1-y^2)}^\alpha}{{|x-y|}^{1+2s}}\;dy\\
&&=-
{(1-x^2)}^{\alpha-2s}\left[
\frac1{2s}\left(
(1-x)^{2s}+(1+x)^{2s}\right) 
+\frac{\Gamma(\alpha+1)\,\Gamma(-s)}{\Gamma(\alpha-s+1)}
\hf\left(-s,\alpha-s+\frac12;\frac12\,\bigg|\,x^2\right)
\right].
\end{eqnarray*}
Plugging this identity into~\eqref{jfwoiehfnwfklnczxlkh}, and recalling~\eqref{jdjdfjfjfjgnvndjdmxmcj} and~\eqref{jdjdfjfjfjgnvndjdmxmcj22}, we conclude that
\begin{align*}
{(-\Delta)}^su_\alpha(x)=-c_{1,s}\frac{\Gamma(\alpha+1)\,\Gamma(-s)}{\Gamma(\alpha-s+1)}
{(1-x^2)}^{\alpha-2s}
\hf\left(-s,\alpha-s+\frac12;\frac12\,\bigg|\,x^2\right)
\qquad\text{for }|x|<1,
\end{align*}
as desired.
\end{proof}

Recalling the definition of~$c_{1,s}$ given in~\eqref{cns}, from~\eqref{fl-ball-1d} we may also write
\begin{align*}
{(-\Delta)}^su_\alpha(x)=\frac{2^{2s}\Gamma\big(\frac12+s\big)\,\Gamma(\alpha+1)}{\sqrt\pi\,\Gamma(\alpha-s+1)}
{(1-x^2)}^{\alpha-2s}
\hf\left(-s,\alpha-s+\frac12;\frac12\,\bigg|\,x^2\right)
\qquad\text{for }|x|<1,
\end{align*}
or, using the hypergeometric transformation~\eqref{hyp-transf1}, also
\begin{align}\label{c0c0c0c0c0c0c0c}
{(-\Delta)}^su_\alpha(x)=\frac{2^{2s}\Gamma\big(\frac12+s\big)\,\Gamma(\alpha+1)}{\sqrt\pi\,\Gamma(\alpha-s+1)}
\hf\left(s-\alpha,\frac12+s;\frac12\,\bigg|\,x^2\right)
\qquad\text{for }|x|<1.
\end{align}
Equation~\eqref{c0c0c0c0c0c0c0c} has some interesting consequences,
by performing suitable choices of the power~$\alpha$, as we now illustrate.

\begin{proposition}
We have that
\begin{eqnarray}&&
{(-\Delta)}^su_s(x) =
\frac{2^{2s}\Gamma\big(\frac12+s\big)\,\Gamma(1+s)}{\sqrt\pi}
\qquad \text{for }|x|<1, \label{torsion1d} \\
{\mbox{and }} && {(-\Delta)}^su_{s-1}(x) =0
\qquad \text{for }|x|<1. \label{sing1d}
\end{eqnarray}
Moreover, for any~$m\in\N$,
\begin{align}\label{poly1d}
{(-\Delta)}^su_{s+m}(x)
\quad\text{is a polynomial of degree~$2m$ in~$x$ for~$|x|<1$.}
\end{align}
\end{proposition}

\begin{proof}
Equation~\eqref{c0c0c0c0c0c0c0c} gives that
\begin{align*}
{(-\Delta)}^su_s(x) =\frac{2^{2s}\Gamma\big(\frac12+s\big)\,\Gamma(1+s)}{\sqrt\pi}
\hf\left(0,\frac12+s;\frac12\,\bigg|\,x^2\right).
\end{align*}
{F}rom~\eqref{hyp(0)} we also have that
\begin{align*}
\hf\left(0,\frac12+s;\frac12\,\bigg|\,x^2\right)=1
\end{align*}
and this completes the proof of~\eqref{torsion1d}.

For~$\alpha:=s+m$, with~$m\in\N$, formula~\eqref{c0c0c0c0c0c0c0c} coupled with~\eqref{hyp-poly} renders a polynomial,
and this proves~\eqref{poly1d}.

As for the proof of~\eqref{sing1d}, one can formally deduce it from the singularity 
of the factor~$\Gamma(\alpha-s+1)$ at~$\alpha=s-1$ in~\eqref{c0c0c0c0c0c0c0c}.
We present here a rigorous proof, whose details are as follows.

We recall the notation in~\eqref{isinforceBIS} and~\eqref{isinforceBIS2} and,
for the sake of readability, we denote by~$\mathcal{K}:=\mathcal{K}_{\sqrt 2,-1}$.
In this way, we consider the function~$h_\alpha$ in~\eqref{accadialfa},
with~$\alpha\in(-1,2s)$), and we can write
\begin{align*}&
\big(h_\alpha\big)_{\mathcal{K}}(x) =
{|x+1|}^{2s-1}h_\alpha\left(2\frac{x+1}{{|x+1|}^2}-1\right) \\
&\qquad=
{|x+1|}^{2s-1}h_\alpha\left(\frac{1-x^2}{{|x+1|}^2}\right)
=\left\lbrace\begin{aligned}
& {(x+1)}^{2s-1-2\alpha}\big(1-x^2\big)^\alpha
&& \text{for }|x|<1, \\
& 0 && \text{for }|x|\geq 1.
\end{aligned}\right.
\end{align*}
Now, using~\eqref{fl-inversion-general} and Lemma~\ref{lem:halfline}, we find that, for~$|x|<1$,
\begin{eqnarray*}
{(-\Delta)}^s \big(h_\alpha\big)_{\mathcal{K}}(x)
&=&2^{2s}|x+1|^{-1-2s} {(-\Delta)}^s h_\alpha \left(\mathcal{K}(x)
\right)
\\&=&2^{2s}|x+1|^{-1-2s} c_{1,s}\kappa(\alpha)
\left(\frac{1-x^2}{{|x+1|}^2}\right)^{\alpha-2s}
\\&=&
2^{2s}c_{1,s}\kappa(\alpha)|x+1|^{2s-1-2\alpha}
\big(1-x^2\big)^{\alpha-2s}.
\end{eqnarray*}

In particular, from the values of~$\kappa(\alpha)$ in~\eqref{kappalfa},
we see that
\begin{align*}
{(-\Delta)}^s \big(h_s\big)_{\mathcal{K}}(x)=0={(-\Delta)}^s \big(
h_{s-1}\big)_{\mathcal{K}}(x)
\qquad\text{for any }|x|<1.
\end{align*}
As a consequence,
\begin{align}\label{turieowfhdjsdvsb1234567890987654e3}
{(-\Delta)}^s\Big( \big(h_s\big)_{\mathcal{K}}+\big(
h_{s-1}\big)_{\mathcal{K}}\Big)(x)=0
\qquad\text{for any }|x|<1
\end{align}

Moreover, we observe that
\begin{align*}
\big(h_s\big)_{\mathcal{K}}(x) &=\left\lbrace\begin{aligned}
& \frac{\big(1-x^2\big)^s}{x+1}
&& \text{for }|x|<1, \\
& 0 && \text{for }|x|\geq 1,\end{aligned}\right.
\\&=\left\lbrace\begin{aligned}
& (1-x)^s(1+x)^{s-1}
&& \text{for }|x|<1 ,\\
& 0 && \text{for }|x|\geq 1
\end{aligned}\right.
\end{align*}
and\begin{align*}
\big(h_{s-1}\big)_{\mathcal{K}}(x)&=\left\lbrace\begin{aligned}
& (x+1)\big(1-x^2\big)^{s-1}
&& \text{for }|x|<1, \\
& 0 && \text{for }|x|\geq 1,
\end{aligned}\right. \\
&=\left\lbrace\begin{aligned}
& (1-x)^{s-1}(1+x)^s
&& \text{for }|x|<1, \\
& 0 && \text{for }|x|\geq 1
\end{aligned}\right.
\end{align*}
and therefore
\begin{align*}
\big(h_s\big)_{\mathcal{K}}(x)+\big(h_{s-1}\big)_{\mathcal{K}}(x)&=\left\lbrace\begin{aligned}
& 2\big(1-x^2\big)^{s-1}
&& \text{for }|x|<1, \\
& 0 && \text{for }|x|\geq 1,
\end{aligned}\right. \\
&=2 u_{s-1}(x).
\end{align*}
{F}rom this and~\eqref{turieowfhdjsdvsb1234567890987654e3}
we obtain the desired result in~\eqref{sing1d}.
\end{proof}

\section{Functions supported on a half-space}

Let~$\alpha>-1$ and define the function~$h_{n,\alpha}:\R^n \to\R$ as
\begin{align}\label{yueiwgdsh2wsxcfr56yhnjui890po-09876543}
h_{n,\alpha}(x):=\left\lbrace\begin{aligned}
& x_1^\alpha && \text{if }x_1>0, \\
& 0 && \text{if }x_1\leq 0,
\end{aligned}\right.
\end{align}
for all~$x=(x_1,\dots,x_n)\in\R^n$.

\begin{proposition}
We have that
\begin{align}\label{fl-halfspace}
{(-\Delta)}^s h_{n,\alpha}(x)=c_{1,s}\kappa(\alpha)x_1^{\alpha-2s}
\qquad\text{in }\{x_1>0\},
\end{align}
where~$\kappa(\alpha)$ is the
constant given by Lemma~\ref{lem:halfline}.

In particular,
\begin{align}\label{fl-halfspaceBIS}
{(-\Delta)}^s h_{n,s}(x)=0 
\qquad\text{and}\qquad
{(-\Delta)}^s h_{n,s-1}(x)=0
\qquad\text{in }\{x_1>0\}.
\end{align}
\end{proposition}

\begin{proof}The desired result follows from
Lemmata~\ref{lem:mute} and~\ref{lem:halfline}.
\end{proof}

\section{Functions supported on a ball or an ellipsoid}

Here we compute the fractional Laplacian of some special functions possessing rotational or elliptic symmetry (see~\cite{MR4181195}
for explicit formulas of, possibly higher-order, fractional Laplacians
and applications to maximum principle results).

For this,
let~$a=(a_1,\dots,a_n)\in\R^n$, with~$a_i>0$ for any~$i\in\{1,\ldots,n\}$, and let~$A$ be the diagonal matrix induced by the vector~$a$, i.e.,
\begin{align*}
A_{ij}=a_i\delta_{i,j}
\qquad\text{for }i,j\in\{1,\ldots,n\}.
\end{align*}
For any~$x,y\in\R^n$, we set
\begin{align*}
\langle x,y\rangle_a:=Ax\cdot y=\sum_{i=1}^n a_ix_iy_i
\qquad\text{and}\qquad
|x|_a=\sqrt{\langle x,x\rangle_a},
\end{align*}
which respectively define a scalar product and a norm on~$\R^n$.

Let~$E_a\subseteq\R^n$ be the open unitary ball with respect to this new norm, that is
\begin{align*}
E_a:=\{x\in\R^n:|x|_a<1\}.
\end{align*}
{F}rom the geometric point of view,~$E_a$
corresponds to an open ellipsoid in~$\R^n$, 
reducing to~$B_1$ as soon as~$a_i=1$ for any~$i\in\{1,\ldots,n\}$. 

For~$\beta>-1$, we define
\begin{align*}
v_\beta(x):=\left\lbrace\begin{aligned}
& \big(1-|x|_a^2\big)^\beta && \text{for }x\in E_a, \\
& 0 && \text{for }x\in\R^n\setminus E_a.
\end{aligned}\right.
\end{align*}
One can calculate the fractional Laplacian of~$v_\beta$, as described by the following result:

\begin{proposition}
For all~$x\in E_a$, we have that
\begin{align}\label{fl-ellipsoids}
\begin{split}&
{(-\Delta)}^s v_\beta(x)=
\frac{c_{n,s}\,\Gamma(\beta+1)\,\Gamma(-s)}{2\Gamma(\beta-s+1)}
\\
&\qquad \times\ \int_{\partial E_a}\big(  v_1(x)+\langle x,\theta\rangle_a^2\big)^{\beta-s} \;
\hf\left(\frac12+s,s-\beta;\frac12\Big|\frac{\langle x,\theta\rangle_a^2}{ v_1(x)+\langle x,\theta\rangle_a^2}\right) \; \frac{
d{\mathcal{H}}^{n-1}_\theta}{|\theta|^{n+2s}|A\theta|}.
\end{split}
\end{align}

In particular, for all~$x\in E_a$,
\begin{eqnarray*}&&
{(-\Delta)}^s v_{s}(x) = \frac{2^{2s-1}\Gamma(1+s)\Gamma(\frac{n}2+s)}{\pi^{n/2}}
\int_{\partial E_a}\frac{d{\mathcal{H}}^{n-1}_\theta}{|\theta|^{n+2s}|A\theta|} \\
{\mbox{and }}&&
{(-\Delta)}^s v_{s-1}(x) = 0 
\end{eqnarray*}
and, for any~$m\in\N$, 
\begin{align*}
(-\Delta)^s v_{s+m}\quad\text{is a polynomial of degree~$2m$ in~$E_a$.}
\end{align*}
\end{proposition}

\begin{proof}
We consider spherical coordinates with respect to the~$a$-norm by writing
any~$y\in\R^n$ as~$y=t\theta$ with~$t>0$ and~$\theta\in\partial E_a$. 
In this way, noticing also that~$\nabla|x|_a=Ax/|x|_a$,
by the co-area formula, we see that
\begin{eqnarray*}
\int_{\R^n}f(x)\;dx&=&\int_{\R^n}\frac{f(x)}{\big|\nabla|x|_a\big|}\big|\nabla|x|_a\big|\;dx\\
&=&\int_{\R^n}\frac{f(x)\,|x|_a}{|Ax|}\big|\nabla|x|_a\big|\;dx
\\&=&
\int_0^{+\infty}\int_{\partial E_a}\frac{f(t\theta) \,|t\theta|_a}{|A t\theta|} \; d{\mathcal{H}}^{n-1}_\theta\; t^{n-1} \;dt
\\&=&
\int_0^{+\infty}\int_{\partial E_a}\frac{f(t\theta)}{|A\theta|} \; d{\mathcal{H}}^{n-1}_\theta\; t^{n-1} \;dt.
\end{eqnarray*}

Using this formula and the representation in~\eqref{pv-def}, we find that
\begin{align*}
{(-\Delta)}^s v_\beta(x) 
&= 
c_{n,s}\pv\int_{\R^n}\frac{v_\beta(x)-v_\beta(x+y)}{{|y|}^{n+2s}}\;dy \\
&=
c_{n,s}\pv\int_{\partial E_a}\int_0^{+\infty}\frac{v_\beta(x)-v_\beta(x+t\theta)}{t^{1+2s}}\;dt\;\frac{d{\mathcal{H}}^{n-1}_\theta}{|\theta|^{n+2s}|A\theta|} \\
&= 
\frac{c_{n,s}}2\int_{\partial E_a}\pv\int_\R\frac{v_\beta(x)-v_\beta(x+t\theta)}{{|t|}^{1+2s}}\;dt\;\frac{d{\mathcal{H}}^{n-1}_\theta}{|\theta|^{n+2s}|A\theta|}.
\end{align*}

We now focus on the inner integral.
We apply the change of variables
\begin{align*}
t &= -\langle x,\theta\rangle_a + \tau \sqrt{1-|x|_a^2+\langle x,\theta\rangle_a^2}
\end{align*}
and, for~$k\in\{0,1\}$, we see that
\begin{align*}
& 1-|x+kt\theta|_a^2=1-\left|
x-k\langle x,\theta\rangle_a\theta+k\tau\theta \sqrt{1-|x|_a^2+\langle x,\theta\rangle_a^2}\right|_a^2 \\
&= 
1-|x|_a^2-k^2\langle x,\theta\rangle_a^2-k^2\tau^2\Big(
1-|x|_a^2+\langle x,\theta\rangle_a^2\Big)
+2k\langle x,\theta\rangle_a^2\\&\qquad
-2k(1-k)\langle x,\theta\rangle_a\tau\sqrt{1-|x|_a^2+\langle x,\theta\rangle_a^2} \\
&=\left(
\frac{1-|x|_a^2}{1-|x|_a^2+\langle x,\theta\rangle_a^2}+(2k-k^2)\frac{\langle x,\theta\rangle_a^2}{1-|x|_a^2+\langle x,\theta\rangle_a^2}-k^2\tau^2
-2k(1-k)\tau\frac{\langle x,\theta\rangle_a}{\sqrt{1-|x|_a^2+\langle x,\theta\rangle_a^2}}
\right)\\
&\qquad \times \Big( 1-|x|_a^2+\langle x,\theta\rangle_a^2\Big) \\
&=\left(
1-(1-k)^2\frac{\langle x,\theta\rangle_a^2}{1-|x|_a^2+\langle x,\theta\rangle_a^2}-k^2\tau^2
-2k(1-k)\tau\frac{\langle x,\theta\rangle_a}{\sqrt{1-|x|_a^2+\langle x,\theta\rangle_a^2}}
\right) \Big( 1-|x|_a^2+\langle x,\theta\rangle_a^2\Big) \\
&=\left(1-\left(
(1-k)\frac{\langle x,\theta\rangle_a}{\sqrt{1-|x|_a^2+\langle x,\theta\rangle_a^2}}+k\tau\right)^2\right)\Big( 1-|x|_a^2+\langle x,\theta\rangle_a^2\Big).
\end{align*}

We now recall the function~$u_\beta$ introduced in~\eqref{ytruieow54738564783574839fgdhsja}.
We observe that, if~$x\in E_a$ then~$\langle x,\theta\rangle_a< {\sqrt{1-|x|_a^2+\langle x,\theta\rangle_a^2}}$, and therefore
utilizing the formula above with~$k=0$ we have that
\begin{eqnarray*}
v_\beta(x)&=&\big(1-|x|_a^2\big)^\beta\\&=&
\left(1-\frac{\langle x,\theta\rangle_a^2}{{1-|x|_a^2+\langle x,\theta\rangle_a^2}} \right)^\beta\Big( 1-|x|_a^2+\langle x,\theta\rangle_a^2\Big)^\beta 
\\&=&
u_\beta\left(\frac{\langle x,\theta\rangle_a}{\sqrt{1-|x|_a^2+\langle x,\theta\rangle_a^2}}\right)\Big( 1-|x|_a^2+\langle x,\theta\rangle_a^2\Big)^\beta
.\end{eqnarray*}
Similarly, exploiting the formula with~$k=1$, we get that
$$ 1-|x+t\theta|_a^2=\big(1-\tau^2\big)\Big( 1-|x|_a^2+\langle x,\theta\rangle_a^2\Big),$$
and thus~$x+t\theta\in E_a$ if and only if~$|\tau|<1$. This entails that
\begin{eqnarray*}
v_\beta(x+t\theta)=
u_\beta(\tau)\Big( 1-|x|_a^2+\langle x,\theta\rangle_a^2\Big)^\beta
. \end{eqnarray*}

We thus deduce that
\begin{align*}
&\pv\int_\R\frac{v_\beta(x)-v_\beta(x+t\theta)}{{|t|}^{1+2s}}\;dt
\\&=
\Big(1-|x|_a^2+\langle x,\theta\rangle_a^2\Big)^{\beta-s}  \pv\int_\R\frac{\displaystyle
u_\beta\left(\frac{\langle x,\theta\rangle_a}{\sqrt{1-|x|_a^2+\langle x,\theta\rangle_a^2}}\right)
-u_\beta(\tau)}
{\displaystyle\bigg|\tau-\frac{\langle x,\theta\rangle_a}{\sqrt{1-|x|_a^2+\langle x,\theta\rangle_a^2}}\bigg|^{1+2s}}\;d\tau.
\end{align*}
We let~$\tilde x_\theta := \langle x,\theta\rangle_a \, (1-|x|_a^2+\langle x,\theta\rangle_a^2)^{-1/2}$ and, using~\eqref{fl-ball-1d} and~\eqref{hyp-transf1}, we find that
\begin{align*}
& \pv\int_\R\frac{\displaystyle
u_\beta\left(\frac{\langle x,\theta\rangle_a}{\sqrt{1-|x|_a^2+\langle x,\theta\rangle_a^2}}\right)
-u_\beta(\tau)}
{\displaystyle\bigg|\tau-\frac{\langle x,\theta\rangle_a}{\sqrt{1-|x|_a^2+\langle x,\theta\rangle_a^2}}\bigg|^{1+2s}}\;d\tau  \\
&= \pv\int_\R\frac{\displaystyle
u_\beta\left(\tilde x_\theta\right)
-u_\beta(\tau)}
{\displaystyle|\tau-\tilde x_\theta|^{1+2s}}\;d\tau\\
& =
-\frac{\Gamma(\beta+1)\,\Gamma(-s)}{\Gamma(\beta-s+1)}
\big(1-\widetilde x_\theta^2\big)^{\beta-2s}
\hf\left(-s,\beta-s+\frac12;\frac12\,\bigg|\,\widetilde x_\theta^2\right) \\ 
& = 
-\frac{\Gamma(\beta+1)\,\Gamma(-s)}{\Gamma(\beta-s+1)}
\hf\left(\frac12+s,s-\beta;\frac12\,\bigg|\,\widetilde x_\theta^2\right).
\end{align*}
Therefore, for~$x\in E_a$,
\begin{align*}&
{(-\Delta)}^s u_\beta(x)\\& = 
-\frac{\Gamma(\beta+1)\,\Gamma(-s)}{\Gamma(\beta-s+1)}\frac{c_{n,s}}2
\int_{\partial E_a}\big( 1-|x|_a^2+\langle x,\theta\rangle_a^2\big)^{\beta-s} \;
\hf\left(\frac12+s,s-\beta;\frac12\Big|\widetilde x_\theta^2\right) \; \frac{d{\mathcal{H}}_\theta^{n-1}}{|\theta|^{n+2s}|A\theta|},
\end{align*}
which gives the desired result in~\eqref{fl-ellipsoids}.

In particular, 
when~$\beta=s$ we use~\eqref{propsym75655535} and~\eqref{hyp(0)}
to obtain that
\begin{eqnarray*}{(-\Delta)}^s u_s(x)&=&
-\Gamma(1+s)\,\Gamma(-s) \frac{c_{n,s}}2
\int_{\partial E_a}
\hf\left(\frac12+s,0;\frac12\Big|\widetilde x_\theta^2\right) \; \frac{d{\mathcal{H}}_\theta^{n-1}}{|\theta|^{n+2s}|A\theta|}
\\&=&-\Gamma(1+s)\,\Gamma(-s) \frac{c_{n,s}}2
\int_{\partial E_a}\frac{d{\mathcal{H}}_\theta^{n-1}}{|\theta|^{n+2s}|A\theta|}
,\end{eqnarray*}
which entails the desired result, recalling the expression of~$c_{n,s}$
in~\eqref{cns}.

When~$\beta=s-1$, we deduce from~\eqref{sing1d} and the computations
above that~$(-\Delta)^s u_{s-1}=0$ in~$E_a$.

When~$\beta=s+m$ with~$m\in\N$, we use~\eqref{hyp-poly}
and obtain that~$(-\Delta)^s u_{s+m}$ is a polynomial of degree~$2m$
in~$E_a$.
\end{proof}

It is worth mentioning that when~$A$ coincides with the identity matrix, 
then~$E_a=B_1$ and~\eqref{fl-ellipsoids} simplifies drastically.

\begin{theorem}\label{thm:flonballs}
If~$A$ is the identity matrix then~$E_a=B_1$ and, for any~$x\in B_1$, we have that
\begin{align*}
{(-\Delta)}^s u_\beta(x) =
-\frac{c_{n,s}\,\pi^{n/2}\,\Gamma(\beta+1)\,\Gamma(-s)}{\Gamma(\beta-s+1)\,\Gamma(\frac{n}2)}
\hf\left(\frac{n}2+s,s-\beta;\frac{n}2\bigg||x|^2\right).
\end{align*}
\end{theorem}

\begin{proof}
Under our assumptions, equation~\eqref{fl-ellipsoids} reduces to
\begin{align*}
{(-\Delta)}^s u_\beta(x) = 
-\frac{c_{n,s}\,\Gamma(\beta+1)\,\Gamma(-s)}{2\Gamma(\beta-s+1)}
\int_{\partial B_1}\big( 1-|x|^2+\langle x,\theta\rangle^2\big)^{\beta-s} \;
\hf\left(\frac12+s,s-\beta;\frac12\Big|\widetilde x_\theta^2\right) \; d{\mathcal{H}}_\theta^{n-1},
\end{align*}
where~$\tilde x_\theta := \langle x,\theta\rangle \, (1-|x|^2+\langle x,\theta\rangle^2)^{-1/2}$.

By spherical symmetry, we can suppose that~$x=|x|e_1$, which gives
\begin{align*}
& {(-\Delta)}^s u_\beta(x)  \\
&=
-\frac{c_{n,s}\,\Gamma(\beta+1)\,\Gamma(-s)}{2\Gamma(\beta-s+1)}
\int_{\partial B_1}\big(1-|x|^2+|x|^2\theta_1^2\big)^{\beta-s} \;
\hf\left(\frac12+s,s-\beta;\frac12\bigg|\frac{|x|^2\theta_1^2}{1-|x|^2+|x|^2\theta_1^2} \right) \; d{\mathcal{H}}_\theta^{n-1}.
\end{align*}
To compute this expression, we transform the hypergeometric function via~\eqref{hyp-transf3} by writing
\begin{align*}
\hf\left(\frac12+s,s-\beta;\frac12\bigg|\frac{|x|^2\theta_1^2}{1-|x|^2+|x|^2\theta_1^2}\right)
=
\left(1+\frac{|x|^2\theta_1^2}{1-|x|^2}\right)^{s-\beta}
\hf\left(-s,s-\beta;\frac12\bigg|-\frac{|x|^2\theta_1^2}{1-|x|^2}\right)
\end{align*}
and therefore
\begin{align*}
& {(-\Delta)}^s u_\beta(x) \\&= 
-\frac{c_{n,s}\,\Gamma(\beta+1)\,\Gamma(-s)}{2\Gamma(\beta-s+1)}
\big(1-|x|^2\big)^{\beta-s}\int_{\partial B_1}
\hf\left(-s,s-\beta;\frac12\bigg|-\frac{|x|^2\theta_1^2}{1-|x|^2}\right) \; d{\mathcal{H}}_\theta^{n-1} \\
&=
-\frac{c_{n,s}\,|\mathbb{S}^{n-2}|
\,\Gamma(\beta+1)\,\Gamma(-s)}{\Gamma(\beta-s+1)}
\big(1-|x|^2\big)^{\beta-s}\int_0^1
\hf\left(-s,s-\beta;\frac12\bigg|-\frac{|x|^2\theta_1^2}{1-|x|^2}\right)\,\big(1-\theta_1^2\big)^{\frac{n-3}2} \; d\theta_1.
\end{align*}

To evaluate the integral in~$\theta_1$, 
we use the definition of hypergeometric function~$\hf$, see~\eqref{def2:hyp}, finding that
\begin{align*}
& 
\int_0^1\hf\left(-s,s-\beta;\frac12\bigg|-\frac{|x|^2\theta_1^2}{1-|x|^2}\right)\,\big(1-\theta_1^2\big)^{\frac{n-3}2} \; d\theta_1 \\
&=
\frac{\Gamma(\frac12)}{\Gamma(-s)\Gamma(s-\beta)}\sum_{k=0}^{+\infty}
\frac{\Gamma(-s+k)\,\Gamma(s-\beta+k)}{\Gamma(\frac12+k)}\frac1{k!}\left(-\frac{|x|^2}{1-|x|^2}\right)^k
\int_0^1\theta_1^{2k}\,\big(1-\theta_1^2\big)^{\frac{n-3}2} \; d\theta_1 \\
&=
\frac{\Gamma(\frac12)}{2\Gamma(-s)\Gamma(s-\beta)}\sum_{k=0}^{+\infty}
\frac{\Gamma(-s+k)\,\Gamma(s-\beta+k)}{\Gamma(\frac12+k)}\frac1{k!}\left(-\frac{|x|^2}{1-|x|^2}\right)^k
\int_0^1 t^{k-1/2}\,(1-t)^{\frac{n-3}2} \; dt \\
&=
\frac{\sqrt\pi\,\Gamma(\frac{n-1}2)}{2\Gamma(-s)\Gamma(s-\beta)}\sum_{k=0}^{+\infty}
\frac{\Gamma(-s+k)\Gamma(s-\beta+k)}{\Gamma(\frac{n}2+k)}\frac1{k!}\left(-\frac{|x|^2}{1-|x|^2}\right)^k \\
&=
\frac{\sqrt\pi\,\Gamma(\frac{n-1}2)}{2\Gamma(\frac{n}2)}
\hf\left(-s,s-\beta;\frac{n}2\bigg|-\frac{|x|^2}{1-|x|^2}\right)
\end{align*}
where we have used~\eqref{beta-identity} and~\eqref{gamma-one-half}.

We thus obtain that
\begin{align*}
{(-\Delta)}^s u_\beta(x) &= 
-\frac{c_{n,s}\,\pi^{n/2}\,\Gamma(\beta+1)\,\Gamma(-s)}{\Gamma(\beta-s+1)\,\Gamma(\frac{n}2)}
\big(1-|x|^2\big)^{\beta-s}\hf\left(-s,s-\beta;\frac{n}2\bigg|-\frac{|x|^2}{1-|x|^2}\right) \\
&=
-\frac{c_{n,s}\,\pi^{n/2}\,
\Gamma(\beta+1)\,\Gamma(-s)}{\Gamma(\beta-s+1)\,\Gamma(\frac{n}2)}
\hf\left(\frac{n}2+s,s-\beta;\frac{n}2\bigg||x|^2\right) 
\end{align*}
by using~\eqref{hyp-transf3} once again.
\end{proof}

Theorem~\ref{thm:flonballs} can be used as a basis for further examples. To this end, given~$\beta>-1$, we consider the function
\begin{align*}
\overline{u}_\beta(x):=\left\lbrace\begin{aligned}
& x_j\big(1-|x|^2\big)^\beta && \text{for }x\in B_1, \\
& 0 && \text{for }x\in\R^n\setminus B_1,
\end{aligned}\right.
\end{align*}
and we have that:

\begin{corollary}
For~$x\in B_1$, it holds that
\begin{align*}
{(-\Delta)}^s \overline{u}_\beta(x) =
-\frac{c_{n,s}(n+2s)\pi^{n/2}\,
\Gamma(\beta+1)\,\Gamma(-s)}{2\Gamma(\beta-s+1)\,\Gamma(\frac{n}2+1)}
\,x_j\,
\hf\left(\frac{n}2+s+1,s-\beta;\frac{n}2+1\bigg||x|^2\right).
\end{align*}
\end{corollary}

\begin{proof}
We notice that the following identity holds true:
\begin{align*}
\overline{u}_\beta(x)=-\frac1{2(\beta+1)}\partial_{x_j}u_{\beta+1}(x).
\end{align*}
Therefore, from Proposition~\ref{prop:exchange}
we have that
\begin{eqnarray*}
&& (-\Delta)^s\overline{u}_\beta(x)
=-\frac1{2(\beta+1)}(-\Delta)^s \big(\partial_{x_j}u_{\beta+1}\big)(x)
=-\frac1{2(\beta+1)}
\partial_{x_j}(-\Delta)^s u_{\beta+1}(x)
.\end{eqnarray*}
Hence, the desired claim follows from Theorem~\ref{thm:flonballs}
and formula~\eqref{hyp'}.
\end{proof}

In the next example, we give a class of (unbounded)~$s$-harmonic functions in the ball. This is quite surprising, since classical harmonic functions
with zero Dirichlet datum vanish identically. This is one of the several
differences between the classical and the fractional worlds, see~\cite{getting}
for more of them.

\begin{proposition}
For any~$\theta\in\partial B_1$, the function
\begin{align*}
M_\alpha(x,\theta):=\left\lbrace\begin{aligned}
& \frac{\big(1-|x|^2\big)^\alpha}{{|x-\theta|}^{n-2s+2\alpha}} && \text{if }x\in B_1, \\
& 0 && \text{if }x\in\R^n\setminus B_1,
\end{aligned}\right.
\end{align*}
satisfies
\begin{align*}
(-\Delta)^s M_\alpha(\cdot,\theta)=
4^sc_{1,s}\kappa(\alpha)\frac{\big(1-|x|^2\big)^{\alpha-2s}}{{|x-\theta|}^{n-2s+2\alpha}}
\qquad\text{in }B_1.
\end{align*}
In particular,
\begin{align}\label{xvdgyry3453ygtku9967ry-5ydgd}
(-\Delta)^s M_s(\cdot,\theta)=0
\qquad\text{and}\qquad
(-\Delta)^s M_{s-1}(\cdot,\theta)=0
\qquad\text{in }B_1.
\end{align}
\end{proposition}

\begin{proof}
We prove the claim for~$\theta=-e_1$.
The general case then follows by the rotational invariance of the fractional Laplacian given by Lemma~\ref{lem:rotation}.

We use the notation in~\eqref{isinforceBIS} with~$R:=\sqrt 2$ and~$x_0:=-e_1$
and, for the facility of the reader, we write~$\mathcal{K}:=\mathcal{K}_{\sqrt 2,-e_1}$. In this way,
\begin{eqnarray*}&&
\mathcal{K}(x) = 2\frac{x+e_1}{{|x+e_1|}^2}-e_1\\{\mbox{and }}&&
\big(\mathcal{K}(x)\big)_1 = 2\frac{x_1+1}{{|x+e_1|}^2}-1=\frac{1-|x|^2}{{|x+e_1|}^2}.
\end{eqnarray*}
We also recall the function~$h_{n,\alpha}$
introduced in~\eqref{yueiwgdsh2wsxcfr56yhnjui890po-09876543}.

{F}rom Proposition~\ref{prop:inversion-general} and
formula~\eqref{fl-halfspace}
we deduce that, in the set~$\big\{(\mathcal{K}(x))_1>0\big\}$,
\begin{equation*}\begin{split}
(-\Delta)^s \big(h_{n,\alpha}\big)_{\mathcal{K}}
(x)
&=4^{s}{|x+e_1|}^{-n-2s}(-\Delta)^s  h_{n,\alpha}\big(\mathcal{K}(x)\big)
\\&
=4^sc_{1,s}\kappa(\alpha){|x+e_1|}^{-n-2s}
\big(\mathcal{K}(x)\big)^{\alpha-2s}_1\\&=
4^sc_{1,s}\kappa(\alpha){|x+e_1|}^{-n-2s}
\left(\frac{1-|x|^2}{{|x+e_1|}^2}
\right)^{\alpha-2s}\\&=
4^sc_{1,s}\kappa(\alpha)\frac{\big(1-|x|^2\big)^{\alpha-2s}}{
{|x+e_1|}^{n-2s+2\alpha}}.
\end{split}\end{equation*}

We observe that
$$
\big(h_{n,\alpha}\big)_{\mathcal{K}}(x) =
{|x+e_1|}^{2s-n}\left(\frac{1-|x|^2}{{|x+e_1|}^2}\right)^\alpha
= M_\alpha(x,-e_1)$$
and thus
$$
(-\Delta)^s M_\alpha(x,-e_1) =
4^sc_{1,s}\kappa(\alpha)\frac{\big(1-|x|^2\big)^{\alpha-2s}}{
{|x+e_1|}^{n-2s+2\alpha}}
\qquad\text{in }B_1,
$$ as desired. 

The claim in~\eqref{xvdgyry3453ygtku9967ry-5ydgd} follows
from these observations and~\eqref{fl-halfspaceBIS}.
\end{proof}

\section{Power functions of the norm}

Let~$\alpha\in(-n,2s)$ and consider the function
\begin{align*}
w_\alpha(x):={|x|}^\alpha
\qquad\text{for any }x\in\R^n\setminus\{0\}.
\end{align*}
The restriction on~$\alpha$ is needed in order to be able to compute the fractional Laplacian
of~$w_\alpha$ in the pointwise sense, as this requires~$w_\alpha$ to lie in~$L^1_s(\R^n)$, as given by~\eqref{w-l1-space}.

In this situation, one obtains the following result:

\begin{proposition}\label{prop:inversion2}
We have that
\begin{align*}
(-\Delta)^s w_\alpha(x)= |x|^{\alpha-2s}(-\Delta)^s w_\alpha(e_1)
\qquad\text{for any }x\in\R^n\setminus\{0\},
\end{align*}
where
\begin{align*}
\text{if }2s<n,\ &\text{ then }\ (-\Delta)^s w_\alpha(e_1)\ 
\left\lbrace\begin{aligned}
& <0 && \text{for }\alpha\in(0,2s) ,\\
& =0 && \text{for }\alpha=0 ,\\
& >0 && \text{for }\alpha\in(2s-n,0), \\
& =0 && \text{for }\alpha=2s-n ,\\
& <0 && \text{for }\alpha\in(-n,2s-n),
\end{aligned}\right. \\
\text{if }n=1\text{ and } s>\frac12,\ &\text{ then }\ (-\Delta)^s w_\alpha(e_1)\ 
\left\lbrace\begin{aligned}
& <0 && \text{for }\alpha\in(2s-1,2s) ,\\
& =0 && \text{for }\alpha=2s-1, \\
& >0 && \text{for }\alpha\in(0,2s-1), \\
& =0 && \text{for }\alpha=0, \\
& <0 && \text{for }\alpha\in(-1,0),
\end{aligned}\right. \\
\text{if }n=1\text{ and } s=\frac12,\ &\text{ then }\ (-\Delta)^s w_\alpha(e_1)\ 
\left\lbrace\begin{aligned}
& <0 && \text{for }\alpha\in(0,1) ,\\
& =0 && \text{for }\alpha=0, \\
& <0 && \text{for }\alpha\in(-1,0).
\end{aligned}\right.
\end{align*}
\end{proposition}

We point out that the case~$\alpha:=2s-n$
is also a consequence of formula~\eqref{fundsol} in Proposition~\ref{prop:inversion}, which gives that
\begin{align*}
(-\Delta)^s w_{2s-n}(e_1)=0.
\end{align*}

\begin{proof}[Proof of Proposition~\ref{prop:inversion2}]
The function~$w_\alpha$ is homogeneous of degree~$\alpha$ and radial: as a result,
recalling Lemmata~\ref{lem:rotation} and~\ref{lem:homogeneity}, we see that, for any~$x\in\R^n\setminus\{0\}$,
\begin{align*}
(-\Delta)^s w_\alpha(x)=
|x|^\alpha {(-\Delta)}^s\left( w_\alpha\left(\frac{ x}{|x|}\right)\right)
|x|^{\alpha-2s}(-\Delta)^s w_\alpha
\left(\frac{x}{|x|}\right)
=|x|^{\alpha-2s}(-\Delta)^s w_\alpha(e_1).
\end{align*}

With this observation, if we are able to characterize~$(-\Delta)^s w_\alpha(e_1)$, 
then we fully characterize~$(-\Delta)^s w_\alpha$. To this end, we remark that,
in the notation of~\eqref{kelvin-0} and~\eqref{kelvin-00},
for any~$x\in\R^n\setminus\{0\}$,
\begin{align*}
\big(w_\alpha\big)_{\mathcal{K}}(x)=|x|^{2s-n}w_\alpha\left(\frac{x}{|x|^2}\right)=w_{2s-n-\alpha}(x),
\end{align*}
and, owing to~\eqref{fl-inversion},
\begin{align*}
(-\Delta)^s w_{2s-n-\alpha}(x)=|x|^{-n-2s}(-\Delta)^s w_\alpha\left(
\frac{x}{|x|^2}\right),
\end{align*}
which in particular gives that
\begin{align*}
(-\Delta)^s w_{2s-n-\alpha}(e_1)=(-\Delta)^s w_\alpha(e_1).
\end{align*}

We use this information to write
\begin{align*}
(-\Delta)^s w_\alpha(e_1) &=
\frac12(-\Delta)^s w_\alpha(e_1)+\frac12(-\Delta)^s w_{2s-n-\alpha}(e_1) \\
&=
\frac{c_{n,s}}2\pv\int_{\R^n}\frac{2-|y|^{\alpha}-|y|^{2s-n-\alpha}}{{|e_1-y|}^{n+2s}}\;dy \\
&=
\frac{c_{n,s}}2\lim_{\eps\searrow 0}\int_{\R^n\setminus B_\eps(e_1)}\frac{2-|y|^{\alpha}-|y|^{2s-n-\alpha}}{{|e_1-y|}^{n+2s}}\;dy.
\end{align*}
We now split the integration over~$\R^n\setminus B_\eps(e_1)$ into~$B_1\setminus B_\eps(e_1)$
and~$(\R^n\setminus B_1)\setminus B_\eps(e_1)$.

In~$(\R^n\setminus B_1)\setminus B_\eps(e_1)$ we
apply the change of variable~$y:=z/|z|^2$
(which produces~$dy=dz/|z|^{2n}$, see formula~(2.6.7) and footnote~9 in~\cite{2021arXiv210107941D}).
In this framework, we observe that the condition~$|y-e_1|>\eps$ translates into
$$\left|\frac{z}{|z|^2}-e_1\right|>\eps,$$ which, by~\eqref{inverted-distance},
is equivalent to~$ |z-e_1|>\eps|z|$.

We now remark that
\begin{equation}\label{tyie4875634yhliedshgfliewsyht}
|z-e_1|>\eps|z| \quad{\mbox{ is equivalent to }}\quad 
\left|z-\frac{e_1}{1-\eps^2}\right|>\frac{\eps}{1-\eps^2}.\end{equation}
Indeed,
\begin{align*}
& |z-e_1|^2>\eps^2|z|^2 \\
\Leftrightarrow\ & (1-\eps^2)|z|^2-2z_1+1>0 \\
\Leftrightarrow\ & |z|^2-\frac{2z_1}{1-\eps^2}+\frac1{1-\eps^2}>0 \\
\Leftrightarrow\ & |z|^2-\frac{2z_1}{1-\eps^2}+\frac1{{(1-\eps^2)}^2}>\frac{\eps^2}{{(1-\eps^2)}^2} \\
\Leftrightarrow\ & \Big|z-\frac{e_1}{1-\eps^2}\Big|^2>\frac{\eps^2}{{(1-\eps^2)}^2},
\end{align*}
which establishes~\eqref{tyie4875634yhliedshgfliewsyht}.

Hence, also in light of~\eqref{inverted-distance},
\begin{align*}
\int_{(\R^n\setminus B_1)\setminus B_\eps(e_1)}
\frac{2-|y|^{\alpha}-|y|^{2s-n-\alpha}}{{|e_1-y|}^{n+2s}}\;dy
&=
\int_{B_1\setminus B_{\eps/(1-\eps^2)}(\frac{e_1}{1-\eps^2})}
\frac{2-|z|^{-\alpha}-|z|^{-2s+n+\alpha}}{\big|e_1-\frac{z}{|z|^2}\big|^{n+2s}}\;\frac{dz}{|z|^{2n}} \\
&=
\int_{B_1\setminus B_{\eps/(1-\eps^2)}(\frac{e_1}{1-\eps^2})}
\frac{2|z|^{2s-n}-|z|^{2s-n-\alpha}-|z|^{\alpha}}{{|e_1-z|}^{n+2s}}\;dz.
\end{align*}

\begin{figure}
\centering
\begin{tikzpicture} 
    \draw[->] (-1,0) -- (7,0);
    \node at (6.85,-.25) {$x_1$};
    \draw (0,-4) -- (0,4);
    \node at (-.27,3.85) {$x'$};
    \draw[ultra thick] (3,0) arc (0:110:3);
    \draw[ultra thick] (3,0) arc (0:-110:3);
    \draw[ultra thick,red] (3,0) circle (1.5);
    \filldraw[red] (3,0) circle (2pt);
    \draw[ultra thick,blue] (4,0) circle (2);
    \filldraw[blue] (4,0) circle (2pt);
    \draw[densely dotted,very thick,blue] (4,0) -- (4,2.7);
    \node at (4,3.1) {\color{blue}~$x_1=\frac1{1-\eps^2}$};
    \draw[densely dotted,very thick,red] (3,0) -- (3,3.5);
    \node at (3,3.7) {\color{red}~$x_1=1$};
    \draw[densely dotted,very thick] (2.625,-3) -- (2.625,3);
    \node at (2.625,-3.5) {$x_1=1-\frac{\eps^2}{2}$};
\end{tikzpicture}
\caption{A depiction of the statement of Lemma~\ref{lem-strange-ballssss}:
in black the ball~$B_1$;
in red, the ball~$B_\eps(e_1)$; 
in blue, the ball~$B_{\eps/(1-\eps^2)}(\frac{e_1}{1-\eps^2})$.}
\label{superballs-fig}
\end{figure}

Also, we remark that, for any~$\eps\in(0,1)$,
\begin{align}\label{jdsocndk902wdL}
B_{\eps/(1-\eps^2)}\Big(\frac{e_1}{1-\eps^2}\Big)\cap B_1\subseteq B_\eps(e_1)\cap B_1.
\end{align}
We postpone the proof of this fact to a separate lemma (see Lemma~\ref{lem-strange-ballssss} below and also Figure~\ref{superballs-fig}).

We then have 
\begin{align}
(-\Delta)^s w_\alpha(e_1) &=
\frac{c_{n,s}}2\lim_{\eps\searrow 0}\Bigg(\int_{B_1\setminus B_\eps(e_1)}\frac{2-|y|^{\alpha}-|y|^{2s-n-\alpha}}{{|e_1-y|}^{n+2s}}\;dy
 \nonumber \\
& \qquad
+\int_{B_1\setminus B_{\eps/(1-\eps^2)}(\frac{e_1}{1-\eps^2})}
\frac{2|y|^{2s-n}-|y|^{2s-n-\alpha}-|y|^{\alpha}}{{|e_1-y|}^{n+2s}}\;dy\Bigg) \nonumber \\
&=
c_{n,s}\lim_{\eps\searrow 0}\Bigg(\int_{B_1\setminus B_\eps(e_1)}\frac{1+|y|^{2s-n}-|y|^{\alpha}-|y|^{2s-n-\alpha}}{{|e_1-y|}^{n+2s}}\;dy \nonumber \\
& \qquad
+\frac12\int_{B_\eps(e_1)\setminus B_{\eps/(1-\eps^2)}(\frac{e_1}{1-\eps^2})}
\frac{2|y|^{2s-n}-|y|^{2s-n-\alpha}-|y|^{\alpha}}{{|e_1-y|}^{n+2s}}\;dy\Bigg) \nonumber \\
&=
c_{n,s}\lim_{\eps\searrow 0}\Bigg(\int_{B_1\setminus B_\eps(e_1)}\frac{\big(|y|^\alpha-1\big)\big(|y|^{2s-n-\alpha}-1\big)}{{|e_1-y|}^{n+2s}}\;dy \label{9epijkdm03iru94p3i} \\
& \qquad
+\frac12\int_{B_\eps(e_1)\setminus B_{\eps/(1-\eps^2)}(\frac{e_1}{1-\eps^2})}
\frac{2|y|^{2s-n}-|y|^{2s-n-\alpha}-|y|^{\alpha}}{{|e_1-y|}^{n+2s}}\;dy\Bigg). \label{hiog3jnf9fioj3nl}
\end{align}
The limit in~\eqref{hiog3jnf9fioj3nl} is zero: the proof of this fact is postponed to Lemma~\ref{ballsss-integral}.

We are then left with the limit in~\eqref{9epijkdm03iru94p3i}, which entails that
\begin{align}\label{cisiamoquasi}
(-\Delta)^s w_\alpha(e_1) = c_{n,s}
\int_{B_1}\frac{\big(|y|^\alpha-1\big)\big(|y|^{2s-n-\alpha}-1\big)}{{|e_1-y|}^{n+2s}}\;dy.
\end{align}
By inspection, one can see that, for~$y\in B_1$, if~$2s<n$ then 
\begin{align*}
\big(|y|^\alpha-1\big)\big(|y|^{2s-n-\alpha}-1\big)\ 
\left\lbrace\begin{aligned}
& <0 && \text{for }\alpha\in(0,2s), \\
& =0 && \text{for }\alpha=0 ,\\
& >0 && \text{for }\alpha\in(2s-n,0), \\
& =0 && \text{for }\alpha=2s-n, \\
& <0 && \text{for }\alpha\in(-n,2s-n).
\end{aligned}\right.\  \end{align*}
Similarly, if~$n=1$ and~$ 2s>1$ then
\begin{align*}
\big(|y|^\alpha-1\big)\big(|y|^{2s-n-\alpha}-1\big)\ 
\left\lbrace\begin{aligned}
& <0 && \text{for }\alpha\in(2s-1,2s) ,\\
& =0 && \text{for }\alpha=2s-1, \\
& >0 && \text{for }\alpha\in(0,2s-1) ,\\
& =0 && \text{for }\alpha=0, \\
& <0 && \text{for }\alpha\in(-1,0)
\end{aligned}\right.\ \end{align*}
and if~$n=1=2s$ then
\begin{align*}\big(|y|^\alpha-1\big)\big(|y|^{2s-n-\alpha}-1\big)\ 
\left\lbrace\begin{aligned}
& <0 && \text{for }\alpha\in(0,1) ,\\
& =0 && \text{for }\alpha=0, \\
& <0 && \text{for }\alpha\in(-1,0).
\end{aligned}\right.
\end{align*}
The above signs on the integrand of~\eqref{cisiamoquasi} give the corresponding signs of~$(-\Delta)^s w_\alpha(e_1)$, concluding the proof.
\end{proof}

In order to complete the proof of Proposition~\ref{prop:inversion2}
it remains to prove~\eqref{jdsocndk902wdL} and the fact that the quantity in~\eqref{hiog3jnf9fioj3nl}
is infinitesimal. These proofs are presented here below.

\begin{lemma}\label{lem-strange-ballssss}
For any~$\eps\in(0,1)$, it holds that
\begin{align*}
B_{\eps/(1-\eps^2)}\left(\frac{e_1}{1-\eps^2}\right)\cap B_1\subseteq B_\eps(e_1)\cap B_1.
\end{align*}
\end{lemma}

\begin{proof}
Notice that
\begin{align}\label{superballs}
\begin{split}
B_\eps(e_1)\cap B_1
&=
\big\{x\in\R^n\;{\mbox{ s.t. }}\;
|x'|^2<1-x_1^2\;\text{ and }\;|x'|^2<\eps^2-x_1^2+2x_1-1\big\} \\
&=
\big\{x\in\R^n\;{\mbox{ s.t. }}\;|x'|^2<\min\{\eps^2-x_1^2+2x_1-1,1-x_1^2\}\big\}.
\end{split}\end{align}
Similarly,
\begin{align}\label{superballs2}\begin{split}
B_{\eps/(1-\eps^2)}\left(\frac{e_1}{1-\eps^2}\right)\cap B_1
&=
\Big\{x\in\R^n\;{\mbox{ s.t. }}\;|x'|^2<1-x_1^2\;\text{ and } \\
&\qquad\qquad
(1-\eps^2)^2|x'|^2<\eps^2-(1-\eps^2)^2x_1^2+2(1-\eps^2)x_1-1\Big\} \\
&=
\left\{x\in\R^n\;{\mbox{ s.t. }}\;|x'|^2<1-x_1^2\;\text{ and }\;|x'|^2<-x_1^2+\frac{2x_1-1}{1-\eps^2}\right\} \\
&=
\left\{x\in\R^n\;{\mbox{ s.t. }}\;|x'|^2<\min\left\{-x_1^2+\frac{2x_1-1}{1-\eps^2},1-x_1^2\right\}\right\}.
\end{split}
\end{align}

Also, we see that
\begin{equation}\label{yuie73224erdsgfa43675}
{\mbox{if~$x\in B_{\eps/(1-\eps^2)}\left(\frac{e_1}{1-\eps^2}\right)\cap B_1$ then~$x_1\in\left(\frac1{1+\eps},1\right)$.}}\end{equation}
Indeed,
\begin{eqnarray*}
1>x_1>\frac{1}{1-\eps^2}-\frac{\eps}{1-\eps^2}=\frac{1-\eps}{1-\eps^2}=\frac{1}{1+\eps},
\end{eqnarray*}
which proves~\eqref{yuie73224erdsgfa43675}.



In light of~\eqref{yuie73224erdsgfa43675}, we consider two cases,
according to whether~$x_1\in\left(\frac1{1+\eps},1-\frac{\eps^2}2\right)$
or~$x_1\in\left[1-\frac{\eps^2}2,1\right)$.

If~$x_1\in\left(\frac1{1+\eps},1-\frac{\eps^2}2\right)$,
we have that
$$\eps^2+2x_1-1<\eps^2+2-\eps^2-1=1
$$
and therefore
$$ 
\min\big\{\eps^2-x_1^2+2x_1-1,1-x_1^2\big\}=\eps^2-x_1^2+2x_1-1.$$
Moreover,
$$ \frac{2x_1-1}{1-\eps^2}<\frac{2-\eps^2-1}{1-\eps^2}=1
$$
and so
$$
\min\left\{-x_1^2+\frac{2x_1-1}{1-\eps^2},1-x_1^2\right\}
=-x_1^2+\frac{2x_1-1}{1-\eps^2}.
$$
Furthermore, 
$$ (2x_1-1)\left(\frac1{1-\eps^2}-1\right)<
(1-\eps^2)\left(\frac1{1-\eps^2}-1\right)=\eps^2,$$
which implies that
\begin{eqnarray*}
-x_1^2+\frac{2x_1-1}{1-\eps^2}
<
\eps^2-x_1^2+2x_1-1.
\end{eqnarray*}
These observations prove that
\begin{align*}
B_{\eps/(1-\eps^2)}\left(\frac{e_1}{1-\eps^2}\right)\cap \left\{
x_1\in\left(\frac1{1+\eps},1-\frac{\eps^2}2\right)\right\}
\subseteq 
B_\eps(e_1)\cap\left\{x_1\in\left(\frac1{1+\eps},1-\frac{\eps^2}2\right)\right\}.
\end{align*}

If instead~$x_1\in\left[1-\frac{\eps^2}2,1\right)$ we have that
$$ \eps^2+2x_1-1>\eps^2+2-\eps^2-1=1$$
which gives that
$$\min\big\{\eps^2-x_1^2+2x_1-1,1-x_1^2\big\}=1-x_1^2.$$
Similarly,
$$ \frac{2x_1-1}{1-\eps^2}>\frac{2-\eps^2-1}{1-\eps^2}=1
$$
and so
$$
\min\left\{-x_1^2+\frac{2x_1-1}{1-\eps^2},1-x_1^2\right\}
=1-x_1^2. $$
Accordingly,
\begin{align*}
B_{\eps/(1-\eps^2)}\left(\frac{e_1}{1-\eps^2}\right)
\cap \left\{x_1\in\left[1-\frac{\eps^2}2,1\right)\right\}\cap B_1
=
B_\eps(e_1)\cap\left\{x_1\in\left[1-\frac{\eps^2}2,1\right)\right\}\cap B_1.
\end{align*}

Gathering these observations, we obtain the desired inclusion.
\end{proof}

\begin{lemma}\label{ballsss-integral}
For any~$s\in(0,1)$ and~$n\in\N\setminus\{0\}$, it holds that
\begin{align}\label{ballsss-integral-eq}
\lim_{\eps\searrow 0}\int_{B_\eps(e_1)\setminus B_{\eps/(1-\eps^2)}(\frac{e_1}{1-\eps^2})}
\frac{2|y|^{2s-n}-|y|^{2s-n-\alpha}-|y|^{\alpha}}{{|e_1-y|}^{n+2s}}\;dy=0.
\end{align}
\end{lemma}

\begin{proof}
For~$y\in B_\eps(e_1)$, a first order Taylor expansion gives that
\begin{eqnarray*}
&&2|y|^{2s-n}-|y|^{2s-n-\alpha}-|y|^{\alpha}\\
&=&2\big( 1+(2s-n) (y_1-1)\big)-\big(1+(2s-n-\alpha) (y_1-1)\big)
-\big(1+\alpha  (y_1-1)\big)+o(|y-e_1|)
\\&=&(2s-n)(y_1-1)+o(|y-e_1|)
\end{eqnarray*}
and therefore 
\begin{equation*}\begin{split}
&\left|\int_{B_\eps(e_1)\setminus B_{\eps/(1-\eps^2)}(\frac{e_1}{1-\eps^2})}
\frac{2|y|^{2s-n}-|y|^{2s-n-\alpha}-|y|^{\alpha}}{{|e_1-y|}^{n+2s}}\;dy\right|\\
&\qquad\leq 
C\int_{B_\eps(e_1)\setminus B_{\eps/(1-\eps^2)}(\frac{e_1}{1-\eps^2})}
\frac{dy}{{|e_1-y|}^{n+2s-1}}.
\end{split}\end{equation*}

Moreover, if~$y\in B_\eps(e_1)\setminus B_{\eps/(1-\eps^2)}(\frac{e_1}{1-\eps^2})$, then
$$ \frac\eps{1+\eps}\le |e_1-y|\le \eps$$
because, since $e_1\in B_{\eps/(1-\eps^2)}(\frac{e_1}{1-\eps^2})$ for~$\eps$ small, we have that
\begin{align*}
&|e_1-y|\geq\min\left\{|e_1-z|\;{\mbox{ with }}\; z\in\partial B_{\eps/(1-\eps^2)}\left(\frac{e_1}{1-\eps^2}\right)\right\}
\\&\qquad =\left|e_1-\left(\frac1{1-\eps^2}-\frac\eps{1-\eps^2}\right)e_1\right|
=1-\frac1{1+\eps}=\frac\eps{1+\eps}.
\end{align*}
This entails that
\begin{equation}\label{misurapallle}\begin{split}
&\left|\int_{B_\eps(e_1)\setminus B_{\eps/(1-\eps^2)}(\frac{e_1}{1-\eps^2})}
\frac{2|y|^{2s-n}-|y|^{2s-n-\alpha}-|y|^{\alpha}}{{|e_1-y|}^{n+2s}}\;dy\right|\\
&\qquad\leq
C\eps^{1-n-2s}\left|B_\eps(e_1)\setminus B_{\eps/(1-\eps^2)}\left(\frac{e_1}{1-\eps^2}\right)\right|,
\end{split}\end{equation}
up to renaming~$C$.

We now claim that
\begin{align}\label{misurapallle2}
\left|B_\eps(e_1)\setminus B_{\eps/(1-\eps^2)}\left(\frac{e_1}{1-\eps^2}\right)\right|
\leq C\eps^{n+1},
\end{align}
which, together with~\eqref{misurapallle} and the fact that~$s\in(0,1)$, gives~\eqref{ballsss-integral-eq}.

In order to prove~\eqref{misurapallle2}, 
we notice that
\begin{equation}\label{palle7565yfhnfekjupalle06588}\begin{split}&
B_\eps(e_1)\setminus B_{\eps/(1-\eps^2)}\Big(\frac{e_1}{1-\eps^2}\Big)\\&
=
\left\{x\in\R^n\;{\mbox{ s.t. }}\; 1-\sqrt{\eps^2-|x'|^2}<x_1<\frac1{1-\eps^2}-\sqrt{\frac{\eps^2}{(1-\eps^2)^2}-|x'|^2} \;{\mbox{ and }} \;|x'|^2<\eps^2-\frac{\eps^4}4\right\},\end{split}
\end{equation}
see also Figure~\ref{superballs-fig} for a depiction of the situation that we are in.

To check this, we observe that if~$x\in B_\eps(e_1)$ then~$|x'|<\eps$ and
$$  1-\sqrt{\eps^2-|x'|^2}<x_1<  1+\sqrt{\eps^2-|x'|^2}.$$
Also, if~$x\in\R^n\setminus B_{\eps/(1-\eps^2)}\Big(\frac{e_1}{1-\eps^2}\Big)$ then 
either~$|x'|>\frac\eps{1-\eps^2}$ or
\begin{eqnarray*} && |x'|\le\frac\eps{1-\eps^2} \quad {\mbox{ and }}\\
&&{\mbox{either }}\quad 
x_1<\frac1{1-\eps^2}-\sqrt{\frac{\eps^2}{(1-\eps^2)^2}-|x'|^2}
\quad{\mbox{ or }}\quad x_1>\frac1{1-\eps^2}+\sqrt{\frac{\eps^2}{(1-\eps^2)^2}-|x'|^2}
.\end{eqnarray*}
As a consequence, if~$x\in B_\eps(e_1)\setminus B_{\eps/(1-\eps^2)}\left(\frac{e_1}{1-\eps^2}\right)$ then~$ |x'|<\eps $ and
\begin{eqnarray*}1-\sqrt{\eps^2-|x'|^2}<x_1<
\min\left\{1+\sqrt{\eps^2-|x'|^2},\;
\frac1{1-\eps^2}-\sqrt{\frac{\eps^2}{(1-\eps^2)^2}-|x'|^2} \right\}
.\end{eqnarray*}
{F}rom this, it follows that if~$x\in B_\eps(e_1)\setminus B_{\eps/(1-\eps^2)}\left(\frac{e_1}{1-\eps^2}\right)$ then
$$1-\sqrt{\eps^2-|x'|^2} < \frac1{1-\eps^2}-\sqrt{\frac{\eps^2}{(1-\eps^2)^2}-|x'|^2} $$
and therefore
$$ |x'|^2< \eps^2-\frac{\eps^4}4.$$
This also gives that
$$ \frac1{1-\eps^2}-\sqrt{\frac{\eps^2}{(1-\eps^2)^2}-|x'|^2}<1<1+\sqrt{\eps^2-|x'|^2}.$$
Gathering these pieces of information, we obtain~\eqref{palle7565yfhnfekjupalle06588}.

{F}rom~\eqref{palle7565yfhnfekjupalle06588}, we have that
\begin{align*}
& \left|B_\eps(e_1)\setminus B_{\eps/(1-\eps^2)}\left(\frac{e_1}{1-\eps^2}\right)\right| \\
& = 
\int_{\{|x'|<{\sqrt{\eps^2-\eps^4/4}}\}}\int_{1-\sqrt{\eps^2-|x'|^2}}^{\frac1{1-\eps^2}-\sqrt{\frac{\eps^2}{(1-\eps^2)^2}-|x'|^2}}dx_1\;dx' \\
& =
|\mathbb{S}^{n-2}|\int_0^{\sqrt{\eps^2-\eps^4/4}}\rho^{n-2}
\left[
\frac1{1-\eps^2}-\sqrt{\frac{\eps^2}{(1-\eps^2)^2}-\rho^2}
-1+\sqrt{\eps^2-\rho^2}
\right]\;d\rho \\
& =
|\mathbb{S}^{n-2}|\int_0^{\sqrt{\eps^2-\eps^4/4}}\rho^{n-2}\left[
\frac{\eps^2}{1-\eps^2}-\sqrt{\frac{\eps^2}{(1-\eps^2)^2}-\rho^2}
+\sqrt{\eps^2-\rho^2}
\right]\;d\rho .
\end{align*}
Thus, the change of variable~$r:=\rho^2/\eps^2$ gives that
\begin{align*}
& \left|B_\eps(e_1)\setminus B_{\eps/(1-\eps^2)}\left(\frac{e_1}{1-\eps^2}\right)\right| \\
& \leq
|\mathbb{S}^{n-2}|\eps^n\int_0^{{1-\eps^2/4}}r^{\frac{n-3}2}\left[
\frac{\eps}{1-\eps^2}-\sqrt{\frac1{(1-\eps^2)^2}-r}
+\sqrt{1-r}
\right]\;dr .
\end{align*}
We also observe that
\begin{align*}
\sqrt{1-r}-\sqrt{\frac1{(1-\eps^2)^2}-r}\leq 0,
\end{align*}
and therefore
\begin{align*}
\left|B_\eps(e_1)\setminus B_{\eps/(1-\eps^2)}\left(\frac{e_1}{1-\eps^2}\right)\right| 
&\leq
|\mathbb{S}^{n-2}|\frac{\eps^{n+1}}{1-\eps^2}\int_0^{{1-\eps^2/4}}r^{\frac{n-3}2}\;dr\\
& \le
\frac{2|\mathbb{S}^{n-2}|}{n-1}\frac{\eps^{n+1}}{1-\eps^2}.\end{align*}
This proves~\eqref{misurapallle2} and completes the proof of~\eqref{ballsss-integral-eq}.
\end{proof}

\subsection{A particular example for \texorpdfstring{$n=1$}{n=1} and \texorpdfstring{$s=1/2$}{s=1/2}}

We point out that Proposition~\ref{prop:inversion2} provides 
nontrivial~$s$-harmonic functions in~$\R^n\setminus\{0\}$,
namely~$w_{2s-n}$ 
whenever~$2s\neq n$. Since~$s\in(0,1)$, the equality~$2s=n$ can only hold 
for~$n=1$ and~$s=1/2$. 

We fill here below this gap left by Proposition~\ref{prop:inversion2}
by giving the example of a nontrivial~$1/2$-harmonic function on~$\R\setminus\{0\}$ (compare this result
also with Corollary~\ref{FURIEZ}).

\begin{lemma}
We have that
\begin{align*}
(-\Delta)^{1/2}\ln|x|=0
\qquad\text{for any }x\in\R\setminus\{0\}.
\end{align*}
\end{lemma}

\begin{proof}
We start with the evaluation of the~$1/2$-Laplacian at~$x=\pm1$.
We use representation~\eqref{pv-def} and the change of variables~$y:=1/z$
to write that
\begin{align*}&
\Big[(-\Delta)^{1/2}\ln|x|\Big]\Big|_{x=\pm1} =
\frac1\pi\pv\int_\R\frac{-\ln|y|}{{|\pm1-y|}^2}\;dy
=\frac1\pi\pv\int_\R\frac{\ln|z|}{\big|\pm1-\frac1z\big|^2}\frac{dz}{z^2} \\
&\qquad=
\frac1{\pi}\pv\int_\R\frac{\ln|z|}{{|z\pm1|}^2}\;dz
=-\Big[(-\Delta)^{1/2}\ln|x|\Big]\Big|_{x=\pm1},
\end{align*}
which yields that
\begin{align}\label{vwolkvmc 2oink}
\Big[(-\Delta)^{1/2}\ln|x|\Big]\Big|_{x=\pm1}=0.
\end{align}

For a general~$x\in\R^n\setminus\{-1,0,1\}$, we apply the change of variables~$y:=|x|z$
and we use~\eqref{vwolkvmc 2oink} to find that
\begin{eqnarray*}&&
(-\Delta)^{1/2}\ln|x| =
\frac1\pi\pv\int_\R\frac{\ln|x|-\ln|y|}{{|x-y|}^2}\;dy\\&&\qquad
=\frac1{\pi|x|}\pv\int_\R\frac{-\ln|z|}{{|z\pm1|}^2}\;dz 
=\frac1{|x|}\Big[(-\Delta)^{1/2}\ln|x|\Big]\Big|_{x=\pm1}=0,
\end{eqnarray*}
as desired.\end{proof}

\chapter{Liouville-type results}\label{LIOUV:CHAP}

In complex analysis, Liouville's Theorem (named after Joseph Liouville,
who presented it in~1847,
although probably first proved by Augustin-Louis Cauchy in~1844)
states that a bounded entire function of complex variable is necessarily bounded.
This result had a striking impact on the development of complex analysis and reached out to other fields (for instance,
it provides a quick and elegant proof of the Fundamental Theorem of Algebra).

The natural counterpart of this result in the theory of elliptic partial differential equations states that
a bounded harmonic function in~$\R^n$ is necessarily constant. This type of results serves as a cornerstone for a solid regularity and classification theory, see~\cite{MR2569331} and the references therein for a full account of the importance of Liouville-type results.

Liouville's Theorem was extended to the fractional world in~\cite[Lemma~3.2]{MR1936936} and several different versions and refinements are nowadays available in the literature, see e.g.~\cites{MR2759038, MR3311908, MR3318148, MR3385173, MR3482695, MR3511811, MR3538413, MR3959045, MR4149690, MR4395952}. In these pages, we will present an approach proposed by~\cite{MR3348929}, relying on Fourier methods and distribution theory (see~\cite{MR3477075} for a different approach based on Cauchy-type estimates).

\begin{theorem}\label{BiknsoU98ikjmd9}
Assume that
\begin{equation}\label{BiknsoU98ikjmd} \int_{\R^n}\frac{|u(x)|}{1+|x|^{n+2s}}\,dx<+\infty\end{equation}
and that~$(-\Delta)^su=0$ in~$\R^n$.

Then,
\begin{itemize}
\item If~$s\in\left(\frac12,1\right)$, we have that~$u$ is necessarily an affine function.
\item If~$s\in\left(0,\frac12\right]$, we have that~$u$ is necessarily constant.
\end{itemize}
\end{theorem}

We observe that if~$u$ is bounded, then condition~\eqref{BiknsoU98ikjmd} is automatically satisfied
(and bounded affine functions are necessarily constant): in this sense, Theorem~\ref{BiknsoU98ikjmd9} is actually even stronger than the classical Liouville's statement. Interestingly, condition~\eqref{BiknsoU98ikjmd} is also required for a pointwise definition
of the fractional Laplacian, and therefore this assumption is sharp (see however Theorem~1.5 in~\cite{MR3988080}
for a more general Liouville-type result allowing polynomial growths).

\begin{proof}[Proof of Theorem~\ref{BiknsoU98ikjmd9}] The core of this argument consists in considering~$\check u$ as a distribution (see Appendix~\ref{APPECC}
for more details about the theory of distributions)
and in proving that, for every~$\psi\in C^\infty_c(\R^n\setminus\{0\})$,
\begin{equation}\label{BiknsoU98ikjmd9232}
\check u(\psi)=0.
\end{equation}

Once~\eqref{BiknsoU98ikjmd9232} is established, the proof of Theorem~\ref{BiknsoU98ikjmd9} can be completed in this way.
By~\eqref{BiknsoU98ikjmd9232}, one has that the distribution~$\check u$ has support contained in the singleton~$\{0\}$.
Accordingly (see Theorem~\ref{DE:CALFA:TG}), one deduces that the distribution~$\check u$ is a finite combination of derivatives of Dirac Delta Functions at the origin, namely there exist~$N\in\N$ and~$\{c_\alpha\}_{{\alpha\in\N^n}\atop{|\alpha|\le N}}$ such that
$$ \check u=\sum_{{\alpha\in\N^n}\atop{|\alpha|\le N}} c_\alpha D^\alpha \delta_{0}.$$
That is, for every~$f\in C^\infty_c(\R^n)$,
\begin{equation}\label{BiknsoU98ikjmd92323}
\int_{\R^n} u(x)\,\check f(x)\,dx=
\check u(f)=\sum_{{\alpha\in\N^n}\atop{|\alpha|\le N}} c_\alpha D^\alpha \delta_{0}(f)
=\sum_{{\alpha\in\N^n}\atop{|\alpha|\le N}} (-1)^{|\alpha|} c_\alpha D^\alpha f(0).
\end{equation}
Now, using that
$$ f(y)=\overline{\int_{\R^n}\widehat f(x)\,e^{2\pi ix\cdot y}\,dy}=
\int_{\R^n}\check f(x)\,e^{-2\pi ix\cdot y}\,dy,$$
we deduce that
$$ D^\alpha f(y)=\int_{\R^n}(-1)^{|\alpha|}(2\pi ix)^\alpha \check f(x)\,e^{2\pi ix\cdot y}\,dx.$$
Plugging this information into~\eqref{BiknsoU98ikjmd92323}, we obtain that
\begin{eqnarray*}
\int_{\R^n} u(x)\,\check f(x)\,dx
=\sum_{{\alpha\in\N^n}\atop{|\alpha|\le N}} c_\alpha \int_{\R^n}(2\pi ix)^\alpha \check f(x)\,dx
\end{eqnarray*}
and therefore
\begin{equation*} u(x)=\sum_{{\alpha\in\N^n}\atop{|\alpha|\le N}}c_\alpha (2\pi ix)^\alpha,\end{equation*}
whence~$u$ is a polynomial.

{F}rom~\eqref{BiknsoU98ikjmd}, we deduce that~$u$ is necessarily affine for all~$s\in(0,1)$, and actually constant if~$s\in\left(0,\frac12\right]$.

This would complete the proof of Theorem~\ref{BiknsoU98ikjmd9}, therefore it remains to prove~\eqref{BiknsoU98ikjmd9232}.
To this end, we argue as follows.
For every~$\varphi$ in the Schwartz space of smooth and rapidly decreasing functions,
\begin{equation}\label{BiknsoU98ikjmd91}
0=\int_{\R^n}u(x)\,(-\Delta)^s\varphi(x)\,dx=\int_{\R^n}u(x)\,{\mathcal{F}}^{-1} \Big(  (2\pi|\xi|)^{2s}\widehat\varphi(\xi)\Big)(x)\,dx.
\end{equation}
Also, if~$\psi\in C^\infty_c(\R^n\setminus\{0\})$, we have that the function~$\beta(\xi):=\frac{\psi(\xi)}{(2\pi|\xi|)^{2s}}$ is also in~$\psi\in C^\infty_c(\R^n\setminus\{0\})$.
Thus, we take~$\varphi:={\mathcal{F}}^{-1}(\beta)$ and deduce from~\eqref{BiknsoU98ikjmd91} that
\begin{equation*}
\begin{split}
0&=\int_{\R^n}u(x)\,{\mathcal{F}}^{-1} \Big(  (2\pi|\xi|)^{2s}\widehat\varphi(\xi)\Big)(x)\,dx\\&=\int_{\R^n}u(x)\,{\mathcal{F}}^{-1} \Big(  (2\pi|\xi|)^{2s}\beta(\xi)\Big)(x)\,dx\\&=\int_{\R^n}u(x)\,{\mathcal{F}}^{-1} \psi (x)\,dx.
\end{split}\end{equation*}
This completes the proof of~\eqref{BiknsoU98ikjmd9232}, as desired.
\end{proof}

As a variant of Theorem~\ref{BiknsoU98ikjmd9}, we point out a (more rigid)
Liouville-type result\footnote{The notation concerning~$(1-\Delta)^s$ is not uniform in the literature. Fpr example,
this operator is denoted 
by~${\mathcal{I}}_{-2s}$ in~\cite{MR0290095},
by~${\mathcal{G}}_{-2s}$ in~\cite{MR1411441}, and
by~$\langle D_x\rangle^s$
in~\cite{MR2884718}. Other notations in the literature
include~$X^s$ and~$\Lambda^s$. To avoid confusion,
we will stick to the notation~$(1-\Delta)^s$ in this book.}
for the operator~$(1-\Delta)^s$.

\begin{theorem} \label{BiknsoU98ikjmd9f3}
Let~$s>0$.
Assume that~$u\in L^1_{\rm loc}(\R^n)$ is a distributional solution of~$(1-\Delta)^su=0$ in~$\R^n$.
Then, $u$ vanishes identically.
\end{theorem}

\begin{proof} Let~$T$ be the distribution~$(1-\Delta)^su$. For every~$\varphi$ in the Schwartz space of smooth and rapidly decreasing functions we have that
\begin{eqnarray*}
0=T(\varphi)= \check u\Big[ (1+2\pi|\xi|^2)^s\widehat\varphi(\xi)\Big].
\end{eqnarray*}
So, given any~$\psi$ in the Schwartz space of smooth and rapidly decreasing functions, we set~$\varphi(x):={\mathcal{F}}^{-1}
\big((1+2\pi|\xi|^2)^{-s}\psi(\xi)\big)$. In this way, we have that~$(1+2\pi|\xi|^2)^s\widehat\varphi(\xi)=\psi(\xi)$ and thus~$
0= \check u[\psi]$. As a result,~$\check u$ vanishes identically and so does~$u$.
\end{proof}

\chapter{Regularity theory in Lebesgue spaces for global solutions}\label{CHAP6}

\section{Baloney around the regularity theory in Lebesgue spaces}

In the following pages, we address the global and interior regularity theory of fractional equations. To this end, different approaches are possible. The strategy that we follow here consists of three steps.

The first step introduces the so-called Bessel potential spaces. Roughly speaking, these can be seen as spaces of the type~$(1-\Delta)^{-s}(L^p)$.
That is, these spaces encode, somewhat in a tautological sense, that if~$u$ is a global solution of~$(1-\Delta)^s u = f$ with~$f \in L^p(\R^n)$, then~$u=(1-\Delta)^{-s} f$, whence~$u$ belongs, by default, to the corresponding Bessel potential space.

The advantage of this setting is that one can develop a suitable functional analysis to deduce regularity estimates for a related, but structurally different equation, namely for global solutions of~$(-\Delta)^s u=f\in L^p(\R^n)$. This will rely on a result (namely Theorem~\ref{TH:9191}) which states, roughly speaking, that global solutions of~$(-\Delta)^s u = f$ are as good as those of~$(1-\Delta)^s u = f$, provided they are in~$L^p(\R^n)$.

In this way, this first step allows one to place global solutions of~$(-\Delta)^s u=f\in L^p(\R^n)$ in the corresponding Bessel potential space. The second step thus consists in understanding ``how nice'', in terms of an appropriate notion of regularity, the functions belonging to a Bessel potential space are. To describe the appropriate notion of regularity, one can rely on Besov spaces (see~\cite{zbMATH03254379}), as well as on the more familiar\footnote{We assume that the reader has already some basic familiarity with Sobolev spaces. For what is needed in this book, it suffices to recall that, for~$p\in[1,+\infty)$
and~$s\in(0,1)$, the Sobolev space \( W^{s,p}(\R^n) \) is defined as the collection of functions~$ u \in L^p(\R^n)$ such that~$ \|u\|_{W^{s,p}(\R^n)} < +\infty $,
where the norm \( \|u\|_{W^{s,p}(\R^n)} \) is given by
\[
\|u\|_{W^{s,p}(\R^n)} := \left( \|u\|_{L^p(\R^n)}^p + \int_{\R^n} \int_{\R^n} \frac{|u(x) - u(y)|^p}{|x-y|^{n+sp}} \, dx \, dy \right)^{\frac{1}{p}}
\]
and
\[
\|u\|_{L^p(\R^n)} := \left( \int_{\R^n} |u(x)|^p \, dx \right)^{\frac{1}{p}}.
\]

Also, as usual, when~$k\in\N$, the space~$
W^{k,p}(\R^n) $ is the set of functions~$ u \in L^p(\R^n) $
which, for all multi-indices~$\alpha=(\alpha_1,\dots,\alpha_n)\in\N^n$ with~$ \ |\alpha| :=\alpha_1 + \alpha_2 + \dots + \alpha_n\leq k$, possess weak derivatives~$ \partial^\alpha u $
which are functions in~$ L^p(\R^n)$.
Namely, for all multi-indices~$\alpha\in\N^n$ with~$ \ |\alpha| \leq k$,
there exists a function (denoted, with a slight abuse of notation~$ \partial^\alpha u$) belonging to~$L^p(\R^n)$
and such that\[
\int_{\R^n} u(x) \, \partial^\alpha \phi(x) \, dx = (-1)^{|\alpha|} \int_{\R^n} (\partial^\alpha u)(x) \, \phi(x) \, dx,
\]
for all \( \phi \in C_c^\infty(\R^n) \), where \( |\alpha| = \alpha_1 + \alpha_2 + \cdots + \alpha_n \).

The norm \( \|u\|_{W^{k,p}(\R^n)} \) is given by
\[
\|u\|_{W^{k,p}(\R^n)} := \left( \sum_{|\alpha| \leq k} \|\partial^\alpha u\|_{L^p(\R^n)}^p \right)^{\frac{1}{p}}.
\]

As a notational remark, we recall that the notation for Sobolev spaces is not uniform in the literature.
For example, the space which is denoted here by~$W^{k,p}(\R^n)$ corresponds to the space~$L^p_k(\R^n)$
in~\cite[Chapter~V, Section~2]{MR0290095}.}
Sobolev spaces
(the bridge between different spaces will be discussed in Theorems~\ref{TH0-104-okn5KMD3} and~\ref{THPS2}, as well as in Corollaries~\ref{COROLEEP2r34t5},~\ref{0pirj09365-4g4eZXCHJKLRFG-544-NO1-I2L3C2O25R213t4TY} and~\ref{0pirj09365-4g4eZXCHJKLRFG-544-NO1-I2L3C2O25R213t4TYFGHSJDFRAMISJMFG}).
This second step thus provides a regularity theory for fractional equations in suitable Besov or Sobolev spaces, specifically for global solutions of~$(1-\Delta)^s u = f\in L^p(\R^n)$, as well as of~$(-\Delta)^s u = f\in L^p(\R^n)$ (provided, in the latter case, that~$u\in L^p(\R^n)$, see also Propositions~\ref{HSZUONDCCONT} and~\ref{HSZUONDCCONT2} to appreciate the importance of global assumptions on~$u$ when dealing with the fractional Laplacian).

Then, the third step consists in using suitable cutoff arguments to deduce, from the previous work, interior regularity theories for solutions of~$(1-\Delta)^s u = f\in L^p(\Omega)$, as well as of~$(-\Delta)^s u = f\in L^p(\Omega)$, for a given open and bounded set~$\Omega$ in~$\R^n$. In this, once again, Bessel potential spaces play the role of a useful pivot to join Besov and Sobolev spaces. An important ingredient for this regularity theory also comes from the setting of pseudodifferential operators (which is sketchily recalled, in a rather essential version, in Appendix~\ref{ojdlfwenSTRFGbdollDeltafRn}).

Let us now start our journey\footnote{In this set of notes, we will not go into the details of the boundary regularity theory in Lebesgue spaces for solutions
of fractional equations. For this, we refer to~\cite{MR3276603, MR3293447}.
See also~\cite{MR3168912, MR3482695, NLBFR}
for the interior and boundary regularity theories in H\"older spaces.}
towards this global and interior regularity theory in Lebesgue spaces.

\section{Riesz potential analysis}\label{R:P:A:SEC}

Given~$s\in(0,1)$, we now turn our attention to the regularity theory in Lebesgue spaces for solutions of~$(-\Delta)^s u=f$ in all~$\R^n$.
A first natural question is whether the facts that~$(-\Delta)^s u\in L^p(\R^n)$ and~$u$ goes to zero at infinity
can automatically guarantee that~$u\in L^p(\R^n)$. The answer to this question is negative, even when~$n=1$, as pointed out here below:

\begin{proposition}\label{HSZUONDCCONT}
Let~$p\in[1,+\infty)$,~$s\in(0,1)$ and~$\beta\in\left(\max\left\{0,\frac1p-2s\right\},\frac1p\right)$.

Let
$$\R\ni x\longmapsto u(x):=\frac{1}{(1+x^2)^{\frac\beta2}}.$$
Then,
\begin{equation}\label{PLSK:001}
u\not\in L^p(\R).
\end{equation}
Also, if~$|x|\ge1$,
\begin{equation}\label{PLSK:002}
|(-\Delta)^s u(x)|\le\frac{C}{|x|^{\beta+2s}},
\end{equation}
for some constant~$C>0$.

Finally,
\begin{equation}\label{PLSK:003}
(-\Delta)^s u \in L^p(\R).
\end{equation}
\end{proposition}

\begin{proof} The claim in~\eqref{PLSK:001} follows from the fact that~$\beta p<1$.

Moreover, the claim in~\eqref{PLSK:003} would follow from~\eqref{PLSK:002} and the fact that~$(\beta+2s)p>1$, therefore we can focus our attention on the proof of~\eqref{PLSK:002}.

To this end, we pick~$x\in\R\setminus(-1,1)$. Actually, since~$u$ is even, to prove~\eqref{PLSK:002} we can restrict ourselves to the case~$x\ge1$: we note that
$$ \int_{\{|y|\in(x/2,2x)\}} \frac{|u(x+y)|}{|y|^{1+2s}}\,dy\le
C\int_{\{|y|\in(x/2,2x)\}} \frac{dy}{|y|^\beta |y|^{1+2s}}\le C\int_{\{|y|\in(x/2,2x)\}} \frac{dy}{x^\beta x^{1+2s}}
\le \frac{C}{x^{\beta+2s}},$$
up to renaming~$C$ at any step of the computation.

Accordingly,
\begin{eqnarray*}&& \int_{\{|y|>x/2\}} \frac{|u(x+y)|}{|y|^{1+2s}}\,dy\le
\frac{C}{x^{\beta+2s}}+\int_{\{|y|\ge2x\}} \frac{|u(x+y)|}{|y|^{1+2s}}\,dy\\&&\qquad\qquad\le
\frac{C}{x^{\beta+2s}}+C\int_{\{|y|\ge2x\}} \frac{dy}{|x+y|^\beta |y|^{1+2s}}\\&&\qquad\qquad
\le\frac{C}{x^{\beta+2s}}+C\int_{\{|y|\ge 2x\}} \frac{dy}{|y|^{1+\beta+2s}}
\le\frac{C}{x^{\beta+2s}}\end{eqnarray*}
and therefore
\begin{equation}\label{PLSK:008}\begin{split}&
\left| \int_{\{ |y|>x/2\}}\frac{u(x+y)-u(x)}{|y|^{1+2s}}\,dy\right|
\le\frac{C}{x^{\beta+2s}}+\int_{\{ |y|>x/2\}}\frac{|u(x)|}{|y|^{1+2s}}\,dy\\&\qquad
\le\frac{C}{x^{\beta+2s}}+\frac1{x^\beta}\int_{\{ |y|>x/2\}}\frac{dy}{|y|^{1+2s}}\le\frac{C}{x^{\beta+2s}}.\end{split}
\end{equation}

In addition,
\begin{equation*}\begin{split}&
\left| \int_{\{ |y|\le x/2\}}\frac{u(x+y)-u(x)}{|y|^{1+2s}}\,dy\right|
=\left| \int_{\{ |y|\le x/2\}}\frac{u(x+y)-u(x)-u'(x)y}{|y|^{1+2s}}\,dy\right|\\&\qquad
\le \int_{\{ |y|\le x/2\}}\frac{\sup_{|t|\le x/2}|u''(x+t)|}{|y|^{2s-1}}\,dy
\le 
\sup_{(x/2,2x)}|u''(t)|
\int_{\{ |y|\le x/2\}}\frac{dy}{|y|^{2s-1}}\\&\qquad\le
\frac{C}{x^{2+\beta}}\;x^{2-2s}=\frac{C}{x^{\beta+2s}}.\end{split}
\end{equation*}
{F}rom this and~\eqref{PLSK:008} we obtain the desired result in~\eqref{PLSK:002}.
\end{proof}

In view of this result, some care is required when dealing with solutions of~$(-\Delta)^s u=f$ in all~$\R^n$.
A first step is to somewhat select the ``appropriate'' notion of solution, since~$u+c$ would also satisfy the same equation.
For this, it is useful to consider the kernel
$$ \R^n\setminus\{0\}\ni x\mapsto {\mathcal{R}}(x):=\begin{dcases} \displaystyle\frac{\displaystyle\Gamma\left(\frac{n-2s}{2}\right)}{2^{2s}\pi^{\frac{n}2}\,\Gamma(s)} \,|x|^{2s-n} & {\mbox{ if }} n>2s,\\
	\frac{1}{2\cos(\pi s)\Gamma(2s-1)}|x|^{2s-1} & {\mbox{ if~$ n=1$ and~$s\in\left(\displaystyle\frac12,1\right)$,}}\\ -\frac{1}{\pi}\ln|x|
	& {\mbox{ if~$ n=1$ and~$s=\displaystyle\frac12$.}}
\end{dcases}$$
This is sometimes called the Riesz kernel \index{Riesz kernel} and, for our purposes, it provides the fundamental solution \index{fundamental solution} of the fractional Laplace operator (see Corollary~\ref{FURIEZ}), namely the function~$u:={\mathcal{R}}*f$
is a distributional solution of~$(-\Delta)^s u=\delta_0*f=f$ in~$\R^n$.

Let us observe that~${\mathcal{R}}\in L^1_{\rm loc}(\R^n)$, hence
\begin{equation}\label{PRODpoikjhr3}
{\mbox{${\mathcal{R}}*f$ is properly defined whenever, e.g.,~$f\in C^\infty_c(\R^n)$.}}\end{equation}


When~$n>2s$ we also have the following result:

\begin{lemma}\label{PRODpoikjhr3-2}
If~$s\in(0,1)$ and~$n>2s$, then
\begin{equation*}
\widehat{\mathcal{R}}(\xi)=(2\pi|\xi|)^{-2s} 
\qquad\text{for }\xi\in\R^n,\ \text{in the sense of distributions.}
\end{equation*}
\end{lemma}

\begin{proof}
Notice that we need here the condition~$n>2s$ in order for the term~$|\xi|^{-2s}$ to be locally integrable (i.e., to define
a distribution).
To prove the claim,
we take~$\varphi\in C^\infty_c(\R^n)$ and~$\e>0$. We use~\eqref{CAL:GA:098} 
with~$ f:={\mathcal{F}}^{-1}\left(\frac{\varphi(\xi)}{(\e+4\pi^2|\xi|^2)^{s}}\right)$. This gives that
\begin{eqnarray*}&&
\int_{\R^n} \widehat {\mathcal{R}}(\xi)\,{\frac{(4\pi^2|\xi|^{2})^s\,\varphi(\xi)}{(\e+4\pi^2|\xi|^2)^{s}}}\,d\xi=
(2\pi)^{2s} \int_{\R^n}|\xi|^{2s}\widehat {\mathcal{R}}(\xi)\,\overline{\frac{\varphi(\xi)}{(\e+4\pi^2|\xi|^2)^{s}}}\,d\xi\\&&\qquad=
(2\pi)^{2s} \int_{\R^n}|\xi|^{2s}\widehat {\mathcal{R}}(\xi)\,\overline{\widehat f(\xi)}\,d\xi= f(0)=\int_{\R^n}
\frac{\varphi(\xi)}{(\e+4\pi^2|\xi|^2)^{s}}\,d\xi
\end{eqnarray*}
and thus, taking the limit as~$\e\searrow0$,
\begin{eqnarray*}&&
\int_{\R^n} \widehat {\mathcal{R}}(\xi)\,\varphi(\xi)\,d\xi=
\int_{\R^n}\frac{\varphi(\xi)}{(4\pi^2|\xi|^2)^{s}}\,d\xi
,\end{eqnarray*}
which proves the desired result.
\end{proof}

Now we remark that the Riesz kernels enjoy a natural semigroup property with respect to the fractional parameter: 

\begin{lemma}\label{PRODpoikjhr3-2:le}
If~$s$,~$\sigma\in(0,1)$,~$s+\sigma\in(0,1)$ and~$2(s+\sigma)<n$, denoting by~${\mathcal{R}}^{(s)}$ the Riesz kernel with fractional parameter~$s$
and by~${\mathcal{R}}^{(\sigma)}$ the Riesz kernel with fractional parameter~$\sigma$, we have that
\begin{equation*}
{\mathcal{R}}^{(s)}*{\mathcal{R}}^{(\sigma)}=
{\mathcal{R}}^{(s+\sigma)}
\end{equation*}
in the sense of distributions.
\end{lemma}

\begin{proof} By Lemma~\ref{PRODpoikjhr3-2}, in the sense of distributions we have that
\[
\widehat{\mathcal{R}}^{(s)}(\xi)\widehat{\mathcal{R}}^{(\sigma)}(\xi)=(2\pi|\xi|)^{-2(s+\sigma)}=\widehat{\mathcal{R}}^{(s+\sigma)}(\xi)
\qquad\text{for }\xi\in\R^n,
\]
from which one obtains the desired result by the inversion of the Fourier Transform.
\end{proof}

We now point out that~\eqref{PRODpoikjhr3} can be conveniently generalized in Lebesgue spaces as follows:

\begin{lemma}\label{P:RIE:1}
Let~$n>2sp$ and~$p\in[1,+\infty)$. Let~$f\in L^p(\R^n)$. Then, the integral defining~${\mathcal{R}}*f(x)$ converges absolutely for a.e.~$x\in\R^n$.
\end{lemma}

We stress that the condition~$p\in[1,+\infty)$ in Lemma~\ref{P:RIE:1} cannot be improved, since if~$f:=1$ we have that~${\mathcal{R}}*f$ is never finite.\medskip

Moreover, the condition~$n>2sp$ in Lemma~\ref{P:RIE:1} cannot be removed. Indeed, if~$n\le2sp$
we consider the function 
$$ \R^n\ni x\longmapsto f(x):=\frac{\chi_{\R^n\setminus B_{e^2}}(x)}{|x|^{\frac{n}p}\,\ln |x|\, (\ln(\ln |x|))^\gamma},$$
with
$$ \gamma:=\begin{dcases}
2 & {\mbox{ if }}p=1,\\0 &{\mbox{ if }}p>1.
\end{dcases}$$
We observe that
$$\frac{\|f\|_{L^p(\R^n)}^p}{{\mathcal{H}}^{n-1}(\partial B_1)}=\int_{e^2}^{+\infty}\frac{dr}{r\,(\ln r)^{p}(\ln(\ln r))^{\gamma p}}=
\int_{2}^{+\infty}\frac{dy}{y^{p}(\ln y)^{\gamma p}}
=\int_{\ln2}^{+\infty}\frac{dt}{e^{(p-1)t} t^{\gamma p}}
<+\infty,$$ due to our choice of~$\gamma$,
whence~$f\in L^p(\R^n)$.

But, in this case
\begin{equation}\label{P:RIE:1e}
{\mbox{$|{\mathcal{R}}*f(x)|=+\infty$ for every~$x\in\R^n$.}}\end{equation}
Indeed, given every~$x\in\R^n$, if~$\beta\in\R$, 
$$ \lim_{y\to+\infty} \frac{|x-y|^\beta}{|y|^\beta}=1\qquad{\mbox{and}}\qquad
\lim_{y\to+\infty} \frac{\ln|x-y|}{\ln|y|}=1.$$
Hence, we can take~$k_x\ge e^2$ sufficiently large such that~$|{\mathcal{R}}(x-y)|\ge\frac{|{\mathcal{R}}(y)|}2$
and~$\ln(\ln|y|)>0$ whenever~$|y|\ge k_x$.

We now distinguish two cases, the first being when~$n=1$ and~$s=\frac12$ and the second being the complementary situation.
In the first case,
\begin{equation}\label{P:RIE:1eJMBIS}\begin{split}
		&\left|\int_{\R} {\mathcal{R}}(x-y)\,f(y)\,dy\right|=
		\frac1\pi\left|\int_{\{ |y|\ge e^2\}} \frac{-\ln|x-y|}{|y|^{\frac{1}p}\, \ln| y|\,(\ln(\ln |y|))^{2}}\,dy\right|\\&\qquad
		\ge \int_{{\{ |y|\ge e^2\}}\atop{\{|x-y|>1\}}} \frac{\ln|x-y|}{|y|^{\frac{1}p}\, \ln| y|\,(\ln(\ln |y|))^{2}}\,dy
		-\int_{{\{ |y|\ge e^2\}}\atop{\{|x-y|\le1\}}} \frac{|\ln|x-y||}{|y|^{\frac{1}p}\, \ln| y|\,(\ln(\ln |y|))^{2}}\,dy\\
		&\qquad \ge \frac12\int_{{\{ |y|\ge k_x\}}\atop{\{|x-y|>1\}}} \frac{dy}{|y|^{\frac{1}p}\, (\ln(\ln |y|))^{2}}\,dy
		-C_x,
\end{split}\end{equation}
for some~$C_x>0$ possibly depending also on~$x$.
We observe that, up to a multiplicative constant, the last integral term in~\eqref{P:RIE:1eJMBIS} takes the form
\begin{eqnarray*}
	\int_{k_x}^{+\infty} \frac{dy}{y^{\frac{1}p}\,(\ln(\ln y))^{\gamma}}=\int_{\ln(\ln{k_x})}^{+\infty} \frac{(\exp(e^t))^{1-\frac{1}p}e^t dt}{t^{\gamma}}=+\infty.
\end{eqnarray*}

In the second case,
\begin{equation}\label{P:RIE:1eJM}\begin{split}
\left|\int_{\R^n} {\mathcal{R}}(x-y)\,f(y)\,dy\right|&=
\int_{\R^n\setminus B_{e^2}} \frac{ |{\mathcal{R}}(x-y)|}{|y|^{\frac{n}p}\, \ln| y|\,(\ln(\ln |y|))^{2}}\,dy\\&\ge
\frac12\int_{\R^n\setminus B_{k_x}} \frac{| {\mathcal{R}}(y)|}{|y|^{\frac{n}p}\,\ln| y|\,(\ln(\ln| y|))^{2}}\,dy.\end{split}
\end{equation}
Also, up to a multiplicative constant, the last term in~\eqref{P:RIE:1eJM} becomes
\begin{eqnarray*}&&\int_{\R^n\setminus B_{k_x}} \frac{ |y|^{2s-n-\frac{n}p}}{\ln| y|\,(\ln(\ln| y|))^{\gamma}}\,dy
= {\mathcal{H}}^{n-1}(\partial B_1) \int_{k_x}^{+\infty} \frac{ r^{2s-1-\frac{n}p}}{\ln r\,(\ln(\ln r))^{\gamma}}\,dr\\&&\qquad={\mathcal{H}}^{n-1}(\partial B_1)\int_{\ln(\ln{k_x})}^{+\infty} \frac{ (\exp(e^t))^{2s-\frac{n}p}}{t^{\gamma}}\,dt.
\end{eqnarray*}
This is also equal to~$+\infty$ because if~$n<2sp$ the divergence comes from the exponential term, and if instead~$n=2sp$
we have that~$p>1$ (otherwise~$n=2s$ and we go back to the first case) and therefore the divergence comes from the fact that~$\gamma=0$.

These considerations show~\eqref{P:RIE:1e} and thus confirm that the condition~$n>2sp$ in Lemma~\ref{P:RIE:1} cannot be avoided.
\medskip

In our setting, Lemma~\ref{P:RIE:1} is also related to two auxiliary results of independent interest, namely
Lemmata~\ref{P:RIE:2} and~\ref{P:RIE:3} here below.

\begin{lemma}\label{P:RIE:2}
Let~$n> 2s$,~$f\in L^1(\R^n)$ and
\begin{equation}\label{P:RIE:5p1}q:=\frac{n}{n-2s}.\end{equation}

Then, for every~$\lambda>0$,
\begin{equation}\label{P:RIE:5p2} \left| \big\{x\in\R^n{\mbox{ s.t. }}|{\mathcal{R}}*f(x)|>\lambda\big\}\right|\le C\left( \frac{\|f\|_{L^1(\R^n)}}{\lambda}\right)^q,\end{equation}
for some constant~$C>0$ depending only on~$n$ and~$s$.
\end{lemma}

\begin{lemma}\label{P:RIE:3}
Let~$p\in(1,+\infty)$,~$n> 2sp$ and~$f\in L^p(\R^n)$.

Set
\begin{equation}\label{P:RIE:5}q:=\frac{np}{n-2sp}.\end{equation}
Then,
\begin{equation}\label{P:RIE:4} \| {\mathcal{R}}*f\|_{L^q(\R^n)}\le C\|f\|_{L^p(\R^n)},\end{equation}
for some constant~$C>0$ depending only on~$n$,~$s$ and~$p$.
\end{lemma}

In jargon, one can rephrase~\eqref{P:RIE:5p2} by saying that, if~$n>2s$ and~$q$ is as in~\eqref{P:RIE:5p1}, the convolution against the Riesz kernel is of weak type~$(1,q)$, according to the following notation:

\begin{definition} Let~$p\in[1,+\infty]$ and~$q\in[1,+\infty)$.
Let~$T$ be a map on some measurable class of functions.
We say that~$T$ is of weak type~$(p,q)$ \index{weak type~$(p,q)$} if, for every~$\lambda>0$, it satisfies an estimate of the form
$$ \left| \big\{ x\in\R^n {\mbox{ s.t. }} |Tf(x)|>\lambda\big\}\right|\leq C\left(\frac{\|f\|_{L^p(\R^n)}}{\lambda}\right)^q,~$$
for some~$C>0$.

We say that an operator~$T$ is of strong type~$(p,q)$ \index{strong type~$(p,q)$} if it satisfies an estimate of the form
$$ \|Tf\|_{L^q(\R^n)}\leq C\|f\|_{L^p(\R^n)},~$$
for some~$C>0$.
\end{definition}

Note that if~$T$ is of strong type~$(p,q)$ then it is of weak type~$(p,q)$,
thanks to the Chebyshev's Inequality.\medskip

We stress that the condition on~$q$ in~\eqref{P:RIE:5} of Lemma~\ref{P:RIE:3} is sharp. Namely, if~\eqref{P:RIE:4} holds true for every~$f\in L^p(\R^n)$
then necessarily~$q$ has the form requested in~\eqref{P:RIE:5}. This can be seen through a scaling argument, that is
by taking~$\vartheta>0$,~$f\in L^p(\R^n)$ and applying~\eqref{P:RIE:4} to~$f_\vartheta(x):=f(\vartheta x)$. This gives that
\begin{eqnarray*}&&
\frac{\| {\mathcal{R}}*f\|_{L^q(\R^n)}}{\vartheta^{n+\frac{2s}q}}=
\left(\frac{1}{\vartheta^{2s}}\int_{\R^n}\left| \frac1{\vartheta^n}\int_{\R^n} {\mathcal{R}}(X-Y)f(Y)\,dY
\right|^q\,dX\right)^{\frac1q}\\&&\qquad=
\left({\vartheta^{2s(q-1)-nq}}\int_{\R^n}\left| \frac1{\vartheta^n}\int_{\R^n} {\mathcal{R}}\left(\frac{X-Y}\vartheta\right)f(Y)\,dY
\right|^q\,dX\right)^{\frac1q}\\&&\qquad =
\left({\vartheta^{2s(q-1)-nq+n}}\int_{\R^n}\left| \int_{\R^n} {\mathcal{R}}(x-y)f(\vartheta y)\,dy
\right|^q\,dx\right)^{\frac1q}\\&&\qquad={\vartheta^{\frac{2s(q-1)-n(q-1)}{q}}}
\| {\mathcal{R}}*f_\vartheta\|_{L^q(\R^n)}\le C{\vartheta^{\frac{(2s-n)(q-1)}{q}}}
\|f_\vartheta\|_{L^p(\R^n)}=\frac{C{\vartheta^{\frac{(2s-n)(q-1)}{q}}}\|f\|_{L^p(\R^n)}}{\vartheta^{\frac{n}p}}.
\end{eqnarray*}
By taking the limits~$\vartheta\searrow0$ and~$\vartheta\nearrow+\infty$, it follows that necessarily~${n+\frac{2s}q}=
{\frac{n}p}+\frac{(n-2s)(q-1)}{q}$, which gives~\eqref{P:RIE:5}.
\medskip

We also note that condition~\eqref{P:RIE:5} naturally reduces to~\eqref{P:RIE:5p1} when~$p=1$.
In this sense, Lemma~\ref{P:RIE:2} can be considered the natural counterpart of
Lemma~\ref{P:RIE:3} when~$p=1$, also in terms of the corresponding exponent~$q$. Once again, condition~\eqref{P:RIE:5p1} cannot be removed from Lemma~\ref{P:RIE:2}: to see this, one can just assume that~\eqref{P:RIE:5p2} holds true and
apply it to~$f_\vartheta$, thus obtaining
\begin{eqnarray*}&&
\theta^{-n}\left| \big\{X\in\R^n{\mbox{ s.t. }}|{\mathcal{R}}*f(X)|>\lambda\theta^{2s}\big\}\right|=
\theta^{-n}\left| \left\{X\in\R^n{\mbox{ s.t. }}\left|\int_{\R^n}{\mathcal{R}}(X-Y)f(Y)\,dY\right|>\lambda\theta^{2s}\right\}\right|\\&&\qquad=
\left| \left\{x\in\R^n{\mbox{ s.t. }}\left|\int_{\R^n}{\mathcal{R}}(\theta x-Y)f(Y)\,dY\right|>\lambda\theta^{2s}\right\}\right|\\&&\qquad=
\left| \left\{x\in\R^n{\mbox{ s.t. }}\left|\int_{\R^n}{\mathcal{R}}\left(\frac{\theta x-Y}\theta\right)f(Y)\,dY\right|>\lambda\theta^n\right\}\right|\\&&\qquad=
\left| \left\{x\in\R^n{\mbox{ s.t. }}\left|\int_{\R^n}{\mathcal{R}}(x-y)f(\theta y)\,dy\right|>\lambda\right\}\right|\\&&\qquad=
\left| \big\{x\in\R^n{\mbox{ s.t. }}|{\mathcal{R}}*f_\theta(x)|>\lambda\big\}\right|\\&&\qquad\le C\left( \frac{\|f_\theta\|_{L^1(\R^n)}}{\lambda}\right)^q\\&&\qquad= C\left( \frac{\|f\|_{L^1(\R^n)}}{\lambda\theta^n}\right)^q.\end{eqnarray*}
Therefore, taking~$\mu$ such that~$\{|{\mathcal{R}}*f |>\mu\}\neq\varnothing$
and using the notation~$\lambda:=\mu/\theta^{2s}$,
$$ \left| \big\{X\in\R^n{\mbox{ s.t. }}|{\mathcal{R}}*f(X)|>\mu\big\}\right|\le
C\theta^{(2s-n)q+n}\|f\|_{L^1(\R^n)}^q.$$
By taking the limits~$\vartheta\searrow0$ and~$\vartheta\nearrow+\infty$, it follows that necessarily~$(2s-n)q+n=0$, which gives~\eqref{P:RIE:5p1}.
\medskip

We also remark that Lemma~\ref{P:RIE:3} does not hold when~$p=1$ (its natural replacement being the weaker result stated in Lemma~\ref{P:RIE:2}). Indeed, one could take~$f\in C^\infty_c(\R^n,[0,+\infty))$ with unit mass and, for every~$\e>0$, define~$f_\e(x):=\frac{1}{\e^n}f\left(\frac{x}{\e}\right)$.
In this setting, if~\eqref{P:RIE:4} held true for~$p=1$, using~\eqref{P:RIE:5} 
we would find that
\begin{eqnarray*} 
\left(\int_{\R^n} \left|
\int_{\R^n} {\mathcal{R}}(x-y)f_\e(y)\,dy\right|^{\frac{n}{n-2s}}\,dx
\right)^{\frac{n-2s}{n}}=
\| {\mathcal{R}}*f_\e\|_{L^{\frac{n}{n-2s}}(\R^n)}\le C\|f_\e\|_{L^1(\R^n)}=C.
\end{eqnarray*}
But this is impossible since, by Fatou's Lemma,
\begin{eqnarray*}&&
\lim_{\e\searrow0}
\int_{\R^n} \left|\int_{\R^n} |x-y|^{2s-n}f_\e(y)\,dy\right|^{\frac{n}{n-2s}}\,dx
=
\lim_{\e\searrow0}
\int_{\R^n} \left|\int_{\R^n} \frac{|x-y|^{2s-n}}{\e^n}f\left(\frac{y}\e\right)\,dy\right|^{\frac{n}{n-2s}}\,dx\\&&\qquad
=\lim_{\e\searrow0}
\int_{\R^n} \left|\int_{\R^n} |x-\e Y|^{2s-n}f(Y)\,dY\right|^{\frac{n}{n-2s}}\,dx\ge
\int_{\R^n} \left(\int_{\R^n} |x|^{2s-n}f(Y)\,dY\right)^{\frac{n}{n-2s}}\,dx\\&&\qquad=
\int_{\R^n} \left( |x|^{2s-n}\right)^{\frac{n}{n-2s}}\,dx=\int_{\R^n} \frac{dx}{|x|^n}=+\infty
,\end{eqnarray*}
thus showing that the case~$p=1$ must be excluded from
Lemma~\ref{P:RIE:3}.
\medskip

It is also worth to remark that Lemma~\ref{P:RIE:3} is violated when~$n\le2sp$. Indeed, when~$n<2sp$ the exponent~$q$ in~\eqref{P:RIE:5} becomes negative, making it impossible to define~$L^q(\R^n)$. When instead~$n=2sp$,
formally~\eqref{P:RIE:5} would return~$q=\infty$, but~\eqref{P:RIE:4} does not hold for~$q=\infty$. To check this,
take~$\e\in\left(0,1-\frac1p\right)$ and
$$ f(x):=\frac{\chi_{B_{1/2}}(x)}{|x|^{2s}\, \big|\ln|x|\big|^{\frac{2s}{n}+\e}}.$$
Using polar coordinates and substituting for~$t:=-\ln r$, using Fatou's Lemma we see that
\begin{eqnarray*}&&
\lim_{x\to0}\int_{B_{1/2}}\frac{ |x-y|^{2s-n}}{|y|^{2s}\, \big|\ln|y|\big|^{\frac{2s}{n}+\e}}\,dy\ge
\int_{B_{1/2}}\frac{ |y|^{2s-n}}{|y|^{2s}\, \big|\ln|y|\big|^{\frac{2s}{n}+\e}}\,dy\\&&\qquad=
{\mathcal{H}}^{n-1}(\partial B_1)
\int_0^{{1/2}}\frac{dr}{r\,|\ln r|^{\frac{1}{p}+\e}}=
{\mathcal{H}}^{n-1}(\partial B_1)\int_{\ln2}^{+\infty}\frac{dt}{t^{\frac{1}{p}+\e}}=+\infty.
\end{eqnarray*}
Since~$n=2sp>2s$, this gives that
$$\lim_{x\to0} {\mathcal{R}}*f(x)=+\infty
$$
and therefore, in this case,
\begin{equation}\label{P:RIE:200}
\| {\mathcal{R}}*f\|_{L^\infty(\R^n)}=+\infty.
\end{equation}

On the other hand,
\begin{eqnarray*}&&
\|f\|_{L^p(\R^n)}^p=
\int_{B_{1/2}}\frac{dx}{|x|^{2sp}\, \big|\ln|x|\big|^{\frac{2sp}{n}+\e p}}=
{\mathcal{H}}^{n-1}(\partial B_1)
\int_0^{{1/2}}\frac{dr}{r\, |\ln r|^{1+\e p}}\\&&\qquad\qquad\qquad=
{\mathcal{H}}^{n-1}(\partial B_1)
\int_{\ln2}^{+\infty} \frac{dt}{t^{1+\e p}}<+\infty.
\end{eqnarray*}
Combining this with~\eqref{P:RIE:200}, we see that Lemma~\ref{P:RIE:3} does not hold when~$n=2sp$.\medskip

We are now ready to prove the results stated in Lemmata~\ref{P:RIE:1},~\ref{P:RIE:2} and~\ref{P:RIE:3}.

\begin{proof}[Proof of Lemma~\ref{P:RIE:1}]
We define
\begin{equation}\label{P:RIE:3MA56}
{\mathcal{R}}_0:={\mathcal{R}}\chi_{B_1}\qquad{\mbox{and}}\qquad{\mathcal{R}}_\infty:={\mathcal{R}}\chi_{\R^n\setminus B_1}.\end{equation}
By construction, we have that~${\mathcal{R}}_0\in L^1(\R^n)$. Thus, we can consider the function
$$ \Phi(x):=\int_{\R^n} |{\mathcal{R}}_0(x-y)|\,|f(y)|\,dy =|{\mathcal{R}}|*|f|(x)$$
and note that, by Young's Convolution Inequality,
\begin{equation}\label{P:RIE:12J} \|\Phi\|_{L^p(\R^n)}\le \|{\mathcal{R}}_0\|_{L^1(\R^n)}\| f\|_{L^p(\R^n)}<+\infty.\end{equation}
As a consequence,
\begin{equation}\label{P:RIE:12}
{\mbox{the integral defining~${\mathcal{R}}_0*f(x)$ converges absolutely for a.e.~$x\in\R^n$.}}\end{equation}

Furthermore, since~$n>2s$ we have that the function~$\R^n\ni x\mapsto |x|^{2s-n}$ belongs to~$L^m(\R^n\setminus B_1)$,
and consequently~${\mathcal{R}}_\infty\in L^m(\R^n)$, for all~$m\in\left(\frac{n}{n-2s},+\infty\right]$.
Hence, we let~$m_p$ be the dual exponent of~$p$, namely
\begin{equation}\label{P:RIE:12J:01} 
m_p:=\frac{p}{p-1}.\end{equation}
We stress that the condition~$n>2sp$ in Lemma~\ref{P:RIE:1} is equivalent to the fact that~$m_p\in\left(\frac{n}{n-2s},+\infty\right]$.
Hence, we have that~${\mathcal{R}}_\infty\in L^{m_p}(\R^n)$ and accordingly the H\"older's Inequality entails that
\begin{equation}\label{P:RIE:12J:02} \int_{\R^n}|{\mathcal{R}}_\infty(x-y)|\,|f(y)|\,dy\le\|{\mathcal{R}}_\infty\|_{L^{m_p}(\R^n)}\|f\|_{L^p(\R^n)}.\end{equation}
The desired result in Lemma~\ref{P:RIE:1} now follows from this inequality,
\eqref{P:RIE:12} and the fact that~${\mathcal{R}}={\mathcal{R}}_0+{\mathcal{R}}_\infty$.
\end{proof}

\begin{proof}[Proof of Lemmata~\ref{P:RIE:2} and~\ref{P:RIE:3}]
A core computation is common for both the proofs of Lemmata~\ref{P:RIE:2} and~\ref{P:RIE:3} and goes as follows.
Let~$p\in[1,+\infty)$ and~$n>2sp$. Let also~$q$ be as in~\eqref{P:RIE:5}. We claim that, for every~$\lambda>0$,
\begin{equation}\label{P:RIE:5p2-0} \left| \big\{x\in\R^n{\mbox{ s.t. }}|{\mathcal{R}}*f(x)|>\lambda\big\}\right|\le C\left( \frac{\|f\|_{L^p(\R^n)}}{\lambda}\right)^q,\end{equation}
for some constant~$C>0$ depending only on~$n$,~$s$ and~$p$.

To prove this, we can assume, without loss of generality that
\begin{equation}\label{P:RIE:5p2-0-p}{\mbox{$f$ does not vanish identically,}}\end{equation}
otherwise we are done. The strategy to prove~\eqref{P:RIE:5p2-0} is to revisit in a more quantitative way the proof of Lemma~\ref{P:RIE:1}.
To this end, we slightly modify the notation in~\eqref{P:RIE:3MA56} by allowing ourself the additional freedom of a new parameter~$\ell>0$, to be conveniently chosen in dependence of the given~$\lambda>0$ in what follows. Namely, we set
\begin{equation*}
{\mathcal{R}}_0:={\mathcal{R}}\chi_{B_\ell}\qquad{\mbox{and}}\qquad{\mathcal{R}}_\infty:={\mathcal{R}}\chi_{\R^n\setminus B_\ell}.\end{equation*}
Since~${\mathcal{R}}={\mathcal{R}}_0+{\mathcal{R}}_\infty$, we see that
\begin{equation*}
\big\{x\in\R^n{\mbox{ s.t. }}|{\mathcal{R}}*f(x)|>\lambda\big\}\subseteq
\left\{x\in\R^n{\mbox{ s.t. }}|{\mathcal{R}}_0*f(x)|>\frac\lambda2\right\}\cup\left\{x\in\R^n{\mbox{ s.t. }}|{\mathcal{R}}_\infty*f(x)|>\frac\lambda2\right\}
\end{equation*}
and, as a consequence,
\begin{equation}\label{P:RIE:12Ja}\begin{split}&
\left| \big\{x\in\R^n{\mbox{ s.t. }}|{\mathcal{R}}*f(x)|>\lambda\big\}\right|\\&\qquad\le
\left| \left\{x\in\R^n{\mbox{ s.t. }}|{\mathcal{R}}_0*f(x)|>\frac\lambda2\right\}\right|+\left| \left\{x\in\R^n{\mbox{ s.t. }}|{\mathcal{R}}_\infty*f(x)|>\frac\lambda2\right\}\right|.\end{split}
\end{equation}

Furthermore, recalling the Young's Convolution Inequality in~\eqref{P:RIE:12J} and using the Chebyshev's Inequality, we find that
\begin{equation}\label{P:RIE:12Jb}
\begin{split}&
\|{\mathcal{R}}_0\|_{L^1(\R^n)}^p \| f\|_{L^p(\R^n)}^p\ge
\int_{\R^n}\left(\int_{\R^n} |{\mathcal{R}}_0(x-y)|\,|f(y)|\,dy\right)^p\,dx\ge
\int_{\R^n}\left|\int_{\R^n} {\mathcal{R}}_0(x-y) f(y)\,dy\right|^p\,dx
\\&\qquad=\int_{\R^n}|{\mathcal{R}}_0*f(x)|^p\,dx\ge\left(\frac\lambda2\right)^p\left| \left\{x\in\R^n{\mbox{ s.t. }}|{\mathcal{R}}_0*f(x)|>\frac\lambda2\right\}\right|.
\end{split}
\end{equation}

Now, we use the H\"older's Inequality (recall~\eqref{P:RIE:12J:01}
and~\eqref{P:RIE:12J:02}) and we obtain
\begin{eqnarray*}
&\|{\mathcal{R}}_\infty\|_{L^{\frac{p}{p-1}}(\R^n)}\|f\|_{L^p(\R^n)}\ge
|{\mathcal{R}}_\infty*f(x)|
\end{eqnarray*}
and, as a byproduct,
\begin{equation}\label{P:RIE:12J:0131}
\begin{split}&
\left| \left\{x\in\R^n{\mbox{ s.t. }}|{\mathcal{R}}_\infty*f(x)|>\frac\lambda2\right\}\right|\le
\left| \left\{x\in\R^n{\mbox{ s.t. }}\|{\mathcal{R}}_\infty\|_{L^{\frac{p}{p-1}}(\R^n)}\|f\|_{L^p(\R^n)}>\frac\lambda2\right\}\right|\\ &\qquad=\begin{dcases}
0 & {\mbox{ if }} \;\displaystyle\|{\mathcal{R}}_\infty\|_{L^{\frac{p}{p-1}}(\R^n)}\|f\|_{L^p(\R^n)}\le\frac\lambda2,\\
+\infty & {\mbox{ if }} \;\displaystyle\|{\mathcal{R}}_\infty\|_{L^{\frac{p}{p-1}}(\R^n)}\|f\|_{L^p(\R^n)}>\frac\lambda2.
\end{dcases}\end{split}
\end{equation}

Hence, it is now convenient to choose the free parameter~$\ell$ appropriately in order to tune the result of~\eqref{P:RIE:12J:0131}. Namely, since
$$ \|{\mathcal{R}}_\infty\|_{L^{\frac{p}{p-1}}(\R^n)}^{\frac{p}{p-1}}\le
C_1 \int_{\R^n\setminus B_\ell} |x|^{ {\frac{p(2s-n)}{p-1}}}\,dx
\le C_2 \ell^{\frac{n - 2 sp}{1 - p}},$$
for some constants~$C_1$ and~$C_2$, recalling~\eqref{P:RIE:5p2-0-p}, the choice
$$ \ell:= \left(\frac{2C_2^{\frac{p-1}p}\,\|f\|_{L^p(\R^n)} }{\lambda}\right)^{\frac{p}{n-2sp}}
$$
entails that
$$ \|{\mathcal{R}}_\infty\|_{L^{\frac{p}{p-1}}(\R^n)}\|f\|_{L^p(\R^n)}\le\frac\lambda2,$$
whence~\eqref{P:RIE:12J:0131} returns that
$$ \left| \left\{x\in\R^n{\mbox{ s.t. }}|{\mathcal{R}}_\infty*f(x)|>\frac\lambda2\right\}\right|=0.$$

Combining this information with~\eqref{P:RIE:12Ja} and~\eqref{P:RIE:12Jb} we find that
\begin{eqnarray*}
&&\left| \big\{x\in\R^n{\mbox{ s.t. }}|{\mathcal{R}}*f(x)|>\lambda\big\}\right|\le
\frac{2^p\|{\mathcal{R}}_0\|_{L^1(\R^n)}^p \| f\|_{L^p(\R^n)}^p}{\lambda^p}\\&&\qquad\le \frac{C_3\, \| f\|_{L^p(\R^n)}^p}{\lambda^p} \left(\int_{B_\ell} |x|^{2s-n}\,dx\right)^p =\frac{C_4\, \| f\|_{L^p(\R^n)}^p \,\ell^{2sp}}{\lambda^p} =\frac{C_5\, \| f\|_{L^p(\R^n)}^p }{\lambda^p} 
\left(\frac{\|f\|_{L^p(\R^n)} }{\lambda}\right)^{\frac{2sp^2}{n-2sp}}\\&&\qquad\le
C_6\left(\frac{\|f\|_{L^p(\R^n)}}{\lambda}\right)^{\frac{np}{n-2sp}},
\end{eqnarray*}
for some positive constants~$C_3$,~$C_4$,~$C_5$ and~$C_6$, which completes the proof of~\eqref{P:RIE:5p2-0}.

We observe that for~$p=1$,~\eqref{P:RIE:5p2-0} reduces to~\eqref{P:RIE:5p2} and the proof of Lemma~\ref{P:RIE:2} is thereby complete.

We now refine the previous estimates to establish Lemma~\ref{P:RIE:3} as well. For this, we define
$$  \left(1,\frac{n}{2s}\right)\ni \zeta\longmapsto \varphi(\zeta):=
\frac{(p-1)\zeta}{p(\zeta-1)}=\frac{p-1}{p}+\frac{p-1}{p(\zeta-1)}.$$
We point out that~$\varphi$ is decreasing in~$\zeta$. Besides, since~$p\in\left(1,\frac{n}{2s}\right)$,
$$ \lim_{\zeta\searrow1}\varphi(\zeta)=+\infty\qquad{\mbox{and}}\qquad\lim_{\zeta\nearrow n/2s}\varphi(\zeta)=
\frac{n}{n-2s}\left(1-\frac1p\right)\in(0,1).
$$
Therefore we can pick~$p_1\in \left(1,\frac{n}{2s}\right)$ such that~$\varphi(p_1)\in(0,1)$.
We let
$$ p_0:=1,\qquad q_0:=\frac{n}{n-2s}\quad{\mbox{ and }}\quad q_1:=\frac{np_1}{n-2sp_1}.$$
We observe that the convolution against the Riesz kernel is of weak type~$(p_0,q_0)$, thanks to Lemma~\ref{P:RIE:2},
and of weak type~$(p_1,q_1)$, thanks to~\eqref{P:RIE:5p2-0} (used here with~$p$ replaced by~$p_1$ and~$q$ replaced by~$q_1$,
which is possible since~$q_1$ respects the structure of~\eqref{P:RIE:5}).

As a consequence,
by the Marcinkiewicz Interpolation Theorem (see e.g.~\cite[Appendix~B]{MR0290095}),
for every~$\theta\in(0,1)$,
\begin{equation}\label{P:RIE:3MA}
{\mbox{the convolution against the Riesz kernel is of strong type~$(p_\theta,q_\theta)$,}}\end{equation}
where the exponents~$p_\theta$ and~$q_\theta$ are given by the relation
$$\frac1{p_\theta}=\frac{1-\theta}{p_0}+\frac\theta{p_1}\qquad{\mbox{and}}\qquad
\frac1{q_\theta}=\frac{1-\theta}{q_0}+\frac\theta{q_1}.$$

The choice~$\theta:=\varphi(p_1)$ corresponds to
\begin{eqnarray*}
\frac1{p_\theta}=1-\varphi(p_1)+\frac{\varphi(p_1)}{p_1}=1+\frac{\varphi(p_1)}{p_1}(1-p_1)=1+\frac{(p-1)}{p(p_1-1)}(1-p_1)=\frac1p
\end{eqnarray*}
and, by~\eqref{P:RIE:5},
\begin{eqnarray*}&&
\frac1{q_\theta}=\frac{(1-\varphi(p_1))(n-2s)}{n}+\frac{\varphi(p_1)\,(n-2sp_1)}{np_1}\\&&\qquad=
\frac{\left(1-\frac{(p-1)p_1}{p(p_1-1)}\right)(n-2s)}{n}+\frac{\frac{(p-1)p_1}{p(p_1-1)}(n-2sp_1)}{np_1}=\frac{n-2sp}{np}=\frac1q.
\end{eqnarray*}
Hence, with this choice, we have that~$(p_\theta,q_\theta)=(p,q)$
and, in light of~\eqref{P:RIE:3MA}, this completes the proof of Lemma~\ref{P:RIE:3}.
\end{proof}

We stress that the positive result in Lemma~\ref{P:RIE:3} is somewhat limited by the negative example showcased in Proposition~\ref{HSZUONDCCONT}.
More explicitly, the condition that~$f\in L^p(\R^n)$ does not guarantee that~${\mathcal{R}}*f\in L^p(\R^n)$, as illustrated here below:

\begin{proposition}\label{HSZUONDCCONT2}
Let~$p\in[1,+\infty)$ and~$\gamma\in(n,n+2sp]$.

Let
$$\R^n\ni x\longmapsto f(x):=\frac{1}{(1+|x|^2)^{\frac{\gamma}{2p}}}.$$
Then,~$f\in L^p(\R^n)$ but~${\mathcal{R}}*f\not\in L^p(\R^n)$.
\end{proposition}

\begin{proof} The fact that~$\gamma>n$ ensures that~$f\in L^p(\R^n)$.

The claim that~${\mathcal{R}}*f\not\in L^p(\R^n)$ will follow from the assumption~$\gamma\le n+2sp$ once we prove that there exists~$R>0$ such that, for every~$x\in\R^n\setminus B_R$,
\begin{equation}\label{HSZUONDCCONT3}
|{\mathcal{R}}*f(x)|\ge\frac{c \ell(x)}{|x|^{\frac{\gamma}p-2s}},
\end{equation}
for some constant~$c>0$, where
$$ \ell(x):=\begin{cases}
\ln|x| & {\mbox{ if~$n=1$ and }}\displaystyle s=\frac12,\\
1&{\mbox{ otherwise.}}
\end{cases}$$
To check~\eqref{HSZUONDCCONT3}, we first consider the case~$n\ne 2s$. We use that
in this case both~$f$ and~${\mathcal{R}}$ do not change sign to find that, for all~$x\in\R^n\setminus B_2$,
\begin{eqnarray*}&&
|{\mathcal{R}}*f(x)|\ge\int_{B_{2|x|}\setminus B_{|x|}} |{\mathcal{R}}(x-y)|\,f(y)\,dy\ge
c\int_{B_{2|x|}\setminus B_{|x|}} |x-y|^{2s-n}\,|y|^{-\frac\gamma{p}}\,dy\\&&\qquad
\ge c\int_{B_{2|x|}\setminus B_{|x|}} |x|^{2s-n}\,|x|^{-\frac\gamma{p}}\,dy=c|x|^{2s-\frac\gamma{p}}
\end{eqnarray*}
for some~$c>0$ possibly varying at each step of the calculation.

This proves~\eqref{HSZUONDCCONT3} when~$n\ne 2s$ and we can therefore now focus on the case~$n=2s$, that is~$n=1$ and~$s=\frac12$. In this situation, if~$x\in\R\setminus (-4,4)$
and~$|x-y|<1$, we have that~$|y|\ge|x|-1\ge\frac{|x|}2$ and therefore
\begin{eqnarray*}
\int_{\{|x-y|<1\}} \big|\ln|x-y|\big| f(y)\,dy\le
\frac{C}{|x|^{\frac{\gamma}p}} \int_{\{|x-y|<1\}} \big|\ln|x-y|\big|\,dy\le\frac{C}{|x|^{\frac{\gamma}p}},
\end{eqnarray*}
for some~$C>0$ possibly varying at each step of the calculation.

Furthermore, if~$|y|\in\left(\frac{|x|}4,\frac{|x|}2\right)$ then~$|x-y|\ge|x|-|y|\ge\frac{|x|}2\ge1$, therefore
\begin{eqnarray*}&&
\int_{\{|x-y|\ge1\}} \ln|x-y| \,f(y)\,dy\ge\int_{\{|y|\in(|x|/4,|x|/2)\}}\ln|x-y| \,f(y)\,dy\\&&\qquad
\ge c\int_{\{|y|\in(|x|/4,|x|/2)\}}\frac{\ln|x|}{|x|^{\frac\gamma{p}}}\,dy=
\frac{c|x|\ln|x|}{|x|^{\frac\gamma{p}}}.
\end{eqnarray*}

As a result, up to renaming constants,
$$ |{\mathcal{R}}*f(x)|\ge \frac{c|x|\ln|x|}{|x|^{\frac\gamma{p}}}-
\frac{C}{|x|^{\frac{\gamma}p}}.$$
This gives that~\eqref{HSZUONDCCONT3} holds true whenever~$|x|$ is sufficiently large,
as desired.
\end{proof}

It is also useful to combine the Liouville-type result in Theorem~\ref{BiknsoU98ikjmd9} and
the bound obtained in Lemma~\ref{P:RIE:3} to obtain an estimate in Lebesgue spaces for the equation~$(-\Delta)^su=f$
in~$\R^n$:

\begin{corollary}
Let~$p\in(1,+\infty)$,~$n> 2sp$ and~$f\in L^p(\R^n)$.
Set~$q:=\frac{np}{n-2sp}$.

Let~$u$ be a distributional solution of~$(-\Delta)^su=f$ in~$\R^n$, with
$$ \int_{\R^n}\frac{|u(x)|}{1+|x|^{n+2s}}\,dx<+\infty.$$
Then, there exists a function~$\ell$, which is affine when~$s\in\Big(\frac12,1\Big)$ and constant when~$s\in\Big(0,\frac12\Big]$,
such that
$$ \| u-\ell\|_{L^q(\R^n)}\le C\|f\|_{L^p(\R^n)}
$$
for some constant~$C>0$ depending only on~$n$,~$s$ and~$p$.
\end{corollary}

\begin{proof} Let~$v:={\mathcal{R}}*f$ and~$\ell:=u-v$. Using H\"older's Inequality and~\eqref{P:RIE:4}, we note that
$$ \int_{\R^n}\frac{|v(x)|}{1+|x|^{n+2s}}\,dx \le\|v\|_{L^q(\R^n)}\left( \int_{\R^n}\frac{dx}{\left(1+|x|^{\frac{(n+2s)q}{q-1}}\right)}\right)^{\frac{q-1}q}<+\infty$$
and therefore
$$ \int_{\R^n}\frac{|\ell(x)|}{1+|x|^{n+2s}}\,dx <+\infty.$$
Since~$(-\Delta)^s\ell=0$, we thus deduce from Theorem~\ref{BiknsoU98ikjmd9} that~$\ell$ is affine, and constant whenever~$s\in\Big(0,\frac12 \Big]$.

The desired result thus follows from~\eqref{P:RIE:4}.
\end{proof}

See e.g.~\cite{MR0290095, MR0350027} for further information on the Riesz kernel.

\section{Bessel potential analysis}\label{bess:R:P:A:Sec}

The Riesz kernel presented in Section~\ref{R:P:A:SEC} has the interesting feature of being singular (at least when~$n\ge2s$, and non differentiable when~$n=2s$) at the origin
(also, it is locally integrable at the origin). This singular behavior is expected to be useful
in the development of a regularity theory (because if the convolution with a singular operator produces a nice function, necessarily the initial function needs to be ``even nicer''; conversely,
convolutions with smooth and compactly supported kernels always produce a smooth function,
hence smooth and compactly supported kernels have smoothing effects, but do not entail any regularity theory on the initial function).

On the other hand, the Riesz kernel is not integrable at infinity, producing somewhat inconvenient technical complications (as well as specific pathologies, as pointed out in Propositions~\ref{HSZUONDCCONT} and~\ref{HSZUONDCCONT2}).

To circumvent this difficulty, it is customary to consider, instead of the fractional Laplacian, the fractional operator obtained from the identity minus the Laplacian, that is~$(1-\Delta)^s$.
This operator\footnote{As a notational remark, we mention that the operator~$(1-\Delta)^s$ is denoted by~$\langle D_x\rangle^{2s}$ in~\cite[Definition~2.39]{MR2884718}.}
 is defined by its action in the Fourier space for all~$u\in C^\infty_c(\R^n)$ by
\begin{equation}\label{HSZUONDCCONT-FF} {\mathcal{F}}\Big( (1-\Delta)^s u\Big)(\xi)= \big(1+4\pi^2|\xi|^2\big)^{s}\widehat u(\xi).\end{equation}
Interestingly, in this section, we will not confine ourselves to the case~$s\in(0,1)$ and we will instead embrace the full range~$s\in(0,+\infty)$ (and sometimes even~$s\in\R$).
The Bessel kernel~${\mathcal{B}}$ \index{Bessel kerel} will thus be defined as the convolution operator\footnote{The function~${\mathcal{B}}$ appears in various forms in the literature and it is sometimes called the Bessel function, or the
Hankel function, see e.g.~\cite[page~260]{MR0209834}.}
induced by the fundamental solution of such a kernel.

To this end, we point out that: \index{fundamental solution}

\begin{lemma}\label{HSZUONDCCONT-FF-G} Let~$s\in(0,+\infty)$.
The fundamental solution of the operator~$(1-\Delta)^s$ can be written in the form
\begin{equation*}
{\mathcal{B}}(x)=
\frac1{(4\pi)^{\frac{n}2}\Gamma (s)}\int _{0}^{+\infty }\frac{e^{-\frac{ |x|^2}{4\tau}-\tau}}{\tau^{\frac{n}2+1-s} } \,d\tau,
\qquad\text{for }x\in\R^n.\end{equation*}
\end{lemma}

\begin{proof} We start from the definition of Euler Gamma Function for~$z\in\C$ with~$\Re z>0$, namely
$$\Gamma (z)=\int _{0}^{+\infty }t^{z-1}e^{-t}\,dt.$$

Given~$b>0$, the substitution~$\tau:=\frac{t}b$ gives
$$ \Gamma (z)=b^z\int _{0}^{+\infty }\tau^{z-1}e^{-b\tau}\,d\tau.$$
Choosing~$z:=s$ and~$b:=1+4\pi^2|\xi|^2$, we find
$$ (1+4\pi^2|\xi|^2)^{-s}=\frac1{\Gamma (s)}\int _{0}^{+\infty }\tau^{s-1}e^{-(1+4\pi^2|\xi|^2)\tau}\,d\tau.$$
Thus, by~\eqref{HSZUONDCCONT-FF}, the fundamental solution~${\mathcal{B}}$ of the operator~$(1-\Delta)^s$ satisfies
\begin{equation}\label{HSZUONDCCONT-FFg} \widehat {\mathcal{B}}(\xi)=\big(1+4\pi^2|\xi|^2\big)^{-s}=\frac1{\Gamma (s)}\int _{0}^{+\infty }\tau^{s-1}e^{-(1+4\pi^2|\xi|^2)\tau}\,d\tau\end{equation}
and therefore, using the inverse Fourier Transform of the Gau{\ss}ian,
\begin{eqnarray*}
{\mathcal{B}}(x)=\frac1{\Gamma (s)}\int _{0}^{+\infty }\tau^{s-1} e^{-\tau} {\mathcal{F}^{-1}}\Big(e^{-4\pi^2\tau |\xi|^2}\Big)(x)\,d\tau
=\frac1{(4\pi)^{\frac{n}2}\Gamma (s)}\int _{0}^{+\infty }\tau^{s-1-\frac{n}2} e^{-\tau} e^{-\frac{ |x|^2}{4\tau}} \,d\tau,
\end{eqnarray*} as desired.
\end{proof}

For clarity, we make explicit the relation in~\eqref{HSZUONDCCONT-FFg},
also to stress its analogy with Lemma~\ref{PRODpoikjhr3-2}:

\begin{lemma}\label{PRODpoikjhr3-2:le-7}
If~$s\in(0,+\infty)$ then
\begin{equation*}
{\mbox{$\widehat{\mathcal{B}}(\xi)=\big(1+4\pi^2|\xi|^2\big)^{-s}$.}}\end{equation*}\end{lemma}
 
Also, in analogy with Lemma~\ref{PRODpoikjhr3-2:le},
we have that the Bessel kernels enjoy a natural semigroup property with respect to the fractional parameter: 

\begin{lemma}\label{L912-xpsmv}
If~$s$,~$\sigma\in(0,+\infty)$, denoting by~${\mathcal{B}}^{(s)}$ the Bessel kernel with parameter~$s$
and by~${\mathcal{B}}^{(\sigma)}$ the Bessel kernel with parameter~$\sigma$, we have that
\begin{equation*}
{\mathcal{B}}^{(s)}*{\mathcal{B}}^{(\sigma)}=
{\mathcal{B}}^{(s+\sigma)}
\end{equation*}
in the sense of distributions.
\end{lemma}

\begin{proof} By Lemma~\ref{PRODpoikjhr3-2:le-7},
\begin{equation*}\begin{split}&
{\mathcal{F}}\big({\mathcal{B}}^{(s)}*{\mathcal{B}}^{(\sigma)}\big)(\xi)=
\widehat{\mathcal{B}}^{(s)}(\xi)\widehat{\mathcal{B}}^{(\sigma)}(\xi)
=\big(1+4\pi^2|\xi|^2\big)^{-s}\big(1+4\pi^2|\xi|^2\big)^{-{\sigma}}\\&\qquad\qquad=\big(1+4\pi^2|\xi|^2\big)^{-{(s+\sigma)}}=
\widehat{\mathcal{B}}^{(s+\sigma)}(\xi)\end{split}
\end{equation*}
and the desired result follows.
\end{proof}

Differently from the case of Riesz potentials, we have that Bessel potentials have finite (in fact, unit) mass:

\begin{lemma}\label{L419} For all~$s\in(0,+\infty)$
it holds that
$$\int_{\R^n}{\mathcal{B}}(x)\,dx=1.$$
\end{lemma}

\begin{proof} Let~$\varphi\in C^\infty_c(B_2,[0,1])$ with~$\varphi=1$ in~$B_1$. Given~$\e>0$, let~$\varphi_\e(x):=\varphi(\e x)$. By Lemma~\ref{PRODpoikjhr3-2:le-7},
using the changes of variables~$X:=\e x$ and~$Y:=\frac\xi\e$,
\begin{equation*} \begin{split}&\int_{\R^n}{\mathcal{B}}(x)\,\varphi_\e(x)\,dx=
\int_{\R^n}\widehat{\mathcal{B}}(\xi)\,\check\varphi_\e(x)\,dx=
\int_{\R^n}\big(1+4\pi^2|\xi|^2\big)^{-{s}}\,\check\varphi_\e(\xi)\,d\xi\\&=\iint_{\R^n\times\R^n}\big(1+4\pi^2|\xi|^2\big)^{-s}\,\varphi_\e(x)\,e^{2\pi ix\cdot\xi}\,dx\,d\xi=\iint_{\R^n\times\R^n}\big(1+4\e^2\pi^2|Y|^2\big)^{-s}\,\varphi(X)\,e^{2\pi i X\cdot Y}\,dX\,dY.
\end{split}
\end{equation*}
{F}rom this, by taking the limit as~$\e\searrow0$, we arrive at
\begin{equation*} \begin{split}&\int_{\R^n}{\mathcal{B}}(x)\,dx=
\iint_{\R^n\times\R^n}\varphi(X)\,e^{2\pi i X\cdot Y}\,dX\,dY=
\int_{\R^n}\check\varphi(Y)\,dY=\varphi(0)=1,
\end{split}
\end{equation*}
as desired.
\end{proof}

\begin{corollary}\label{CO914athcal}
For all~$p\in[1,+\infty]$ and~$f\in L^p(\R^n)$,
$$\|{\mathcal{B}}*f\|_{L^p(\R^n)}\le\|f\|_{L^p(\R^n)}.$$
\end{corollary}

\begin{proof} The claim follows from Young's Convolution Inequality and Lemma~\ref{L419}.
\end{proof}

One can also detect precisely the decay of the Bessel potential:

\begin{lemma}\label{HSZUONDCCONT-FF-G-DECA} Let~$s\in(0,+\infty)$.
We have that
$$ \lim_{|x|\to+\infty} |x|^{\frac{n+1}2-s}\,e^{|x|}\, {\mathcal{B}}(x)=
\frac{1}{2^{\frac{n-1}{2}+s}\,\pi^{\frac{n-1}{2}}\,\Gamma (s)}.$$
\end{lemma}

\begin{proof} By Lemma~\ref{HSZUONDCCONT-FF-G}, and using the substitution~$t:=\frac{ |x|}{2\sqrt\tau}-\sqrt\tau$, with inverse
\[
\tau= \frac12 \left(|x| + t^2-t\sqrt{2 |x| + t^2}\right)>0,
\]
we find that
\begin{eqnarray*}
e^{|x|} \,{\mathcal{B}}(x)&=&
\frac1{(4\pi)^{\frac{n}2}\Gamma (s)}\int _{0}^{+\infty }\frac{e^{-\left(\frac{ |x|^2}{4\tau}+\tau-|x|\right)}}{\tau^{\frac{n}2+1-s} } \,d\tau
\\&=&
\frac1{(4\pi)^{\frac{n}2}\Gamma (s)}\int _{0}^{+\infty }\frac{e^{-\left(\frac{ |x|}{2\sqrt\tau}-\sqrt\tau\right)^2}}{\tau^{\frac{n}2+1-s} } \,d\tau\\&=&
\frac{2^{\frac{n}2+1-s}}{(4\pi)^{\frac{n}2}\Gamma (s)}\int _{-\infty}^{+\infty }\frac{e^{-t^2}}{\left(|x| + t^2-t\sqrt{2 |x| + t^2}\right)^{\frac{n}2-s} } 
\frac{ 1}{\sqrt{2 |x| + t^2}}\,dt
\\&=&
\frac{2^{\frac{1-n}2-s}}{\pi^{\frac{n}2}\,\Gamma (s)\,|x|^{\frac{n+1}2-s}}\int _{-\infty}^{+\infty }\frac{e^{-t^2}}{\left(1 + \frac{t^2-t\sqrt{2 |x| + t^2}}{|x|}\right)^{\frac{n}2-s} \,\sqrt{1 +\frac{ t^2}{2 |x|}}}\,dt
\end{eqnarray*}
and the Dominated Convergence Theorem yields the desired result.
\end{proof}

\begin{corollary}\label{210pkf0ikj23456789iytdxscvbnm-934098i7ytghbnfraCxfr}
Let~$s\in(0,+\infty)$.
Then, there exist positive constants~$R$ and~$C$, depending only on~$n$ and~$s$, such that for every~$x\in\R^n\setminus B_R$
$$ |\nabla {\mathcal{B}}(x)|\le \frac{C}{|x|^{\frac{n+1}2-s}\,e^{|x|}}.$$
\end{corollary}

\begin{proof} For the sake of clarity, we denote here by~${\mathcal{B}}^{(s,n)}$ the Bessel potential in~$\R^n$ with fractional parameter~$s$.
In light of Lemma~\ref{HSZUONDCCONT-FF-G}, we see that
\begin{equation}\label{HSZUONDCCONT-FF-G-kefmgEGSND-99}\begin{split}
|\partial_j{\mathcal{B}}^{(s,n)}(x)|&\le
C\,\left| \frac{\partial}{\partial x_j}\int _{0}^{+\infty }\frac{e^{-\frac{ |x|^2}{4\tau}-\tau}}{\tau^{\frac{n}2+1-s} } \,d\tau\right|\\&=
C\,\left|\int _{0}^{+\infty }\frac{x_j e^{-\frac{ |x|^2}{4\tau}-\tau}}{\tau^{\frac{n}2+2-s} } \,d\tau\right|\\&\le
C|x|\,\int _{0}^{+\infty }\frac{e^{-\frac{ |x|^2}{4\tau}-\tau}}{\tau^{\frac{n+2}2+1-s} } \,d\tau\\&\le
C|x|\,|{\mathcal{B}}^{(s,n+2)}(x)|,
\end{split}\end{equation}
up to renaming~$C$ line after line.

Also, by Lemma~\ref{HSZUONDCCONT-FF-G-DECA} (used here with~$n+2$ instead of~$n$) we can find~$R>0$ such that if~$|x|\ge R$ then
$$ |{\mathcal{B}}^{(s,n+2)}(x)|\le\frac{C}{|x|^{\frac{n+3}2-s}\,e^{|x|}}.$$
The desired result plainly follows from this estimate and~\eqref{HSZUONDCCONT-FF-G-kefmgEGSND-99}.
\end{proof}

While Lemma~\ref{HSZUONDCCONT-FF-G-DECA} and Corollary~\ref{210pkf0ikj23456789iytdxscvbnm-934098i7ytghbnfraCxfr}
take into account the behavior of the Bessel potential at infinity, the situation at the origin depends on the parameters. For instance, we point out this estimate: 

\begin{lemma} \label{HSZUONDCCONT-FF-G-0ojr236-201-ILLEM}
Let~$s\in(0,+\infty)$. For all~$x\in B_5\setminus\{0\}$,
\begin{equation}\label{HSZUONDCCONT-FF-G-0ojr236-201}
\big|{\mathcal{B}}^{(s)}(x)\big|\leq\begin{dcases} \displaystyle\frac{C}{|x|^{n-2s}} &{\mbox{ if~$n>2s$,}}\\
C(1+|\ln|x||)&{\mbox{ if~$n=2s$,}}\\ C&{\mbox{ if~$n<2s$.}}
\end{dcases}\end{equation}
\end{lemma}

\begin{proof} We claim that, for all~$x\in B_5\setminus\{0\}$,~$m$,~$s>0$,
\begin{equation}\label{TEMP:ES:BES1}
\int _{0}^{+\infty }\frac{e^{-\frac{ |x|^2}{4\tau}-\tau}}{\tau^{\frac{m}2+1-s} } \,d\tau \le\begin{dcases} \displaystyle\frac{C}{|x|^{m-2s}} &{\mbox{ if~$m>2s$,}}\\
C(1+|\ln|x||)&{\mbox{ if~$m=2s$,}}\\ C&{\mbox{ if~$m<2s$.}}
\end{dcases}
\end{equation}
where~$C>0$ depends only on~$m$ and~$s$.

To check this, we use the change of variable~$\vartheta:=\frac{ |x|^2}{\tau}$ and we
observe that
\begin{equation}\label{TEMP:ES:BES1-PTRE}\begin{split}&
\int _{0}^{|x|^2}\frac{e^{-\frac{ |x|^2}{4\tau}-\tau}}{\tau^{\frac{m}2+1-s} } \,d\tau \le
\int _{0}^{|x|^2}\frac{e^{-\frac{ |x|^2}{4\tau}}}{\tau^{\frac{m}2+1-s} } \,d\tau\\&\qquad=
\frac1{|x|^{m-2s}}\int _{1}^{+\infty} \vartheta^{\frac{m}2-1-s}
e^{-\frac{ \vartheta}{4}}\,d\vartheta\le\frac{C}{|x|^{m-2s}}.
\end{split}\end{equation}
Besides,
\begin{equation}\label{TEMP:ES:BES10erqwtegh}
\int _{|x|^2}^{+\infty }\frac{e^{-\frac{ |x|^2}{4\tau}-\tau}}{\tau^{\frac{m}2+1-s} } \,d\tau \le 
\int _{|x|^2}^{+\infty }\frac{e^{-\tau}}{\tau^{\frac{m}2+1-s} } \,d\tau.
\end{equation}
Thus, when~$m=2s$
\begin{eqnarray*}
\int _{|x|^2}^{+\infty }\frac{e^{-\frac{ |x|^2}{4\tau}-\tau}}{\tau^{\frac{m}2+1-s} } \,d\tau \le\int _{|x|^2}^{+\infty }\frac{e^{-\tau}}{\tau } \,d\tau
\le\int _{|x|^2}^{25 }\frac{d\tau}{\tau } +\int _{25}^{+\infty } e^{-\tau} \,d\tau\le C(1+|\ln|x||).
\end{eqnarray*}
This and~\eqref{TEMP:ES:BES1-PTRE} give~\eqref{TEMP:ES:BES1} when~$m=2s$.

Additionally, if~$m<2s$, we infer from~\eqref{TEMP:ES:BES10erqwtegh}
that~$$\int _{|x|^2}^{+\infty }\frac{e^{-\frac{ |x|^2}{4\tau}-\tau}}{\tau^{\frac{m}2+1-s} } \,d\tau \le\int _{0}^{+\infty }\frac{e^{-\tau}}{\tau^{\frac{m}2+1-s} } \,d\tau
\le C,$$ which, together with~\eqref{TEMP:ES:BES1-PTRE}, establishes~\eqref{TEMP:ES:BES1}
in this case.

Furthermore, when~$m>2s$, employing~\eqref{TEMP:ES:BES10erqwtegh} we see that
\begin{eqnarray*}
\int _{|x|^2}^{+\infty }\frac{e^{-\frac{ |x|^2}{4\tau}-\tau}}{\tau^{\frac{m}2+1-s} } \,d\tau \le
\int _{|x|^2}^{25 }\frac{d\tau}{\tau^{\frac{m}2+1-s} }+
\int _{25}^{+\infty } e^{-\tau} \,d\tau\le C\left(\frac{1}{|x|^{m-2s}}+1\right)\le\frac{C}{|x|^{m-2s}}.
\end{eqnarray*} 
This and~\eqref{TEMP:ES:BES1-PTRE} yield~\eqref{TEMP:ES:BES1} also in this case, completing the proof of the desired result.

{F}rom Lemma~\ref{HSZUONDCCONT-FF-G} and~\eqref{TEMP:ES:BES1} (used here with~$m:=n$) we arrive at~\eqref{HSZUONDCCONT-FF-G-0ojr236-201}.\end{proof}

\begin{corollary}\label{HSdfbhncO0nTAktNZUONDCCONT-FF-G-0ojr236-201}
Let~$s\in(0,+\infty)$. For all~$x\in B_5\setminus\{0\}$,
$$ 
|\nabla{\mathcal{B}}^{(s,n)}(x)|\le
\begin{dcases} \displaystyle\frac{C}{|x|^{n+1-2s}} &{\mbox{ if~$n+2>2s$,}}\\
C|x|(1+|\ln|x||)&{\mbox{ if~$n+2=2s$,}}\\ C|x|&{\mbox{ if~$n+2<2s$.}}\end{dcases}$$
\end{corollary}

\begin{proof} We adopt here the notation~${\mathcal{B}}^{(s,n)}$ that was introduced in Corollary~\ref{210pkf0ikj23456789iytdxscvbnm-934098i7ytghbnfraCxfr}.
Thus, by~\eqref{HSZUONDCCONT-FF-G-kefmgEGSND-99} and~\eqref{HSZUONDCCONT-FF-G-0ojr236-201},
\begin{eqnarray*}
|\nabla{\mathcal{B}}^{(s,n)}(x)|\le
C|x|\,|{\mathcal{B}}^{(s,n+2)}(x)|\le
\begin{dcases} \displaystyle\frac{C}{|x|^{n+1-2s}} &{\mbox{ if~$n+2>2s$,}}\\
C|x|(1+|\ln|x||)&{\mbox{ if~$n+2=2s$,}}\\ C|x|&{\mbox{ if~$n+2<2s$,}}\end{dcases}\end{eqnarray*}
as desired.
\end{proof}

We now present a comparison between Riesz and Bessel potentials:

\begin{proposition}\label{HSZUONDCCONT-FF-GMMA1-2-Pr}
Let~$s\in(0,1)$. Then, there exist finite signed Radon measures~$M_1$ and~$M_2$ such that
\begin{equation}\label{HSZUONDCCONT-FF-GMMA1-2}
\begin{split}
& (2\pi|\xi|)^{2s}=(1+4\pi^2|\xi|^2)^s\, \widehat M_1(\xi)\\
{\mbox{and }}\qquad&(1+4\pi^2|\xi|^2)^s=\widehat M_2(\xi)+(2\pi|\xi|)^{2s} \,\widehat M_2(\xi).
\end{split}
\end{equation}
\end{proposition}

\begin{proof} By the Generalized Binomial Theorem, for every~$t\in(0,1)$,
\begin{equation}\label{HSZUONDCCONT-FF-GMMA1}
(1-t)^s=1-\sum_{k=1}^{+\infty} c_k t^k,\qquad{\mbox{where }}\,
c_k:=-\frac1{k!}\prod_{j=0}^{k-1} (j-s) 
\end{equation}
and note that~$c_k>0$ for all~$k\ge1$.

As a result, for every~$N\in\N$, with~$N\ge1$,
\begin{equation*}1=\lim_{t\nearrow1}
\Big(1-(1-t)^s\Big)=\lim_{t\nearrow1}\sum_{k=1}^{+\infty} c_k t^k=\lim_{t\nearrow1}
\sum_{k=1}^{+\infty} |c_k| t^k\ge\lim_{t\nearrow1}\sum_{k=1}^{N} |c_k| t^k=\sum_{k=1}^{N} |c_k| 
\end{equation*}
and consequently
\begin{equation}\label{HSZUONDCCONT-FF-GMMA1.2} \sum_{k=1}^{+\infty} |c_k| \le1.\end{equation}

Now we apply~\eqref{HSZUONDCCONT-FF-GMMA1} with~$t:=\frac1{1+4\pi^2|\xi|^2}$ and obtain that
\begin{equation}\label{HSZUONDCCONT-FF-GMMA1.2-XCVB90} T(\xi):=\frac{(2\pi|\xi|)^{2s}}{(1+4\pi^2|\xi|^2)^s}=1-\sum_{k=1}^{+\infty} \frac{c_k}{(1+4\pi^2|\xi|^2)^k}=
1-\sum_{k=1}^{+\infty} c_k \widehat{\mathcal{B}}^{(k)}(\xi),\end{equation} where we have denoted by~${\mathcal{B}}^{(k)}$ the Bessel kernel with parameter~$k$ (recall Lemma~\ref{PRODpoikjhr3-2:le-7}).

Thus, utilizing~\eqref{HSZUONDCCONT-FF-GMMA1.2} to pass the Fourier Transform
under the series,
$$ T=\widehat M_1,\qquad{\mbox{ where }}\,M_1:=\delta_0-\sum_{k=1}^{+\infty} c_k {\mathcal{B}}^{(k)}.$$

We also remark that, for every~$f\in C^\infty_c(\R^n)$,
\begin{equation}\label{HSZUONDCCONT-FF-GMMA1.2.3}
\begin{split} 
|M_1(f)| &\le
|f(0)|+\sum_{k=1}^{+\infty} c_k |{\mathcal{B}}^{(k)}(f)|\le\|f\|_{L^\infty(\R^n)}+\sum_{k=1}^{+\infty} c_k
\int_{\R^n}|{\mathcal{B}}^{(k)}(x)|\,|f(x)|\,dx \\
&\leq \|f\|_{L^\infty(\R^n)}\left(1+\sum_{k=1}^{+\infty} c_k\right)\le2\|f\|_{L^\infty(\R^n)},
\end{split}\end{equation}
thanks to Lemma~\ref{L419} and~\eqref{HSZUONDCCONT-FF-GMMA1.2}. Therefore,~$M_1$ can be extended to a linear function on the space
of continuously and compacted supported functions. Hence, by
the Riesz-Markov-Kakutani Representation Theorem, we conclude that~$M_1$ can be identified with a signed Radon measure, with~$|M_1(\R^n)|\le2$, in view of~\eqref{HSZUONDCCONT-FF-GMMA1.2.3}.
This establishes the first claim in~\eqref{HSZUONDCCONT-FF-GMMA1-2}.

Now we define
$$ \phi_1:={\mathcal{B}}^{(s)}-\sum_{k=1}^{+\infty} c_k {\mathcal{B}}^{(k)}.$$
Note that~$\phi_1\in L^1(\R^n)$, due to Lemma~\ref{L419}
and~\eqref{HSZUONDCCONT-FF-GMMA1.2}.

Moreover, by Lemma~\ref{PRODpoikjhr3-2:le-7} and~\eqref{HSZUONDCCONT-FF-GMMA1.2-XCVB90},
$$ \widehat\phi_1(\xi)+1=\frac{1}{\big(1+4\pi^2|\xi|^2\big)^s}-\sum_{k=1}^{+\infty} c_k\widehat {\mathcal{B}}^{(k)}(\xi)+1=\frac{1+(2\pi|\xi|)^{2s}}{\big(1+4\pi^2|\xi|^2\big)^s},$$
whose infimum is bounded away from zero.

Therefore, by Theorem~\ref{WI:TH:-0eirojeg}, \label{WI:TH:-0eirojeg:page} we deduce that there exists~$\phi_2\in L^1(\R^n)$ such that
$$ \widehat\phi_2=\frac{1}{\widehat\phi_1+1}-1=
\frac{\big(1+4\pi^2|\xi|^2\big)^s}{1+(2\pi|\xi|)^{2s}}-1.$$
The second claim in~\eqref{HSZUONDCCONT-FF-GMMA1-2} is thereby proved by choosing
\begin{equation*} M_2(A):=\int_A \phi_2(x)\,dx+\chi_A(0).\qedhere\end{equation*}
\end{proof}

The following two results are useful consequences of the two identities in~\eqref{HSZUONDCCONT-FF-GMMA1-2}
(these results will be applied in the forthcoming Section~\ref{SEC:BPSPA} to prove
Lemma~\ref{HSZUONDCCONT-FF-GMMA1-2-LE}).

\begin{corollary}\label{PRODpoikjhr3-2:le-7-00-corol}
Let~$s\ge\frac12$,~$p\in(1,+\infty)$,~$f\in L^p(\R^n)$ and~$u:={\mathcal{B}}^{(s)}*f$.

Then, for each~$j\in\{1,\dots,n\}$,
\begin{equation} \label{PRODpoikjhr3-2:le-7-00-FA} \partial_ju ={\mathcal{B}}^{(s-1/2)}*f^{(j)},\end{equation}
where
\begin{equation} \label{PRODpoikjhr3-2:le-7-00}
f^{(j)}:={\mathcal{F}}^{-1}\left( \frac{2\pi i\xi_j}{\sqrt{ 1+4\pi^2|\xi|^2}}\,\widehat f(\xi)\right).\end{equation}

Furthermore,
\begin{equation}\label{SAMIHOLMS:kmsdf2}
\|f^{(j)}\|_{L^p(\R^n)} \le C\,\|f\|_{L^p(\R^n)},
\end{equation}
for some constant~$C>0$ depending only on~$n$,~$p$ and~$s$.
\end{corollary}

\begin{proof} Assume first that~$f$ is in the Schwartz space of smooth and rapidly decreasing functions. Then, by Lemma~\ref{PRODpoikjhr3-2:le-7}, for each~$j\in\{1,\dots,n\}$,
\begin{eqnarray*}&&
{\mathcal{F}}(\partial_j u)(\xi)=2\pi i\xi_j\widehat u=2\pi i\xi_j \widehat{\mathcal{B}}^{(s)}(\xi)\,\widehat f(\xi)=2\pi i\xi_j \big(1+4\pi^2|\xi|^2\big)^{-s}\widehat f(\xi)\\&&\qquad= \big(1+4\pi^2|\xi|^2\big)^{\frac12-s}{\mathcal{F}}( f^{(j)})(\xi)=\widehat{\mathcal{B}}^{(s-1/2)}(\xi){\mathcal{F}}( f^{(j)})(\xi).
\end{eqnarray*}
By taking the inversion of the Fourier Transform, we obtain~\eqref{PRODpoikjhr3-2:le-7-00-FA} in this case.

When~$f\in L^p(\R^n)$ we take a sequence of smooth and rapidly decreasing functions~$f_k$ such that~$f_k\to f$ in~$L^p(\R^n)$ as~$k\to+\infty$. Thus, if~$u_k:={\mathcal{B}}^{(s)}*f_k$
and~$f_k^{(j)}$ is as in~\eqref{PRODpoikjhr3-2:le-7-00} with~$f$ replaced by~$f_k$, we already know that
\begin{equation}\label{HSZUONDCCONT-FF-GMMA1-2-ij} \partial_ju_k ={\mathcal{B}}^{(s-1/2)}*f^{(j)}_k.\end{equation}
The objective is now to pass to the limit as~$k\to+\infty$. For this, we utilize the first identity in~\eqref{HSZUONDCCONT-FF-GMMA1-2} with~$s=\frac12$ and we see that, if~$h_k:=f-f_k$
and~$h^{(j)}_k$ is as in~\eqref{PRODpoikjhr3-2:le-7-00} with~$f$ replaced by~$h_k$, then
\begin{eqnarray*} h_k^{(j)}={\mathcal{F}}^{-1}\left(
\frac{2\pi i\xi_j}{\sqrt{ 1+4\pi^2|\xi|^2}}\,\widehat h_k(\xi)\right)={\mathcal{F}}^{-1}\left(
\frac{i\xi_j}{|\xi|}\,\frac{2\pi |\xi|}{\sqrt{ 1+4\pi^2|\xi|^2}}\,\widehat h_k(\xi)\right)
={\mathcal{F}}^{-1}\left(\frac{i\xi_j}{|\xi|}\,\widehat M_1(\xi)\,\widehat h_k(\xi)\right).
\end{eqnarray*}
Moreover, by the singular integral theory of Calder\'on-Zygmund type (see e.g.
Theorem~3 on page~39 in~\cite{MR0290095}
or Theorem~5.2.1 in~\cite{2021arXiv210107941D}), one has that
$$\left\| {\mathcal{F}}^{-1}\left(\frac{i\xi_j}{|\xi|}\right)*  h_k\right\|_{L^p(\R^n)}\le C\,\|h_k\|_{L^p(\R^n)}$$
for some~$C>0$.

{F}rom these observations and the fact that~$M_1$ has finite mass (recall Proposition~\ref{HSZUONDCCONT-FF-GMMA1-2-Pr}), using the notation~$\phi_{j,k}:={\mathcal{F}}^{-1}\left(\frac{i\xi_j}{|\xi|}\right)*  h_k$
and Minkowski's Integral Inequality (see Theorem~\ref{MLAerSM:ijfKKSMdf02}), we arrive at
\begin{equation}\label{SAMIHOLMS:kmsdf}
\begin{split}
& \|h_k^{(j)}\|_{L^p(\R^n)}=
\left\| {\mathcal{F}}^{-1}\left(\frac{i\xi_j}{|\xi|}\right)*  h_k*M_1\right\|_{L^p(\R^n)}
=\left\| \phi_{j,k}*M_1\right\|_{L^p(\R^n)}\\&\qquad=
\left(\int_{\R^n} \left| \int_{\R^n}\phi_{j,k}(x-y)\,dM_1(y)\right|^p\,dx\right)^{\frac1p}\le C
\int_{\R^n} \left(\int_{\R^n}|\phi_{j,k}(x-y)|^p\,dx\right)^{\frac1p}\,d|M_1|(y) \\ &\qquad=C\int_{\R^n} \|\phi_{j,k}
\|_{L^p(\R^n)}\,d|M_1|(y) =C\|\phi_{j,k}
\|_{L^p(\R^n)}\le C\,\|h_k\|_{L^p(\R^n)},\end{split}
\end{equation}
up to renaming~$C$ at each step of the computation.

This,~\eqref{HSZUONDCCONT-FF-GMMA1-2-ij}, Lemma~\ref{HSZUONDCCONT-FF-G-DECA}
and Young's Convolution Inequality yield that
\begin{eqnarray*} &&\|{\mathcal{B}}^{(s-1/2)}*f^{(j)}-\partial_j u_k \|_{L^p(\R^n)}
=\| {\mathcal{B}}^{(s-1/2)}*h^{(j)}_k\|_{L^p(\R^n)}
\le \|h_k^{(j)}\|_{L^p(\R^n)}
\le \|h_k\|_{L^p(\R^n)},
\end{eqnarray*}
which is infinitesimal as~$k\to+\infty$.

This ensures the desired convergence and gives~\eqref{PRODpoikjhr3-2:le-7-00-FA}.

The claim in~\eqref{SAMIHOLMS:kmsdf2} follows in the line of~\eqref{SAMIHOLMS:kmsdf}.
\end{proof}

\begin{corollary} \label{COK90FR:918}
Let~$p\in(1,+\infty)$ and~$u\in W^{1,p}(\R^n)$.
Then, there exists a locally integrable function~$f$ such that~$u={\mathcal{B}}^{(1/2)}*f$. 

Moreover, $f\in L^p(\R^n)$ and
$$ \|f\|_{L^p(\R^n)}\le C\, \|u\|_{W^{1,p}(\R^n)},$$
for some positive constant~$C$ depending only on~$n$ and~$p$.
\end{corollary}

\begin{proof} Let~$\{u_j\}_{j\in\N}$ be a sequence of smooth and compactly supported functions converging to~$u$ in~$W^{1,p}(\R^n)$ as~$j\to+\infty$. Let also~$f_j:=(1-\Delta)^{1/2} u_j$. Then,~$\widehat f_j= \sqrt{1+4\pi^2|\xi|^2}\,\widehat u_j$.

Using the second identity in~\eqref{HSZUONDCCONT-FF-GMMA1-2} with~$s:=\frac12$, we have that
$$ \sqrt{1+4\pi^2|\xi|^2}=(1+2\pi|\xi|) \,\widehat M_2$$
and therefore
$$ f_j={\mathcal{F}}^{-1}\Big(\sqrt{1+4\pi^2|\xi|^2}\,\widehat u_j\Big)={\mathcal{F}}^{-1}\Big((1+2\pi|\xi|) \,\widehat M_2\,\widehat u_j\Big)={\mathcal{F}}^{-1}\Big((1+2\pi|\xi|)\widehat u_j\Big) *M_2.$$
Thus, since~$M_2$ has finite mass, using the H\"older's Inequality we obtain that
\begin{equation}\label{29ik3HSZUONDCCONT-FF-GMMA1-2}
\begin{split}
\|f_j\|_{L^p(\R^n)}&=\left(
\int_{\R^n}\left| {\mathcal{F}}^{-1}\Big((1+2\pi|\xi|)\widehat u_j\Big) *M_2(x)\right|^p\,dx
\right)^{\frac1p}\\&=\left(
\int_{\R^n}\left|\int_{\R^n} {\mathcal{F}}^{-1}\Big((1+2\pi|\xi|)\widehat u_j\Big) (x-y)\,dM_2(y)\right|^p\,dx
\right)^{\frac1p}\\&\le C\left(
\int_{\R^n}\int_{\R^n}\left| {\mathcal{F}}^{-1}\Big((1+2\pi|\xi|)\widehat u_j\Big) (x-y)\right|^p\,d|M_2(y)|\,dx
\right)^{\frac1p}\\&\le C\,\left\| {\mathcal{F}}^{-1}\Big((1+2\pi|\xi|)\widehat u_j\Big)\right\|_{L^p(\R^n)}\\
&\le C\,\left\|u_j+ {\mathcal{F}}^{-1}\big( 2\pi|\xi|\widehat u_j\big)\right\|_{L^p(\R^n)}\\&\le
C\left(\|u_j\|_{L^p(\R^n)}+ \left\|{\mathcal{F}}^{-1}\big( 2\pi|\xi|\widehat u_j\big)\right\|_{L^p(\R^n)}\right).
\end{split}\end{equation}
Since, for every~$k\in\N$, we have that~$\partial_k u_j={\mathcal{F}}^{-1}\big(2\pi i\xi_k\widehat u_j\big)$, we see that
\begin{eqnarray*}
&&{\mathcal{F}}^{-1}\big( 2\pi|\xi|\widehat u_j\big)=
\sum_{k=1}^n{\mathcal{F}}^{-1}\left(\frac{2\pi \xi_k^2}{|\xi|}\widehat u_j\right)=
\sum_{k=1}^n{\mathcal{F}}^{-1}\left(\frac{ \xi_k}{|\xi|}\right)*{\mathcal{F}}^{-1}\left(2\pi \xi_k\widehat u_j\right)=\frac1i\sum_{k=1}^n{\mathcal{F}}^{-1}\left(\frac{ \xi_k}{|\xi|}\right)*(\partial_ku_j)
\end{eqnarray*}
and therefore, by the singular integral theory of Calder\'on-Zygmund type (see e.g.
Theorem~3 on page~39 in~\cite{MR0290095} or Theorem~5.2.1 in~\cite{2021arXiv210107941D}),
\begin{eqnarray*}
\left\|{\mathcal{F}}^{-1}\big( 2\pi|\xi|\widehat u_j\big)\right\|_{L^p(\R^n)}
\le\sum_{k=1}^n\left\|{\mathcal{F}}^{-1}\left(\frac{ \xi_k}{|\xi|}\right)*(\partial_ku_j)\right\|_{L^p(\R^n)}\le
C\sum_{k=1}^n \| \partial_ku_j\|_{L^p(\R^n)}.
\end{eqnarray*}

Combining this estimate and~\eqref{29ik3HSZUONDCCONT-FF-GMMA1-2} we obtain that
\begin{equation}\label{832r789yf2fy23yfh823f}
\|f_j\|_{L^p(\R^n)}\le C\, \|u_j\|_{W^{1,p}(\R^n)}.
\end{equation}
Also, since~$f_j-f_m=(1-\Delta)^{1/2}(u_j-u_m)$, we have that
\[
\|f_j-f_m\|_{L^p(\R^n)}\le C\, \|u_j-u_m\|_{W^{1,p}(\R^n)},
\]
and therefore~$\{f_j\}_{j\in\N}$ is a Cauchy sequence in~$L^p(\R^n)$.
As a consequence, it admits a limit~$f\in L^p(\R^n)$.

Moreover, taking the limit as~$j\to+\infty$ in~\eqref{832r789yf2fy23yfh823f},
we obtain that
\[
\|f\|_{L^p(\R^n)}\le C\, \|u\|_{W^{1,p}(\R^n)}.
\]

Finally, since~$u_j={\mathcal{B}}^{(1/2)}*f_j$, by Corollary~\ref{CO914athcal} we have that
\begin{eqnarray*}
\|u-{\mathcal{B}}*f\|_{L^p(\R^n)}&\le&
\|u-u_j\|_{L^p(\R^n)}
+\|{\mathcal{B}}^{(1/2)}*f_j-{\mathcal{B}}^{(1/2)}*f\|_{L^p(\R^n)}\\&\le&
\|u-u_j\|_{L^p(\R^n)}+\|f_j-f\|_{L^p(\R^n)}
\\& \le& C\|u-u_j\|_{L^p(\R^n)}.
\end{eqnarray*}
Thus, taking the limit as~$j\to+\infty$,
we obtain that~$u={\mathcal{B}}^{(1/2)}*f$ and this concludes the proof.
\end{proof}

Another useful consequence of Proposition~\ref{HSZUONDCCONT-FF-GMMA1-2-Pr} consists in the following two bounds between
the norms of the two operators~$(-\Delta)^s$ and~$(1-\Delta)^s$:

\begin{theorem}\label{TH:9191} Let~$s\in(0,1)$ and~$p\in[1,+\infty)$. Then,
\begin{equation}\label{MAKlLMSTIuntYYTi}
\begin{split}&\|(-\Delta)^s u\|_{L^p(\R^n)}\le C\,\|(1-\Delta)^s u\|_{L^p(\R^n)}\\ {\mbox{and }}\quad&\|(1-\Delta)^s u\|_{L^p(\R^n)}\le C\Big(\| u\|_{L^p(\R^n)}+\|(-\Delta)^s u\|_{L^p(\R^n)}\Big),
\end{split}
\end{equation}
for some constant~$C>0$ depending only on~$n$,~$p$ and~$s$.
\end{theorem}

\begin{proof} By the first identity in~\eqref{HSZUONDCCONT-FF-GMMA1-2},
\begin{eqnarray*}&& (-\Delta)^s u ={\mathcal{F}}^{-1}\big( (2\pi|\xi|)^{2s}\widehat u\big)={\mathcal{F}}^{-1}\big( (1+4\pi^2|\xi|^2)^s\, \widehat M_1\widehat u\big)\\&&\qquad=M_1*{\mathcal{F}}^{-1}\big( (1+4\pi^2|\xi|^2)^s\widehat u\big)=M_1*((1-\Delta)^s u) .\end{eqnarray*}
{F}rom this and the Minkowski's Integral Inequality (see Theorem~\ref{MLAerSM:ijfKKSMdf02}), recalling that~$M_1$ has finite mass (recall Proposition~\ref{HSZUONDCCONT-FF-GMMA1-2-Pr}), we infer that
\begin{eqnarray*}
&&\|(-\Delta)^s u\|_{L^p(\R^n)}=\|M_1*((1-\Delta)^s u)\|_{L^p(\R^n)}
=\left(\int_{\R^n}\big|M_1*((1-\Delta)^s u)(x)\big|^p\,dx\right)^{\frac1p}\\&&\qquad=
\left(\int_{\R^n}\left|\int_{\R^n} (1-\Delta)^s u(x-y)\,dM_1(y)\right|^p\,dx\right)^{\frac1p}\\&&\qquad\le C
\int_{\R^n}\left(\int_{\R^n} \big|(1-\Delta)^s u(x-y)\big|^p\,dx\right)^{\frac1p}\,d|M_1|(y)=C\,|M_1(\R^n)|\,\|(1-\Delta)^s u\|_{L^p(\R^n)},
\end{eqnarray*}
which establishes the first estimate in~\eqref{MAKlLMSTIuntYYTi}.

Furthermore, owing to the second identity in~\eqref{HSZUONDCCONT-FF-GMMA1-2},
\begin{eqnarray*}&& (1-\Delta)^s u= {\mathcal{F}}^{-1}\big((1+4\pi^2|\xi|^2)^s\widehat u\big)=
{\mathcal{F}}^{-1}\big(\widehat M_2\widehat u+(2\pi|\xi|)^{2s} \,\widehat M_2\widehat u\big)\\&&\qquad=
M_2*u+M_2*{\mathcal{F}}^{-1}\big((2\pi|\xi|)^{2s} \widehat u\big)
=M_2*\big(u+(-\Delta)^su\big).\end{eqnarray*}
For this reason, if~$v:=u+(-\Delta)^su$, using the
Minkowski's Integral Inequality (see Theorem~\ref{MLAerSM:ijfKKSMdf02})
and the fact that~$M_2$ has finite mass
we find that
\begin{align*}
&
\|(1-\Delta)^s u\|_{L^p(\R^n)}=\|M_2*v\|_{L^p(\R^n)}=
\left(\int_{\R^n} \big|M_2*v(x)\big|^p\,dx\right)^{\frac1p}=\left(\int_{\R^n} \left|\int_{\R^n} v(x-y)\,dM_2(y)\right|^p\,dx\right)^{\frac1p} \\
&\qquad\le C\,
\int_{\R^n}\left( \int_{\R^n} |v(x-y)|^p\,dx\right)^{\frac1p}\,d|M_2|(y)=C\,|M_2(\R^n)|\,\|v\|_{L^p(\R^n)}.
\end{align*}
This proves the second estimate in~\eqref{MAKlLMSTIuntYYTi}.
\end{proof}

\medskip

See e.g.~\cite{MR143935, MR228702, MR0290095, MR0374877, MR698780, MR1349110, MR1411441} and the references therein for further information on the Bessel kernel.

\section{Bessel potential spaces}\label{SEC:BPSPA}

In light of~\eqref{HSZUONDCCONT-FF} and the Bessel kernel setting,
to understand solutions of~$ (1-\Delta)^s u=f$ it is useful to look at the space containing all the functions which can be written as~${\mathcal{B}}*f$.
Namely, given~$p\in[1,+\infty]$ we consider the Bessel potential space defined\footnote{For simplicity, we assumed in the definition~\eqref{HSZUONDCCONT-FF0iejdfujhnHSNDJNdik9238475tyhfj} of \label{VEONORL123r3FVSjjMDMFP}
Bessel potential space that the elements of the space are functions in~$L^p(\R^n)$.
One could also work with tempered distributions, instead of functions. These settings would be essentially equivalent when~$s\ge0$, since the Lebesgue norm can be reabsorbed into the
Bessel potential space  norm in~\eqref{ERFGHJN6789-09tftd90u8yhgiug8erhISB}, due to Lemma~\ref{VEONORLP}
and Young's Convolution Inequality.

We point out that the notation for
Bessel potential spaces is not uniform across the literature. For example, the space denoted here by~${\mathcal{L}}^p_{2s}(\R^n)$
corresponds to~${\mathcal{L}}^p_{\alpha}(\R^n)$
in~\cite[Section~3.3, Chapter~V]{MR0290095} with~$\alpha:=2s$
and to~$H^{2s}_p(\R^n)$ in~\cite[Definition~2.42]{MR2884718}.

The Bessel potential spaces are also seen as members of a larger family of spaces called Triebel-Lizorkin spaces, see~\cite{MR1328645} and~\cite[Remark~6.10]{MR2884718}.
See also~\cite[Section~2]{MR1044427} for a review of several functional
spaces commonly used in the analysis of fractional operators.

For the sake of simplicity, here
we deal with spaces with positive regularity parameter~$s$ but we mention that the case of negative exponents
can be considered too, also obtaining duality results and useful results for distribution solutions.}
as \index{Bessel potential space}
\begin{equation} \label{HSZUONDCCONT-FF0iejdfujhnHSNDJNdik9238475tyhfj}{\mathcal{L}}^p_{2s}(\R^n):=\Big\{ u\in L^p(\R^n) {\mbox{ s.t. }}u={\mathcal{B}}^{(s)}*f, {\mbox{ with }}f\in L^p(\R^n)\Big\},\end{equation}
where~${\mathcal{B}}^{(s)}$ is the Bessel kernel corresponding
to the parameter~$s>0$.

For notational convenience, we also define~${\mathcal{L}}^p_{2s}(\R^n):=L^p(\R^n)$ when~$s=0$
(which is consistent\footnote{This choice can be considered as the \label{234erV3E432ON24OR12t35uLP}
counterpart of defining the Sobolev space~$W^{\sigma,p}(\R^n)$ to be~$L^p(\R^n)$ when~$\sigma=0$.} with the choice of taking~${\mathcal{B}}^{(s)}$ to be the identity when~$s=0$,
to validate Lemma~\ref{L912-xpsmv} in this case too).

\begin{lemma}\label{VEONORLP}
If~$u\in{\mathcal{L}}^p_{2s}(\R^n)$ there exists a unique~$f_u\in L^p(\R^n)$ such that~$u={\mathcal{B}}^{(s)}*f_u$.\end{lemma}

\begin{proof} The existence of~$f_u$ follows from the definition of~${\mathcal{L}}^p_{2s}(\R^n)$.
For its uniqueness, suppose that~${\mathcal{B}}^{(s)}*f_1={\mathcal{B}}^{(s)}*f_2$, for some~$f_1$,~$f_2\in L^p(\R^n)$.
Let~$f_3:=f_2-f_1$. Then, for all~$\varphi$ in the Schwartz space of smooth and rapidly decreasing functions, using the Plancherel Theorem and Lemma~\ref{PRODpoikjhr3-2:le-7} we see that
\begin{eqnarray*}
&&0=\int_{\R^n}{\mathcal{B}}^{(s)}*f_3(x)\,\varphi(x)\,dx=\int_{\R^n}{\mathcal{F}}\Big({\mathcal{B}}^{(s)}*f_3\Big)(\xi)\,\overline{\widehat\varphi(\xi)}\,d\xi\\&&\qquad=\int_{\R^n} \big(1+4\pi^2|\xi|^2\big)^{-s}\widehat f_3(\xi)\,\overline{\widehat\varphi(\xi)}\,d\xi
=\int_{\R^n} \widehat f_3(\xi)\,\overline{\big(1+4\pi^2|\xi|^2\big)^{-s}\widehat\varphi(\xi)}\,d\xi.
\end{eqnarray*}
Therefore, given~$\psi\in C^\infty_c(\R^n)$, we take~$\varphi:={\mathcal{F}}^{-1}\Big(\big(1+4\pi^2|\xi|^2\big)^{s}\widehat\psi \Big)$ and we conclude that
$$ 0=\int_{\R^n} \widehat f_3(\xi)\,\overline{\widehat\psi(\xi)}\,d\xi=\int_{\R^n} f_3(\xi)\,\psi(\xi)\,d\xi,$$
yielding that~$f_3$ vanishes identically and thus proving the desired uniqueness claim.
\end{proof}

\begin{corollary}\label{BANACHCOEL}
We have that
\begin{equation}\label{ERFGHJN6789-09tftd90u8yhgiug8erhISB} \|u\|_{ {\mathcal{L}}^p_{2s}(\R^n)}:=\| f_u\|_{L^p(\R^n)}\end{equation}
is a norm on~${\mathcal{L}}^p_{2s}(\R^n)$.

Endowed with this norm,~${\mathcal{L}}^p_{2s}(\R^n)$ is a Banach space.
\end{corollary}

\begin{proof} The fact that~$\|\cdot\|_{ {\mathcal{L}}^p_{2s}(\R^n)}$ is properly defined
follows from Lemma~\ref{VEONORLP}. Also, the triangle and homogeneity properties of~$\|\cdot\|_{ {\mathcal{L}}^p_{2s}(\R^n)}$ are straightforward.

Suppose now that~$\|u\|_{ {\mathcal{L}}^p_{2s}(\R^n)}=0$. Then,~$\| f_u\|_{L^p(\R^n)}=0$, whence~$f_u$ vanishes identically and thus~$u={\mathcal{B}}^{(s)}*f_u$ vanishes identically.

These observations show that~$\|\cdot\|_{ {\mathcal{L}}^p_{2s}(\R^n)}$ is a norm.

Let us now check that~${\mathcal{L}}^p_{2s}(\R^n)$ is complete.
For this, let us consider a Cauchy sequence~$\{u_j\}_{j\in\N}$ of functions in~${\mathcal{L}}^p_{2s}(\R^n)$ and let~$f_j:=f_{ u_j}$. We observe that~$f_j$ is a Cauchy sequence in~$L^p(\R^n)$, therefore there exists~$f\in L^p(\R^n)$ such that~$f_j\to f$ in~$L^p(\R^n)$.

We let~$u:={\mathcal{B}}^{(s)}*f$. Then, by construction,~$u\in{\mathcal{L}}^p_{2s}(\R^n)$. Moreover, by Lemma~\ref{VEONORLP}, we know that~$f=f_u$. Accordingly,
$$ \|u-u_j\|_{ {\mathcal{L}}^p_{2s}(\R^n)}=\| f_u-f_j\|_{L^p(\R^n)}
=\|f-f_j\|_{ L^p(\R^n)},$$
which is infinitesimal, thus showing the completeness of~${\mathcal{L}}^p_{2s}(\R^n)$.
\end{proof}

\begin{corollary}\label{09ijnuiathcaRn}
If~$\sigma\ge s\ge0$ and~$u\in{\mathcal{L}}^p_{2\sigma}(\R^n)$, then~$u\in{\mathcal{L}}^p_{2s}(\R^n)$ and
$$ \|u\|_{{\mathcal{L}}^p_{2s}(\R^n)}\le\|u\|_{{\mathcal{L}}^p_{2\sigma}(\R^n)}.$$
\end{corollary}

\begin{proof} Let~$f\in L^p(\R^n)$ be such that~$u={\mathcal{B}}^{(\sigma)}*f$. By Lemma~\ref{L912-xpsmv} we know that~$u=
{\mathcal{B}}^{(s)}*{\mathcal{B}}^{(\sigma-s)}*f={\mathcal{B}}^{(s)}*g$, with~$g:={\mathcal{B}}^{(\sigma-s)}*f$.
In light of Corollary~\ref{CO914athcal} we also know that
$$\|g\|_{L^p(\R^n)}=\|{\mathcal{B}}^{(\sigma-s)}*f\|_{L^p(\R^n)}\le\|f\|_{L^p(\R^n)}$$
and therefore~$g\in L^p(\R^n)$.

These observations show that~$u\in{\mathcal{L}}^p_{2s}(\R^n)$ and that
\begin{equation*}\|u\|_{{\mathcal{L}}^p_{2s}(\R^n)}=\|g\|_{L^p(\R^n)}\le\|f\|_{L^p(\R^n)}=\|u\|_{{\mathcal{L}}^p_{2\sigma}(\R^n)}.\qedhere
\end{equation*}
\end{proof}

\begin{corollary}\label{XL912-xpsmv-qw-map}
Let~$\sigma\ge s\ge0$. Then, the map
\begin{equation}\label{XL912-xpsmv-qw} u \longmapsto{\mathcal{B}}^{(\sigma-s)}*u \end{equation}
is an isomorphism from~${\mathcal{L}}^p_{2s}(\R^n)$ to~${\mathcal{L}}^p_{2\sigma}(\R^n)$.
\end{corollary}

\begin{proof} Firstly, let us observe that
\begin{equation}\label{L912-xpsmv-qw}
{\mbox{if~$u\in{\mathcal{L}}^p_{2s}(\R^n)$ then~${\mathcal{B}}^{(\sigma-s)}*u \in{\mathcal{L}}^p_{2\sigma}(\R^n)$.}}\end{equation}
Indeed, let~$f\in L^p(\R^n)$ such that~$u={\mathcal{B}}^{(s)}*f~$.
Then, by Lemma~\ref{L912-xpsmv},~${\mathcal{B}}^{(\sigma-s)}*u={\mathcal{B}}^{(\sigma-s)}*
{\mathcal{B}}^{(s)}*f={\mathcal{B}}^{(\sigma)}*f$, proving~\eqref{L912-xpsmv-qw}.

This argument also shows that
$$ \|{\mathcal{B}}^{(\sigma-s)}*u\|_{ {\mathcal{L}}^p_{2\sigma}(\R^n)}=\| f\|_{L^p(\R^n)}=
\|u\|_{ {\mathcal{L}}^p_{2s}(\R^n)}.$$
The injectivity of the map in~\eqref{XL912-xpsmv-qw} is also guaranteed by this identity.

As for the surjectivity, given any~$v\in{\mathcal{L}}^p_{2\sigma}(\R^n)$, let~$g\in L^p(\R^n)$ be such that~$v={\mathcal{B}}^{(\sigma)}*g$. One defines~$u:={\mathcal{B}}^{(s)}*g$. In this way,~$u\in{\mathcal{L}}^p_{2s}(\R^n)$ and, by Lemma~\ref{L912-xpsmv},~$${\mathcal{B}}^{(\sigma-s)}*u={\mathcal{B}}^{(\sigma-s)}*{\mathcal{B}}^{(s)}*g=
{\mathcal{B}}^{(\sigma)}*g=v,$$
which establishes that the map in~\eqref{XL912-xpsmv-qw} is surjective.
\end{proof}

The main result of this section consists in relating Bessel potential spaces and Sobolev spaces for half-integer exponents.
This will be accomplished in Theorem~\ref{TH0-104-okn5KMD3} below. To this end, one needs an intermediate result
relating Bessel potential spaces and derivatives:

\begin{lemma}\label{HSZUONDCCONT-FF-GMMA1-2-LE}
Let~$s\ge\frac12$ and~$p\in(1,+\infty)$. Then, the following conditions are equivalent:
\begin{itemize}
\item~$u\in{\mathcal{L}}^p_{2s}(\R^n)$,
\item~$u$,~$\partial_1u$,~$\dots$,~$\partial_nu\in{\mathcal{L}}^p_{2s-1}(\R^n)$.
\end{itemize}

Moreover, the norm
\begin{equation} \label{ERFGHJN6789-09tftd90u8yhgiug8erhISB-eq}
\|u\|_{{\mathcal{L}}^p_{2s-1}(\R^n)}+\sum_{j=1}^n\|\partial_ju\|_{{\mathcal{L}}^p_{2s-1}(\R^n)}\end{equation}
is equivalent to the one in~\eqref{ERFGHJN6789-09tftd90u8yhgiug8erhISB}.
\end{lemma}

\begin{proof} Assume that~$u\in{\mathcal{L}}^p_{2s}(\R^n)$. Let~$f\in L^p(\R^n)$ be such that~$u={\mathcal{B}}^{(s)}*f$.
Then, by Corollary~\ref{09ijnuiathcaRn},
$u\in{\mathcal{L}}^p_{2s-1}(\R^n)$ and~$ \|u\|_{{\mathcal{L}}^p_{2s-1}(\R^n)}\le\|u\|_{{\mathcal{L}}^p_{2s}(\R^n)}$.

Additionally, by~\eqref{PRODpoikjhr3-2:le-7-00-FA} and~\eqref{PRODpoikjhr3-2:le-7-00},
$$ \partial_ju ={\mathcal{B}}^{(s-1/2)}*f^{(j)},\qquad{\mbox{ where }}\qquad f^{(j)}:=
{\mathcal{F}}^{-1}\left( \frac{2\pi i\xi_j}{\sqrt{ 1+4\pi^2|\xi|^2}}\,\widehat f(\xi)\right).$$
As a result, using~\eqref{SAMIHOLMS:kmsdf2} and~\eqref{ERFGHJN6789-09tftd90u8yhgiug8erhISB},
\begin{eqnarray*}&&\|
\partial_ju \|_{{\mathcal{L}}^p_{2s-1}(\R^n)}=
\|f^{(j)}\|_{L^p(\R^n)} \le C\,\|f\|_{L^p(\R^n)}=C\,
\|u \|_{{\mathcal{L}}^p_{2s}(\R^n)},\end{eqnarray*}
for some~$C>0$.

This shows that~$\partial_ju\in{\mathcal{L}}^p_{2s-1}(\R^n)$ for each~$j\in\{1,\dots,n\}$ and that the norm in~\eqref{ERFGHJN6789-09tftd90u8yhgiug8erhISB-eq} is bounded by the one in~\eqref{ERFGHJN6789-09tftd90u8yhgiug8erhISB}, up to a constant.

Now, the other way around, suppose instead that~$u$,~$\partial_1u$,~$\dots$,~$\partial_nu\in{\mathcal{L}}^p_{2s-1}(\R^n)$.
Let~$g\in L^p(\R^n)$ be such that
\begin{equation}\label{PRODpoikjhr3-2:le-7-00-c912ork-1}
u={\mathcal{B}}^{(s-1/2)}*g.\end{equation}
Moreover, we know that,
for each~$j\in\{1,\dots,n\}$,
\begin{equation}\label{ZXC-2u03r9ogu3ro656789rjUS} \partial_ju ={\mathcal{B}}^{(s-1/2)}*g^{(j)},\qquad{\mbox{with }} g^{(j)}\in L^p(\R^n).\end{equation}

Now we claim that, in the weak sense,
\begin{equation}\label{CO914athcal-e2p}
{\mbox{$\partial_jg=g^{(j)}\in L^p(\R^n)$.}}
\end{equation}
For this, we observe that, for every~$\phi$ in the Schwartz space of smooth and rapidly decreasing functions,
\begin{equation*}\begin{split}&
-\int_{\R^n}\partial_ju(x)\,\phi(x)\,dx=\int_{\R^n}u(x)\,\partial_j\phi(x)\,dx
=\int_{\R^n}({\mathcal{B}}^{(s-1/2)}*g)(x)\,\partial_j\phi(x)\,dx\\&\qquad
=\int_{\R^n}g(x)\,({\mathcal{B}}^{(s-1/2)}*\partial_j\phi)(x)\,dx=\int_{\R^n}g(x)\,\partial_j({\mathcal{B}}^{(s-1/2)}*\phi)(x)\,dx.\end{split}\end{equation*}
Combining this and~\eqref{ZXC-2u03r9ogu3ro656789rjUS} we see that
\begin{equation}\label{CO914athcal-e2pp}\begin{split}&
-\int_{\R^n}g^{(j)}(x)\,({\mathcal{B}}^{(s-1/2)}*\phi)(x)\,dx=
-\int_{\R^n}({\mathcal{B}}^{(s-1/2)}*g^{(j)})(x)\,\phi(x)\,dx\\&\qquad=-\int_{\R^n}\partial_ju(x)\,\phi(x)\,dx
=\int_{\R^n}g(x)\,\partial_j({\mathcal{B}}^{(s-1/2)}*\phi)(x)\,dx.\end{split}\end{equation}
Now, recalling Lemma~\ref{PRODpoikjhr3-2:le-7}, given~$\psi\in C^\infty_c(\R^n)$ we define
$$ \phi:={\mathcal{F}}^{-1}\left( \big(1+4\pi^2|\xi|^2\big)^{s-1/2}\,\widehat\psi\right) ,$$
which gives that
\begin{eqnarray*}
{\mathcal{B}}^{(s-1/2)}*\phi={\mathcal{F}}^{-1}\Big( \widehat{\mathcal{B}}^{(s-1/2)}\,\widehat\phi\Big)
={\mathcal{F}}^{-1}\Big( \big(1+4\pi^2|\xi|^2\big)^{1/2-s}\,\widehat\phi\Big)
={\mathcal{F}}^{-1}\big(\widehat\psi\big)=\psi.\end{eqnarray*}
This and~\eqref{CO914athcal-e2pp} yield that
$$ -\int_{\R^n}g^{(j)}(x)\,\psi(x)\,dx=\int_{\R^n}g(x)\,\partial_j\psi(x)\,dx,$$
whence~$\partial_jg=g^{(j)}$. This proves~\eqref{CO914athcal-e2p}, since~$g^{(j)}\in L^p(\R^n)$.

{F}rom~\eqref{CO914athcal-e2p} we infer that~$g\in W^{1,p}(\R^n)$ and therefore there exists a sequence of
smooth compactly supported functions~$g_\ell$ such that~$g_\ell\to g$ in~$W^{1,p}(\R^n)$ as~$\ell\to+\infty$.
We also let~$u_\ell:={\mathcal{B}}^{(s-1/2)}*g_\ell$.

In light of Lemma~\ref{PRODpoikjhr3-2:le-7} we define
$$h_\ell:={\mathcal{F}}^{-1}\Big(\sqrt{1+4\pi^2|\xi|^2}\,\widehat g_\ell\Big),$$
which entails that
$$ {\mathcal{B}}^{(1/2)}*h_\ell={\mathcal{F}}^{-1}\big(\widehat{\mathcal{B}}^{(1/2)}\,\widehat h_{\ell}\big)=
{\mathcal{F}}^{-1}\left(\frac{\widehat h_\ell}{\sqrt{1+4\pi^2|\xi|^2}}\right)=g_\ell.$$
Thus, by Lemma~\ref{L912-xpsmv},
\begin{equation*}u_\ell=
{\mathcal{B}}^{(s-1/2)}*g_\ell=
{\mathcal{B}}^{(s-1/2)}*{\mathcal{B}}^{(1/2)}*h_\ell=
{\mathcal{B}}^{(s)}*h_\ell.
\end{equation*}

For this reason, recalling~\eqref{ERFGHJN6789-09tftd90u8yhgiug8erhISB} and Corollary~\ref{COK90FR:918},
\begin{equation}\label{COK90FR:918aaoskvdb02A}
\|u_\ell\|_{ {\mathcal{L}}^p_{2s}(\R^n)}=\| h_\ell\|_{L^p(\R^n)}\le
C\, \|{\mathcal{B}}^{(1/2)}*h_\ell\|_{W^{1,p}(\R^n)}=C\, \|g_\ell\|_{W^{1,p}(\R^n)}.
\end{equation}
Similarly, for all~$\ell$,~$\ell'\in\N$,
\begin{eqnarray*}
&&\|u_\ell-u_{\ell'}\|_{ {\mathcal{L}}^p_{2s}(\R^n)}\le C\, \|g_\ell-g_{\ell'}\|_{W^{1,p}(\R^n)},
\end{eqnarray*}
giving that~$u_\ell$ is a Cauchy sequence in~${\mathcal{L}}^p_{2s}(\R^n)$.

Thus, by Corollary~\ref{BANACHCOEL}, we know that~$u_\ell$ converges to some~$u_\star$ in~${\mathcal{L}}^p_{2s}(\R^n)$ as~$\ell\to+\infty$. Accordingly, there exists~$h_\star\in L^p(\R^n)$ such that~$u_\star={\mathcal{B}}^{(s)}*h_\star$ and
$$ \lim_{\ell\to+\infty} \|h_\star-h_\ell\|_{L^p(\R^n)}=0.$$
Moreover, by Corollary~\ref{CO914athcal},
$$ \|u_\star-u_\ell\|_{L^p(\R^n)}
=\|{\mathcal{B}}^{(s)}*(h_\star-h_\ell)\|_{L^p(\R^n)}\le\|h_\star-h_\ell\|_{L^p(\R^n)}$$
and
$$ \|u-u_\ell\|_{L^p(\R^n)}
=\|{\mathcal{B}}^{(s-1/2)}*(g-g_\ell)\|_{L^p(\R^n)}\le\|g-g_\ell\|_{L^p(\R^n)}.$$
Therefore,
$$ \|u-u_\star\|_{L^p(\R^n)}\le
\lim_{\ell\to+\infty} \Big(\|u-u_\ell\|_{L^p(\R^n)}+\|u_\ell-u_\star\|_{L^p(\R^n)}\Big)
\le\lim_{\ell\to+\infty} \Big(\|g-g_\ell\|_{L^p(\R^n)}+\|h_\star-h_\ell\|_{L^p(\R^n)}\Big)=0,$$
from which we infer that~$u=u_\star\in {\mathcal{L}}^p_{2s}(\R^n)$.

Hence, to complete the proof of Lemma~\ref{HSZUONDCCONT-FF-GMMA1-2-LE}, it remains to check that
the norm in~\eqref{ERFGHJN6789-09tftd90u8yhgiug8erhISB} is bounded by the one in~\eqref{ERFGHJN6789-09tftd90u8yhgiug8erhISB-eq}, up to a constant.
To this end, we observe that
\begin{eqnarray*}&&
\|u\|_{ {\mathcal{L}}^p_{2s}(\R^n)}=\lim_{\ell\to+\infty}\|u_\ell\|_{ {\mathcal{L}}^p_{2s}(\R^n)}
\le C\,\lim_{\ell\to+\infty} \|g_\ell\|_{W^{1,p}(\R^n)}=C\, \|g\|_{W^{1,p}(\R^n)}\\&&\qquad\le
C\left(\|g\|_{L^p(\R^n)}+\sum_{j=1}^n\|\partial_jg\|_{L^p(\R^n)}\right)
\le C\left(\|u\|_{{\mathcal{L}}^p_{2s-1}(\R^n)}+\sum_{j=1}^n\|\partial_ju\|_{{\mathcal{L}}^p_{2s-1}(\R^n)}\right),
\end{eqnarray*}
thanks to~\eqref{PRODpoikjhr3-2:le-7-00-c912ork-1} and~\eqref{COK90FR:918aaoskvdb02A}, as desired.
\end{proof}

With this, we can prove the following\footnote{The functional setting of Sobolev spaces will be completed
by comparing Theorem~\ref{TH0-104-okn5KMD3} with the forthcoming Corollary~\ref{0pirj09365-4g4eZXCHJKLRFG-544-NO1-I2L3C2O25R213t4TY}, see footnote~\ref{sdcjnPwqdkjfcmvDCtVn3sd} on page~\pageref{sdcjnPwqdkjfcmvDCtVn3sd}.} structural result:

\begin{theorem}\label{TH0-104-okn5KMD3}
Let~$2s\in\N$ and~$p\in(1,+\infty)$. Then,~$ {\mathcal{L}}^p_{2s}(\R^n)=W^{2s,p}(\R^n)$ and the corresponding norms are equivalent.
\end{theorem}

\begin{proof}
When~$s=0$, then both~${\mathcal{L}}^p_{2s}(\R^n)$ and~$W^{2s,p}(\R^n)$ reduce to~$L^p(\R^n)$
(recall footnote~\ref{234erV3E432ON24OR12t35uLP} on page~\pageref{234erV3E432ON24OR12t35uLP}), whence the desired result holds true.

When~$s=\frac12$, we have that~${\mathcal{L}}^p_{2s-1}(\R^n)$ reduces to~$L^p(\R^n)$ and accordingly the norm in~\eqref{ERFGHJN6789-09tftd90u8yhgiug8erhISB-eq} equivalently reduces to the norm in~$W^{1,p}(\R^n)$. Therefore,
in this case, the result follows from Lemma~\ref{HSZUONDCCONT-FF-GMMA1-2-LE}.

We now proceed by induction: we assume that the desired result is true when~$2s\in \N\cap[0,\dots,N-1]$ for some~$N\ge2$ and we prove it for~$2s=N$. To this end, we first use Lemma~\ref{HSZUONDCCONT-FF-GMMA1-2-LE} to deduce that, when~$2s=N$,
we have that~$u\in{\mathcal{L}}^p_{2s}(\R^n)$ if and only if~$u$,~$\partial_1u$,~$\dots$,~$\partial_nu\in{\mathcal{L}}^p_{N-1}(\R^n)$, with~$\|u\|_{{\mathcal{L}}^p_{2s}(\R^n)}$ being bounded from above and below by~$
\|u\|_{{\mathcal{L}}^p_{N-1}(\R^n)}+\sum_{j=1}^n\|\partial_ju\|_{{\mathcal{L}}^p_{N-1}(\R^n)}$, up to constants.

By inductive assumption, we know that~${\mathcal{L}}^p_{N-1}(\R^n)=W^{N-1,p}(\R^n)$
and the norm~$\|\cdot\|_{{\mathcal{L}}^p_{N-1}(\R^n)}$ is equivalent to~$\|\cdot\|_{W^{N-1,p}(\R^n)}$.

{F}rom these bits of information, we deduce that~$u\in{\mathcal{L}}^p_{2s}(\R^n)$ if and only if~$u$,~$\partial_1u$,~$\dots$,~$\partial_nu\in W^{N-1,p}(\R^n)$, with~$\|u\|_{{\mathcal{L}}^p_{2s}(\R^n)}$ being equivalent to~$
\|u\|_{W^{N-1,p}(\R^n)}+\sum_{j=1}^n\|\partial_ju\|_{W^{N-1,p}(\R^n)}$, which yields the desired result.
\end{proof}

We stress that Theorem~\ref{TH0-104-okn5KMD3} does not hold when~$p=1$ and when~$p=+\infty$,
but partial results hold true, see e.g. comment~6.6 on page~160 of~\cite{MR0290095} for these cases.

\section[Regularity theory in Bessel potential spaces 
and Sobolev spaces]{Regularity theory in Bessel potential spaces 
and in Sobolev spaces
for global solutions}\label{Ijsmiazme3sKJS03-krodgh}

The analysis of Bessel potential spaces is intimately related to the regularity theory of fractional operators, in view of the following
observations:

\begin{proposition}\label{BiknsoU98ikjmd9f4}
If~$s\ge0$,~$p\in[1,+\infty]$ and~$f\in L^p(\R^n)$, then~$u:={\mathcal{B}}^{(s)}*f$ is the unique distributional solution of~$(1-\Delta)^su=f$ in~$\R^n$.
\end{proposition}

\begin{proof} In light of Corollary~\ref{CO914athcal} we note that~$u\in L^p(\R^n)\subseteq L^1_{\rm loc}(\R^n)$,
hence it defines a distribution. The fact that~$(1-\Delta)^su=f$ follows from the definition of Bessel potential below~\eqref{HSZUONDCCONT-FF}. The uniqueness of this solution is a consequence of Theorem~\ref{BiknsoU98ikjmd9f3}.
\end{proof}

\begin{theorem}\label{09ijnuiathcaRn-th}
Let~$s\ge0$,~$p\in(1,+\infty)$,~$f\in L^p(\R^n)$, and~$u$ be a distributional solution of~$(1-\Delta)^su=f$ in~$\R^n$.

Then,~$u\in {\mathcal{L}}^p_{2s}(\R^n)$ and
\begin{equation}\label{ERFGHJN6789-09tftd90u8yhgiug8erhISB-Xke} \|u\|_{ {\mathcal{L}}^p_{2s}(\R^n)}=\| f\|_{L^p(\R^n)}.\end{equation}
\end{theorem}

\begin{proof} By Proposition~\ref{BiknsoU98ikjmd9f4}, we have that~$u={\mathcal{B}}^{(s)}*f$.
Also, on account of Corollary~\ref{CO914athcal}, we know that~$u\in L^p(\R^n)$, with~$\|u\|_{L^p(\R^n)}\le\|f\|_{L^p(\R^n)}$.
Consequently,~$u\in {\mathcal{L}}^p_{2s}(\R^n)$ and~\eqref{ERFGHJN6789-09tftd90u8yhgiug8erhISB-Xke} follows from~\eqref{ERFGHJN6789-09tftd90u8yhgiug8erhISB}.
\end{proof}

As a general comment, let us point out that, as a rule of thumb, there are two separate issues in connection with the regularity of solutions of an equation of the type~$Lu=f$:
\begin{eqnarray}&&\label{asd0scjovndskghjsdckvhrta9djcvmouosd-01} \nonumber
{\mbox{On the one side, to account for a variety of function spaces}}\\ &&{\mbox{and their interrelations (inclusion, interpolation, etc.),}}\\
\label{asd0scjovndskghjsdckvhrta9djcvmouosd-02}&&\nonumber{\mbox{On the other side, to get, as precisely as possible, the information}}
\\&&{\mbox{on where the solution~$u$ lies when~$f$ is in a given space.}} \end{eqnarray}
Ideally, the best information one can strive to get occurs when the information on~$u$ holds {\em if and only if} a certain information on~$f$ is known. This perfect equivalence is not always possible (for example,
for distributional solutions of~$\Delta u=f$, it is not true that the continuity of~$f$
is equivalent to the continuity of the second derivatives of~$u$, see e.g.~\cite[Section~4.1]{2021arXiv210107941D}).
For Bessel potential spaces however, the following complete characterization holds true:

\begin{corollary}
Let~$u$, $f\in L^1_{{\rm{loc}}}(\R^n)$ be such that~$(1-\Delta)^su=f$ in~$\R^n$, in the sense of distributions.

Then, $u\in{\mathcal{L}}^p_{2s}(\R^n)$ if and only if~$f\in L^p(\R^n)$.
\end{corollary}

\begin{proof} On the one hand, if~$f\in L^p(\R^n)$, we have that~$u\in{\mathcal{L}}^p_{2s}(\R^n)$,
owing to Theorem~\ref{09ijnuiathcaRn-th}.

On the other hand, if~$u\in{\mathcal{L}}^p_{2s}(\R^n)$, the definition
of Bessel potential space in~\eqref{HSZUONDCCONT-FF0iejdfujhnHSNDJNdik9238475tyhfj}
guarantees the existence of a function~$g\in L^p(\R^n)$ such that~$u={\mathcal{B}}^{(s)}*g$.
Hence, by Proposition~\ref{BiknsoU98ikjmd9f4}, we have that~$(1-\Delta)^su=g$ in~$\R^n$, in the sense of distributions,
and this gives that~$f=g\in L^p(\R^n)$.
\end{proof}

The pieces of information in~\eqref{asd0scjovndskghjsdckvhrta9djcvmouosd-01} and~\eqref{asd0scjovndskghjsdckvhrta9djcvmouosd-02} can also fruitfully intertwine: for example,
after showing the best possible regularity result
according to~\eqref{asd0scjovndskghjsdckvhrta9djcvmouosd-02},
one can employ the knowledge of the functional spaces set up in~\eqref{asd0scjovndskghjsdckvhrta9djcvmouosd-01}
to draw consequences on the solvability and regularity in other spaces. In the case under consideration,
for example, one can combine the regularity theory
in Bessel potential spaces and the knowledge of how these spaces interplay with Sobolev spaces and obtain:

\begin{theorem}\label{TH0-104-okn5KMD3-09}
Let~$k\in\N$,~$s\ge \frac{k}2$,~$p\in(1,+\infty)$,~$f\in L^p(\R^n)$, and~$u$ be a distributional solution of~$(1-\Delta)^su=f$ in~$\R^n$.

Then,~$u\in W^{k,p}(\R^n)$ and
\begin{equation*} \|u\|_{ W^{k,p}(\R^n)}\le C\,\| f\|_{L^p(\R^n)}\end{equation*}
for some positive constant~$C$ depending only on~$n$,~$p$,~$s$ and~$k$.
\end{theorem}

\begin{proof} In virtue of Corollary~\ref{09ijnuiathcaRn} and Theorem~\ref{09ijnuiathcaRn-th}, we know that
$$ \|u\|_{ {\mathcal{L}}^p_{k}(\R^n)}\le\|u\|_{ {\mathcal{L}}^p_{2s}(\R^n)}=\| f\|_{L^p(\R^n)}.$$
The desired result thus follows from Theorem~\ref{TH0-104-okn5KMD3}.
\end{proof}

\begin{corollary}\label{ILCOLPSQRTPASKLX}
Let~$p\in(1,+\infty)$,~$f\in L^p(\R^n)$, and~$u$ be a distributional solution of~$\sqrt{1-\Delta}\,u=f$ in~$\R^n$.

Then,~$u\in W^{1,p}(\R^n)$ and
\begin{equation*} \|u\|_{ W^{1,p}(\R^n)}\le C\,\| f\|_{L^p(\R^n)}\end{equation*}
for some positive constant~$C$ depending only on~$n$ and~$p$.
\end{corollary}

\begin{proof} This is a particular case of Theorem~\ref{TH0-104-okn5KMD3-09}, picking~$s:=\frac12$ and~$k:=1$ there.
\end{proof}

\begin{theorem}\label{6.33THCORNSTNE}
Let~$s\in(0,1)$,~$p\in(1,+\infty)$,~$f\in L^p(\R^n)$, and~$u$ be a distributional solution of~$(-\Delta)^su=f$ in~$\R^n$.

Assume that~$u\in L^p(\R^n)$.

Then,~$u\in {\mathcal{L}}^p_{2s}(\R^n)$ and
\begin{equation*} \|u\|_{ {\mathcal{L}}^p_{2s}(\R^n)}\le
C\,\Big(\|u\|_{L^p(\R^n)}+\| f\|_{L^p(\R^n)}\Big)\end{equation*}
for some positive constant~$C$ depending only on~$n$,~$p$ and~$s$. 
\end{theorem}

\begin{proof} Let~$g:=(1-\Delta)^s u$.
By virtue of Theorem~\ref{TH:9191},
\begin{equation}\label{PK:S10uojr321912} \begin{split}&\|g\|_{L^p(\R^n)}=\|(1-\Delta)^s u\|_{L^p(\R^n)}\le C\Big(\| u\|_{L^p(\R^n)}+\|(-\Delta)^su\|_{L^p(\R^n)}\Big)\\&\qquad\qquad\qquad\qquad\le C\Big(\| u\|_{L^p(\R^n)}+\|f\|_{L^p(\R^n)}\Big).\end{split}\end{equation}
Accordingly, using Theorem~\ref{09ijnuiathcaRn-th},
$u\in {\mathcal{L}}^p_{2s}(\R^n)$ and
\begin{equation*} \|u\|_{ {\mathcal{L}}^p_{2s}(\R^n)}=\| g\|_{L^p(\R^n)}\le C\Big(\| u\|_{L^p(\R^n)}+\|f\|_{L^p(\R^n)}\Big).\qedhere\end{equation*}
\end{proof}

\begin{theorem}\label{PKJSM0-COSK1424654856ij45gf}
Let~$s\in\left[ \frac{1}2,1\right)$,~$p\in(1,+\infty)$,~$f\in L^p(\R^n)$, and~$u$ be a distributional solution of~$(-\Delta)^su=f$ in~$\R^n$.

Assume that~$u\in L^p(\R^n)$.

Then,~$u\in W^{1,p}(\R^n)$ and
\begin{equation*} \|u\|_{ W^{1,p}(\R^n)}\le C\,\Big(\|u\|_{L^p(\R^n)}+\| f\|_{L^p(\R^n)}\Big)\end{equation*}
for some positive constant~$C$ depending only on~$n$,~$p$ and~$s$.
\end{theorem}

\begin{proof} Let~$g:=(1-\Delta)^s u$ and recall~\eqref{PK:S10uojr321912}.
We infer from Theorem~\ref{TH0-104-okn5KMD3-09} (used here with~$k:=1$) that~$u\in W^{1,p}(\R^n)$ and
\begin{equation*} \|u\|_{ W^{1,p}(\R^n)}\le C\,\| g\|_{L^p(\R^n)}\le C\,\Big(\|u\|_{L^p(\R^n)}+\| f\|_{L^p(\R^n)}\Big).\qedhere\end{equation*}
\end{proof}

The result in Theorem~\ref{PKJSM0-COSK1424654856ij45gf} will be sharpened in the forthcoming Theorem~\ref{PKJSM0-COSK1424654856ij45gf-SHA}.

\begin{corollary}\label{0oijnhb8ygf6tfed35r7268wq-2oiejh09i28ytnbbVzUJDO03-1i40-1jhtrh9c6hv76b8706bhvc870}
Let~$p\in(1,+\infty)$,~$f\in L^p(\R^n)$, and~$u$ be a distributional solution of~$\sqrt{-\Delta}\,u=f$ in~$\R^n$.

Assume that~$u\in L^p(\R^n)$.

Then,~$u\in W^{1,p}(\R^n)$ and
\begin{equation*} \|u\|_{ W^{1,p}(\R^n)}\le C\,\Big(\|u\|_{L^p(\R^n)}+\| f\|_{L^p(\R^n)}\Big)\end{equation*}
for some positive constant~$C$ depending only on~$n$ and~$p$.
\end{corollary}

\begin{proof} This is a special case of Theorem~\ref{PKJSM0-COSK1424654856ij45gf} when~$s=\frac12$.\end{proof}

\section{Besov spaces}

Given~$s>0$ and~$p$,~$q\in[1,+\infty)$, it comes in handy to introduce\footnote{In jargon~$p$ is often called
the ``integration exponent'' and~$q$ the ``summation exponent''.
The reason for the latter name will be more transparent
in the forthcoming formula~\eqref{0pirj09365-4g4eZXCHJKLRFG-544-NO3}. One could also consider the case~$p=\infty$ and/or~$q=\infty$ with appropriate modifications
in the statements and in the proofs of the main results, but we restrict ourselves to the case of finite
integration and summation exponents for the sake of simplicity.

Also, one can define in a similar way Besov spaces corresponding to values of~$s$ in~$(-\infty,0]$, by considering tempered distributions instead of functions, which is also useful to discuss dual spaces, but we will not enter in this detail here, see e.g.~\cite{MR0461123, MR2884718} and the references therein for further information.

Similarly, we will only consider here the case of ``inhomogeneous'' Besov spaces, in which low frequencies
are treated separately (they would correspond to the initial bump~$\varphi_0$ in the dyadic decomposition
of Lemma~\ref{0pirj09365-4g4eZXCHJKLRFG-544-NO1-LL}). This choice is motivated by the heuristic fact that
high frequencies are the ones usually creating more troubles in terms of regularity, hence, in a sense,
these spaces are more naturally related to other functional spaces.

We mention that it is also possible to study ``homogeneous'' Besov spaces in which the frequency
decomposition is made indistinctly over the whole space, with the aim of better preserving  the scaling invariance of some problem, see e.g.~\cite{MR0461123} and the references therein.

As a notational remark, we mention that
the Besov space~$B^{s,p,q}(\R^n)$ is denoted by~$\Lambda^{p,q}_\alpha$, with~$\alpha=s$, in~\cite[Section~5, Chapter~V]{MR0290095}.} the \index{Besov space} Besov spaces~$B^{s,p,q}(\R^n)$.

We warn the reader that the treatment of Besov spaces in this pages is far\footnote{Besides, as a historical remark,
we point out that we will define Besov
spaces by partitions of unity in the frequency space on page~\pageref{DEFBEScsovkgRFVblop0-23}, linking them with the original
setting based on translations by Besov later, with several equivalent approaches outlined in Propositions~\ref{F58uuRS-2CALEREG:LEAbc-CC} and~\ref{0pirj09365-4g4eZXCHJKLRFG-544-NOP},
and Corollaries~\ref{corod34958834y689u2eryue9wighrejhgfhdjehfdjkjfhj}
and~\ref{0pirj09365-4g4eZXCHJKLRFG-544-NO3-COR}).} from being comprehensive: for a detailed theory of these spaces, see~\cite{MR0290095, MR0461123, MR0482275, MR891189, MR1163193, MR1419319, MR2768550, MR2884718, MR3308364, MR3726909, MR3771838, MR3839617}.

\begin{figure}
\centering
\begin{tikzpicture}[
    declare function= { gin(\x)     = exp(1/(\x*\x-1)+1);
                        myfun(\x)   = 3*gin(\x)*(1-gin(1-\x));
                      },
                   ]
    \draw[thick,->] (0,0) -- (10,0);
    \draw[thick,->] (0,0) -- (0,3.6);
    \draw (.1,1.5) -- (-.1,1.5)  node[left] {$\frac12$};
    \draw (.1,3) -- (-.1,3)  node[left] {$1$};
    \draw[densely dotted] (10,1.5) -- (-.1,1.5);
    \draw[densely dotted] (10,3) -- (-.1,3);
    \filldraw (0,0) circle (1.5pt);
    \node at (0,-.3) {$0$};
    \filldraw (.5,0) circle (1.5pt);
    \node at (.5,-.3) {$1$};
    \filldraw (1,0) circle (1.5pt);
    \node at (1,-.3) {$2$};
    \filldraw (2,0) circle (1.5pt);
    \node at (2,-.3) {$2^2$};
    \filldraw (4,0) circle (1.5pt);
    \node at (4,-.3) {$2^3$};
    \filldraw (8,0) circle (1.5pt);
    \node at (8,-.3) {$2^4$};
    \draw[very thick,red] (0,3) -- (.5,3);
    \draw[very thick,red] plot[domain=.501:1] (\x,{myfun(2*\x-1)});
    \node at (.3,3.3) {\color{red}~$\varphi_0$};
    \draw[very thick,blue] plot[domain=.501:1] (\x,{3-myfun(2*\x-1)});
    \draw[very thick,blue] plot[domain=1.01:2] (\x,{myfun(\x-1)});
    \node at (1,3.3) {\color{blue}~$\varphi_1$};
    \draw[very thick,magenta] plot[domain=1.01:2] (\x,{3-myfun(\x-1)});
    \draw[very thick,magenta] plot[domain=2.02:4] (\x,{myfun(\x/2-1)});
    \node at (2,3.3) {\color{magenta}~$\varphi_2$};
    \draw[very thick,teal] plot[domain=2.02:4] (\x,{3-myfun(\x/2-1)});
    \draw[very thick,teal] plot[domain=4.01:8] (\x,{myfun(\x/4-1)});
    \node at (4,3.3) {\color{teal}~$\varphi_3$};
    \draw[very thick,orange] plot[domain=4.01:8] (\x,{3-myfun(\x/4-1)});
    \draw[very thick,orange] plot[domain=8.01:10] (\x,{myfun(\x/8-1)});
    \node at (8,3.3) {\color{orange}~$\varphi_4$};
    \draw[very thick,brown] plot[domain=8.01:10] (\x,{3-myfun(\x/8-1)}) node[right] {\color{brown}~$\varphi_5$};
\end{tikzpicture}
\caption{A sketch of the dyadic partition of unity used in Lemma~\ref{0pirj09365-4g4eZXCHJKLRFG-544-NO1-LL}.}\label{fig:sub2}
\end{figure}


To define Besov spaces, it is helpful to introduce a dyadic partition of unity
in the frequency space, which will somewhat detect the fine regularity of a given function at different scales through a decomposition via bump  functions that are spectrally supported\footnote{This kind of partitions
of unity are quite useful in harmonic analysis. For example, they are also employed in the proof of the
Mikhlin Multiplier Theorem, see e.g.~\cite[Section~4.6]{MR2884718}.} in dyadic shells
(see Figure~\ref{fig:sub2} for a very rough diagram of a dyadic partition of unity). For this target, we point out that:

\begin{lemma}\label{0pirj09365-4g4eZXCHJKLRFG-544-NO1-LL}
For all~$j\in\N$ there exists a radially symmetric function~$\varphi_j:\R^n\to[0,1]$ such that
\begin{equation}\label{0pirj09365-4g4eZXCHJKLRFG-544-NO1-LL-eq1}
\sum_{j=0}^{+\infty}\varphi_j(\xi)=1\quad {\mbox{ for all~$\xi\in\R^n$, }}
\end{equation}
with~$\varphi_0\in C^\infty_c(B_2)$
and, for all~$j\ge1$,
\begin{equation}\label{0pirj09365-4g4eZXCHJKLRFG-544-NO1-LL-eq2}
\varphi_j\in C^\infty_c(B_{2^{j+1}}\setminus B_{2^{j-1}}) .
\end{equation}

Moreover,  for all~$j\ge1$
\begin{equation}\label{0pirj09365-4g4eZXCHJKLRFG-544-NO1-LL-eq4-09}
\varphi_j(\xi)=\varphi_1(2^{1-j}\xi),
\end{equation}
\begin{equation}\label{0pirj09365-4g4eZXCHJKLRFG-544-NO1-LL-eq4-09-TRIS}
\check\varphi_j(x)=2^{(j-1)n}\check \varphi_1(2^{j-1}x),\end{equation}
\begin{equation}\label{0pirj09365-4g4eZXCHJKLRFG-544-NO1-LL-eq4-09-BIS}
\|\check\varphi_j\|_{L^1(\R^n)}=\|\check\varphi_1\|_{L^1(\R^n)},
\end{equation}
and, given~$\alpha\in\N^n$,
\begin{equation}\label{0pirj09365-4g4eZXCHJKLRFG-544-NO1-LL-eq4}
\|D^\alpha\varphi_j\|_{L^\infty(\R^n)}\le \frac{C}{2^{|\alpha|j}} \|D^\alpha\varphi_1\|_{L^\infty(\R^n)}
\end{equation}
for a positive constant~$C$ depending only on~$\alpha$.


Additionally, for every~$j\ge1$ and~$\alpha\in\N^n$,
\begin{equation}\label{UTILTAYMAA8ikjf}
\int_{\R^n} \check\varphi_j(y)\,y^\alpha\,dy=0.
\end{equation}

Furthermore, if, for~$k\ge1$, we define
\begin{equation}\label{UTILTAYMAA8ikjf-x589} \Phi_k(\xi):=1-\sum_{j=0}^k\varphi_j(\xi),\end{equation}
we have that
\begin{equation}\label{UTILTAYMAA8ikjf-x5}\begin{split}
&\Phi_k\in C^\infty(\R^n,[0,1]),\\ &{\mbox{$\Phi_k=0$ in~$B_{2^k}$,}}\\
&{\mbox{$\Phi_k=1$ in~$\R^n\setminus B_{2^{k+1}}$}}\\ {\mbox{and, for all~$\alpha\in\N^n$, }}\,&
\|D^\alpha\Phi_k\|_{L^\infty(\R^n)}\le C,\end{split}
\end{equation}
for a positive constant~$C$ depending only on~$n$ and~$\alpha$ (in particular, independent of~$k$).

Finally, if we set~$\varphi_{-1}$ to be the null function and, for all~$j\in\N$,
\begin{equation}\label{UTILTAYMAA8ikjf-x5-LAPSIN} \psi_j:=\varphi_{j-1}+\varphi_{j}+\varphi_{j+1},\end{equation}
we have that
\begin{equation}\label{UTILTAYMAA8ikjf-x5-LAPSIN2} \psi_j\varphi_j=\varphi_{j}.\end{equation}
\end{lemma}

\begin{proof} Let~$\varphi_0\in C^\infty_c(B_{19/10},[0,1])$ be radial and radially nonincreasing and such that~$\varphi_0=1$ in~$B_{11/10}$.
Then, for all~$j\ge1$, let
\begin{equation}\label{0984ujr0-3uthgGSBDEFIj}\varphi_j(\xi):=\varphi_0(2^{-j}\xi)-\varphi_0(2^{-j+1}\xi).\end{equation}

Let us check that
\begin{equation}\label{0984ujr0-3uthgGSBDEFIj2}
{\mbox{$\varphi_j$ takes values in~$[0,1]$.}}\end{equation} Since this is obvious for~$j=0$ let us assume that~$j\ge1$.
Thus, since~$\varphi_0\ge0$, it follows from~\eqref{0984ujr0-3uthgGSBDEFIj} that~$\varphi_j\le1$.
Additionally, since~$\varphi_0$ is radially nonincreasing, we have that~$\varphi_0(2^{-j}\xi)\ge\varphi_0(2^{-j+1}\xi)$, which, together with~\eqref{0984ujr0-3uthgGSBDEFIj}, gives that~$\varphi_j\ge0$.
The proof of~\eqref{0984ujr0-3uthgGSBDEFIj2} is thereby complete.
 
Now we check the validity of~\eqref{0pirj09365-4g4eZXCHJKLRFG-544-NO1-LL-eq2}.
Firstly, if~$|\xi|\le\frac{11\;2^{j-1}}{10}$ we have that~$
|2^{-j}\xi|\le|2^{-j+1}\xi|\leq\frac{11}{10}$, therefore~$\varphi_0(2^{-j}\xi)=1=\varphi_0(2^{-j+1}\xi)$, whence~$\varphi_j(\xi)=0$.

Similarly, if~$|\xi|\ge\frac{19\;2^{j+1}}{20}$ it follows that~$|2^{-j+1}\xi|\ge|2^{-j}\xi|\ge\frac{19}{10}$, yielding that~$\varphi_0(2^{-j}\xi)=0=\varphi_0(2^{-j+1}\xi)$, and accordingly~$\varphi_j(\xi)=0$.

These observations establish~\eqref{0pirj09365-4g4eZXCHJKLRFG-544-NO1-LL-eq2}.

We now prove~\eqref{0pirj09365-4g4eZXCHJKLRFG-544-NO1-LL-eq1}. 
For this, we employ a telescopic sum generated by~\eqref{0984ujr0-3uthgGSBDEFIj}
\[
\sum_{j=0}^{+\infty}\varphi_j(\xi)
=\lim_{N\nearrow+\infty}\sum_{j=0}^{N}\varphi_j(\xi)
=\lim_{N\nearrow+\infty}\varphi_0(2^{-N}\xi)=\varphi_0(0)=1
\]

Furthermore, if~$j\ge1$ we can use~\eqref{0984ujr0-3uthgGSBDEFIj} to find that
\begin{eqnarray*}
\varphi_j(\xi)-\varphi_1(2^{1-j}\xi)=\big(\varphi_0(2^{-j}\xi)-\varphi_0(2^{-j+1}\xi)\big)-
\big(\varphi_0(2^{-1}\,2^{1-j}\xi)-\varphi_0(2^{1-j}\xi)\big)=0,
\end{eqnarray*}
which establishes~\eqref{0pirj09365-4g4eZXCHJKLRFG-544-NO1-LL-eq4-09}.

As a consequence of~\eqref{0pirj09365-4g4eZXCHJKLRFG-544-NO1-LL-eq4-09} we also have that
$$ \check\varphi_j(x)=\int_{\R^n}\varphi_1(2^{1-j}\xi)\,e^{2\pi ix\cdot\xi}\,d\xi=
2^{(j-1)n}\int_{\R^n}\varphi_1(\eta)\,e^{2\pi i2^{j-1}x\cdot\eta}\,d\eta=
2^{(j-1)n}\check\varphi_1(2^{j-1}x),$$
which is~\eqref{0pirj09365-4g4eZXCHJKLRFG-544-NO1-LL-eq4-09-TRIS}, and accordingly, changing variables~$y:=2^{j-1}x$,
\begin{eqnarray*}
\|\check\varphi_j\|_{L^1(\R^n)}=2^{(j-1)n}\int_{\R^n}|\check\varphi_1(2^{j-1}x)|\,dx
=\|\check\varphi_1\|_{L^1(\R^n)},
\end{eqnarray*}
which is~\eqref{0pirj09365-4g4eZXCHJKLRFG-544-NO1-LL-eq4-09-BIS}.

Additionally, taking derivatives of~\eqref{0pirj09365-4g4eZXCHJKLRFG-544-NO1-LL-eq4-09}
we conclude that, for all~$j\ge1$,~$\alpha\in\N^n$ and~$\xi\in\R^n$,
\begin{eqnarray*}
|D^\alpha\varphi_j(\xi)|&=&2^{(1-j)|\alpha|}|D^\alpha\varphi_1(2^{1-j}\xi)|
\end{eqnarray*}
from which the claim in~\eqref{0pirj09365-4g4eZXCHJKLRFG-544-NO1-LL-eq4} plainly follows.


Besides, we use~\eqref{0pirj09365-4g4eZXCHJKLRFG-544-NO1-LL-eq2} and~\eqref{0pirj09365-4g4eZXCHJKLRFG-544-NO1-LL-eq4-09}, together with the change of variable~$(\eta,z):=(2^{j-1}\xi,2^{1-j}y)$, and observe that, for every~$j\ge1$ and~$\alpha\in\N^n$,
\begin{eqnarray*}
0&=& D^\alpha\varphi_1(0)\\
&=& \int_{\R^n} {\mathcal{F}}^{-1} (D^\alpha\varphi_1) (y)\,dy
\\&=&\iint_{\R^n\times\R^n} D^\alpha\varphi_1(\xi)\, e^{2\pi i\xi\cdot y}\,d\xi\,dy
\\&=&(-1)^{|\alpha|}\iint_{\R^n\times\R^n} \varphi_1(\xi)\, D^\alpha_\xi e^{2\pi i\xi\cdot y}\,d\xi\,dy
\\&=& (-2\pi i)^{|\alpha|}\iint_{\R^n\times\R^n} \varphi_1(\xi)\,y^\alpha\, e^{2\pi i\xi\cdot y}\,d\xi\,dy
\\&=& (-2^j\pi i)^{|\alpha|}\iint_{\R^n\times\R^n} \varphi_1(2^{1-j}\eta)\,z^\alpha\, e^{2\pi i\eta\cdot z}\,d\eta\,dz
\\&=& (-2^j\pi i)^{|\alpha|}\iint_{\R^n\times\R^n} \varphi_j(\eta)\,z^\alpha\, e^{2\pi i\eta\cdot z}\,d\eta\,dz
\\&=& (-2^j\pi i)^{|\alpha|}\int_{\R^n} \check\varphi_j(z)\,z^\alpha\,dz
,\end{eqnarray*}
which proves~\eqref{UTILTAYMAA8ikjf}.

Now we prove~\eqref{UTILTAYMAA8ikjf-x5}.
For this, we use the fact that~$\varphi_j\ge0$ to deduce that~$\Phi_k\le1$. Similarly, by~\eqref{0pirj09365-4g4eZXCHJKLRFG-544-NO1-LL-eq1},
\begin{equation} \label{UTILTAYMAA8ikjf-x5-098uhn3oefnX}\Phi_k=\sum_{k+1}^{+\infty}\varphi_j\ge0.\end{equation}
Let us now pick~$\xi\in\R^n\setminus B_{2^{k+1}}$. Then~$|\xi|\ge 2^{j+1}$, and thus~$\varphi_j(\xi)=0$
owing to~\eqref{0pirj09365-4g4eZXCHJKLRFG-544-NO1-LL-eq2}, for all~$j\in\{0,\dots,k\}$. This and~\eqref{UTILTAYMAA8ikjf-x589} give that~$\Phi_k(\xi)=1$.

Similarly, if we pick~$\xi\in B_{2^k}$, we have that~$|\xi|\le 2^{j-1}$, and thus~$\varphi_j(\xi)=0$
owing to~\eqref{0pirj09365-4g4eZXCHJKLRFG-544-NO1-LL-eq2},
for all~$j\ge k+1$. This and~\eqref{UTILTAYMAA8ikjf-x5-098uhn3oefnX} yield that~$\Phi_k(\xi)=0$.

Moreover, we have that~$\|\Phi_k\|_{L^\infty(\R^n)}\le1$ and,
recalling~\eqref{0pirj09365-4g4eZXCHJKLRFG-544-NO1-LL-eq4}, for all~$\alpha\in\N^n\setminus\{0\}$,
\begin{eqnarray*}
|D^\alpha\Phi_k|=\left|\sum_{j=0}^k D^\alpha\varphi_j\right|\le\sum_{j=0}^k\frac{C}{2^{|\alpha|j}} \|D^\alpha\varphi_1\|_{L^\infty(\R^n)}\le\sum_{j=0}^{+\infty}\frac{C}{2^{|\alpha|j}} \|D^\alpha\varphi_1\|_{L^\infty(\R^n)}\le C,
\end{eqnarray*}
up to renaming~$C$.
These observations establish~\eqref{UTILTAYMAA8ikjf-x5}.

Now we deal with the setting in~\eqref{UTILTAYMAA8ikjf-x5-LAPSIN}.
We observe that~\eqref{UTILTAYMAA8ikjf-x5-LAPSIN2} is obvious at all points outside the support of~$\varphi_j$, since both sides vanish there. Hence, to prove~\eqref{UTILTAYMAA8ikjf-x5-LAPSIN2},
we pick~$\xi$ in the support of~$\varphi_j$. 

This gives that~$\varphi_\ell(\xi)=0$ unless~$\ell\in\{j-1,j,j+1\}$, thanks to~\eqref{0pirj09365-4g4eZXCHJKLRFG-544-NO1-LL-eq2}.

As a result, by~\eqref{0pirj09365-4g4eZXCHJKLRFG-544-NO1-LL-eq1},
$$ 1=\sum_{\ell=0}^{+\infty}\varphi_\ell(\xi)=\sum_{\ell=j-1}^{j+1}\varphi_\ell(\xi)=\psi_j(\xi),$$
from which we obtain~\eqref{UTILTAYMAA8ikjf-x5-LAPSIN2}, as desired.
\end{proof}

The dyadic partition of unity presented in Lemma~\ref{0pirj09365-4g4eZXCHJKLRFG-544-NO1-LL}
allows us to introduce the Besov spaces via the following definition. \label{DEFBEScsovkgRFVblop0-23}
Given~$s>0$ and~$p$,~$q\in[1,+\infty)$, we define~$B^{s,p,q}(\R^n)$ as the space of\footnote{Alternatively, rather than considering functions,
one can also work with spaces of tempered distributions (a similar remark was pointed out in footnote~\ref{VEONORL123r3FVSjjMDMFP} on page~\pageref{VEONORL123r3FVSjjMDMFP} for Bessel potential spaces). The distributional settings would be equivalent to our framework, since, as showcased in
Corollary~\ref{COROLEEP2r34t5},~$B^{s,p,q}(\R^n)\subseteq L^p(\R^n)$ for all positive values of~$s$, and only minor notational modifications would be needed in the proofs presented here.} functions~$u\in L^p(\R^n)$ for which
$$\sum_{j=0}^{+\infty}2^{jsq}\|\check{\varphi}_j*u\|_{L^p(\R^n)}^q<+\infty.$$
This space is equipped with the norm
\begin{equation}\label{BENOGIU}\|u\|_{B^{s,p,q}(\R^n)}:=
\left(\sum_{j=0}^{+\infty}2^{jsq}\|\check{\varphi}_j*u\|_{L^p(\R^n)}^q\right)^{\frac1q}.\end{equation}

For our purposes, Besov spaces are important because they constitute an interesting bridge between Bessel potential spaces and Sobolev spaces, somewhat allowing us to extend, in a suitably modified form, Theorem~\ref{TH0-104-okn5KMD3} beyond the case of half-integers (see e.g.
Corollaries~\ref{LE637:l},~\ref{COROLEEP2r34t5} and~\ref{0pirj09365-4g4eZXCHJKLRFG-544-NO1-I2L3C2O25R213t4TY},
and Theorem~\ref{THPS2} for precise statements).

Interestingly, one can directly observe that the exponents~$s$ and~$q$ in~$B^{s,p,q}(\R^n)$ somewhat measure a ``degree of regularity''. Indeed, we have:

\begin{lemma} \label{F58uuRS-2CALEREG:LE} Let~$s>0$ and~$p$,~$q\in[1,+\infty)$. Let also~$a>0$.

Then,
\begin{equation}\label{F58uuRS-2CALEREG-0} B^{s,p,q+a}(\R^n)\supseteq B^{s,p,q}(\R^n),\end{equation}
with continuous embedding.

Moreover, for all~$q_1$,~$q_2\in[1,+\infty)$,
\begin{equation}\label{F58uuRS-2CALEREG}
B^{s+a,p,q_1}(\R^n)\subseteq B^{s,p,q_2}(\R^n),\end{equation}with continuous embedding.
\end{lemma}

We point out the ``scale of regularity'' detected by the different
exponents in Besov spaces: specifically, we stress that the exponent~$q$ measures the regularity of a function on a ``finer scale'' than~$s$, in the sense given by~\eqref{F58uuRS-2CALEREG}.
Moreover, it can be shown that the inclusions in~\eqref{F58uuRS-2CALEREG-0} and~\eqref{F58uuRS-2CALEREG}
are sharp and cannot be improved, see~\cite[Lemmata~22 and~24]{MR163159}.

\begin{proof}[Proof of Lemma~\ref{F58uuRS-2CALEREG:LE}] Here is a general observation:
if~$\kappa_j\ge0$ and~$a\ge0$, then
$$ \kappa_j^a\le \left(\sum_{m=0}^{+\infty}\kappa_m^q\right)^{\frac{a}q}$$and therefore
\begin{equation}\label{F58uuRS-2CALEREG:LE2o-3rkpjtgr94ejrfmXXs}
\sum_{j=0}^{+\infty}\kappa_j^{q+a}\le
\sum_{j=0}^{+\infty}\kappa_j^{q}\left(\sum_{m=0}^{+\infty}\kappa_m^q\right)^{\frac{a}q}=
\left(\sum_{j=0}^{+\infty}\kappa_j^{q}\right)^{\frac{q+a}q}.
\end{equation}

Thus, to prove~\eqref{F58uuRS-2CALEREG-0},
we pick~$u\in B^{s,p,q}(\R^n)$ and use~\eqref{F58uuRS-2CALEREG:LE2o-3rkpjtgr94ejrfmXXs} with~$\kappa_j:=2^{js}\|\check{\varphi}_j*u\|_{L^p(\R^n)}$ to find that
\begin{eqnarray*}
\|u\|_{B^{s,p,q+a}(\R^n)}=
\left(\sum_{j=0}^{+\infty}2^{js(q+a)}\|\check{\varphi}_j*u\|_{L^p(\R^n)}^{q+a}\right)^{\frac1{q+a}}\le
\left(\sum_{j=0}^{+\infty}2^{jsq}\|\check{\varphi}_j*u\|_{L^p(\R^n)}^{q}\right)^{\frac1q}=\|u\|_{B^{s,p,q}(\R^n)}.
\end{eqnarray*}
This proves~\eqref{F58uuRS-2CALEREG-0} and we now deal with the proof of~\eqref{F58uuRS-2CALEREG}.

To this end, we first show that
\begin{equation}\label{F58uuRS-2CALEREG-p1e27634563r5erew1}
B^{s+a,p,q_1}(\R^n)\subseteq B^{s,p,1}(\R^n),\end{equation}with continuous embedding.

For this, we pick~$u\in B^{s+a,p,q_1}(\R^n)$ and we distinguish two cases. If~$q_1=1$, then we notice that
$$ \|u\|_{B^{s,p,1}(\R^n)}=\sum_{j=0}^{+\infty}2^{js}\|\check{\varphi}_j*u\|_{L^p(\R^n)}\le
\sum_{j=0}^{+\infty}2^{j(s+a)}\|\check{\varphi}_j*u\|_{L^p(\R^n)}=\|u\|_{B^{s+a,p,1}(\R^n)},$$
giving~\eqref{F58uuRS-2CALEREG-p1e27634563r5erew1} in this case.

If instead~$q_1>1$ we use H\"older's Inequality with exponents~$\frac{q_1}{q_1-1}$ and~$q_1$ to observe that
\begin{eqnarray*}
\|u\|_{B^{s,p,1}(\R^n)}&=&\sum_{j=0}^{+\infty}2^{js}\|\check{\varphi}_j*u\|_{L^p(\R^n)}\\
&=&\sum_{j=0}^{+\infty}2^{-ja} 2^{j(s+a)}\|\check{\varphi}_j*u\|_{L^p(\R^n)}
\\&\le&\left(\sum_{j=0}^{+\infty}2^{-\frac{jaq_1}{q_1-1}} \right)^{\frac{q_1-1}{q_1}}
\left(\sum_{j=0}^{+\infty} 2^{j(s+a)q_1}\|\check{\varphi}_j*u\|_{L^p(\R^n)}^{q_1}\right)^{\frac1{q_1}}\\&=&C\|u\|_{B^{s+a,p,q_1}(\R^n)},
\end{eqnarray*}
which establishes~\eqref{F58uuRS-2CALEREG-p1e27634563r5erew1}.

The proof of~\eqref{F58uuRS-2CALEREG} then follows from~\eqref{F58uuRS-2CALEREG-0}
and~\eqref{F58uuRS-2CALEREG-p1e27634563r5erew1}, observing that~$B^{s+a,p,q_1}(\R^n)\subseteq B^{s,p,1}(\R^n)\subseteq B^{s,p,q_2}(\R^n)$, with continuous embedding.
\end{proof}

A useful result when dealing with Besov spaces consists in detecting the link
between differentiation and multiplication for functions
spectrally supported in a shell or in a ball. The set of inequalities obtained in this way is sometimes called
Bernstein's Lemma \index{Bernstein's Lemma} (dating back in its original formulation to~\cite{BERN1912})
and goes as follows:

\begin{lemma}\label{BERLEMMAUT}
Let~$C_1$,~$C_2\in(0,+\infty)$, with~$C_2>C_1$.

There exists~$C\ge1$, depending only on~$n$,~$C_1$ and~$C_2$, for which,
for all~$a$,~$b\ge1$ with~$a\le b$, all~$k\in\N$, all~$\lambda>0$, and all functions~$v$, the following statements hold true.

If the support of the distribution~$\widehat v$ is contained in the ball~$B_{C_1\lambda}$, then
\begin{equation}\label{BERLEMMAUT-1}
\sup_{{\alpha\in\N^n}\atop{|\alpha|=k}}\|D^\alpha v\|_{L^b(\R^n)}\le C^k\lambda^{k+n\left(\frac{1}{a}-\frac{1}b\right)} \|v\|_{L^a(\R^n)}.
\end{equation}

If the support of the distribution~$\widehat v$ is contained in the shell~$B_{C_2\lambda}\setminus B_{C_1\lambda}$, then
\begin{equation}\label{BERLEMMAUT-2}
\frac{\lambda^{k} }{C^k}\|v\|_{L^a(\R^n)}\le
\sup_{{\alpha\in\N^n}\atop{|\alpha|=k}}\|D^\alpha v\|_{L^a(\R^n)}\le C^k\lambda^{k} \|v\|_{L^a(\R^n)}.
\end{equation}
\end{lemma}

\begin{proof} First of all, we observe that 
\begin{equation}\label{LAMBDABERLEMMAUT-2}
{\mbox{it suffices to prove the desired claims when~$\lambda=1$.}}\end{equation}
Indeed, if~\eqref{BERLEMMAUT-1} and~\eqref{BERLEMMAUT-2} hold true when~$\lambda=1$,
we define~$w(x):=\frac1{\lambda^n} v\left(\frac{x}\lambda \right)$ and observe that~$\widehat w(\xi)=\widehat v(\lambda\xi)$.
In this way, the support of~$\widehat w$ satisfies the same hypotheses as that of~$\widehat v$ but with~$\lambda=1$. To recover the appropriate scaling for~$v$ it is thus sufficient to observe that, for all~$\beta\in\R^n$ and~$m\ge1$,
\begin{eqnarray*}
\| D^\beta w\|_{L^m(\R^n)}= \lambda^{-n-|\beta|}\left(\int_{\R^n}
\left|D^\beta v\left(\frac{x}\lambda \right)\right|^m\,dx\right)^{\frac1m}=\lambda^{-n-|\beta|+\frac{n}{m}}\,\| D^\beta v\|_{L^m(\R^n)}.
\end{eqnarray*}
This proves~\eqref{LAMBDABERLEMMAUT-2}.

Now, in light of~\eqref{LAMBDABERLEMMAUT-2}, we assume that~$\lambda=1$ and we carry on with the main argument.
For this, the idea is to multiply~$v$ by a cutoff function in the frequency space and use the interplay between multiplication and convolution under Fourier Transform. 

Namely, let~$\tau\in C^\infty_c(B_{2C_1})$ be such that~$\tau=1$ in~$B_{C_1}$.
Hence, if the support of the distribution~$\widehat v$ is contained
in the ball~$B_{C_1}$, we have that~$\widehat v= \tau\,\widehat v$.

As a consequence, \begin{equation}\label{RnleCk1012oerk00otgb3edcfgyan7y7u} v=\check\tau* v.\end{equation}

Moreover, for all~$\alpha\in\N^n$ with~$|\alpha|=k$,
\begin{eqnarray*} D^\alpha\check\tau(x)=D^\alpha_x\left(\int_{\R^n}\tau(\xi)\,e^{2\pi ix\cdot\xi}\,d\xi\right)=\int_{B_{2C_1}}(2\pi i\xi)^\alpha \tau(\xi)\,e^{2\pi ix\cdot\xi}\,d\xi.
\end{eqnarray*}
On this account,
\begin{equation}\label{USHNDcTRnKrijdf9ijN9ijn3edft76yujikmnb} \|D^\alpha\check\tau\|_{L^\infty(\R^n)}\le
(2\pi)^k (2C_1)^{k}\int_{B_{2C_1}}| \tau(\xi)|\,d\xi\le C^{k}
.\end{equation}

Also, letting
\begin{equation} \label{USHNDcTRnKrijdf9ijN9ijn3edft76yujikmnbQ4567AW}{\mathcal{D}}_m:=\Big\{x\in\R^n\setminus B_1 {\mbox{ s.t. }}|x_m|=\max_{\ell\in\{1,\dots,n\}}|x_\ell|\Big\},\end{equation}
we have that
\begin{eqnarray*} \|D^\alpha\check\tau\|_{L^1(\R^n\setminus B_1)}
&\le&
(2\pi)^k 
\int_{\R^n\setminus B_1}\left|\int_{B_{2C_1}} \xi^\alpha \tau(\xi)\,e^{2\pi i x\cdot\xi}\,d\xi\right|\,dx\\&\le&
(2\pi)^k
\sum_{m=1}^n\int_{{\mathcal{D}}_m}
\left|\int_{B_{2C_1}} \xi^\alpha \tau(\xi)\,e^{2\pi ix\cdot\xi}\,d\xi\right|\,dx\\&=&
(2\pi)^k
\sum_{m=1}^n\int_{{\mathcal{D}}_m}
\left|\int_{B_{2C_1}} \xi^\alpha \tau(\xi)\,\partial_{\xi_m}^{n+1}\left(\frac{e^{2\pi ix\cdot\xi}}{(2\pi i x_m)^{n+1}}\right)\,d\xi\right|\,dx\\&=&
(2\pi)^k 
\sum_{m=1}^n\int_{{\mathcal{D}}_m}
\left|\int_{B_{2C_1}} \partial_{\xi_m}^{n+1}\big(\xi^\alpha \tau(\xi)\big)\,
\frac{e^{2\pi i\cdot\xi}}{(2\pi i x_m)^{n+1}}\,d\xi\right|\,dx\\
&\le&
(2\pi)^{k-n-1}
\sum_{m=1}^n
\iint_{{\mathcal{D}}_m\times B_{2C_1}} \big|\partial_{\xi_m}^{n+1}\big(\xi^\alpha \tau(\xi)\big)\big|\,
\frac{dx}{| x_m|^{n+1}}\,d\xi\\&\le&C\,
(2\pi)^{k-n-1}
\iint_{(\R^n\setminus B_1)\times B_{2C_1}} \big|\partial_{\xi_m}^{n+1}\big(\xi^\alpha \tau(\xi)\big)\big|\,
\frac{dx}{| x|^{n+1}}\,d\xi\\&\le&C\,(2\pi)^{k-n-1}
\int_{B_{2C_1}} \big|\partial_{\xi_m}^{n+1}\big(\xi^\alpha \tau(\xi)\big)\big|\,d\xi
\\&\le&C^{k},
\end{eqnarray*}
up to renaming~$C>0$.

This estimate and~\eqref{USHNDcTRnKrijdf9ijN9ijn3edft76yujikmnb} entail that
\begin{equation}\label{USHNDcTRnKrijdf9ijN9ijn3edft76yujikmnb2} \|D^\alpha\check\tau\|_{L^1(\R^n)}\le C^{k}
,\end{equation}
up to renaming constants.

Now we observe that, calling~$B$ the ball centered at the origin with unit volume, for every function~$f$ and every~$m\ge1$, one has that
\begin{equation}\label{USHNDcTRnKrijdf9ijN9ijn3edft76yujikmnb201o2uerjhfgVVXDIKe78}
\|f\|_{L^m(\R^n)}=\left(\int_{\R^n}|f(x)|^m\,dx\right)^{\frac{1}m}
\leq \|f\|_{L^\infty(\R^n)}^{\frac{m-1}m}\left(\int_{\R^n}|f(x)|\,dx\right)^{\frac{1}m}
=\|f\|_{L^\infty(\R^n)}^{\frac{m-1}m} \|f\|_{L^1(\R^n)}^{\frac1m}.
\end{equation}

Hence, we deduce from~\eqref{USHNDcTRnKrijdf9ijN9ijn3edft76yujikmnb},~\eqref{USHNDcTRnKrijdf9ijN9ijn3edft76yujikmnb2},
and~\eqref{USHNDcTRnKrijdf9ijN9ijn3edft76yujikmnb201o2uerjhfgVVXDIKe78} that, for all~$m\ge1$,
$$ \|D^\alpha\check\tau\|_{L^m(\R^n)}\le C^k.$$
Interestingly,~$C$ does not depend on~$m$,~$\alpha$ or~$k$.

In this way, retaking~\eqref{RnleCk1012oerk00otgb3edcfgyan7y7u}, for all~$a$,~$b\ge1$ with~$a\le b$,
all~$k\in\N^n$ and all~$\alpha\in\N^n$ with~$|\alpha|=k$,
\begin{eqnarray*}
\|D^\alpha v\|_{L^b(\R^n)}=\|D^\alpha \check\tau* v\|_{L^b(\R^n)}\le
\|D^\alpha \check\tau\|_{L^m(\R^n)}\| v\|_{L^a(\R^n)}\le C^k\| v\|_{L^a(\R^n)},
\end{eqnarray*}
where we have used Young's Convolution Inequality with exponent~$m:=\frac{ab}{ab+a-b}$
(and we stress that~$m\ge1$ since~$a\le b$). The proof of~\eqref{BERLEMMAUT-1} is thereby complete.

We now prove~\eqref{BERLEMMAUT-2}. Actually, the second inequality in~\eqref{BERLEMMAUT-2} is a byproduct of~\eqref{BERLEMMAUT-1} with~$b:=a$, hence it only remains to prove the first inequality in~\eqref{BERLEMMAUT-2}.

To this end, we take~$\theta\in C^\infty_c(\R^n)$ such that~$\theta=1$ in~$B_{C_2}\setminus B_{C_1}$.
Thus, if~$\widehat v$ is supported in~$B_{C_2}\setminus B_{C_1}$, it follows that
\begin{equation}\label{BERLEMMAUT-2COMBA0}
\widehat v=\theta\,\widehat v.\end{equation}

It is now useful to observe that, for all~$k\in\N$ and~$\alpha\in\N^n$ with~$|\alpha|=k$, there exists
a set of coefficients~$a_\alpha\in\R$ such that~$|a_\alpha|\le n^k$ and, for every~$\xi\in\R^n$,
\begin{equation}\label{BERLEMMAUT-2COMBA}
\sum_{{\alpha\in\N^n}\atop{|\alpha|=k}}a_\alpha \,(i \xi)^\alpha(-i \xi)^\alpha=
|\xi|^{2k}.\end{equation}
To check this, we can argue by induction over~$k$. When~$k=0$, the claim is obvious, since both sides of~\eqref{BERLEMMAUT-2COMBA} equal to~$1$, and we can take~$a_0:=1$.

Also, when~$k=1$, we can take~$a_\alpha:=1$ for all~$\alpha\in\N^n$ with~$|\alpha|=1$ and arrive at
$$ \sum_{{\alpha\in\N^n}\atop{|\alpha|=1}}a_\alpha \,(i \xi)^\alpha(-i \xi)^\alpha=
\sum_{{\alpha\in\N^n}\atop{|\alpha|=1}}\xi^{2\alpha}=
\sum_{j=1}^n \xi_j^{2}
=|\xi|^{2},$$
which is the desired result in this case.

Let us now suppose that~\eqref{BERLEMMAUT-2COMBA} holds true for~$k\ge1$ and let us establish it for~$k+1$.
To this end, to stress their dependence on~$k$, we denote the coefficients as~$a_{\alpha,k}$ and
we use the inductive hypothesis to calculate that
\begin{eqnarray*}&&
|\xi|^{2(k+1)}=|\xi|^2|\xi|^{2k}=\left(\sum_{j=1}^n\xi_j^2\right)\left(
\sum_{{\alpha\in\N^n}\atop{|\alpha|=k}}a_{\alpha,k} (i \xi)^\alpha(-i \xi)^\alpha\right)\\&&\qquad=
\sum_{{{\alpha\in\N^n}\atop{|\alpha|=k}}\atop{1\le j\le n}} a_{\alpha,k}\, \xi^{2\alpha}\xi_j^2=\sum_{{{\alpha\in\N^n}\atop{|\alpha|=k}}\atop{1\le j\le n}} a_{\alpha,k}\, \xi^{2(\alpha+e_j)}=\sum_{{{\beta\in\N^n}\atop{|\beta|=k+1}}\atop{1\le j\le n}} a_{\beta-e_j,k}\, \xi^{2\beta}\\&&\qquad
=\sum_{{{\beta\in\N^n}\atop{|\beta|=k+1}}} a_{\beta,k+1}\, \xi^{2\beta}=\sum_{{{\beta\in\N^n}\atop{|\beta|=k+1}}} a_{\beta,k+1}(i\xi)^{\beta}(-i\xi)^{\beta}
,\end{eqnarray*}
where
$$ a_{\beta,k+1}:=\sum_{j=1}^n a_{\beta-e_j,k}.$$
Notice that
$$ |a_{\beta,k+1}|\le n\sup_{{\gamma\in\N^n}\atop{|\gamma|=k}}| a_{\beta-e_j,k}|\le n\,n^{k}\le n^{k+1}.$$
This completes the inductive step and provides the proof of~\eqref{BERLEMMAUT-2COMBA}.

Therefore, if
$$ \Theta_\alpha:={\mathcal{F}}^{-1}\left(\frac{a_\alpha\,(i \xi)^\alpha}{(2\pi i)^k|\xi|^{2k}}\theta(\xi)\right),$$
we obtain from~\eqref{BERLEMMAUT-2COMBA0} and~\eqref{BERLEMMAUT-2COMBA} that
\begin{eqnarray*}&&
{\mathcal{F}}\left(\sum_{{\alpha\in\N^n}\atop{|\alpha|=k}}\Theta_\alpha*Dv^\alpha v(x)\right)(\xi)=
\sum_{{\alpha\in\N^n}\atop{|\alpha|=k}}\widehat\Theta_\alpha(\xi) {\mathcal{F}}(D^\alpha v)(\xi)
=\sum_{{\alpha\in\N^n}\atop{|\alpha|=k}}\frac{a_\alpha\,(i \xi)^\alpha}{(2\pi i)^k|\xi|^{2k}}
\,(-2\pi i\xi)^\alpha\,\widehat v(\xi)\\&&\qquad
=\sum_{{\alpha\in\N^n}\atop{|\alpha|=k}}\frac{a_\alpha\,(i \xi)^\alpha (-i\xi)^\alpha}{|\xi|^{2k}}
\,\widehat v(\xi) =\widehat v(\xi)
\end{eqnarray*}
and accordingly
$$ \sum_{{\alpha\in\N^n}\atop{|\alpha|=k}}\Theta_\alpha*D^\alpha v(x)=v(x).$$
This and Young's Convolution Inequality lead to
\begin{equation}\label{BERLEMMAUT-2COMBA56} \begin{split}&\|v\|_{L^a(\R^n)}\le
\sum_{{\alpha\in\N^n}\atop{|\alpha|=k}}\|\Theta_\alpha*D^\alpha v\|_{L^a(\R^n)}\le
\sum_{{\alpha\in\N^n}\atop{|\alpha|=k}}\|\Theta_\alpha\|_{L^1(\R^n)}\|D^\alpha v\|_{L^a(\R^n)}\\&\qquad\le
Ck^{n-1}\sup_{{\alpha\in\N^n}\atop{|\alpha|=k}}\Big(
\|\Theta_\alpha\|_{L^1(\R^n)}
\|D^\alpha v\|_{L^a(\R^n)}\Big)\le C^k\sup_{{\alpha\in\N^n}\atop{|\alpha|=k}}
\Big(\|\Theta_\alpha\|_{L^1(\R^n)}\|D^\alpha v\|_{L^a(\R^n)}\Big)
.\end{split}\end{equation}

We also remark that, using the notation in~\eqref{USHNDcTRnKrijdf9ijN9ijn3edft76yujikmnbQ4567AW},
\begin{eqnarray*}
\|\Theta_\alpha\|_{L^1(\R^n)}&=&\int_{\R^n}\left|\int_{B_{C_2}\setminus B_{C_1}}
\frac{a_\alpha\,(i \xi)^\alpha}{(2\pi i)^k|\xi|^{2k}}\theta(\xi)\,e^{2\pi ix\cdot\xi}\,d\xi\right|\,dx
\\&\le&n^k\iint_{B_1\times(B_{C_2}\setminus B_{C_1})}\frac{|\theta(\xi)|}{|\xi|^k}\,dx\,d\xi+
\int_{\R^n\setminus B_1}\left|\int_{B_{C_2}\setminus B_{C_1}}
\frac{a_\alpha\,\xi^\alpha}{|\xi|^{2k}}\theta(\xi)\,e^{2\pi ix\cdot\xi}\,d\xi\right|\,dx\\&\le&C^k+\sum_{m=1}^n
\int_{{\mathcal{D}}_m}\left|\int_{B_{C_2}\setminus B_{C_1}}
\frac{a_\alpha\,\xi^\alpha}{|\xi|^{2k}}\theta(\xi)\,\partial_{\xi_m}^{n+1}\left(\frac{e^{2\pi ix\cdot\xi}}{(2\pi ix_m)^{n+1}}\right)\,d\xi\right|\,dx\\&=&C^k+\sum_{m=1}^n
\int_{{\mathcal{D}}_m}\left|\int_{B_{C_2}\setminus B_{C_1}}
\partial_{\xi_m}^{n+1}\left(\frac{a_\alpha\,\xi^\alpha}{|\xi|^{2k}}\theta(\xi)\right)\,
\frac{e^{2\pi ix\cdot\xi}}{(2\pi ix_m)^{n+1}}\,d\xi\right|\,dx\\&\le&C^k+C
\iint_{(\R^n\setminus B_1)\times(B_{C_2}\setminus B_{C_1})}
\left|\partial_{\xi_m}^{n+1}\left(\frac{a_\alpha\,\xi^\alpha}{|\xi|^{2k}}\theta(\xi)\right)\right|\,\frac{dx\,d\xi}{|x|^{n+1}}
\\&\le&C^k+C^k
\int_{B_{C_2}\setminus B_{C_1}}
\left|\partial_{\xi_m}^{n+1}\left(\frac{\xi^\alpha}{|\xi|^{2k}}\theta(\xi)\right)\right|\,d\xi\\&\le&C^k,
\end{eqnarray*}
up to freely renaming~$C$ (and note that~$C$ can be taken as independent of~$\alpha$ and~$k$).

{F}rom this and~\eqref{BERLEMMAUT-2COMBA56} we obtain the first inequality in~\eqref{BERLEMMAUT-2}
and the proof of the desired result is thereby complete.
\end{proof}

For our purposes, Lemma~\ref{BERLEMMAUT} is particularly useful when applied to the dyadic partition of unity produced in Lemma~\ref{0pirj09365-4g4eZXCHJKLRFG-544-NO1-LL}:

\begin{corollary} For all~$p\ge1$,~$j\ge1$ and~$k\in\N$, we have that
\begin{equation}\label{BERLEMMAUT-3} \frac{2^{kj}}{C^k} \|\check\varphi_j*u\|_{L^p(\R^n)}\le
\sup_{{\alpha\in\N^n}\atop{|\alpha|=k}}\|D^\alpha (\check\varphi_j*u)\|_{L^p(\R^n)}\le C^k 2^{kj} \|\check\varphi_j*u\|_{L^p(\R^n)},\end{equation}
where~$C\ge1$ depends only on~$n$.

Moreover,
\begin{equation}\label{BERLEMMAUT-4}
\|\check\varphi_j*u\|_{L^p(\R^n)}\le\frac{C^{k}}{2^{kj}}\sup_{{\alpha\in\N^n}\atop{|\alpha|=k}}\|D^\alpha u\|_{L^p(\R^n)},\end{equation}
where~$C\ge1$ depends only on~$n$ (and on~$\varphi_1$).
\end{corollary}

\begin{proof}
The gist of the proof is to apply Lemma~\ref{BERLEMMAUT} to the function~$v:=\check\varphi_j*u$, noticing that~$\widehat v=\varphi_j\widehat u$.
In this way, recalling Lemma~\ref{0pirj09365-4g4eZXCHJKLRFG-544-NO1-LL}, we see that
the support of~$\widehat v$ is contained in~$B_{2^{j+1}}\setminus B_{2^{j-1}}$.

On this account, we can employ~\eqref{BERLEMMAUT-2} with~$C_1:=1/2$,~$C_2:=2$,~$\lambda:=2^j$ and~$a:=p$, obtaining the desired result in~\eqref{BERLEMMAUT-3}.

Furthermore, using Young's Convolution Inequality,~\eqref{0pirj09365-4g4eZXCHJKLRFG-544-NO1-LL-eq4-09-BIS} and~\eqref{BERLEMMAUT-3},
\begin{eqnarray*}&& \|\check\varphi_j*u\|_{L^p(\R^n)}\le\frac{C^k}{2^{kj}}
\sup_{{\alpha\in\N^n}\atop{|\alpha|=k}}\|D^\alpha (\check\varphi_j*u)\|_{L^p(\R^n)}
=\frac{C^k}{2^{kj}}\sup_{{\alpha\in\N^n}\atop{|\alpha|=k}}\|\check\varphi_j*D^\alpha u\|_{L^p(\R^n)}\\&&\qquad
\le\frac{C^k}{2^{kj}}\sup_{{\alpha\in\N^n}\atop{|\alpha|=k}}\|D^\alpha u\|_{L^p(\R^n)}\|\check\varphi_j\|_{L^1(\R^n)}=\frac{C^k\|\check\varphi_1\|_{L^1(\R^n)}}{2^{kj}}\sup_{{\alpha\in\N^n}\atop{|\alpha|=k}}\|D^\alpha u\|_{L^p(\R^n)},\end{eqnarray*}
which is~\eqref{BERLEMMAUT-4} up to renaming~$C$.
\end{proof}

\begin{corollary}\label{LE637:l}
We have that~$C^\infty_c(\R^n)\subseteq B^{s,p,q}(\R^n)$.

More precisely, we have that if~$k\in\N\cap(s,+\infty)$ then~$W^{k,p}(\R^n)\subseteq B^{s,p,q}(\R^n)$,
with continuous embedding.
\end{corollary}

\begin{proof} Let~$u\in W^{k,p}(\R^n)$ with~$k\in\N\cap(s,+\infty)$.
By Young's Convolution Inequality we have that
\begin{equation*}
\|\check{\varphi}_0*u\|_{L^p(\R^n)}\le \|\check{\varphi}_0\|_{L^1(\R^n)}\,\|u\|_{L^p(\R^n)}\le C\|u\|_{L^p(\R^n)}.
\end{equation*}
Also, by~\eqref{BERLEMMAUT-4}, for all~$j\ge1$,
$$\|\check\varphi_j*u\|_{L^p(\R^n)}\le\frac{C}{2^{kj}} \|u\|_{W^{k,p}(\R^n)},$$
with~$C>0$ now depending also on~$k$.

As a consequence, using the Besov norm in~\eqref{BENOGIU}, allowing~$C$ to depend on~$q$ as well, and possibly renaming~$C$ at each stage of the calculation,
\begin{eqnarray*}\|u\|_{B^{s,p,q}(\R^n)}^q\le
C\|u\|_{L^p(\R^n)}^q+
\sum_{j=1}^{+\infty}2^{jq(s-k)} \|u\|_{W^{k,p}(\R^n)}^q\le C\|u\|_{W^{k,p}(\R^n)}^q,\end{eqnarray*}
as desired.
\end{proof}

Functions in Besov spaces can be nicely approximated by functions in~$C^\infty_c(\R^n)$, as the next result points out:

\begin{lemma}\label{0oijvn5TSndgrMSKdmfKTGK3m5tgDksP234-LEMM}
For all~$s>0$ and~$p$,~$q\in[1,+\infty)$, we have that~$C^\infty_c(\R^n)$ is dense in~$B^{s,p,q}(\R^n)$.
\end{lemma}

\begin{proof} First of all, we show that, for all~$k\in\N$,
\begin{equation}\label{0oihdjwe9ryihYCokergIImsd76SdtVilpDGBMKACBNAMS034irjtnoiuhg27}
{\mbox{$C^\infty(\R^n)\cap W^{k,p}(\R^n)$ is dense in~$B^{s,p,q}(\R^n)$.}}
\end{equation}
To this end, let~$u\in B^{s,p,q}(\R^n)$. Let also~$\eta_\e\in C^\infty_c(\R^n)$ be a standard mollifier (see e.g.~\cite[Chapter~9]{MR3381284}) and define~$u_\e:=u*\eta_\e$. In this way, since~$u\in L^p(\R^n)$,
we already know that~$u_\e\to u$ in~$L^p(\R^n)$ (see e.g.~\cite[Theorem~9.6]{MR3381284}) and therefore, by Young's Convolution Inequality, for all~$j\in\N$,
\begin{equation}\label{0oihdjwe9ryihYCokergIImsd76SdtVilpDGBMKACBNAMS034irjtnoiuhg2}
\lim_{\e\searrow0} \|\check\varphi_j* (u_\e-u)\|_{L^p(\R^n)}\le
\lim_{\e\searrow0} \|\check\varphi_j\|_{L^1(\R^n)}\|u_\e-u\|_{L^p(\R^n)}=0.
\end{equation}
Young's Convolution Inequality provides in addition that, for every~$\e>0$,
\begin{equation}\label{0oihdjwe9ryihYCokergIImsd76SdtVilpDGBMKACBNAMS034irjtnoiuhg}
\|\check\varphi_j* u_\e\|_{L^p(\R^n)}=\|(\check\varphi_j* u)*\eta_\e\|_{L^p(\R^n)}\le
\|\check\varphi_j* u\|_{L^p(\R^n)}\|\eta_\e\|_{L^1(\R^n)}
=\| \check\varphi_j* u\|_{L^p(\R^n)}.
\end{equation}

We stress that not only~$u_\e\in C^\infty(\R^n)$ but also that~$u_\e\in W^{k,p}(\R^n)$ for all~$k\in\N$, because, by Young's Convolution Inequality, for every~$\alpha\in\N^n$,
$$\|D^\alpha u_\e\|_{L^p(\R^n)}=\|u*(D^\alpha \eta_\e)\|_{L^p(\R^n)}\le
\|u\|_{L^p(\R^n)}\|D^\alpha \eta_\e\|_{L^1(\R^n)}<+\infty.$$

Thus, given any~$\mu>0$, we recall the Besov norm in~\eqref{BENOGIU} and
use the assumption that~$u\in B^{s,p,q}(\R^n)$ to find~$J_\mu\in\N$ large enough such that
$$ \sum_{j=J_\mu}^{+\infty}2^{jsq}\|\check{\varphi}_j*u\|_{L^p(\R^n)}^q\le\mu.$$
This and~\eqref{0oihdjwe9ryihYCokergIImsd76SdtVilpDGBMKACBNAMS034irjtnoiuhg} give that, for all~$\e>0$,
\begin{equation*}\sum_{j=J_\mu}^{+\infty}2^{jsq}\|\check{\varphi}_j*u_\e\|_{L^p(\R^n)}^q\le\mu.
\end{equation*}
Therefore, by Triangle Inequality,
\begin{equation}\label{0oihdjwe9ryihYCokergIImsd76SdtVilpDGBMKACBNAMS034irjtnoiuhg22}\begin{split}&\sum_{j=J_\mu}^{+\infty}2^{jsq}\|\check{\varphi}_j*(u_\e-u)\|_{L^p(\R^n)}^q\le
\sum_{j=J_\mu}^{+\infty}2^{jsq}\Big( \|\check{\varphi}_j*u_\e\|_{L^p(\R^n)}+\|\check{\varphi}_j*u\|_{L^p(\R^n)}\Big)^q\\&\qquad\le2^q
\sum_{j=J_\mu}^{+\infty}2^{jsq}\Big( \|\check{\varphi}_j*u_\e\|_{L^p(\R^n)}^q+\|\check{\varphi}_j*u\|_{L^p(\R^n)}^q\Big)\le2^{q+1}\mu.
\end{split}
\end{equation}

Moreover, by~\eqref{0oihdjwe9ryihYCokergIImsd76SdtVilpDGBMKACBNAMS034irjtnoiuhg2},
$$ \lim_{\e\searrow0}\sum_{j=0}^{J_\mu-1}2^{jsq}\|\check{\varphi}_j*(u_\e-u)\|_{L^p(\R^n)}^q=0.$$
On this account, we can find~$\e_\mu>0$ such that, for all~$\e\in(0,\e_\mu)$,
$$ \sum_{j=0}^{J_\mu-1}2^{jsq}\|\check{\varphi}_j*(u_\e-u)\|_{L^p(\R^n)}^q\le 2^{q+1}\mu.$$
The latter inequality and~\eqref{0oihdjwe9ryihYCokergIImsd76SdtVilpDGBMKACBNAMS034irjtnoiuhg22} yield that, for all~$\e\in(0,\e_\mu)$,
$$ \|u_\e-u\|_{B^{s,p,q}(\R^n)}^q=\sum_{j=0}^{+\infty}2^{jsq}\|\check{\varphi}_j*(u_\e-u)\|_{L^p(\R^n)}^q\le2^{q+2}\mu,$$
which, since~$\mu$ is arbitrary, completes the proof of~\eqref{0oihdjwe9ryihYCokergIImsd76SdtVilpDGBMKACBNAMS034irjtnoiuhg27}.

Now we consider~$u\in B^{s,p,q}(\R^n)$ and~$\mu>0$. We take~$k\in\N\cap (s,s+1]$ use~\eqref{0oihdjwe9ryihYCokergIImsd76SdtVilpDGBMKACBNAMS034irjtnoiuhg27} to find~$v\in C^\infty(\R^n)\cap W^{k,p}(\R^n)$ such that
\begin{equation}\label{0oihdjwe9ryihYCokergIImsd76SdtVilpDGBMKACBNAMS034irjtnoiuhg27-102eo3rikjgh-0304b}\|u-v\|_{B^{s,p,q}(\R^n)}\le\mu.\end{equation}
We pick~$\tau\in C^\infty_c(B_2,[0,1])$ with~$\tau=1$ in~$B_1$ and, for all~$R\ge1$, let~$\tau_R(x):=\tau\left(\frac{x}R\right)$.
We define~$u_R:=\tau_R v$ and note that~$u_R\in C^\infty_c(\R^n)$.

Let also~$w_R:=v-u_R=(1-\tau_R )v$. We see that, for all~$j\ge1$,
\begin{equation}\label{0oihdjwe9ryihYCokergIImsd76SdtVilpDGBMKACBNAMS034irjtnoiuhg27-102eo3rikjgh-0304}
\|\check\varphi_j* (v-u_R)\|_{L^p(\R^n)}=\|\check\varphi_j *w_R\|_{L^p(\R^n)}
\le\frac{C}{2^{kj}}\sup_{{\alpha\in\N^n}\atop{|\alpha|=k}}\|D^\alpha w_R\|_{L^p(\R^n)},
\end{equation}
due to~\eqref{BERLEMMAUT-4}.

Besides, using the Leibniz Product Rule, for all~$\alpha\in\R^n$ with~$|\alpha|=k$ we have that\footnote{Here,~$\beta\le\alpha$
(or, equivalently, $\alpha\ge\beta$)
means that if~$\alpha=(\alpha_1,\dots,\alpha_n)$ and~$\beta=(\beta_1,\dots,\beta_n)$ then~$\beta_i\le\alpha_i$
for all~$i\in\{1,\dots,n\}$.

Also, as customary, if~$\alpha$, $\beta\in\N$, with~$\alpha\ge\beta$, the binomial coefficient is defined by
$${ {\alpha}\choose{\beta}}:=\frac{\alpha!}{\beta!\,(\alpha-\beta)!}$$
and, if~$\alpha=(\alpha_1,\dots,\alpha_n)$ and~$\beta=(\beta_1,\dots,\beta_n)\in\N^n$, with~$\alpha\ge\beta$, by
$$ {{\alpha}\choose{\beta}}:=\prod_{i=1}^n {{\alpha_i}\choose{\beta_i}}.$$}
\begin{eqnarray*}
|D^\alpha w_R(x)|&=&\left|\sum_{{\beta\in\N^n}\atop{\beta \leq \alpha}}{\alpha \choose \beta }
D^{\beta }(1-\tau_R(x) ) D^{\alpha -\beta}v(x)\right|\\
&\le&(1-\tau_R(x) ) |D^{\alpha}v(x)|+
\sum_{{\beta\in\N^n}\atop{0\ne\beta \leq \alpha}}{\alpha \choose \beta }|D^{\beta }\tau_R|\,| D^{\alpha -\beta}v(x)|\\&\le&C\chi_{\R^n\setminus B_R}(x) |D^\alpha v(x)|+C\chi_{B_{2R}\setminus B_R}(x)
\sum_{{\beta\in\N^n}\atop{0\ne\beta \leq \alpha}}{\alpha \choose \beta }R^{-|\beta|}\,| D^{\alpha -\beta}v(x)|
\\&\le& C\chi_{\R^n\setminus B_R}(x)\sum_{{\beta\in\N^n}\atop{|\beta|\le k}}|D^\beta v(x)|
\end{eqnarray*}
and therefore
$$ \lim_{R\to+\infty}\|D^\alpha w_R\|_{L^p(\R^n)}\le C\lim_{R\to+\infty}
\sum_{{\beta\in\N^n}\atop{|\beta|\le k}}\|D^\beta v\|_{L^p(\R^n\setminus B_R)}=0.$$
Notice that here we are using the fact that~$v\in W^{k,p}(\R^n)$. Thus, we fix~$R>0$ sufficiently large such that
$$ \|D^\alpha w_R\|_{L^p(\R^n)}\le\mu$$
and that
$$ \|\check{\varphi}_0*w_R\|_{L^p(\R^n)}\le
\|\check{\varphi}_0\|_{L^1(\R^n)}\|w_R\|_{L^p(\R^n)}\le\|\check{\varphi}_0\|_{L^1(\R^n)}\|v\|_{L^p(\R^n\setminus B_R)}\le\mu.$$

On the account of these observations and~\eqref{0oihdjwe9ryihYCokergIImsd76SdtVilpDGBMKACBNAMS034irjtnoiuhg27-102eo3rikjgh-0304}, we deduce that
\begin{eqnarray*}\|v-u_R\|_{B^{s,p,q}(\R^n)}^q&=&
\sum_{j=0}^{+\infty}2^{jsq}\|\check{\varphi}_j*w_R\|_{L^p(\R^n)}^q\\&\le&
\mu^q+C
\sum_{j=1}^{+\infty}2^{jq(s-k)}
\sup_{{\alpha\in\N^n}\atop{|\alpha|=k}}\|D^\alpha w_R\|_{L^p(\R^n)}^q\\&\le&
\mu^q+C\mu^q
\sum_{j=1}^{+\infty}2^{jq(s-k)}\\ &\le&C\mu^q
.\end{eqnarray*}
Therefore, in light of~\eqref{0oihdjwe9ryihYCokergIImsd76SdtVilpDGBMKACBNAMS034irjtnoiuhg27-102eo3rikjgh-0304b},
we conclude that~$\|u-u_R\|_{B^{s,p,q}(\R^n)}\le C\mu$, up to renaming~$C$, and, since~$\mu$ is arbitrary,
the proof of Lemma~\ref{0oijvn5TSndgrMSKdmfKTGK3m5tgDksP234-LEMM} is complete.
\end{proof}

Let us now further clarify the role of the partition of unity in the Besov spaces setting.
Roughly speaking, by formally taking the Fourier Transform of~\eqref{0pirj09365-4g4eZXCHJKLRFG-544-NO1-LL-eq1},
we may expect that
\begin{equation}\label{LEHDNfcitjn034ty} \sum_{j=0}^{+\infty}\check\varphi_j=
{\mathcal{F}}\left(\sum_{j=0}^{+\infty}\varphi_j\right)={\mathcal{F}}(1)=\delta,\end{equation}
where~$\delta$ is the Dirac Delta distribution centred at the origin, but one may be suspicious
about switching the Fourier Transform and an infinite series without precise hypotheses.
Thus, to formalize this idea, we present the following result:

\begin{lemma}\label{LEHDNfcitjn034ty-Leapp}
If~$p\in[1,+\infty)$ and~$u\in L^p(\R^n)$, then
$$ \sum_{j=0}^{k}\check\varphi_j*u \,{\mbox{ converges to~$u$ in~$L^p(\R^n)$ as~$k\to+\infty$.}}$$
\end{lemma}

\begin{proof} First of all, we notice that it suffices to prove the desired claim when~$u\in C^\infty_c(\R^n)$.
Indeed, if~$u\in L^p(\R^n)$ and~$\e>0$,
we take~$u_\e\in C^\infty_c(\R^n)$ such that~$\|u-u_\e\|_{L^p(\R^n)}\le\e$.
Also, if the desired result holds true in~$C^\infty_c(\R^n)$ we can pick~$k_\e\in\N$ large enough such that whenever~$k\ge k_\e$ it holds that
$$ \left\| u_\e-\sum_{j=0}^{k}\check\varphi_j*u_\e\right\|_{L^p(\R^n)}\le\e.$$
In this way, it follows that
$$ \left\| u-\sum_{j=0}^{k}\check\varphi_j*u_\e\right\|_{L^p(\R^n)}\le
\|u-u_\e\|_{L^p(\R^n)}+\left\| u_\e-\sum_{j=0}^{k}\check\varphi_j*u_\e\right\|_{L^p(\R^n)}\le2\e,$$
giving the desired result for~$u\in L^p(\R^n)$.

Hence, owing to this observation, we now focus on proving the desired result when~$u\in C^\infty_c(\R^n)$.
For this, recalling the notation in~\eqref{UTILTAYMAA8ikjf-x589}, we have that$$\check \Phi_k=\delta-\sum_{j=0}^k\check\varphi_j,$$ where~$\delta$ is the Dirac Delta Function at the origin,
and note that this computation is legitimate since, differently from~\eqref{LEHDNfcitjn034ty},
here we only deal with a finite sum.
We also stress that~$\Phi_k\widehat u$ is in the Schwartz space of smooth and rapidly decreasing functions,
hence so is~$\check\Phi_k* u={\mathcal{F}}^{-1}(\Phi_k\widehat u)$.

Therefore, 
\begin{equation*}
u-\sum_{j=0}^{k}\check\varphi_j*u=\left(\delta-\sum_{j=0}^k\check\varphi_j\right)*u=\check \Phi_k*u
\end{equation*}
and thus, to prove the desired result, we need to show that
\begin{equation}\label{22d3scfv3456tyhbEHDNfcitjn034ty023rutogh004yiuhjbbBNcfv}
\lim_{k\to+\infty}\|\check \Phi_k*u\|_{L^p(\R^n)}=0.
\end{equation}

To this end, let~$R>0$ be such that the support of~$u$ is contained in~$B_R$.
Then, recalling the definition of~${\mathcal{D}}_m$ in~\eqref{USHNDcTRnKrijdf9ijN9ijn3edft76yujikmnbQ4567AW},
\begin{eqnarray*}
|\check\Phi_k* u(x)|&=&
\left| \iint_{B_R\times\R^n} \Phi_k(\xi)\, u(y)\,e^{2\pi i(x-y)\cdot\xi}\,dy\,d\xi\right|\\&\le&
\sum_{m=1}^n
\left| \iint_{B_R\times{\mathcal{D}}_m} \Phi_k(\xi)\, u(y)\,e^{2\pi i(x-y)\cdot\xi}\,dy\,d\xi\right|\\
&=&\sum_{m=1}^n
\left| \iint_{B_R\times{\mathcal{D}}_m} \Phi_k(\xi)\, u(y)\,\partial^{n+1}_{y_r}\left(\frac{e^{2\pi i(x-y)\cdot\xi}}{(2\pi i \xi_r)^{n+1}}\right)\,dy\,d\xi\right|\\&=&\sum_{m=1}^n
\left| \iint_{B_R\times{\mathcal{D}}_m} \Phi_k(\xi)\, \partial^{n+1}_{y_r} u(y)\,\frac{e^{2\pi i(x-y)\cdot\xi}}{(2\pi i \xi_r)^{n+1}}\,dy\,d\xi\right|.
\end{eqnarray*}
In particular, for all~$x\in\R^n$,
\begin{equation}\label{s0oiujhbv09iuy6trfgvbgbUTILTAYMAA8ikjf-x5}
\begin{split}&
|\check\Phi_k* u(x)|\le C_u \sum_{m=1}^n
\iint_{B_R\times{\mathcal{D}}_m} |\Phi_k(\xi)|\frac{dy\,d\xi}{|\xi_r|^{n+1}}
\le C_u \sum_{m=1}^n
\int_{{\mathcal{D}}_m} |\Phi_k(\xi)|\frac{d\xi}{|\xi|^{n+1}}\\&\qquad\qquad
\le C_u 
\int_{\R^n\setminus B_{2^k}}\frac{d\xi}{|\xi|^{n+1}}
\le \frac{C_u}{2^k}.
\end{split}\end{equation}
Notice that~$C_u>0$ here is allowed to depend on~$u$.

Also, setting~$w:=\check\Phi_k* u$, we have that, for each~$m\in\{1,\dots,n\}$,
$$ \partial^{n+1}_{\xi_m} \widehat w(\xi)={\mathcal{F}}\Big( (-2\pi i x_m)^{n+1} w(x)\Big)(\xi)$$
and therefore, recalling the uniform derivative bound in~\eqref{UTILTAYMAA8ikjf-x5},
\begin{eqnarray*}&&  (2\pi |x_m|)^{n+1} |w(x)|=|(-2\pi i x_m)^{n+1} w(x)|=
\Big|{\mathcal{F}}^{-1}\Big(\partial^{n+1}_{\xi_m} \widehat w(\xi)\Big)(x)\Big|\\&&\qquad=
\Big|{\mathcal{F}}^{-1}\Big(\partial^{n+1}_{\xi_m} \big( 
\Phi_k(\xi)\widehat u(\xi)\big)\Big)(x)\Big|=
\left|\int_{\R^n} \partial^{n+1}_{\xi_m} \big( 
\Phi_k(\xi)\widehat u(\xi)\big) e^{2\pi ix\cdot\xi}\,d\xi\right|\\
&&\qquad\le\int_{\R^n} \Big|\partial^{n+1}_{\xi_m} \big( 
\Phi_k(\xi)\widehat u(\xi)\big)\Big|\,d\xi
\le C_u\int_{\R^n\setminus B_{2^k}} \sum_{|\alpha|\le n+1}
|D^{\alpha}\widehat u(\xi) |\,d\xi\\
&&\qquad\le C_u\int_{\R^n\setminus B_{2^k}} \frac{d\xi}{|\xi|^{n+1}}\le\frac{C_u}{2^k}.
\end{eqnarray*}
For this reason, and utilizing the notation in~\eqref{USHNDcTRnKrijdf9ijN9ijn3edft76yujikmnbQ4567AW}
and the estimate in~\eqref{s0oiujhbv09iuy6trfgvbgbUTILTAYMAA8ikjf-x5},
\begin{eqnarray*}&&
\|\check\Phi_k* u\|_{L^1(\R^n)}\le
\int_{\R^n\setminus B_1}|w(x)|\,dx+ \frac{C_u}{2^k}
\le\sum_{m=1}^n
\int_{{\mathcal{D}}_m}|w(x)|\,dx+ \frac{C_u}{2^k}
\\&&\qquad\le \frac{C_u}{2^k}\left(\sum_{m=1}^n\int_{{\mathcal{D}}_m}\frac{dx}{|x_m|^{n+1}}+ 1\right)
\le \frac{C_u}{2^k}\left(\sum_{m=1}^n\int_{\R^n\setminus B_1}\frac{dx}{|x|^{n+1}}+ 1\right)
\le\frac{C_u}{2^k}.
\end{eqnarray*}
This,~\eqref{USHNDcTRnKrijdf9ijN9ijn3edft76yujikmnb201o2uerjhfgVVXDIKe78} and~\eqref{s0oiujhbv09iuy6trfgvbgbUTILTAYMAA8ikjf-x5} entail that
\begin{equation*} \|\check\Phi_k* u\|_{L^p(\R^n)}\le\frac{C_u}{2^k},\end{equation*}
which gives~\eqref{22d3scfv3456tyhbEHDNfcitjn034ty023rutogh004yiuhjbbBNcfv}, as desired.
\end{proof}

\begin{corollary}\label{LEHDNfcitjn034ty-LeappCOR}
If~$p\in[1,+\infty)$ then, for any~$u\in L^p(\R^n)$,
\begin{equation}\label{COROLEEP2r34t5023opjgrh5tgbhuiuhgsdhj}
\|u\|_{L^p(\R^n)}\le\sum_{j=0}^{+\infty}\|\check\varphi_j*u\|_{L^p(\R^n)}.\end{equation}
\end{corollary}

\begin{proof}
Let~$\e>0$. By Lemma~\ref{LEHDNfcitjn034ty-Leapp}, there exists~$k_\e\in\N$ such that, for all~$k\ge k_\e$,
$$ \e\ge\left\|u-\sum_{j=0}^{k}\check\varphi_j*u\right\|_{L^p(\R^n)}\ge
\|u\|_{L^p(\R^n)}-\left\|\sum_{j=0}^{k}\check\varphi_j*u\right\|_{L^p(\R^n)}.$$
Therefore,
\begin{eqnarray*}&&
\|u\|_{L^p(\R^n)}\le\e+\lim_{k\to+\infty}\left\|\sum_{j=0}^{k}\check\varphi_j*u\right\|_{L^p(\R^n)}\\&&\qquad\le \e+\lim_{k\to+\infty}\sum_{j=0}^{k}\|\check\varphi_j*u\|_{L^p(\R^n)}=\e+\sum_{j=0}^{+\infty}\|\check\varphi_j*u\|_{L^p(\R^n)},
\end{eqnarray*}
which, since~$\e$ is arbitrary, leads to the desired result.
\end{proof}

A counterpart of the inclusion in Corollary~\ref{LE637:l} is given in the following result (where, as customary,
the Sobolev space~$W^{0,p}(\R^n)$ is simply~$L^p(\R^n)$).

\begin{corollary}\label{COROLEEP2r34t5}
For all~$s>0$ and~$p$,~$q\in[1,+\infty)$ we have that~$B^{s,p,q}(\R^n)\subseteq L^p(\R^n)$, with continuous embedding.

More precisely, we have that if either~$k\in\N\cap[0,s)$ or~$(k,q)=(s,1)$
then~$B^{s,p,q}(\R^n)\subseteq W^{k,p}(\R^n)$, with continuous embedding.
\end{corollary}

\begin{proof} First of all, we observe that, without loss of generality,
\begin{equation}\label{COROLEEP2r34t5-redq1}
{\mbox{we can assume that~$q=1$.}}
\end{equation}
Indeed, suppose that the desired result is true for~$q=1$, i.e.
that
\begin{equation}\label{COROLEEP2r34t5-redq2}
{\mbox{$B^{\sigma,p,1}(\R^n)\subseteq W^{\kappa,p}(\R^n)$, with continuous embedding, whenever~$\sigma>0$ and~$\kappa\in[0,\sigma]$.}}
\end{equation}
Let now~$s>0$ and~$k\in\N\cap[0,s)$ or~$(k,q)=(s,1)$.
If~$k=s$, then~$q=1$, therefore~\eqref{COROLEEP2r34t5-redq2}, used here with~$\sigma:=s$ and~$\kappa:=s=k$, gives that~$B^{s,p,q}(\R^n)=B^{s,p,1}(\R^n)\subseteq W^{k,p}(\R^n)$
with continuous embedding, as desired.

If instead~$k\in(0, s)$, then we can define~$\sigma:=\frac{s+k}{2}>0$ and~$a:=s-\sigma=\frac{s-k}{2}>0$
and use~\eqref{F58uuRS-2CALEREG} to see that~$B^{s,p,q}(\R^n)=B^{\sigma+a,p,q}(\R^n)\subseteq B^{\sigma,p,1}(\R^n)$ with continuous embedding. Thus, employing~\eqref{COROLEEP2r34t5-redq2} with~$\kappa:=k=\sigma-a\in(0,\sigma)$, we obtain that~$
B^{s,p,q}(\R^n)\subseteq B^{\sigma,p,1}(\R^n)\subseteq W^{k,p}(\R^n)$
with continuous embeddings, and we are done.

These observations show~\eqref{COROLEEP2r34t5-redq1}, hence we now suppose additionally that~$q=1$ and proceed with the proof of Corollary~\ref{COROLEEP2r34t5}. For this, let~$\alpha\in\N^n$ with~$|\alpha|\le k$
and~$u\in C^\infty_c(\R^n)$.

Then, by Corollary~\ref{LEHDNfcitjn034ty-LeappCOR},
\begin{equation}\label{LEHDNfcitjn034ty-LeappCOReqyi0} \|D^\alpha u\|_{L^p(\R^n)}\le\sum_{j=0}^{+\infty}\|D^\alpha(\check\varphi_j*u)\|_{L^p(\R^n)}.\end{equation}
We also utilize~\eqref{BERLEMMAUT-3} to see that, for all~$j\ge1$,
\begin{equation}\label{LEHDNfcitjn034ty-LeappCOReqyi1} \|D^\alpha(\check\varphi_j*u)\|_{L^p(\R^n)}\le2^{j(s-|\alpha|)}\|D^\alpha (\check\varphi_j*u)\|_{L^p(\R^n)}\le C 2^{js} \|\check\varphi_j*u\|_{L^p(\R^n)},\end{equation}
with~$C$ now possibly renamed and depending on~$k$ as well.

Besides, by~\eqref{BERLEMMAUT-1},
$$ \|D^\alpha (\check\varphi_0*u)\|_{L^p(\R^n)}\le C\|\check\varphi_0*u\|_{L^p(\R^n)},$$
showing that~\eqref{LEHDNfcitjn034ty-LeappCOReqyi1} is valid also when~$j=0$.

Therefore, using~\eqref{LEHDNfcitjn034ty-LeappCOReqyi0} and recalling the Besov space norm in~\eqref{BENOGIU},
\begin{eqnarray*}
\|D^\alpha u\|_{L^p(\R^n)}\le C\sum_{j=0}^{+\infty}2^{js} \|\check\varphi_j*u\|_{L^p(\R^n)}=C\|u\|_{B^{s,p,1}(\R^n)}.
\end{eqnarray*}
This inequality, obtained for~$u\in C^\infty_c(\R^n)$, extends to any~$u\in B^{s,p,1}(\R^n)$, due to the density result in Lemma~\ref{0oijvn5TSndgrMSKdmfKTGK3m5tgDksP234-LEMM} (and the fact that~$C$ is independent of~$u$).

Since this inequality holds true for all~$\alpha\in\N^n$ with~$|\alpha|\le k$, we obtain that~$\|u\|_{W^{k,p}(\R^n)}
\le C\|u\|_{B^{s,p,1}(\R^n)}$ and
the desired result is established.
\end{proof}

\begin{corollary} \label{DERIBEVS}
Let~$s>0$ and~$p,q\ge1$. Then,
$$ \frac{\| u\|_{B^{s+1,p,q}(\R^n)}}C\le\|u\|_{L^p(\R^n)}+\|\nabla u\|_{B^{s,p,q}(\R^n)}\le C\| u\|_{B^{s+1,p,q}(\R^n)},$$
with~$C\ge1$ depending only on~$n$,~$s$,~$p$ and~$q$.
\end{corollary}

\begin{proof} In light of~\eqref{BENOGIU} and~\eqref{BERLEMMAUT-3},
\begin{eqnarray*}
\|u\|_{B^{s+1,p,q}(\R^n)}^q&=&
\sum_{j=0}^{+\infty}2^{j(s+1)q}\|\check{\varphi}_j*u\|_{L^p(\R^n)}^q\\&\le&
\|\check{\varphi}_0*u\|_{L^p(\R^n)}^q
+C\sum_{j=1}^{+\infty}2^{jsq}\|\nabla(\check{\varphi}_j*u)\|_{L^p(\R^n)}^q
\\&\le&C\Big(\|u\|_{L^p(\R^n)}^q+\|\nabla u\|_{B^{s,p,q}(\R^n)}^q\Big),\end{eqnarray*}and the other inequality is similar (using also Corollary~\ref{COROLEEP2r34t5} to see that~$\|u\|_{L^p(\R^n)}\le C\| u\|_{B^{s+1,p,q}(\R^n)}$).
\end{proof}

Now, a natural structural property enjoyed by Besov spaces:

\begin{lemma} 
$B^{s,p,q}(\R^n)$ is a Banach space.
\end{lemma} 
\begin{proof}
We focus on the nondegeneracy and completeness proofs, since the other properties follow from the analogous ones of~$L^p(\R^n)$.

For the nondegeneracy, suppose that~$\|u\|_{B^{s,p,q}(\R^n)}=0$.
Then, for all~$j\in\N$, we have that~$\|\check\varphi_j*u\|_{L^p(\R^n)}=0$ and therefore~$\check\varphi_j*u(x)=0$ for a.e.~$x\in\R^n$.
This gives that~$\varphi_j \widehat u$ is the null distribution, and thus so is
$$ \sum_{j=0}^{+\infty}\varphi_j\widehat u=\widehat u,$$
due to~\eqref{0pirj09365-4g4eZXCHJKLRFG-544-NO1-LL-eq1}.
This says that~$\widehat u$ is null, and thus so is~$u$, proving the nondegeneracy of the Besov norm in~\eqref{BENOGIU}.

Now let us discuss the completeness property.
Let~$u_k$ be a Cauchy sequence in~$B^{s,p,q}(\R^n)$. 
By Corollary~\ref{COROLEEP2r34t5}, we have that~$u_k$ is also a Cauchy sequence in~$L^p(\R^n)$
and therefore there exists a function~$u$ such that~$u_k\to u$ in~$L^p(\R^n)$ as~$k\to+\infty$,
and a subsequence~$u_{k_\ell}$ such that~$u_{k_\ell}\to u$ a.e. in~$\R^n$ as~$\ell\to+\infty$.

Thus, for every~$j$,~$\ell\in\N$ and~$x\in\R^n$, using H\"older's Inequality with exponents~$p$ and~$\frac{p}{p-1}$,
\begin{eqnarray*}&&\lim_{m\to+\infty}
\left| \check\varphi_j*(u_{k_\ell}-u_{k_m})(x) - \check\varphi_j*(u_{k_\ell}-u)(x)\right|^p=\lim_{m\to+\infty}
\left| \check\varphi_j*(u-u_{k_m})(x)\right|^p\\
&&\qquad\le\lim_{m\to+\infty}\left(\int_{\R^n} |\check\varphi_j(x-y)|\,|(u-u_{k_m})(y)|\,dy\right)^p
\le \lim_{m\to+\infty} \|\check\varphi_j\|_{L^{\frac{p}{p-1}}(\R^n)}^p\,\|u-u_{k_m}\|_{L^p(\R^n)}^p=0,
\end{eqnarray*}
that is
$$ \lim_{m\to+\infty} \check\varphi_j*(u_{k_\ell}-u_{k_m})(x) = \check\varphi_j*(u_{k_\ell}-u)(x).$$
On this account, by Fatou's Lemma, for every~$j$,~$\ell\in\N$,
\begin{equation}\label{BENOGIU02o-eprjegk9j2erjXXasdf-323}
\begin{split}&
\liminf_{m\to+\infty}\|\check\varphi_j*(u_{k_\ell}-u_{k_m})\|_{L^p(\R^n)}^q=
\left(\liminf_{m\to+\infty}\int_{\R^n}|\check\varphi_j*(u_{k_\ell}-u_{k_m})(x)|^p\,dx \right)^{\frac{q}p}\\&\qquad\ge
\left(\int_{\R^n}\liminf_{m\to+\infty}|\check\varphi_j*(u_{k_\ell}-u_{k_m})(x)|^p\,dx \right)^{\frac{q}p}=
\left(\int_{\R^n}|\check\varphi_j*(u_{k_\ell}-u)(x)|^p\,dx \right)^{\frac{q}p}\\&\qquad=\|\check\varphi_j*(u_{k_\ell}-u)\|_{L^p(\R^n)}^q.
\end{split}\end{equation}

Now, recalling the Besov norm in~\eqref{BENOGIU},
given~$\e>0$, we use the Cauchy property of the sequence~$u_k$ (and thus of the subsequence~$u_{k_\ell}$) to find~$L_\e\in\N$ sufficiently large such that, if~$\ell$,~$m\ge L_\e$, then, for every~$N\in\N$,
\begin{eqnarray*}
\e^q\ge \|u_{k_\ell}-u_{k_m}\|_{B^{s,p,q}(\R^n)}^q\ge
\sum_{j=0}^{N}2^{jsq}\|\check{\varphi}_j*(u_{k_\ell}-u_{k_m})\|_{L^p(\R^n)}^q.
\end{eqnarray*}
Therefore, taking the limit as~$m\to+\infty$ according to~\eqref{BENOGIU02o-eprjegk9j2erjXXasdf-323},
\begin{eqnarray*}
\e^q\ge\sum_{j=0}^{N}2^{jsq}\|\check{\varphi}_j*(u_{k_\ell}-u)\|_{L^p(\R^n)}^q.
\end{eqnarray*}
Taking now the limit as~$N\to+\infty$, we conclude that
\begin{eqnarray*}
\e^q\ge\sum_{j=0}^{+\infty}2^{jsq}\|\check{\varphi}_j*(u_{k_\ell}-u)\|_{L^p(\R^n)}^q= \|u_{k_\ell}-u\|_{B^{s,p,q}(\R^n)}^q.
\end{eqnarray*}

This shows that~$u_{k_\ell}\to u$ as~$\ell\to+\infty$ in~$B^{s,p,q}(\R^n)$. 
To complete the proof of the desired result, we have to show that~$u_k\to u$ as~$k\to+\infty$ in~$B^{s,p,q}(\R^n)$. For this, pick~$\e>0$ and use the Cauchy property of~$u_k$ to find~$K_\e\in\N$ so large that, for all~$k$,~$h\ge K_\e$,
$$ \|u_k-u_h\|_{B^{s,p,q}(\R^n)}\le \e.$$
Let now choose~$\ell_\e\in\N$ so large that~$k_{\ell_\e}\ge K_\e$.
This gives that whenever~$\ell\ge\ell_\e$ we have that~$k_\ell\ge k_{\ell_\e}\ge K_\e$ and, as a consequence,
$$ \|u_k-u_{k_\ell}\|_{B^{s,p,q}(\R^n)}\le \e.$$
In particular,
$$ \|u_k-u\|_{B^{s,p,q}(\R^n)}=\lim_{\ell\to+\infty}\|u_k-u_{k_\ell}\|_{B^{s,p,q}(\R^n)}\le \e$$
and therefore~$u_k$ converges to~$u$ in~$B^{s,p,q}(\R^n)$, as desired.
\end{proof}

Following is a useful estimate to relate the Besov spaces and the Bessel potential spaces\footnote{Linking
Bessel potential spaces to Besov spaces is a key step to the regularity theory of nonlocal equations:
indeed, while Bessel potential spaces, as introduced in Section~\ref{SEC:BPSPA}, mainly describe the space of solutions of certain nonlocal equations (say,~$ (1-\Delta)^su=f$ in~$\R^n$), we have that Besov spaces
measure, to some extent, the regularity of a function (see e.g. Corollary~\ref{COROLEEP2r34t5}).} presented in Section~\ref{SEC:BPSPA}.
A pivotal role in this setting is played by the Mikhlin Multiplier Theorem, which we now briefly recall.

As customary, given two Banach spaces~$X$ and~$Y$, we will also denote by~$L(X,Y)$ the space of bounded linear operators from~$X$ to~$Y$.
Given~$m:\R^n\setminus\{0\}\to L(X,Y)$, we denote\footnote{The notation relating the symbol in the Fourier space
to the corresponding operator is not uniform in the literature. Here is a (largely incomplete)
table comparing different notations.

\begin{center}
\begin{tabular}{ |p{2cm}||p{4cm}|p{4cm}| }
 \hline
&  Fourier multiplier or symbol&(pseudo)differential operator\\
\hline
 \hline
This book &$ m$& $T_m$\\
\hline
\cite{MR2884718}& $p(\xi)$ & $P$\\
\hline
\cite{MR3170202}& $p(x,\xi)$ & $P$, $p(x,D_x)$, ${\operatorname{OP}}(p)$\\
\hline
\cite{MR2244530}& $L_0(\xi)$, $\sigma_{D_\omega}$ & $L_0(\partial)$,
$D_\omega$\\
 \hline
\cite{MR1996120}& $\sigma$&$T_\sigma$\\
 \hline
\cite{MR1269107}& $\sigma_Q$&$Q$\\
\hline
\cite{MR1385196}& $a(x,\xi)$, $p$ & $A(x,D_x)$, ${\operatorname{Op}}(p)$\\
\hline
\cite{MR2453959}& $p(\xi)$ & $P(D)$, ${\operatorname{Op}}(p)$\\
 \hline
\cite{MR4499500}& $a(x,\xi)$ &  $T_a$\\
 \hline
\cite{MR1996773}& $p(x,\xi)$ & $P(x,\partial)$\\
 \hline
 \cite{RUZPR14}& $a(x,\xi)$ & $T_a$, $a(x,D)$\\
 \hline
\cite{MR2567604} & $a$ & ${\operatorname{Op}}(a)$, $a(x, D)$, $T_a$\\
 \hline
 \cite{MR1211419}& $a(x,\xi)$ & $a(x,D)$\\
 \hline
 \cite{MR1852334}& $\sigma_A$ & $A$\\
 \hline
 \cite{MR2587583}& $\sigma$ & $T_\sigma$\\
 \hline  
 \cite{MR597144}& $P(\xi)$ & $P(D)$\\
 \hline 
\cite{MR4436039}& $a$ & $A$, ${\operatorname{Op}}\,a$\\
 \hline
\end{tabular}\end{center}

The notation adopted here tries to be streamlined from the typographical point of view
and avoid the use of multiple letters
(recalling Feynman's comment~\cite{FEYNM}:
``While I was doing all this trigonometry, I didn't like the symbols for sine, cosine, tangent, and so on. To me, $\sin \phi$ looked like $s$ times $i$ times $n$ times $\phi$!'').}
by~$T_m$ the operator~${\mathcal{F}}^{-1}\circ m\circ{\mathcal{F}}$, that is
\begin{equation}\label{DEFTMMIH} T_m g:={\mathcal{F}}^{-1}\Big( m \widehat{g}\Big).\end{equation}
Roughly speaking,~$m$ acts as a multiplication operator in frequency space
and~$T_m$ is the corresponding operator in the usual space.

Though more general statements are available in the literature, the kind of result that we use here goes as follows:

\begin{theorem}[Mikhlin Multiplier Theorem]\label{Mikhlin Multiplier Theorem}
Let~$p\in(1,+\infty)$.
Let~$X$,~$Y$ be Hilbert spaces.

Assume that~$m\in C^{n+2}(\R^n\setminus\{0\}, \,L(X,Y))$, with
\begin{equation}\label{Mikhlin Multiplier Theorem-ASSUN}
A:= \sup_{{\xi\in\R^n\setminus\{0\}}\atop{{\alpha\in\N^n}\atop{|\alpha|\le n+2}}}|\xi|^{|\alpha|} \|D^\alpha m(\xi)\|_{L(X,Y)}<+\infty.
\end{equation}
Then,
$$ \| T_m \|_{L(L^p(\R^n,X),\,L^p(\R^n,Y))}\le CA,$$
for a suitable~$C>0$ depending only on~$n$,~$p$,~$X$, and~$Y$.\end{theorem}

We will not give the proof of this result here:
for more details (and proofs), see e.g.~\cite[Theorem~5.8]{MR2884718},
\cite[Proposition~4.2.14 and Theorem~5.3.18]{MR3617205}, \cite[Theorem~4.7.2]{MR3930629}, or~\cite[Remark~3.26 and Theorem~3.27]{MR4249415}.

With this, we can now state and prove a careful estimate\footnote{Estimates of this type somewhat aim at extending methods and results about functions in~$L^2(\R^n)$ (or spaces modeled on~$L^2(\R^n)$) to~$L^p (\R^n)$ and they sit in the so-called \index{Littlewood-Paley Theory} Littlewood-Paley Theory.
The link between spaces with integration indexes~$p$ and~$2$ becomes also more apparent in~\eqref{MikhlinEQWF} if one writes
$$ \left(
\sum_{j=0}^{+\infty} 2^{2js}|\check\varphi_j*u|^2\right)^{\frac12}=
\left\|\big\{
2^{js}|\check\varphi_j*u|\big\}_{j\in\N}
\right\|_{\ell^2(\N)}
.$$

To show the power of the estimates in~\eqref{MikhlinEQWF}, we provide in Appendix~\ref{APPE:INTERPLEMM} an interpolation
inequality in the Bessel potential spaces setting, with a direct proof relying on Theorem~\ref{0pirj09365-4g4eZXCHJKLRFG-544-NO3-COR-0oerkCC}
(other approaches are possible through complex interpolation theory).}
relating Besov and Bessel potential spaces:

\begin{theorem}\label{0pirj09365-4g4eZXCHJKLRFG-544-NO3-COR-0oerkCC}
For all~$s>0$ and~$p\ge1$ there exists a constant~$C\ge1$, depending only on~$n$,~$s$ and~$p$, such that, for all locally integrable functions~$u$,
\begin{equation}\label{MikhlinEQWF}
\frac{ \|u\|_{{\mathcal{L}}^p_{s} (\R^n)}}{C}\le\left\|\,\left(
\sum_{j=0}^{+\infty} 2^{2js}|\check\varphi_j*u|^2\right)^{\frac12} \,\right\|_{L^p(\R^n)}\le C\|u\|_{{\mathcal{L}}^p_{s}(\R^n)}.
\end{equation}
\end{theorem}

\begin{proof} We recall that, by~\eqref{ERFGHJN6789-09tftd90u8yhgiug8erhISB},
the norm on the Bessel potential spaces is given by
$$ \|u\|_{ {\mathcal{L}}^p_{s}(\R^n)}=\| f_u\|_{L^p(\R^n)}, \qquad{\mbox{with}}\qquad u={\mathcal{B}}^{(s/2)}*f_u.$$
It is also convenient to use the short notation
$$ \langle\xi\rangle:=\sqrt{1+4\pi^2|\xi|^2},$$
so that, by Lemma~\ref{PRODpoikjhr3-2:le-7},
$$ \widehat{\mathcal{B}}^{(s/2)}=\langle\xi\rangle^{-s}.$$

Given~$\xi\in\R^n$, we now define the multiplier
\begin{eqnarray*}
m(\xi):\C&\rightarrow&\ell^2(\N)\\
z&\longmapsto& \Big\{ 2^{js}\varphi_j(\xi)\langle\xi\rangle^{-s} z\Big\}_{j\in\N}
\end{eqnarray*}
and we claim that, for all~$N\in\N$,
\begin{equation}\label{LASTIPERUSAMI1}
\sup_{{\xi\in\R^n\setminus\{0\}}\atop{{\alpha\in\N^n}\atop{|\alpha|\le N}}}|\xi|^{|\alpha|} \|D^\alpha m(\xi)\|_{L(\C,\ell^2(\N))}<+\infty.
\end{equation}
Let us postpone the proof of~\eqref{LASTIPERUSAMI1}.

Similarly, we define the multiplier
\begin{eqnarray*}
\mu(\xi):\ell^2(\N)&\rightarrow&\C\\
\big\{a_j\big\}_{j\in\N}&\longmapsto& \sum_{j=0}^{+\infty} 2^{-js}\psi_j(\xi)\langle\xi\rangle^{s}\,a_j
\end{eqnarray*}
where~$\psi_j$ is as in~\eqref{UTILTAYMAA8ikjf-x5-LAPSIN},
and we claim that, for all~$N\in\N$,
\begin{equation}\label{LASTIPERUSAMI2}
 \sup_{{\xi\in\R^n\setminus\{0\}}\atop{{\alpha\in\N^n}\atop{|\alpha|\le N}}}|\xi|^{|\alpha|} \|D^\alpha \mu(\xi)\|_{L(\ell^2(\N),\C)}<+\infty.
\end{equation}
Let us postpone the proof of~\eqref{LASTIPERUSAMI2} as well. Instead, let us see how~\eqref{LASTIPERUSAMI1}
and~\eqref{LASTIPERUSAMI2} yield the desired result owing to Theorem~\ref{Mikhlin Multiplier Theorem}.

For this aim, we observe that~\eqref{LASTIPERUSAMI1} entails assumption~\eqref{Mikhlin Multiplier Theorem-ASSUN}
with~$X:=\C$ and~$Y:=\ell^2(\N)$, hence Theorem~\ref{Mikhlin Multiplier Theorem} yields that
$$ \| T_m \|_{L(L^p(\R^n,\C),\,L^p(\R^n,\ell^2(\N)))}\le C,$$
up to renaming~$C$ (and writing here explicitly~$L^p(\R^n,\C)$, instead of the short notation~$L^p(\R^n)$,
for consistency with the general notation in Theorem~\ref{Mikhlin Multiplier Theorem}).

As a result, recalling the definition of~$T_m$ in~\eqref{DEFTMMIH}, for
functions~$g:\R^n\to\C$,
\begin{equation}\label{DEFTMMIH0owfjlg9ijrfvgb6yu}\begin{split}&
\left(\int_{\R^n}\left( \sum_{j=0}^{+\infty}\Big| {\mathcal{F}}^{-1}\Big( 2^{js}\varphi_j(\xi)\langle\xi\rangle^{-s} \widehat{g}(\xi)\Big)\Big|^2(x)
\right)^{\frac{p}2}\,dx\right)^{\frac1p}=
\left(\int_{\R^n}\left( \sum_{j=0}^{+\infty}\Big| {\mathcal{F}}^{-1}\Big( m_j \widehat{g}\Big)(x)\Big|^2\right)^{\frac{p}2}\,dx\right)^{\frac1p}\\&\qquad=
\left(\int_{\R^n} \Big\| {\mathcal{F}}^{-1}\Big( m \widehat{g}\Big)(x)\Big\|_{\ell^2(\N)}^p\,dx\right)^{\frac1p}=
\Big\| {\mathcal{F}}^{-1}\Big( m \widehat{g}\Big)\Big\|_{L^p(\R^n,\ell^2(\N))}
=\| T_m g\|_{L^p(\R^n,\ell^2(\N))}\\&\qquad\le
\| T_m \|_{L(L^p(\R^n,\C),\,L^p(\R^n,\ell^2(\N)))}\,\|g\|_{L^p(\R^n,\C)}\le C\|g\|_{L^p(\R^n,\C)}.\end{split}
\end{equation}
In particular, choosing~$g:=f_u$ we have that
$$\langle\xi\rangle^{-s} \widehat{g}=\widehat{\mathcal{B}}^{(s/2)} \widehat{f_u}=
{\mathcal{F}}\big( {\mathcal{B}}^{(s/2)}*f_u\big)={\mathcal{F}}(u)=\widehat u$$
and thus~\eqref{DEFTMMIH0owfjlg9ijrfvgb6yu} returns that
\begin{eqnarray*}C
\|u\|_{ {\mathcal{L}}^p_{s}(\R^n)}&=&
C\|f_u\|_{L^p(\R^n,\C)}\\&\ge&
\left(\int_{\R^n}\left( \sum_{j=0}^{+\infty}\Big| {\mathcal{F}}^{-1}\Big( 2^{js}\varphi_j(\xi)\widehat u(\xi)\Big)\Big|^2\right)^{\frac{p}2}\,dx\right)^{\frac1p}\\&=&
\left(\int_{\R^n}\left( \sum_{j=0}^{+\infty}2^{2js} | \check \varphi_j* u(x)|^2\right)^{\frac{p}2}\,dx\right)^{\frac1p}
,\end{eqnarray*}
which proves the second inequality in~\eqref{MikhlinEQWF}.

Similarly, we see that~\eqref{LASTIPERUSAMI2} guarantees the validity of~\eqref{Mikhlin Multiplier Theorem-ASSUN}
with~$X:=\ell^2(\N)$ and~$Y:=\C$ (and~$m$ renamed as~$\mu$). Therefore, by Theorem~\ref{Mikhlin Multiplier Theorem},
$$ \| T_\mu \|_{L(L^p(\R^n,\ell^2(\N)),\,L^p(\R^n,\C))}\le C$$
and accordingly, for functions~$h:\R^n\to\ell^2(\N)$,
\begin{equation}\label{LASTIPERUSAMI281uifhnRefTGBHugrt12rq3tw4y}\begin{split}&
\left(\int_{\R^n} \left| {\mathcal{F}}^{-1}\left( \sum_{j=0}^{+\infty} 2^{-js}\psi_j(\xi)\langle\xi\rangle^{s}\,\widehat h_j(\xi)\right)(x)\right|^p\,dx\right)^{\frac1p}
=\left(\int_{\R^n} \Big| {\mathcal{F}}^{-1}\Big( \mu \widehat{h}\Big)(x)\Big|^p\,dx\right)^{\frac1p}\\&\qquad
=\left(\int_{\R^n} |T_\mu h(x)|^p\,dx\right)^{\frac1p}=
\| T_\mu h\|_{L^p(\R^n,\C)}\\&\qquad\le \| T_\mu \|_{L(L^p(\R^n,\ell^2(\N)),\,L^p(\R^n,\C))}\,\|h\|_{L^p(\R^n,\ell^2(\N))}\\&\qquad\le
C\|h\|_{L^p(\R^n,\ell^2(\N))}=C\left(\int_{\R^n} \|h(x)\|^p_{\ell^2(\N)}\,dx\right)^{\frac1p}\\&\qquad
=C\left(\int_{\R^n} \left( \sum_{j=0}^{+\infty} |h_j(x)|^2\right)^{\frac{p}2}\,dx\right)^{\frac1p}.
\end{split}\end{equation}
In particular, choosing~$h_j:=2^{js} \check \varphi_j* u$ and making use of~\eqref{UTILTAYMAA8ikjf-x5-LAPSIN2}, we see that
\begin{eqnarray*}
&& 2^{-js}\psi_j(\xi)\langle\xi\rangle^{s}\,\widehat h_j(\xi)=
\psi_j(\xi)\langle\xi\rangle^{s}\,\varphi_j(\xi)\widehat u(\xi)
=\langle\xi\rangle^{s}\,\varphi_j(\xi)\widehat u(\xi)=\varphi_j(\xi)\widehat f_u(\xi)
\end{eqnarray*}
and therefore (using Lemma~\ref{LEHDNfcitjn034ty-Leapp} to swap series and Fourier Transforms)
$$ {\mathcal{F}}^{-1}\left( \sum_{j=0}^{+\infty} 2^{-js}\psi_j(\xi)\langle\xi\rangle^{s}\,\widehat h_j(\xi)\right)(x)=
\sum_{j=0}^{+\infty}\check\varphi_j*f_u(x)=f_u(x).$$
Consequently, we deduce from~\eqref{LASTIPERUSAMI281uifhnRefTGBHugrt12rq3tw4y} that
\begin{eqnarray*}
\|u\|_{ {\mathcal{L}}^p_{s}(\R^n)}=\|f_u\|_{L^p(\R^n,\C)}=
\left(\int_{\R^n} |f_u(x)|^p\,dx\right)^{\frac1p}\le C
\left(\int_{\R^n} \left( \sum_{j=0}^{+\infty} 2^{2js}| \check \varphi_j* u(x)|^2\right)^{\frac{p}2}\,dx\right)^{\frac1p}.
\end{eqnarray*}
This establishes the first inequality in~\eqref{MikhlinEQWF} and proves the desired result in
Theorem~\ref{0pirj09365-4g4eZXCHJKLRFG-544-NO3-COR-0oerkCC}.

It remains however to check that~\eqref{LASTIPERUSAMI1}
and~\eqref{LASTIPERUSAMI2} hold true. To attain this goal, we claim that, for all~$\alpha\in\N^n$ and~$b\in\R$,
\begin{equation}\label{0pirj09365-4g4eZXCHJKLRFG-544-NO1-LL-eq49-09i23w}
\begin{split}
|D^{\alpha} \langle\xi\rangle^{b}|&\le C\langle\xi\rangle^{b-|\alpha|},
\end{split}
\end{equation}
with~$C>0$ depending only on~$n$,~$\alpha$ and~$b$.

To check this, let~$X:=(t,\xi)\in\R\times\R^{n}$ and
$$ F(X)=F(t,\xi):=(t^2+4\pi^2|\xi|^2)^{\frac{b}2}.$$
Since~$F$ is positively homogeneous of degree~$b$, we see that, for all~$\lambda>0$ and~$\beta\in\N^{n+1}$,
$$ \lambda^{|\beta|} D^\beta F(\lambda X)=
D^\beta\big(F(\lambda X)\big)=D^\beta\big(\lambda^{b} F(X)\big)=\lambda^{b} D^\beta F(X).$$
Choosing~$\beta:=(0,\alpha)$,~$\lambda:=|X|^{-1}$, with~$X:=(1,\xi)$, we thereby conclude that
$$ (1+|\xi|^2)^{\frac{b-|\alpha|}2}\sup_{{{Y\in\R^{n+1}}\atop{|Y|=1}}\atop{{\gamma\in\N^n}\atop{|\gamma|=|\alpha|}}} |D^\gamma F(Y)|\ge
|X|^{b-|\alpha|} \left|D^{(0,\alpha)} F\left(\frac{X}{|X|}\right)\right|= \big|D^\alpha (1+4\pi^2|\xi|^2)^{\frac{b}2}\big|,$$
which establishes~\eqref{0pirj09365-4g4eZXCHJKLRFG-544-NO1-LL-eq49-09i23w}.

We also recall that, for all~$j\ge1$,
the function~$\varphi_j$ is supported in~$B_{2^{j+1}}\setminus B_{2^{j-1}}$,
thanks to~\eqref{0pirj09365-4g4eZXCHJKLRFG-544-NO1-LL-eq2}.
Thus, making use of the Leibniz Product Rule and~\eqref{0pirj09365-4g4eZXCHJKLRFG-544-NO1-LL-eq4},
for all~$b\in\R$,
\begin{equation}\label{PKJMDS0pew90g43yt980hf67n324vct8SDnt876nv-hgr}
\begin{split}&
\Big|D^{\alpha}_\xi\Big(\langle\xi\rangle^{b}\,\varphi_j(\xi)\Big)\Big|
=\left|\sum_{{\beta \in\N^n}\atop{\beta \leq \alpha}}{\alpha \choose \beta }\,
D^{\beta }\langle\xi\rangle^{b}\, D^{\alpha -\beta }\varphi_j(\xi)\right|\\&\qquad
\le C
\sum_{{\beta \in\N^n}\atop{\beta \leq \alpha}}
\frac{\langle\xi\rangle^{{b-|\beta|}}\,\chi_{B_{2^{j+1}}\setminus B_{2^{j-1}}}(\xi)}{2^{|\alpha-\beta|j}} \le C
\sum_{{\beta \in\N^n}\atop{\beta \leq \alpha}}
\frac{\langle\xi\rangle^{{b-|\beta|}}\,\chi_{B_{2^{j+1}}\setminus B_{2^{j-1}}}(\xi)}{|\xi|^{|\alpha|-|\beta|}}\\&\qquad \le C
\sum_{{\beta \in\N^n}\atop{\beta \leq \alpha}}
\frac{\langle\xi\rangle^{{b-|\beta|}}\,\chi_{B_{2^{j+1}}\setminus B_{2^{j-1}}}(\xi)}{\langle\xi\rangle^{{|\alpha|-|\beta|}}}=C\langle\xi\rangle^{{b-|\alpha|}}\,\chi_{B_{2^{j+1}}\setminus B_{2^{j-1}}}(\xi)\\&\qquad
\le C 2^{j b}\langle\xi\rangle^{{-|\alpha|}}\,\chi_{B_{2^{j+1}}\setminus B_{2^{j-1}}}(\xi)
\end{split}
\end{equation}
for a positive constant~$C$, varying from line to line, and depending only on~$n$,~$\alpha$ and~$b$.

For this reason, choosing~$b:=-s$, we obtain that, for every~$z\in\C$,
\begin{eqnarray*}&&
\|D^\alpha m(\xi) z\|_{\ell^2(\N)}^2
= |z|^2\sum_{j=0}^{+\infty}\Big|D^\alpha \Big( 2^{js}\varphi_j(\xi)\langle\xi\rangle^{-s} \Big)\Big|^2\\&&\qquad\qquad \le
C|z|^2\left(\chi_{B_2}(\xi)+\sum_{j=1}^{+\infty}
\langle\xi\rangle^{-2|\alpha|} \chi_{B_{2^{j+1}}\setminus B_{2^{j-1}}}(\xi)\right)\\&&\qquad\qquad\le\frac{
C|z|^2}{|\xi|^{2\alpha}}\left(1+\sum_{(\log_2|\xi|-1)_+\le j\le\log_2|\xi|+1}1\right)=\frac{
C\,\|z\|_{\C}^2}{|\xi|^{2\alpha}}
,\end{eqnarray*}
which establishes~\eqref{LASTIPERUSAMI1}, as desired.

Let us now check the validity of~\eqref{LASTIPERUSAMI2}. To this end, we make use of~\eqref{PKJMDS0pew90g43yt980hf67n324vct8SDnt876nv-hgr} with~$b:=s$ 
(and~$\psi_j$ replacing~$\varphi_j$, which only slightly changes the support)
and calculate that, for~$a\in\ell^2(\N)$,
\begin{eqnarray*}
\|D^\alpha \mu(\xi)a\|_{\C}^2&=&\left|D^\alpha\left(\sum_{j=0}^{+\infty} 2^{-js}\psi_j(\xi)\langle\xi\rangle^{s}\,a_j\right)\right|^2\\&\le&
C\left(|a_0|\chi_{B_4}(\xi)+|a_1|\chi_{B_8}(\xi)+\sum_{j=2}^{+\infty}
\langle\xi\rangle^{{-|\alpha|}}\,\chi_{B_{2^{j+2}}\setminus B_{2^{j-2}}}(\xi)\,|a_j|\right)^2\\&\le&
\frac{C}{|\xi|^{2\alpha}}\left(|a_0|+|a_1|+\sum_{ (\log_2|\xi|-2)_+\le j\le \log_2|\xi|+2}|a_j|\right)^2\\&\le&
\frac{C}{|\xi|^{2\alpha}}\left[|a_0|^2+|a_1|^2+\left(\sum_{ (\log_2|\xi|-2)_+\le j\le \log_2|\xi|+2}|a_j|\right)^2\right].
\end{eqnarray*}
Hence, since
\begin{eqnarray*}&& \left(\sum_{ (\log_2|\xi|-2)_+\le j\le \log_2|\xi|+2}|a_j|\right)^2\le
\sum_{ (\log_2|\xi|-2)_+\le j\le \log_2|\xi|+2}|a_j|^2
\sum_{ (\log_2|\xi|-2)_+\le j\le \log_2|\xi|+2}1\\&&\qquad\qquad
\le C\sum_{ (\log_2|\xi|-2)_+\le j\le \log_2|\xi|+2}|a_j|^2
\le C\sum_{ j=0}^{+\infty}|a_j|^2=\|a\|^2_{\ell^2(\N)},
\end{eqnarray*}
we infer that~$\|D^\alpha \mu(\xi)a\|_{\C}^2\le C|\xi|^{-2\alpha}\|a\|^2_{\ell^2(\N)}$. The proof of~\eqref{LASTIPERUSAMI2} is thereby complete.
\end{proof}

We remark that one of the difficulties of dealing with Besov spaces comes from the interplay between Fourier Transform
of Lebesgue spaces (which enjoys several simplifications in~$L^2(\R^n)$ which are not available in general).
When~$s\in(0,1)$, however, there are equivalent definitions of Besov spaces which get around such a complication.
To show some of these equivalent formulations, we use the notation
\begin{equation}\label{TAUNOTA}\tau_y u(x):=u(x+y)\end{equation}
and we point out that:

\begin{proposition}\label{F58uuRS-2CALEREG:LEAbc-CC}
If~$s\in(0,1)$ then the Besov space norm in~\eqref{BENOGIU}
is equivalent to 
\begin{equation}\label{0pirj09365-4g4eZXCHJKLRFG-544-NO1}
\|u\|_{L^p(\R^n)}+\left(\int_{\R^n}\frac{\| \tau_y u-u \|_{L^p(\R^n)}^q}{|y|^{n+sq}}\,dy\right)^{\frac1q}.\end{equation}
\end{proposition}

The advantage of this formulation is indeed that no Fourier Transform, and no partition of unity, appears in~\eqref{0pirj09365-4g4eZXCHJKLRFG-544-NO1}. In this way, when~$s\in(0,1)$, the Besov space~$B^{s,p,q}(\R^n)$
can be equivalently defined as the space of locally integrable functions for which~\eqref{0pirj09365-4g4eZXCHJKLRFG-544-NO1} is finite.

\begin{proof}[Proof of Proposition~\ref{F58uuRS-2CALEREG:LEAbc-CC}] By Corollary~\ref{COROLEEP2r34t5},
we know that the term in~$L^p(\R^n)$ is already included in the Besov space norm in~\eqref{BENOGIU},
therefore the core of the proof of Proposition~\ref{F58uuRS-2CALEREG:LEAbc-CC} consists in bounding
the integral term in~\eqref{0pirj09365-4g4eZXCHJKLRFG-544-NO1} from above and from below in terms of the Besov space norm in~\eqref{BENOGIU}.

For this, we start by bounding~\eqref{0pirj09365-4g4eZXCHJKLRFG-544-NO1} from above
in terms of~\eqref{BENOGIU}.
We observe that, for any function~$v$,
\begin{eqnarray*}&&
\|\tau_y v-v\|_{L^p(\R^n)}^p=
\int_{\R^n} |v(x+y)-v(x)|^p\,dx\le
\int_{\R^n} \left|\int_0^1 \nabla v(x+ty)\cdot y \,dt\right|^p\,dx\\&&\qquad\le 
\iint_{\R^n\times(0,1)} |\nabla v(x+ty)|^p |y|^p \,dx\,dt=\iint_{\R^n\times(0,1)} |\nabla v(z)|^p |y|^p \,dz\,dt
=\|\nabla v\|_{L^p(\R^n)}^p\,|y|^p.
\end{eqnarray*}
Hence, taking~$u\in B^{s,p,q}(\R^n)$, using the notation in~\eqref{UTILTAYMAA8ikjf-x5-LAPSIN} and the result in~\eqref{UTILTAYMAA8ikjf-x5-LAPSIN2}, and choosing, for all~$j\in\N$,
$$ v:=\check\varphi_j*u={\mathcal{F}}^{-1}(\psi_j\varphi_j)*u=\check\psi_j*\check\varphi_j*u,$$
we see that
\begin{eqnarray*}&&
\|\tau_y (\check\varphi_j*u)-\check\varphi_j*u\|_{L^p(\R^n)}\le\|\nabla (\check\psi_j*\check\varphi_j*u)\|_{L^p(\R^n)}\,|y|.
\end{eqnarray*}
Since~$\tau_y(f*g)=\tau_yf*g=g*\tau_yf$, we obtain that
\begin{eqnarray*}&&
\| \check\varphi_j*\tau_y u-\check\varphi_j*u\|_{L^p(\R^n)}\le\|\nabla\check\varphi_j*\check\psi_j*u\|_{L^p(\R^n)}\,|y|\le
\|\nabla\check\varphi_j\|_{L^1(\R^n)}\|\check\psi_j*u\|_{L^p(\R^n)}\,|y|.
\end{eqnarray*}
We observe that
$$ \|\nabla\check\varphi_0\|_{L^1(\R^n)}\le C,$$
and also,
recalling~\eqref{0pirj09365-4g4eZXCHJKLRFG-544-NO1-LL-eq4-09-TRIS},
for all~$j\ge1$,
$$ \|\nabla\check\varphi_j\|_{L^1(\R^n)}=2^{(j-1)(n+1)}\int_{\R^n}|\nabla\check \varphi_1(2^{j-1}x)|\,dx=
2^{j-1}\int_{\R^n}|\nabla\check \varphi_1(z)|\,dz
\le C\,2^j.$$

Therefore, we conclude that, for all~$j\in\N$,
\begin{eqnarray*}&&
\| \check\varphi_j*\tau_y u-\check\varphi_j*u\|_{L^p(\R^n)}\le C\,2^j\|\check\psi_j*u\|_{L^p(\R^n)}\,|y|.
\end{eqnarray*}
{F}rom this we arrive at
\begin{eqnarray*}
\frac{\| \check\varphi_j*(\tau_y u-u)\|_{L^p(\R^n)}}{|y|}
&\le& C\,2^j\|\check\psi_j*u\|_{L^p(\R^n)}
\\&\le& C\,2^j\Big(\|\check\varphi_{j-1}*u\|_{L^p(\R^n)}+\|\check\varphi_{j}*u\|_{L^p(\R^n)}+\|\check\varphi_{j+1}*u\|_{L^p(\R^n)}\Big)
\\&\le& C\Big(2^{j-1}\|\check\varphi_{j-1}*u\|_{L^p(\R^n)}+2^j\|\check\varphi_{j}*u\|_{L^p(\R^n)}+2^{j+1}\|\check\varphi_{j+1}*u\|_{L^p(\R^n)}\Big)
\end{eqnarray*}
and thus, by Corollary~\ref{LEHDNfcitjn034ty-LeappCOR},
\begin{eqnarray*}
\frac{\|\tau_y u-u\|_{L^p(\R^n)}}{|y|}&\le&
\sum_{j=0}^{+\infty}\frac{\| \check\varphi_j*(\tau_y u-u)\|_{L^p(\R^n)}}{|y|}\\&\le&C\sum_{j=0}^{+\infty}
\Big(2^{j-1}\|\check\varphi_{j-1}*u\|_{L^p(\R^n)}+2^j\|\check\varphi_{j}*u\|_{L^p(\R^n)}+2^{j+1}\|\check\varphi_{j+1}*u\|_{L^p(\R^n)}\Big)\\&\le&C\sum_{j=0}^{+\infty}2^j\|\check\varphi_{j}*u\|_{L^p(\R^n)}.
\end{eqnarray*}

As a consequence,
\begin{equation}\label{C823eirjfg0987yghvdc3erfvghnmlpoEP2r34t5}
\begin{split}
\int_{B_1}\frac{\| \tau_y u-u \|_{L^p(\R^n)}^q}{|y|^{n+sq}}\,dy&= \sum_{m=0}^{+\infty}
\int_{B_{2^{-m}}\setminus B_{2^{-m-1}}}\frac{\| \tau_y u-u \|_{L^p(\R^n)}^q}{|y|^{n+sq}}\,dy\\&\le C\sum_{m=0}^{+\infty}
\int_{B_{2^{-m}}\setminus B_{2^{-m-1}}} \left(\sum_{j=0}^{+\infty}2^j\|\check\varphi_{j}*u\|_{L^p(\R^n)}\right)^q\frac{dy}{|y|^{n-(1-s)q}}\\&\le
C\sum_{m=0}^{+\infty}
\left(\sum_{j=0}^{+\infty}2^{j-(1-s)m}\|\check\varphi_{j}*u\|_{L^p(\R^n)}\right)^q\\&=C\sum_{m=0}^{+\infty}
\left(\sum_{j=0}^{+\infty}a_{m-j}\,b_j\right)^q,
\end{split}\end{equation}
where
$$ a_j:=2^{(s-1)j}\qquad{\mbox{and}}\qquad b_j:=2^{js}\|\check\varphi_{j}*u\|_{L^p(\R^n)} .$$

It is therewith convenient to use the discrete convolution notation and the corresponding Young's Convolution Inequality
to see that, since~$a_j$,~$b_j\ge0$,
\begin{eqnarray*}&&
\left[\sum_{m=0}^{+\infty}\left(\sum_{j=0}^{+\infty}a_{m-j}\,b_j\right)^q\right]^{\frac1q}=\left[\sum_{m=0}^{+\infty}\left|\sum_{j=0}^{+\infty}a_{m-j}\,b_j\right|^q\right]^{\frac1q}=
\left[\sum_{m=0}^{+\infty} |a*b|^q\right]^{\frac1q}=\|a*b\|_{\ell^q(\N)}\\&&\qquad\le
\|a\|_{\ell^1(\N)}\|b\|_{\ell^q(\N)}=\left(
\sum_{j=0}^{+\infty}2^{(s-1)j}\right)\left(\sum_{j=0}^{+\infty}2^{jsq}\|\check\varphi_{j}*u\|^q_{L^p(\R^n)}\right)^{\frac1q}\le C
\|u\|_{B^{s,p,q}(\R^n)}
,\end{eqnarray*}
where in the last step we have used the Besov space norm in~\eqref{BENOGIU}
and the fact that~$s\in(0,1)$.

Plugging this information into~\eqref{C823eirjfg0987yghvdc3erfvghnmlpoEP2r34t5} and recalling Corollary~\ref{COROLEEP2r34t5}, we conclude that
\begin{eqnarray*}&&
\|u\|_{L^p(\R^n)}+\left(\int_{\R^n}\frac{\| \tau_y u-u \|_{L^p(\R^n)}^q}{|y|^{n+sq}}\,dy\right)^{\frac1q}\\&
\le& \|u\|_{L^p(\R^n)}+
\left(\int_{\R^n\setminus B_1}\frac{\big(\| \tau_y u\|_{L^p(\R^n)}+\|u \|_{L^p(\R^n)}\big)^q}{|y|^{n+sq}}\,dy+
\int_{B_1}\frac{\| \tau_y u-u \|_{L^p(\R^n)}^q}{|y|^{n+sq}}\,dy\right)^{\frac1q}\\&\le&\|u\|_{L^p(\R^n)}+C
\left(\|u\|_{L^p(\R^n)}^q+\sum_{m=0}^{+\infty}
\left(\sum_{j=0}^{+\infty}a_{m-j}\,b_j\right)^q\right)^{\frac1q}\\&\le&\|u\|_{L^p(\R^n)}+C
\left(\|u\|_{L^p(\R^n)}^q+\|u\|_{B^{s,p,q}(\R^n)}^q\right)^{\frac1q}\\&\le&C
\big(\|u\|_{L^p(\R^n)}+\|u\|_{B^{s,p,q}(\R^n)}\big)\\&\le& C\|u\|_{B^{s,p,q}(\R^n)},
\end{eqnarray*}
which provides the desired bound from above of~\eqref{0pirj09365-4g4eZXCHJKLRFG-544-NO1}
in terms of~\eqref{BENOGIU}.

We now address the bound from below, to complete the proof of Proposition~\ref{F58uuRS-2CALEREG:LEAbc-CC}.
To this end, we use~\eqref{UTILTAYMAA8ikjf} with~$\alpha:=0$ and we see that, for every~$x\in\R^n$,
$$ \int_{\R^n} \check\varphi_j(y)\big(u(x-y)- u(x)\big)\,dy=
\int_{\R^n} \check\varphi_j(y) u(x-y)\,dy=\check\varphi_j*u(x).$$
This, Minkowski's Integral Inequality (see Theorem~\ref{MLAerSM:ijfKKSMdf02})
and~\eqref{0pirj09365-4g4eZXCHJKLRFG-544-NO1-LL-eq4-09-TRIS} yield that
\begin{eqnarray*}
\|\check\varphi_j*u\|_{L^p(\R^n)}&=&\left(\int_{\R^n}\left| \int_{\R^n} \check\varphi_j(y)\big(u(x-y)- u(x)\big)\,dy\right|^p\,dx\right)^{\frac1p}\\&\le&\int_{\R^n} 
\left(\int_{\R^n}\left| \check\varphi_j(y)\big(u(x-y)- u(x)\big)\right|^p\,dx\right)^{\frac1p}\,dy\\&=&2^{(j-1)n}\int_{\R^n} 
\left(\int_{\R^n}\left| \check \varphi_1(2^{j-1}y)\big(u(x-y)- u(x)\big)\right|^p\,dx\right)^{\frac1p}\,dy\\&=&2^{(j-1)n}\int_{\R^n} |\check \varphi_1(2^{j-1}y)|\,\|\tau_{-y}u-u\|_{L^p(\R^n)}\,dy\\&\le&2^{(j-1)n}\int_{\R^n} |\check \varphi_1(2^{j-1}y)|\,\omega(|y|)\,dy,
\end{eqnarray*}
where
\begin{equation}\label{eqomehgaokmOMSDFn} \omega(r):=\sup_{y\in B_r}\|\tau_{y}u-u\|_{L^p(\R^n)}.\end{equation}

Therefore, using the change of variables~$z:=2^{j-1}y$ and Minkowski's Integral Inequality,
\begin{equation}\label{0pirj09365-4g4eZXCHJKLRFG-544-NO1-LL-eq4-09-x34rt-012rutj}\begin{split}
\sum_{j=1}^{+\infty} 2^{jsq}\|\check\varphi_{j}*u\|^q_{L^p(\R^n)}&\le
\sum_{j=1}^{+\infty} \left(
2^{js+(j-1)n}\int_{\R^n} |\check \varphi_1(2^{j-1}y)|\,\omega(|y|)\,dy
\right)^q\\&=
\sum_{j=1}^{+\infty}\left(
\int_{\R^n} 2^{js}|\check \varphi_1(z)|\,\omega(2^{1-j}|z|)\,dz
\right)^q\\&\le\left[ \int_{\R^n} \left(
\sum_{j=1}^{+\infty}
2^{jsq}|\check \varphi_1(z)|^q\,\omega^q(2^{1-j}|z|)
\right)^{\frac1q}\,dz\right]^q.
\end{split}\end{equation}

Notice now that~$\omega$ is nondecreasing, whence for all~$t\in[2^{1-j},2^{2-j}]$ we have that~$\omega(2^{1-j}|z|)\le
\omega(t|z|)$. For this reason,
$$ 2^{1-j}\omega^q(2^{1-j}|z|)=\int_{2^{1-j}}^{2^{2-j}}\omega^q(2^{1-j}|z|)\,dt\le\int_{2^{1-j}}^{2^{2-j}}\omega^q(t|z|)\,dt$$
and accordingly
\begin{eqnarray*}&&
\sum_{j=1}^{+\infty} 2^{jsq} \omega^q(2^{1-j}|z|)\le C\sum_{j=1}^{+\infty} 2^{j(1+sq)} \int_{2^{1-j}}^{2^{2-j}}\omega^q(t|z|)\,dt\\&&\qquad\le C\sum_{j=1}^{+\infty} \int_{2^{1-j}}^{2^{2-j}}\frac{\omega^q(t|z|)}{t^{1+sq}}\,dt=C \int_0^{1} \frac{\omega^q(t|z|)}{t^{1+sq}}\,dt.
\end{eqnarray*}

{F}rom this observation and~\eqref{0pirj09365-4g4eZXCHJKLRFG-544-NO1-LL-eq4-09-x34rt-012rutj}, using the change of variable~$\theta:=t|z|$, we arrive at
\begin{equation}\label{eqomq23424wsehgaokmOMSDFn-a0iqouwjdfweio7uyhbRFGSDHUJN0aosidjh-0}
\begin{split}
\sum_{j=1}^{+\infty} 2^{jsq}\|\check\varphi_{j}*u\|^q_{L^p(\R^n)}&\le
\left[ \int_{\R^n} |\check \varphi_1(z)|\left(
\sum_{j=1}^{+\infty}
2^{jsq}\omega^q(2^{1-j}|z|)
\right)^{\frac1q}\,dz
\right]^q\\&\le C
\left[ \int_{\R^n} |\check \varphi_1(z)|\left(
\int_0^{1} \frac{\omega^q(t|z|)}{t^{1+sq}}\,dt\right)^{\frac1q}\,dz
\right]^q\\&= C\left[ \int_{\R^n} |\check \varphi_1(z)|\,|z|^s\left(
\int_0^{|z|} \frac{\omega^q(\theta)}{\theta^{1+sq}}\,d\theta\right)^{\frac1q}\,dz
\right]^q\\&\le C \int_{0}^{+\infty}\frac{\omega^q(\theta)}{\theta^{1+sq}}\,d\theta,
\end{split}\end{equation}
up to renaming~$C$.

Let us now have a further look at the function~$\omega$ in~\eqref{eqomehgaokmOMSDFn}.
We notice that, for all~$y$,~$z\in\R^n$,
\begin{equation}\label{9ikm23-12oi3rktrmm}\begin{split}&
\|\tau_{y+z} u-u\|_{L^p(\R^n)}\le\|\tau_{y+z} u-\tau_y u\|_{L^p(\R^n)}+\|\tau_{y} u-u\|_{L^p(\R^n)}\\&\qquad=\|\tau_{z} u- u\|_{L^p(\R^n)}+\|\tau_{y} u-u\|_{L^p(\R^n)},\end{split}
\end{equation}
where the translation invariance of the Lebesgue norm has been used.

As a result, writing~$y=\left(\frac{y}{2}+z\right)+\left(\frac{y}{2}-z\right)$,
\begin{equation*}\begin{split}&
\|\tau_{y} u-u\|_{L^p(\R^n)}\le\|\tau_{\frac{y}{2}+z} u- u\|_{L^p(\R^n)}+\|\tau_{\frac{y}{2}-z} u-u\|_{L^p(\R^n)},\end{split}\end{equation*}
and therefore
$$ \|\tau_{y} u-u\|_{L^p(\R^n)}^q\le C\left(\|
\tau_{\frac{y}{2}+z} u- u\|^q_{L^p(\R^n)}+\|\tau_{\frac{y}{2}-z} u-u\|^q_{L^p(\R^n)}\right).~$$
On that account, for all~$r>0$,
\begin{equation}\label{0pirj09365-4g4eZXCHJKLRFG-544}\begin{split}&
\|\tau_{y} u-u\|^q_{L^p(\R^n)}\le C\fint_{B_{r/2}}
\Big(\|\tau_{\frac{y}{2}+z} u- u\|^q_{L^p(\R^n)}+\|\tau_{\frac{y}{2}-z} u-u\|^q_{L^p(\R^n)}\Big)\,dz.\end{split}
\end{equation}
Also, if~$y\in B_r$ and~$z\in B_{r/2}$, we have that~$\left|\frac{y}{2}\pm z\right|\le\frac{|y|}{2}+| z|<r$, whence~$\frac{y}{2}\pm z\in B_r$. {F}rom this and~\eqref{0pirj09365-4g4eZXCHJKLRFG-544} it follows that
\begin{equation*}\begin{split}
\sup_{y\in B_r}\|\tau_{y} u-u\|_{L^p(\R^n)}^q&\le \frac{C}{|B_{r/2}|}\sup_{y\in B_r}\int_{ B_{r/2}}
\Big(\|\tau_{\frac{y}{2}+z} u- u\|^q_{L^p(\R^n)}+\|\tau_{\frac{y}{2}-z} u-u\|^q_{L^p(\R^n)}\Big)\,dz\\&\le
\frac{C}{|B_{r/2}|}\sup_{y\in B_r}\int_{B_{r}}
\|\tau_{w} u- u\|^q_{L^p(\R^n)}\,dw\\&=
\frac{C}{r^n} \int_{B_r}\|\tau_{w} u- u\|^q_{L^p(\R^n)}\,dw.
\end{split}
\end{equation*}

For that reason, using Fubini's Theorem, and possibly renaming constants line after line, we see that
\begin{equation}\label{eqomq23424wsehgaokmOMSDFn-a0iqouwjdfweio7uyhbRFGSDHUJN0aosidjh}
\begin{split}
&\int_0^{+\infty}\frac{\displaystyle\sup_{y\in B_r}\| \tau_y u-u \|_{L^p(\R^n)}^q}{r^{1+sq}}\,dr
\le C\int_0^{+\infty}\left[ \int_{B_r}\frac{ \| \tau_w u-u \|_{L^p(\R^n)}^q}{r^{n+1+sq}}\,dw\right]\,dr
\\&\qquad=C\int_{\R^n}
\left[ \int_{|w|}^{+\infty}\frac{ \| \tau_w u-u \|_{L^p(\R^n)}^q}{r^{n+1+sq}}\,dr\right]\,dw=
C\int_{\R^n}\frac{ \| \tau_w u-u \|_{L^p(\R^n)}^q}{|w|^{n+sq}}\,dw.
\end{split}
\end{equation}

We plug this information into~\eqref{eqomq23424wsehgaokmOMSDFn-a0iqouwjdfweio7uyhbRFGSDHUJN0aosidjh-0} and we conclude that
\begin{equation*}
\begin{split}
\sum_{j=1}^{+\infty} 2^{jsq}\|\check\varphi_{j}*u\|^q_{L^p(\R^n)}\le
C \int_{0}^{+\infty}\frac{\displaystyle\sup_{y\in B_\theta}\|\tau_{y}u-u\|_{L^p(\R^n)}^q}{\theta^{1+sq}}\,d\theta\le
C\int_{\R^n}\frac{ \| \tau_w u-u \|_{L^p(\R^n)}^q}{|w|^{n+sq}}\,dw.
\end{split}\end{equation*}
This and Young's Convolution Inequality give that
$$ \|u\|_{B^{s,p,q}(\R^n)}^q=\sum_{j=0}^{+\infty} 2^{jsq}\|\check\varphi_{j}*u\|^q_{L^p(\R^n)}\le
C\left(\|u\|_{L^p(\R^n)}^q+\int_{\R^n}\frac{ \| \tau_w u-u \|_{L^p(\R^n)}^q}{|w|^{n+sq}}\,dw\right).$$
In this way, we have controlled the norm in~\eqref{0pirj09365-4g4eZXCHJKLRFG-544-NO1} from below by the Besov space norm in~\eqref{BENOGIU}. In doing so, we have completed
the proof of Proposition~\ref{F58uuRS-2CALEREG:LEAbc-CC}.
\end{proof}

We now relate\footnote{In light of Theorem~\ref{TH0-104-okn5KMD3} and
and Corollary~\ref{0pirj09365-4g4eZXCHJKLRFG-544-NO1-I2L3C2O25R213t4TY}, we see that~$W^{s,p}(\R^n)$ has the strange feature that it is a Besov space when~$s$ \label{sdcjnPwqdkjfcmvDCtVn3sd}
is non-integer and a Bessel potential space when~$s$ is integer.

As a notational remark, we point out that, in the literature, the space~$B^{s,p,p}$ is sometimes denoted, for short, either~$B^{s,p}$ or~$B^s_p$.} fractional Sobolev spaces and Besov spaces.

\begin{corollary}\label{0pirj09365-4g4eZXCHJKLRFG-544-NO1-I2L3C2O25R213t4TY}
If~$s\in(0,+\infty)\setminus\N$ and~$p\ge1$, then
$$ B^{s,p,p}(\R^n)=W^{s,p}(\R^n).$$
\end{corollary}

\begin{proof} We write~$s=m+\sigma$, with~$m\in\N$ and~$\sigma\in(0,1)$.
We argue by induction over~$m$. If~$m=0$, then~$s=\sigma\in(0,1)$. Thus, in view of~\eqref{0pirj09365-4g4eZXCHJKLRFG-544-NO1}
we can equivalently identify~$\|u\|_{B^{s,p,p}(\R^n)}$ with
$$ \|u\|_{L^p(\R^n)}+\left(\int_{\R^n}\frac{\| \tau_y u-u \|_{L^p(\R^n)}^p}{|y|^{n+sp}}\,dy\right)^{\frac1p}=
\|u\|_{L^p(\R^n)}+\left(\iint_{\R^n\times\R^n}\frac{|u(x+y)-u(x)|^p}{|y|^{n+sp}}\,dx\,dy\right)^{\frac1p},
$$ which is a norm for~$W^{s,p}(\R^n)$, proving the desired result in this case.

Let us now proceed recursively, assuming the result valid up to~$m-1$.
Then, by Corollary~\ref{DERIBEVS}, a function~$u$ belongs to~$B^{s,p,p}(\R^n)$ if and only if~$u\in L^p(\R^n)$ and~$\partial_ju\in
B^{s-1,p,p}(\R^n)$ for all~$j\in\{1,\dots,n\}$.

Now, notice that~$s-1=m-1+\sigma$, and hence, by inductive hypothesis, we know that~$B^{s-1,p,p}(\R^n)=W^{s-1,p}(\R^n)$.
All in all, we have that~$u$ belongs to~$B^{s,p,p}(\R^n)$ if and only if~$u\in L^p(\R^n)$ and~$\partial_ju\in
W^{s-1,p}(\R^n)$ for all~$j\in\{1,\dots,n\}$, i.e., if and only if~$u\in W^{s,p}(\R^n)$, which completes the inductive step.
\end{proof}

When~$s\in(0,1)$, the result in Proposition~\ref{F58uuRS-2CALEREG:LEAbc-CC} can be rephrased into several equivalent norms.
For instance, one can give explicit relevance to the quantity introduced in~\eqref{eqomehgaokmOMSDFn}:

\begin{proposition}\label{0pirj09365-4g4eZXCHJKLRFG-544-NOP}
If~$s\in(0,1)$ then the Besov space norm in~\eqref{0pirj09365-4g4eZXCHJKLRFG-544-NO1} is equivalent to
\begin{equation}\label{0pirj09365-4g4eZXCHJKLRFG-544-NO2}
\|u\|_{L^p(\R^n)}+\left(\int_0^{+\infty}\frac{\displaystyle\sup_{y\in B_r}\| \tau_y u-u \|_{L^p(\R^n)}^q}{r^{1+sq}}\,dr\right)^{\frac1q}.\end{equation}
\end{proposition}

The relevant contribution to the integral in~\eqref{0pirj09365-4g4eZXCHJKLRFG-544-NO2} is actually the one
close to the origin, as it will be more apparent in~\eqref{0pirj09365-4g4eZXCHJKLRFG-544-NO2bis} below.

\begin{proof}[Proof of Proposition~\ref{0pirj09365-4g4eZXCHJKLRFG-544-NOP}] 
The estimate in~\eqref{eqomq23424wsehgaokmOMSDFn-a0iqouwjdfweio7uyhbRFGSDHUJN0aosidjh} shows that the norm in~\eqref{0pirj09365-4g4eZXCHJKLRFG-544-NO1} controls from above the norm in~\eqref{0pirj09365-4g4eZXCHJKLRFG-544-NO2}.

To prove that, conversely, the norm in~\eqref{0pirj09365-4g4eZXCHJKLRFG-544-NO2} controls from above the norm in~\eqref{0pirj09365-4g4eZXCHJKLRFG-544-NO1}, we notice that, by polar coordinates,
\begin{eqnarray*}&&
\int_{\R^n}\frac{\| \tau_y u-u \|_{L^p(\R^n)}^q}{|y|^{n+sq}}\,dy=
\int_0^{+\infty}\int_{\partial B_1}\frac{\| \tau_{\rho\vartheta} u-u \|_{L^p(\R^n)}^q}{\rho^{1+sq}}\,d{\mathcal{H}}^{n-1}_\vartheta\,d\rho\\&&\qquad\le
\int_0^{+\infty}\int_{\partial B_1}\frac{\displaystyle\sup_{z\in B_{2\rho}}\| \tau_{z} u-u \|_{L^p(\R^n)}^q}{\rho^{1+sq}}\,d{\mathcal{H}}^{n-1}_\vartheta\,d\rho\\
&&\qquad={\mathcal{H}}^{n-1}(\partial B_1)
\int_{0}^{+\infty}\frac{\displaystyle\sup_{z\in B_{2\rho}}\| \tau_{z} u-u \|_{L^p(\R^n)}^q}{\rho^{1+sq}}\,d\rho\\
&&\qquad=2^{sq}\,{\mathcal{H}}^{n-1}(\partial B_1)
\int_{0}^{+\infty}\frac{\displaystyle\sup_{z\in B_{r}}\| \tau_{z} u-u \|_{L^p(\R^n)}^q}{r^{1+sq}}\,dr,
\end{eqnarray*}
which completes the proof of the desired result.
\end{proof}

\begin{corollary}\label{corod34958834y689u2eryue9wighrejhgfhdjehfdjkjfhj}
If~$s\in(0,1)$ then the Besov space norm in~\eqref{0pirj09365-4g4eZXCHJKLRFG-544-NO1} is equivalent to
\begin{equation}\label{0pirj09365-4g4eZXCHJKLRFG-544-NO2bis}
\|u\|_{L^p(\R^n)}+\left(\int_0^1\frac{\displaystyle\sup_{y\in B_r}\| \tau_y u-u \|_{L^p(\R^n)}^q}{r^{1+sq}}\,dr\right)^{\frac1q}.\end{equation}
\end{corollary}

\begin{proof} Owing to Proposition~\ref{0pirj09365-4g4eZXCHJKLRFG-544-NOP}, it suffices to show the equivalence between the norms in~\eqref{0pirj09365-4g4eZXCHJKLRFG-544-NO2} and~\eqref{0pirj09365-4g4eZXCHJKLRFG-544-NO2bis}.

Actually, the norm in~\eqref{0pirj09365-4g4eZXCHJKLRFG-544-NO2} is obviously larger than the one in~\eqref{0pirj09365-4g4eZXCHJKLRFG-544-NO2bis}, therefore we only need to check that the norm in~\eqref{0pirj09365-4g4eZXCHJKLRFG-544-NO2} can be controlled from above by that in~\eqref{0pirj09365-4g4eZXCHJKLRFG-544-NO2bis}, up to a multiplicative constant.

To this end, we observe that
\begin{eqnarray*}&&
\int_1^{+\infty}\frac{\displaystyle\sup_{y\in B_r}\| \tau_y u-u \|_{L^p(\R^n)}^q}{r^{1+sq}}\,dr\le
\int_1^{+\infty}\frac{\displaystyle\sup_{y\in B_r}\big(\| \tau_y u\|_{L^p(\R^n)}+\|u \|_{L^p(\R^n)}\big)^q}{r^{1+sq}}\,dr\\&&\qquad\qquad\qquad
=2^q\|u \|_{L^p(\R^n)}^q\int_1^{+\infty}\frac{dr}{r^{1+sq}}=C\|u \|_{L^p(\R^n)}^q,
\end{eqnarray*}
from which the desired result plainly follows.
\end{proof}

\begin{corollary}\label{0pirj09365-4g4eZXCHJKLRFG-544-NO3-COR} If~$s\in(0,1)$ then the Besov space norm in~\eqref{0pirj09365-4g4eZXCHJKLRFG-544-NO1} is equivalent to
\begin{equation}\label{0pirj09365-4g4eZXCHJKLRFG-544-NO3}
\|u\|_{L^p(\R^n)}+\left(\sum_{k=0}^{+\infty}\left(2^{k sq}\sup_{y\in B_{1/2^k}}\| \tau_y u-u \|_{L^p(\R^n)}^q\right) \right)^{\frac1q}.\end{equation}
\end{corollary}

\begin{proof} We notice that
\begin{equation}\label{0oKSIIMmslRl7u3398j-1}
\int_0^{1}\frac{\displaystyle\sup_{y\in B_r}\| \tau_y u-u \|_{L^p(\R^n)}^q}{r^{1+sq}}\,dr=
\sum_{k=0}^{+\infty} \int^{1/2^{k}}_{1/2^{k+1}}\frac{\displaystyle\sup_{y\in B_r}\| \tau_y u-u \|_{L^p(\R^n)}^q}{r^{1+sq}}\,dr.\end{equation}
On that account,
\begin{eqnarray*}
&& \int_0^{1}\frac{\displaystyle\sup_{y\in B_r}\| \tau_y u-u \|_{L^p(\R^n)}^q}{r^{1+sq}}\,dr\ge
\sum_{k=0}^{+\infty} \int^{1/2^{k}}_{1/2^{k+1}}\frac{\displaystyle\sup_{y\in B_{1/2^{k+1}}}\| \tau_y u-u \|_{L^p(\R^n)}^q}{r^{1+sq}}\,dr\\&&\qquad=
\frac{2^{sq}-1}{sq}
\sum_{k=0}^{+\infty}\left(2^{ksq}\sup_{y\in B_{1/2^{k+1}}}\| \tau_y u-u \|_{L^p(\R^n)}^q\right)\\
&&\qquad=\frac{2^{sq}-1}{sq\,2^{sq}}
\sum_{j=1}^{+\infty}\left(2^{jsq}\sup_{y\in B_{1/2^{j}}}\| \tau_y u-u \|_{L^p(\R^n)}^q\right).
\end{eqnarray*}
As for the term with~$j=0$, we have that
\begin{eqnarray*}
\sup_{y\in B_{1}}\| \tau_y u-u \|_{L^p(\R^n)}^q\le\sup_{y\in B_{1}}\big(\| \tau_y u\|_{L^p(\R^n)}+\|u \|_{L^p(\R^n)}\big)^q=2^q\|u\|_{L^p(\R^n)}^q.
\end{eqnarray*}
These observations show that the norm in~\eqref{0pirj09365-4g4eZXCHJKLRFG-544-NO3} is controlled from above by the norm
in~\eqref{0pirj09365-4g4eZXCHJKLRFG-544-NO2bis}, and therefore by the norm in~\eqref{0pirj09365-4g4eZXCHJKLRFG-544-NO1},
owing to Corollary~\ref{corod34958834y689u2eryue9wighrejhgfhdjehfdjkjfhj}.

Similarly, by~\eqref{0oKSIIMmslRl7u3398j-1},
\begin{eqnarray*}&&\int_0^{1}\frac{\displaystyle\sup_{y\in B_r}\| \tau_y u-u \|_{L^p(\R^n)}^q}{r^{1+sq}}\,dr\leq
\sum_{k=0}^{+\infty} \int^{1/2^{k}}_{1/2^{k+1}}\frac{\displaystyle\sup_{y\in B_{1/2^k}}\| \tau_y u-u \|_{L^p(\R^n)}^q}{r^{1+sq}}\,dr\\
&&\qquad=\frac{2^{sq}-1}{sq}
\sum_{k=0}^{+\infty} \left(2^{ksq}\sup_{y\in B_{1/2^k}}\| \tau_y u-u \|_{L^p(\R^n)}^q\right),
\end{eqnarray*}
showing that the norm in~\eqref{0pirj09365-4g4eZXCHJKLRFG-544-NO3} is controlled from below by the norm
in~\eqref{0pirj09365-4g4eZXCHJKLRFG-544-NO2bis}, and therefore by the norm in~\eqref{0pirj09365-4g4eZXCHJKLRFG-544-NO1},
thanks to Corollary~\ref{corod34958834y689u2eryue9wighrejhgfhdjehfdjkjfhj}.
\end{proof}

We now put forth an interesting relation between
Bessel potential spaces and Besov spaces (for more comprehensive results, see Theorem~5 on page~155
of~\cite{MR0290095} and the references therein).
Roughly speaking, the Bessel potential space~${\mathcal{L}}^p_{s}(\R^n)$ always sits between~$B^{s,p,2}(\R^n)$ and~$B^{s,p,p}(\R^n)$,
but the direction of the inclusion depends on whether~$p\le2$ or~$p\ge2$.

\begin{theorem}\label{THPS2} Let~$s>0$.
If~$p\in(1,2]$, then
$$ B^{s,p,p}(\R^n)\subseteq {\mathcal{L}}^p_{s}(\R^n)\subseteq B^{s,p,2}(\R^n),$$
with continuous embeddings.

Also, if~$p\in[2,+\infty)$, then
$$ B^{s,p,2}(\R^n)\subseteq {\mathcal{L}}^p_{s}(\R^n) \subseteq B^{s,p,p}(\R^n),$$
with continuous embeddings.
\end{theorem}

One can prove that the inclusions presented in Theorem~\ref{THPS2} are optimal and cannot be
improved, see~\cite[Theorem~20]{MR163159} and~\cite[Exercise~6.8]{MR0290095}.

Also, as a particular case,
one immediately infers from Theorem~\ref{THPS2} that:

\begin{corollary} \label{THPS2CO} For all~$s>0$,
$$  B^{s,2,2}(\R^n)=W^{s,2}(\R^n)={\mathcal{L}}^2_{s}(\R^n).$$\end{corollary}

Since one can obtain Corollary~\ref{THPS2CO} also directly, without relying on Theorem~\ref{THPS2},
before giving the more delicate argument needed to establish Theorem~\ref{THPS2} in its general form, let us present a more immediate proof of Corollary~\ref{THPS2CO}.

\begin{proof}[Proof of Corollary~\ref{THPS2CO}] As a norm in~$W^{s,2}(\R^n)$ we can take
$$ \|u\|_{W^{s,2}(\R^n)}:=\sqrt{\int_{\R^n} \big(1+|\xi|^{2s}\big)\,|\widehat u(\xi)|^2\,d\xi},$$
see e.g.~\cite[Section~3.1]{MR2944369}.

Also, if~$u={\mathcal{B}}^{(s/2)}*f_u$, we know by~\eqref{ERFGHJN6789-09tftd90u8yhgiug8erhISB}, Lemma~\ref{PRODpoikjhr3-2:le-7}
and the Plancherel Theorem that
$$ \|u\|_{ {\mathcal{L}}^2_{s}(\R^n)}=\| f_u\|_{L^2(\R^n)}=\| \widehat f_u\|_{L^2(\R^n)}=
\left\|  \frac{\widehat u}{ \widehat{\mathcal{B}}^{(s/2)}}\right\|_{L^2(\R^n)}=
\left\|  \big(1+4\pi^2|\xi|^2\big)^{\frac{s}2}\,\widehat u\right\|_{L^2(\R^n)},$$
which is equivalent to~$\|u\|_{W^{s,2}(\R^n)}$.

Furthermore, by~\eqref{BENOGIU} and the Plancherel Theorem,
\begin{equation} \label{0pirj09365-4g4eZXCHJKLRFG-544-NO1-LL-eq2BEMNofwed}\|u\|_{B^{s,2,2}(\R^n)}^2=
\sum_{j=0}^{+\infty}2^{2sj}\|\check{\varphi}_j*u\|_{L^2(\R^n)}^2=
\sum_{j=0}^{+\infty}2^{2sj}\| \varphi_j\,\widehat u\|_{L^2(\R^n)}^2
.\end{equation}

Now, on the one hand, by~\eqref{0pirj09365-4g4eZXCHJKLRFG-544-NO1-LL-eq2}, we know that, for all~$j\ge1$,~$$
0\le\varphi_j\le\chi_{B_{2^{j+1}}\setminus B_{2^{j-1}}}$$
and therefore
$$ \| \varphi_j\,\widehat u\|_{L^2(\R^n)}^2\le\int_{B_{2^{j+1}}\setminus B_{2^{j-1}}}|\widehat u(\xi)|^2\,d\xi.$$
Plugging this information into~\eqref{0pirj09365-4g4eZXCHJKLRFG-544-NO1-LL-eq2BEMNofwed} we deduce that
\begin{equation}\label{0pirj09365-4g4eZXCHJKLRFG-544-NO1-LL-eq2BEMNofwed22}
\begin{split}
\|u\|_{B^{s,2,2}(\R^n)}^2&\le
\| \varphi_0\,\widehat u\|_{L^2(\R^n)}^2+
\sum_{j=1}^{+\infty}2^{2sj} \int_{B_{2^{j+1}}\setminus B_{2^{j-1}}}|\widehat u(\xi)|^2\,d\xi\\&\le
\| \widehat u\|_{L^2(\R^n)}^2+2^{2s}
\sum_{j=1}^{+\infty} \int_{B_{2^{j+1}}\setminus B_{2^{j-1}}}|\xi|^{2s}|\widehat u(\xi)|^2\,d\xi
\\&\le2^{2s}\left(
\| \widehat u\|_{L^2(\R^n)}^2+\int_{\R^n}|\xi|^{2s}|\widehat u(\xi)|^2\,d\xi
\right)\\&=2^{2s}\|u\|_{W^{s,2}(\R^n)}^2.
\end{split}\end{equation}

On the other hand, by~\eqref{0pirj09365-4g4eZXCHJKLRFG-544-NO1-LL-eq1} and~\eqref{0pirj09365-4g4eZXCHJKLRFG-544-NO1-LL-eq2},
\begin{eqnarray*}
|\xi|^s &=&|\xi|^s\varphi_0(\xi)+\sum_{j=1}^{+\infty}|\xi|^s\varphi_j(\xi)\\
&=&|\xi|^s\varphi_0(\xi)+\sum_{j=1}^{+\infty}\chi_{B_{2^{j+1}}\setminus B_{2^{j-1}}}(\xi)\,|\xi|^s\varphi_j(\xi)\\&\le&
|\xi|^s\varphi_0(\xi)+\sqrt{\sum_{j=1}^{+\infty}\chi_{B_{2^{j+1}}\setminus B_{2^{j-1}}}(\xi)}\;
\sqrt{\sum_{j=1}^{+\infty}|\xi|^{2s}\varphi_j^2(\xi)}\\&\le&
|\xi|^s\varphi_0(\xi)+\sqrt{\sum_{j=1}^{+\infty}\chi_{B_{2^{j+1}}\setminus B_{2^{j-1}}}(\xi)}\;
\sqrt{\sum_{j=1}^{+\infty}\chi_{B_{2^{j+1}}\setminus B_{2^{j-1}}}(\xi)\,|\xi|^{2s}\varphi_j^2(\xi)}\\&\le&
|\xi|^s\varphi_0(\xi)+\sqrt{\sum_{{j\in\N}\atop{\log_2|\xi|-1\le j\le\log_2|\xi|+1}}1}\;
\sqrt{\sum_{j=1}^{+\infty}(2^{j+1})^{2s}\varphi_j^2(\xi)}\\&\le&
|\xi|^s\varphi_0(\xi)+\sqrt{3}\;
\sqrt{2^{2s}\sum_{j=1}^{+\infty}2^{2sj}\varphi_j^2(\xi)}
\end{eqnarray*}
and consequently
$$ |\xi|^{2s}\le C\left( 
|\xi|^{2s}\varphi_0^2(\xi)+\sum_{j=1}^{+\infty}2^{2sj}\varphi_j^2(\xi)\right)\le C\sum_{j=0}^{+\infty}2^{2sj}\varphi_j^2(\xi).
$$

For this reason, recalling~\eqref{0pirj09365-4g4eZXCHJKLRFG-544-NO1-LL-eq2BEMNofwed} and renaming~$C$ over and over,
\begin{eqnarray*}
\|u\|_{W^{s,2}(\R^n)}&=&\int_{\R^n} |\widehat u(\xi)|^2\,d\xi+
\int_{\R^n} |\xi|^{2s}\,|\widehat u(\xi)|^2\,d\xi
\\&\le&\int_{B_1} |\widehat u(\xi)|^2\,d\xi+
2\int_{\R^n} |\xi|^{2s}\,|\widehat u(\xi)|^2\,d\xi\\&\le&\int_{B_1}\varphi_0^2(\xi) |\widehat u(\xi)|^2\,d\xi+C\sum_{j=0}^{+\infty}\int_{\R^n} 2^{2sj}\varphi_j^2(\xi)\,|\widehat u(\xi)|^2\,d\xi\\&\le&
C\sum_{j=0}^{+\infty}\int_{\R^n} 2^{2sj}\varphi_j^2(\xi)\,|\widehat u(\xi)|^2\,d\xi\\&=&C\sum_{j=0}^{+\infty}2^{2sj}\| \varphi_j\,\widehat u\|_{L^2(\R^n)}^2\\&=&C\|u\|_{B^{s,2,2}(\R^n)}^2,
\end{eqnarray*}
which, in combination with~\eqref{0pirj09365-4g4eZXCHJKLRFG-544-NO1-LL-eq2BEMNofwed22}, yields the desired result.
\end{proof}

\begin{proof}[Proof of Theorem~\ref{THPS2}] The argument relies on
Theorem~\ref{0pirj09365-4g4eZXCHJKLRFG-544-NO3-COR-0oerkCC}.

Indeed, if~$q:=\min\{p,2\}$ and~$a:=\max\{p,2\}-q$, applying~\eqref{F58uuRS-2CALEREG:LE2o-3rkpjtgr94ejrfmXXs}, for all~$\kappa_j\ge0$, we see that 
\begin{equation}\label{0pirj09365-4g4eZXCHJKLRFG-544-NO3-COR-0oerkCC-00}\left(
\sum_{j=0}^{+\infty}\kappa_j^{\max\{p,2\}}\right)^{\min\{p,2\}}\le
\left(\sum_{j=0}^{+\infty}\kappa_j^{\min\{p,2\}}\right)^{\max\{p,2\}}.
\end{equation}

We will take~$\kappa_j:=2^{js}|\check\varphi_j*u|$.
In this way, when~$p\ge2$, we infer from~\eqref{0pirj09365-4g4eZXCHJKLRFG-544-NO3-COR-0oerkCC-00} that
$$ \left(\sum_{j=0}^{+\infty} 2^{jsp}|\check\varphi_j*u|^{p}\right)^{2}\le
\left(\sum_{j=0}^{+\infty} 2^{2js}|\check\varphi_j*u|^{2}\right)^{p}
$$
and therefore
\begin{eqnarray*}&&
\|u\|_{B^{s,p,p}(\R^n)}^p=\sum_{j=0}^{+\infty} 2^{jsp}\|\check\varphi_j*u\|_{L^p(\R^n)}^{p}=
\int_{\R^n}\sum_{j=0}^{+\infty} 2^{jsp}|\check\varphi_j*u(x)|^{p}\,dx\\&&\qquad\le\int_{\R^n}
\left(\sum_{j=0}^{+\infty} 2^{2js}|\check\varphi_j*u(x)|^{2}\right)^{\frac{p}2}\,dx=
\left\|\,\left(
\sum_{j=0}^{+\infty} 2^{2js}|\check\varphi_j*u|^2\right)^{\frac12} \,\right\|_{L^p(\R^n)}^p.
\end{eqnarray*}
This and the second inequality in Theorem~\ref{0pirj09365-4g4eZXCHJKLRFG-544-NO3-COR-0oerkCC} lead to
\begin{equation}\label{0pirj09365-4g4eZXCHJKLRFG-544-NO3-COR-0oerkCC-la1}
{\mbox{when~$p\ge2$, we have that~${\mathcal{L}}^p_{s} (\R^n)\subseteq B^{s,p,p}(\R^n)$
with continuous embedding.}}
\end{equation}

Similarly, if~$p\le2$, we deduce from~\eqref{0pirj09365-4g4eZXCHJKLRFG-544-NO3-COR-0oerkCC-00} that
$$ \left(\sum_{j=0}^{+\infty} 2^{2js}|\check\varphi_j*u|^{2}\right)^{p}\le
\left(\sum_{j=0}^{+\infty} 2^{jsp}|\check\varphi_j*u|^{p}\right)^{2}.
$$
On this account,
\begin{eqnarray*}&&\left\|\,\left(
\sum_{j=0}^{+\infty} 2^{2js}|\check\varphi_j*u|^2\right)^{\frac12} \,\right\|_{L^p(\R^n)}^p=
\int_{\R^n}
\left(\sum_{j=0}^{+\infty} 2^{2js}|\check\varphi_j*u(x)|^{2}\right)^{\frac{p}2}\,dx\\&&\qquad\le
\int_{\R^n}\sum_{j=0}^{+\infty} 2^{jsp}|\check\varphi_j*u(x)|^{p}\,dx=
\sum_{j=0}^{+\infty} 2^{jsp}\|\check\varphi_j*u\|_{L^p(\R^n)}^{p}=
\|u\|_{B^{s,p,p}(\R^n)}^p.
\end{eqnarray*}
Combining this estimate with the first inequality in Theorem~\ref{0pirj09365-4g4eZXCHJKLRFG-544-NO3-COR-0oerkCC}, we conclude that
\begin{equation}\label{0pirj09365-4g4eZXCHJKLRFG-544-NO3-COR-0oerkCC-la2}
{\mbox{when~$p\le2$, we have that~$B^{s,p,p}(\R^n)\subseteq {\mathcal{L}}^p_{s} (\R^n)$
with continuous embedding.}}
\end{equation}

Moreover, if~$p\ge2$ we utilize the Minkowski's Integral Inequality (see Theorem~\ref{MLAerSM:ijfKKSMdf02}, used here with a discrete measure~$\mu_1$ and exponent~$\frac{p}2\ge1$) to see that
\begin{eqnarray*}&&\left\|\,\left(\sum_{j=0}^{+\infty} 2^{2js}|\check\varphi_j*u|^2\right)^{\frac12} \,\right\|_{L^p(\R^n)}^p=\int_{\R^n} \left(\sum_{j=0}^{+\infty} 2^{2js}|\check\varphi_j*u(x)|^2\right)^{\frac{p}2}\,dx\\&&\qquad
\le\left[\sum_{j=0}^{+\infty}\left(\int_{\R^n}
2^{pjs}|\check\varphi_j*u(x)|^p
\,dx\right)^{\frac2p}\right]^{\frac{p}2}
=\left[
\sum_{j=0}^{+\infty}\left(
2^{pjs}\|\check\varphi_j*u\|^p_{L^p(\R^n)}\right)^{\frac2p}\right]^{\frac{p}2}\\&&\qquad=\left[
\sum_{j=0}^{+\infty}2^{2js}\|\check\varphi_j*u\|^2_{L^p(\R^n)}\right]^{\frac{p}2}=\|u\|_{B^{s,p,2}(\R^n)}^p.
\end{eqnarray*}
Combining this and
the first inequality in Theorem~\ref{0pirj09365-4g4eZXCHJKLRFG-544-NO3-COR-0oerkCC}, we obtain that
\begin{equation}\label{0pirj09365-4g4eZXCHJKLRFG-544-NO3-COR-0oerkCC-la3}
{\mbox{when~$p\ge2$, we have that~${\mathcal{L}}^p_{s} (\R^n)\subseteq B^{s,p,2}(\R^n)$
with continuous embedding.}}
\end{equation}

In addition, when~$2\ge p$, we use the Minkowski's Integral Inequality (see Theorem~\ref{MLAerSM:ijfKKSMdf02}, employed here with a discrete measure~$\mu_2$ and exponent~$\frac{2}p\ge1$) to see that
\begin{eqnarray*}&&
\|u\|_{B^{s,p,2}(\R^n)}^p=\left[\sum_{j=0}^{+\infty}2^{2js}\|\check\varphi_j*u\|^2_{L^p(\R^n)}\right]^{\frac{p}2}=\left[
\sum_{j=0}^{+\infty}\left(
2^{pjs}\|\check\varphi_j*u\|^p_{L^p(\R^n)}\right)^{\frac2p}\right]^{\frac{p}2}
\\&&\qquad=\left[
\sum_{j=0}^{+\infty}\left(\int_{\R^n}
|2^{js} \check\varphi_j*u(x)|^p\,dx\right)^{\frac2p}\right]^{\frac{p}2}
\le\int_{\R^n}\left( \sum_{j=0}^{+\infty}|2^{js} \check\varphi_j*u(x)|^2 \right)^{\frac{p}2}\,dx\\&&\qquad=\left\|\,\left(\sum_{j=0}^{+\infty} 2^{2js}|\check\varphi_j*u|^2\right)^{\frac12} \,\right\|_{L^p(\R^n)}^p.
\end{eqnarray*}
Hence, by
the second inequality in Theorem~\ref{0pirj09365-4g4eZXCHJKLRFG-544-NO3-COR-0oerkCC},
we conclude that in this case~$B^{s,p,2}(\R^n)\subseteq{\mathcal{L}}^p_{s} (\R^n)$
with continuous embedding.

The combination of this fact,~\eqref{0pirj09365-4g4eZXCHJKLRFG-544-NO3-COR-0oerkCC-la1},
\eqref{0pirj09365-4g4eZXCHJKLRFG-544-NO3-COR-0oerkCC-la2} and~\eqref{0pirj09365-4g4eZXCHJKLRFG-544-NO3-COR-0oerkCC-la3} yield the desired result.
\end{proof}

\begin{corollary}\label{0pirj09365-4g4eZXCHJKLRFG-544-NO1-I2L3C2O25R213t4TYFGHSJDFRAMISJMFG}
If~$s\in(0,+\infty)\setminus\N$ and~$p\in[2,+\infty)$, then
$${\mathcal{L}}^p_{s}(\R^n)\subseteq W^{s,p}(\R^n)$$
with continuous embedding.
\end{corollary}

\begin{proof}The claim is a direct consequence of Corollary~\ref{0pirj09365-4g4eZXCHJKLRFG-544-NO1-I2L3C2O25R213t4TY} and Theorem~\ref{THPS2}.
\end{proof}

We stress that the result in Corollary~\ref{0pirj09365-4g4eZXCHJKLRFG-544-NO1-I2L3C2O25R213t4TYFGHSJDFRAMISJMFG}
does not hold true when~$p\in[1,2)$ (due to
Corollary~\ref{0pirj09365-4g4eZXCHJKLRFG-544-NO1-I2L3C2O25R213t4TY} and the remark
following Theorem~\ref{THPS2} about the optimality of
the inclusion).

\section[Regularity theory in Sobolev spaces and Besov spaces]{Regularity theory in Sobolev spaces and in Besov spaces for global solutions}\label{CHAP:INT-REG-LEB-SPA}

We can now combine the regularity theory in Bessel potential spaces
put forth in Section~\ref{Ijsmiazme3sKJS03-krodgh} and the link between Bessel potential spaces and Besov spaces
to state explicitly some more regularity results.

\begin{theorem}\label{jodncla-pjdmfdmc3r}
Let~$s\in(0,1)\setminus\left\{\frac12\right\}$ and~$p\in(1, 2]$.

Let~$f\in L^p(\R^n)$, and~$u$ be a distributional solution of~$(-\Delta)^su=f$ in~$\R^n$.

Assume that~$u\in L^p(\R^n)$.

Then,~$u\in B^{2s,p,2}(\R^n)$ and
\begin{equation*} \|u\|_{ B^{2s,p,2}(\R^n)}\le
C\,\Big(\|u\|_{L^p(\R^n)}+\| f\|_{L^p(\R^n)}\Big)\end{equation*}
for some positive constant~$C$ depending only on~$n$,~$p$ and~$s$. 
\end{theorem}

\begin{proof} In this range of exponents~$p$, by Theorem~\ref{THPS2} we know that~$ \|u\|_{B^{2s,p,2}(\R^n)}
\le C \|u\|_{ {\mathcal{L}}^p_{2s}(\R^n)}$.
Hence the desired result follows from Theorem~\ref{6.33THCORNSTNE}.
\end{proof}

On the one hand, we stress that, in the range~$p\in(1,2)$, the result in Theorem~\ref{jodncla-pjdmfdmc3r}
is false\footnote{Indeed, by the remark after
Corollary~\ref{0pirj09365-4g4eZXCHJKLRFG-544-NO1-I2L3C2O25R213t4TYFGHSJDFRAMISJMFG},
we have that when~$p\in[1,2)$ there exists~$u_\star\in
{\mathcal{L}}^p_{2s}(\R^n)\setminus W^{2s,p}(\R^n)$. Thus, by~\eqref{ERFGHJN6789-09tftd90u8yhgiug8erhISB},
we have that~$(1-\Delta)^s u_\star\in L^p(\R^n)$ and then, by~\eqref{MAKlLMSTIuntYYTi},
also~$(-\Delta)^s u_\star\in {L^p(\R^n)}$ (yet,~$u_\star\not\in W^{2s,p}(\R^n)$).} if we replace the Besov space~$B^{2s,p,2}(\R^n)$ by the Sobolev space~$W^{2s,p}(\R^n)$.

On the other hand, when~$p\in[2,+\infty)$, we can also improve the Sobolev regularity theory presented in
Theorem~\ref{PKJSM0-COSK1424654856ij45gf}.

\begin{theorem}\label{PKJSM0-COSK1424654856ij45gf-SHA}
Let~$s\in(0,1)$ and~$p\in[2,+\infty)$. Let~$f\in L^p(\R^n)$, and~$u$ be a distributional solution of~$(-\Delta)^su=f$ in~$\R^n$.

Assume that~$u\in L^p(\R^n)$.

Then,~$u \in W^{2s,p}(\R^n)$ and 
\begin{equation*} \|u\|_{ W^{2s,p}(\R^n)}\le
C\,\Big(\|u\|_{L^p(\R^n)}+\| f\|_{L^p(\R^n)}\Big)\end{equation*}
for some positive constant~$C$ depending only on~$n$,~$p$ and~$s$. 
\end{theorem}

\begin{proof} We already know that this result holds true when~$s=\frac12$, owing to Corollary~\ref{0oijnhb8ygf6tfed35r7268wq-2oiejh09i28ytnbbVzUJDO03-1i40-1jhtrh9c6hv76b8706bhvc870}, hence we can suppose that~$s\ne\frac12$.
This allows us to exploit Corollary~\ref{0pirj09365-4g4eZXCHJKLRFG-544-NO1-I2L3C2O25R213t4TY}
and find that~$\|u\|_{B^{2s,p,p}(\R^n)}$ is equivalent to~$\|u\|_{W^{2s,p}(\R^n)}$.

Moreover, in the range of exponents~$p$ considered here,
Theorem~\ref{THPS2} guarantees that~$\|u\|_{B^{2s,p,p}(\R^n)}\le C\|u\|_{ {\mathcal{L}}^p_{2s}(\R^n)}$.

These observations show that
\begin{equation}\label{MojfengTNSDSS9iekrfmgR-4g4eZXCHJKLRFG-544-NO1-I2L3C2O25R213t4TY}
\|u\|_{ W^{2s,p}(\R^n)}\le C \|u\|_{ {\mathcal{L}}^p_{2s}(\R^n)}.\end{equation}
Hence the desired result follows again from Theorem~\ref{6.33THCORNSTNE}.
\end{proof}

Further results can be obtained for negative exponents and for full scales~$s\in\R$, by going deeper into distribution theory.

Also, for sharp regularity results in situations when~$f$ is not in a Lebesgue space,
and for a full-blown application of the knowledge of the 
scales of Besov-Triebel-Lizorkin spaces, see~\cite{MR3293447} and the references therein.

\chapter{Interior regularity theory in Lebesgue spaces}\label{CHAP:INTER-REG-LEB}

We now use the global regularity results developed in Chapter~\ref{CHAP6} and some appropriate cutoff argument to
present a basic interior regularity theory for fractional equations in Lebesgue spaces.
Our treatment of this theory is far from being exhaustive and we refer the interested reader to~\cite{MR2863859, MR3393266, MR3624965, MR3709046, MR3917708, MR4046549, MR4233278, MR4358140, MR4530901}
and the references therein for a thorough panorama.
\medskip

To address the local regularity theory of fractional equations in Bessel potential spaces,
the theory of pseudodifferential operators comes in very handy (a special thank to Gerd Grubb for showing us
some of the beautiful features of this remarkable methodology). 

We introduce a ``localized'' analogue of Bessel potential spaces, for which we proceed as follows.
If~$\Omega$ is an open subset of~$\R^n$, we consider~${\mathcal{L}}^p_s(\Omega)$ to be the space
of all functions~$u:\Omega\to\R$ such that there exists~$\widetilde u:\R^n\to\R$ with~$\widetilde u\in {\mathcal{L}}^p_s(\R^n)$
and~$\widetilde u=u$ in~$\Omega$ (that is, the localized Bessel potential space in~$\Omega$ collects all the restrictions
to~$\Omega$ of the functions in the Bessel potential space).

In this setting, we define
\begin{equation}\label{GSHDJKUNXRTICONSDMUHNShasnd} \|u\|_{{\mathcal{L}}^p_s(\Omega)}:=\inf_{{\widetilde u\in{\mathcal{L}}^p_s(\R^n)}\atop{\tiny{\mbox{$\widetilde u=u$ in~$\Omega$}}}}\|\widetilde u\|_{{\mathcal{L}}^p_s(\R^n)}\end{equation}
and we have:

\begin{theorem} \label{09iuhgf8i1-Deltas023o4irtTH}
Let~$\Omega$ be an open and bounded subset of~$\R^n$ and~$\Omega'\Subset\Omega$.
Let~$s\in(0,1)$,~$p\in(1,+\infty)$ and~$f\in L^p(\Omega)$.

Let~$u\in L^p(\R^n)$ be a distributional solution of
\begin{equation} \label{09iuhgf8i1-Deltas023o4irtTH-equya}
(1-\Delta)^su=f \qquad{\mbox{in~$\Omega$.}}\end{equation}

Then,~$u\in {\mathcal{L}}^p_{2s}(\Omega')$ and\footnote{Strictly speaking, an abuse of notation has been used in Theorem~\ref{09iuhgf8i1-Deltas023o4irtTH}. Namely, in principle the function~$u$ is defined in the whole of~$\R^n$,
so when we write~``$u\in {\mathcal{L}}^p_{2s}(\Omega')$''
what we actually mean is that the restriction of~$u$ to~$\Omega'$
belongs to~${\mathcal{L}}^p_{2s}(\Omega')$.
Thus, in principle it would be more correct to write~``$u\big|_{\Omega'}\in {\mathcal{L}}^p_{2s}(\Omega')$'',
but (albeit the distinction between~$u$ and its restriction~$u\big|_{\Omega'}$ is important, especially when dealing
with distributions) we used the simplified, and somewhat sloppy,
notation for the sake of simplicity.

On a related note, the meaning of~\eqref{09iuhgf8i1-Deltas023o4irtTH-equya} is that its left-hand side,
when restricted to~$\Omega$, coincides with~$f$.
Hence, likewise, a more thorough statement for~\eqref{09iuhgf8i1-Deltas023o4irtTH-equya} would be~``$\big((1-\Delta)^su\big)\big|_\Omega=f$'', but we decided to ease the notation, hoping that
this does not create any confusion.

The restriction notation~$u\big|_\Omega$ is sometimes also denoted by~$r^+u$ and~$r_\Omega u$ in the existing literature, but we will not make use of this setting here.}
\begin{equation*} \|u\|_{ {\mathcal{L}}^p_{2s}(\Omega')}\le C\Big(\| f\|_{L^p(\Omega)}+\|u\|_{L^p(\R^n)}\Big),\end{equation*} for a positive constant~$C$ depending only on~$n$,~$s$,~$p$,~$\Omega$ and~$\Omega'$. 
\end{theorem}

\begin{proof} For all~$s\in\R$, the operator~$(1-\Delta)^s$ is a pseudodifferential operator with symbol~$\big(1+4\pi^2|\xi|^2\big)^{s}$ and order~$2s$ (see
the discussion on page~\pageref{0pirj09365-4g4eZXCHJKLRFG-544-NO1-LL-eq49-09i23w-PAG-t2e32r4s}).
Actually, we will prove this result in a more general setting, possibly
\begin{equation}\label{REOLAKJSHDB2wskjdf2Shdnf09iuyhgf0e9ryfg}
{\mbox{replacing~$(1-\Delta)^s$ by any pseudodifferential operator~$L$ of order~$2s$.}}
\end{equation}
We observe that
\begin{equation}\label{bess:R:P:A:ZEff}
{\mbox{it suffices to prove the desired result with~$f$ vanishing identically.}}
\end{equation}
Indeed, suppose that the desired result holds true
for~$f:=0$ and let us see how the general result follows.
We define
$$ \widetilde{f}(x):=\begin{dcases}
f(x) & {\mbox{ if }}x\in\Omega,\\
0 &{\mbox{ otherwise}}
\end{dcases}$$
and, in the Bessel potential notation of Section~\ref{bess:R:P:A:Sec},~$$w:={\mathcal{B}}^{(s)}*\widetilde{f}.$$
In this way,~$Lw=\widetilde{f}$ in~$\R^n$ and therefore, if~$v:=u-w$, we have that
$$L v=f-\widetilde{f}=0\qquad{\mbox{in }}\Omega.$$
Consequently, since the desired result holds for~$v$, we conclude that~$\|v\|_{ {\mathcal{L}}^p_{2s}(\Omega')}\le C \|v\|_{L^p(\R^n)}$. 

Thus, recalling Corollary~\ref{CO914athcal} and the Bessel potential norm in~\eqref{ERFGHJN6789-09tftd90u8yhgiug8erhISB},
\begin{eqnarray*}
\|u\|_{ {\mathcal{L}}^p_{2s}(\Omega')}&\le& \|v\|_{ {\mathcal{L}}^p_{2s}(\Omega')}+\|w\|_{ {\mathcal{L}}^p_{2s}(\Omega')}\\&
\le& C \|v\|_{L^p(\R^n)}+\|Lw\|_{L^p(\R^n)}\\&
\le &C \Big(\|u\|_{L^p(\R^n)}+\|w\|_{L^p(\R^n)}\Big)+\|\widetilde{f}\|_{L^p(\R^n)}\\&\le&C \Big(\|u\|_{L^p(\R^n)}+\|\widetilde f\|_{L^p(\R^n)}\Big)+\|\widetilde{f}\|_{L^p(\R^n)}
\\&\le&C \Big(\|u\|_{L^p(\R^n)}+\|f\|_{L^p(\Omega)}\Big).
\end{eqnarray*}
These considerations establish~\eqref{bess:R:P:A:ZEff} and, as a consequence, we can now focus
on the proof of Theorem~\ref{09iuhgf8i1-Deltas023o4irtTH} under the additional, not restrictive, assumption that~$f:=0$.
 
To this end,
we consider open sets~$\Omega_1$,~$\Omega_2$ and~$\Omega_3$ such that
$$ \Omega'\Subset\Omega_1\Subset\Omega_2\Subset\Omega_3\Subset\Omega.$$
We take~$\psi\in C^\infty_c(\Omega_1)$ with~$\psi=1$ in~$\Omega'$ and~$\widetilde\psi\in C^\infty_c(\Omega_3)$ with~$\widetilde\psi=1$
in~$\Omega_2$. 

It is now convenient to treat the multiplication by~$\widetilde\psi$ and~$\psi$ as a pseudodifferential operator
of order zero, denoted, in the notation of~\eqref{eqTA6DE4D3DF}, by~$T_{\widetilde\psi}$ and~$T_\psi$ respectively (see Lemma~\ref{98iujrt4grLEetaom90iuy65rf-1kd-1}).

We remark that, since~$\psi$ is supported in~$\Omega$, we deduce from~\eqref{09iuhgf8i1-Deltas023o4irtTH-equya} and~\eqref{bess:R:P:A:ZEff} that
\begin{equation}\label{0pirj09365-4PTTRAadHEwgismhe0pirj09365-4}
T_{\psi}( L u)=0.
\end{equation}
Similarly, considering~$\widetilde\psi$ in lieu of~$\psi$,
\begin{equation}\label{0pirj09365-4PTTRAadHEwgismhe0pirj09365-4-bis}
T_{\widetilde\psi}( L u)=0.
\end{equation}

In light of Theorem~\ref{COMPOTHPSDO-COMMTU215ER}, the commutator
\begin{equation}\label{NSDHN-GSHJDGBGGDBJYUKMPA-23ok-1021}
T:= L\circ T_{\psi}- T_{\psi}\circ L\end{equation} is a pseudodifferential operator of order~$2s-1$ (and in particular of order zero whenever~$s\le\frac12$,
see footnote~\ref{COM78REDU-912ieyrhf-001-IMANSWJh} on page~\pageref{COM78REDU-912ieyrhf-001-IMANSWJh}).
That is, we can say that~$T$ is of order~$\sigma:=(2s-1)_+$.

This and Corollary~\ref{COMPOTHPSDO-CORCON} entail that, for every function~$\zeta$,
\begin{equation}\label{TRAINVAOMMTU215ER0op2jrf-2kkjt-0102-1} \|T \zeta\|_{L^p(\R^n)}\le C\|\zeta\|_{ {\mathcal{L}}^p_{\sigma}(\R^n)}.\end{equation}

Similarly, replacing~$\psi$ by~$\widetilde\psi$, we have that the pseudodifferential operator
$$ \widetilde{T}:= L\circ T_{\widetilde\psi}- T_{\widetilde\psi}\circ L$$
has order~$2s-1$.

Also, by~\eqref{0pirj09365-4PTTRAadHEwgismhe0pirj09365-4-bis},
\begin{equation}\label{KAJSN:jsmdLAJNS01293847r} L(\widetilde\psi u)=L(T_{\widetilde\psi}u)=
\widetilde{T}u- T_{\widetilde\psi}( Lu)=\widetilde{T}u.\end{equation}
Accordingly, a convolution with~${\mathcal{B}}^{(s)}$ (which inverts the operator~$L$, due to Lemma~\ref{L912-xpsmv}) yields that
\begin{equation}\label{KAJSN:jsmdLAJNS01293847r-bis}
\widetilde\psi u=L^{-1}(\widetilde{T}u).\end{equation}

Hence, we now distinguish two cases. When~$s\in\left(0,\frac12\right]$, we have that~$ \widetilde{T}$
has order~$2s-1\le0$.
{F}rom this and Corollary~\ref{COMPOTHPS-q0woeirjUJSDNIie00},
$$ \|\widetilde{T}\zeta\|_{{\mathcal{L}}^p_{1-2s}(\R^n)}\le C\|\zeta\|_{L^p(\R^n)}.$$
We obtain from this inequality,
the Bessel potential norm in~\eqref{ERFGHJN6789-09tftd90u8yhgiug8erhISB} and~\eqref{KAJSN:jsmdLAJNS01293847r} that
\begin{eqnarray*}
\|\widetilde\psi u\|_{ {\mathcal{L}}^p_{2s}(\R^n)}=\|L({\widetilde\psi} u)\|_{L^p(\R^n)}=\| \widetilde{T} u\|_{L^p(\R^n)}\le C\|u\|_{L^p(\R^n)}.
\end{eqnarray*}
This and the definition of the local norm in~\eqref{GSHDJKUNXRTICONSDMUHNShasnd} yield to the desired result in this case, therefore we focus now on the case~$s\in\left(\frac12,1\right)$.

In this situation, we have that~$ \widetilde{T}$
has order~$2s-1$ and therefore~$L^{-1}\circ\widetilde{T}$ has order~$-2s+(2s-1)=-1$, thanks to
Theorem~\ref{COMPOTHPSDO}.
Using this, Corollary~\ref{COMPOTHPS-q0woeirjUJSDNIie00} and~\eqref{KAJSN:jsmdLAJNS01293847r-bis} we arrive at
\begin{equation}\label{09ijhbe8rfiuj34rtfghiorguhbnPOTHPSDO}
\| T_{\widetilde{\psi} }u\|_{{\mathcal{L}}^p_{1}(\R^n)}=
\| \widetilde{\psi} u\|_{{\mathcal{L}}^p_{1}(\R^n)}=
\| L^{-1}(\widetilde{T}u)\|_{{\mathcal{L}}^p_{1}(\R^n)}\le C\|u\|_{L^p(\R^n)}.
\end{equation}

We also observe that~$\sigma=2s-1\le1$ and therefore, by Corollary~\ref{CO914athcal},
$$\|(1-\Delta)^{\frac{\sigma-1}2}\zeta\|_{L^p(\R^n)}=
\|{\mathcal{B}}^{(1-\sigma)/2}*\zeta\|_{L^p(\R^n)}\le\|\zeta\|_{L^p(\R^n)}.$$
On this account and~\eqref{09ijhbe8rfiuj34rtfghiorguhbnPOTHPSDO},
\begin{equation}\label{09ijhbe8rfiuj34rtfghiorguhbnPOTHPSDO-09oikjrfYHNmasdfk02}\begin{split}&
\| T_{\widetilde{\psi} }u\|_{{\mathcal{L}}^p_{\sigma}(\R^n)}=
\| (1-\Delta)^{\frac\sigma2}(T_{\widetilde{\psi} }u)\|_{L^p(\R^n)}=
\| (1-\Delta)^{\frac{\sigma-1}2}(1-\Delta)^{\frac12}(T_{\widetilde{\psi} }u)\|_{L^p(\R^n)}\\&\qquad\le C\| (1-\Delta)^{\frac12}(T_{\widetilde{\psi} }u)\|_{L^p(\R^n)}=C\| T_{\widetilde{\psi} }u\|_{{\mathcal{L}}^p_{1}(\R^n)}\le C\|u\|_{L^p(\R^n)},
\end{split}
\end{equation}
up to renaming~$C$.

Additionally, by Corollary~\ref{COMPOTHPS-q0woeirjUJSDNIie00-cr5tgd},
since~$1-{\widetilde\psi}$ and~$\psi$ have disjoint supports,
we know that the pseudodifferential operator$$S:=T_{1-{\widetilde\psi}}\circ L\circ T_\psi$$ has order~$-\infty$.

Consequently, by Corollary~\ref{COMPOTHPS-q0woeirjUJSDNIie00},
\begin{equation}\label{TRAINVAOMMTU215ER0op2jrf-2kkjt-0102-2} \|S\zeta\|_{L^p(\R^n)}\le C\|\zeta\|_{L^p(\R^n)}.\end{equation}

Moreover, by~\eqref{0pirj09365-4PTTRAadHEwgismhe0pirj09365-4} and~\eqref{NSDHN-GSHJDGBGGDBJYUKMPA-23ok-1021},
\begin{equation*}
L\big( T_{\psi} u\big)=
Tu+T_{\psi}\big( Lu\big)=Tu
\end{equation*}
and, as a result,
\begin{equation}\label{TRAINVAOMMTU215ER0op2jrf-2kkjt-0102-3}\begin{split}
L\big( T_{\psi} u\big)&= (1-\widetilde\psi)L\big( T_{\psi} u\big)+\widetilde\psi L\big( T_{\psi} u\big)
\\&=T_{1-\widetilde\psi}\big(L\big( T_{\psi} u\big)\big)+T_{\widetilde\psi}\big( L\big( T_{\psi} u\big)\big)
\\&=Su+T_{\widetilde\psi}\big( Tu\big)
\\&=Su+Nu+T\big( T_{\widetilde\psi}u\big)
,\end{split}
\end{equation}
where we introduced the commutator
$$ N:=T_{\widetilde\psi}\circ T-T\circ T_{\widetilde\psi}.$$

Furthermore, using again Theorem~\ref{COMPOTHPSDO-COMMTU215ER}, we know that the commutator~$N$ is of order~$\sigma-1\le0$. This, Corollary~\ref{09ijnuiathcaRn} and Corollary~\ref{COMPOTHPS-q0woeirjUJSDNIie00} entail that
\begin{equation*}
\|N\zeta\|_{L^p(\R^n)}\le
\|N \zeta\|_{{\mathcal{L}}^p_{1-\sigma}(\R^n)}\le C\|\zeta\|_{L^p(\R^n)}.
\end{equation*}

We use this inequality,~\eqref{TRAINVAOMMTU215ER0op2jrf-2kkjt-0102-1},~\eqref{TRAINVAOMMTU215ER0op2jrf-2kkjt-0102-2} and~\eqref{TRAINVAOMMTU215ER0op2jrf-2kkjt-0102-3} to find that
\begin{equation*}\begin{split}
\big\|L\big( T_{\psi} u\big)\big\|_{L^p(\R^n)}&\le C\Big( \|Su\|_{L^p(\R^n)}+ \|
Nu\|_{L^p(\R^n)}+ \|T( T_{\widetilde\psi}u)\|_{L^p(\R^n)}\Big)\\&\le C\Big( \|u\|_{L^p(\R^n)}+ \|T_{\widetilde\psi}u\|_{ {\mathcal{L}}^p_{\sigma}(\R^n)}\Big)
.\end{split}\end{equation*}
This and~\eqref{09ijhbe8rfiuj34rtfghiorguhbnPOTHPSDO-09oikjrfYHNmasdfk02} give that
$$ \|T_\psi u\|_{{\mathcal{L}}^p_{2s}(\R^n)}=\big\|L\big( T_{\psi} u\big)\big\|_{L^p(\R^n)}\le
C \|u\|_{L^p(\R^n)},$$
and the proof of the desired result is thereby complete, recalling
the definition of the local norm in~\eqref{GSHDJKUNXRTICONSDMUHNShasnd}.
\end{proof}

{F}rom Theorem~\ref{09iuhgf8i1-Deltas023o4irtTH} one can also deduce an interior regularity theory
when the fractional operator is~$(-\Delta)^s$ instead of~$(1-\Delta)^s$:

\begin{theorem}\label{09o2j3rhngrsdTh09iuhgf8i1-Deltas023o4irtTH}
Let~$\Omega$ be an open and bounded subset of~$\R^n$ and~$\Omega'\Subset\Omega$.
Let~$s\in(0,1)$,~$p\in(1,+\infty)$ and~$f\in L^p(\Omega)$.

Let~$u\in L^p(\R^n)$ be a distributional solution of
\begin{equation*}
(-\Delta)^su=f \qquad{\mbox{in~$\Omega$.}}\end{equation*}

Then,~$u\in {\mathcal{L}}^p_{2s}(\Omega')$ and
\begin{equation}\label{mngtr4AndfDGHNSHYBNSMDINmnhgrISKJND-0} \|u\|_{ {\mathcal{L}}^p_{2s}(\Omega')}\le C\Big(\| f\|_{L^p(\Omega)}+\|u\|_{L^p(\R^n)}\Big),\end{equation} for a positive constant~$C$ depending only on~$n$,~$s$,~$p$,~$\Omega$ and~$\Omega'$. 
\end{theorem}

\begin{proof} Up to a covering argument, we may suppose e.g. that~$\Omega=B_3$ and~$\Omega'=B_1$
(or any other concentric balls).
Also, without loss of generality, we may assume that
\begin{equation}\label{mngtr4AndfDGHNSHYBNSMDINmnhgrISKJND}
{\mbox{$u=0\,$ in~$\,\R^n\setminus\Omega$.}}
\end{equation}
Indeed, suppose that the desired result holds true under this additional assumption and let~$\phi\in C^\infty_c(B_3,[0,1])$
with~$\phi=1$ in~$B_2$.

We observe that (possibly in the sense of principal values and distributions)
\begin{eqnarray*}&&
\int_{\R^n} \frac{\big(v(x)-v(y)\big)\big( w(x)-w(y)\big)}{|x-y|^{n+2s}}\,dy=
\int_{\R^n} \frac{v(x)w(x)-v(x)w(y)+v(y)w(y)-v(y)w(x)}{|x-y|^{n+2s}}\,dy\\&&\qquad=
v(x)\int_{\R^n} \frac{w(x)-w(y)}{|x-y|^{n+2s}}\,dy
+w(x)
\int_{\R^n} \frac{v(x)-v(y)}{|x-y|^{n+2s}}\,dy
+
\int_{\R^n} \frac{v(y)w(y)-v(x)w(x)}{|x-y|^{n+2s}}\,dy
\end{eqnarray*}
and therefore
$$ \Big| (-\Delta)^s(vw)(x)-v(x)(-\Delta)^sw(x)-w(x)(-\Delta)^sv(x)\Big|\le
C\int_{\R^n} \frac{\big|v(x)-v(y)\big|\,\big| w(x)-w(y)\big|}{|x-y|^{n+2s}}\,dy.$$

This observation\footnote{The calculation in~\eqref{mngtr4AndfDGHNSHYBNSMDINmnhgrISKJNDpkedfmv-2dfvb}
also allows one to replace the term~$\|u\|_{L^p(\R^n)}$
in the right-hand side of~\eqref{mngtr4AndfDGHNSHYBNSMDINmnhgrISKJND-0} with other norms, such as
$$ \|u\|_{L^p(\Omega)}+\int_{\R^n}\frac{|u(y)|}{1+|y|^{n+2s}}\,dy.$$
For this, it suffices first to prove~\eqref{mngtr4AndfDGHNSHYBNSMDINmnhgrISKJND-0} as it is and then to apply a cutoff
as in~\eqref{mngtr4AndfDGHNSHYBNSMDINmnhgrISKJNDpkedfmv-2dfvb}. Here, for the sake of simplicity,
we stick to the simpler global norm of Lebesgue type.}
gives that, for every~$x\in B_{3/2}$,
\begin{equation}\label{mngtr4AndfDGHNSHYBNSMDINmnhgrISKJNDpkedfmv-2dfvb}
\begin{split}
\big|(-\Delta)^s (\phi u)(x)\big|&\le |(-\Delta)^s\phi(x)|\,|u(x)|+\phi(x)|f(x)|+
C\int_{\R^n} \frac{\big|u(x)-u(y)\big|\,\big| \phi(x)-\phi(y)\big|}{|x-y|^{n+2s}}\,dy\\
&\le C\left(|u(x)|+|f(x)|+\int_{\R^n\setminus B_2} \frac{\big|u(x)-u(y)\big|\,\big| 1-\phi(y)\big|}{|x-y|^{n+2s}}\,dy\right)\\
&\le C\left(|u(x)|+|f(x)|+\int_{\R^n\setminus B_2} \frac{|u(x)|+|u(y)|}{|y|^{n+2s}}\,dy\right).
\end{split}\end{equation}

On this account,
\begin{eqnarray*}&&
\|(-\Delta)^s (\phi u)\|_{L^p(B_{3/2})}\\&\le& 
C\left[\|u\|_{L^p(B_{3/2})}+\|f\|_{L^p(B_{3/2})}+
C\left( \int_{B_{3/2}}\left(\int_{\R^n\setminus B_2} \frac{|u(x)|+|u(y)|}{|y|^{n+2s}}\,dy\right)^p \,dx
\right)^{\frac1p}
\right]\\&\le& 
C\left[\|u\|_{L^p(B_{3/2})}+\|f\|_{L^p(B_{3/2})}+
C\left( \int_{B_{3/2}}\left(|u(x)|+\int_{\R^n\setminus B_2} \frac{|u(y)|}{|y|^{n+2s}}\,dy\right)^p \,dx
\right)^{\frac1p}
\right]\\&\le& 
C\left[\|u\|_{L^p(B_{3/2})}+\|f\|_{L^p(B_{3/2})}+
C\left( \int_{B_{3/2}}\left(|u(x)|^p+\int_{\R^n\setminus B_2} |u(y)|^p\,dy\right) \,dx
\right)^{\frac1p}
\right]\\&\le& 
C\Big(\|f\|_{L^p(B_{3/2})}+\|u\|_{L^p(\R^n)}\Big).
\end{eqnarray*}
Hence, if the desired result is true under the additional hypothesis in~\eqref{mngtr4AndfDGHNSHYBNSMDINmnhgrISKJND}, we can apply it to~$\phi u$ and deduce that
$$ \|\phi u\|_{ {\mathcal{L}}^p_{2s}(B_1)}\le C\Big(\| (-\Delta)^s(\phi u)\|_{L^p(B_{3/2})}+\|u\|_{L^p(\R^n)}\Big)
\le C\Big(\| f\|_{L^p(B_{3/2})}+\|u\|_{L^p(\R^n)}\Big).$$
Since, by~\eqref{GSHDJKUNXRTICONSDMUHNShasnd}, we know that~$\|u\|_{ {\mathcal{L}}^p_{2s}(B_1)}\le\|\phi u\|_{ {\mathcal{L}}^p_{2s}(B_1)}$, we have obtained the desired result for~$u$.

In view of these considerations, we can now focus on the proof of Theorem~\ref{09o2j3rhngrsdTh09iuhgf8i1-Deltas023o4irtTH} under the additional assumption in~\eqref{mngtr4AndfDGHNSHYBNSMDINmnhgrISKJND}.

Thus, now we apply Theorem~\ref{SQUEZ} and we follow the notation there.
In particular,
\begin{eqnarray*}
\|Su\|_{L^p(\Omega)}\le C\|Su\|_{L^\infty(\R^n)}\le C\|u\|_{L^1(\R^n)}=C\|u\|_{L^1(\Omega)}\le
C\|u\|_{L^p(\Omega)},
\end{eqnarray*}
due to~\eqref{mngtr4AndfDGHNSHYBNSMDINmnhgrISKJND}.

Hence, by Theorem~\ref{09iuhgf8i1-Deltas023o4irtTH} and~\eqref{REOLAKJSHDB2wskjdf2Shdnf09iuyhgf0e9ryfg}, we know that
\begin{equation*}
\begin{split}
\| u\|_{ {\mathcal{L}}^p_{2s}(\Omega')}&\le C\Big(\| T_a u\|_{L^p(\Omega)}+\|u\|_{L^p(\R^n)}\Big)
\\&= C\Big(\| (-\Delta)^s u-Su\|_{L^p(\Omega)}+\|u\|_{L^p(\R^n)}\Big)\\&\le
C\Big(\| (-\Delta)^s u\|_{L^p(\Omega)}+\|Su\|_{L^p(\Omega)}+\|u\|_{L^p(\R^n)}\Big)
\\&\le C\Big(\|f\|_{L^p(\Omega)}+\|u\|_{L^p(\R^n)}\Big).\qedhere
\end{split}\end{equation*}
\end{proof}

\begin{corollary}\label{ILCOLPSQRTPASKLX-k2eojr3lgrbbb9wiefh}
Let~$\Omega$ be an open and bounded subset of~$\R^n$ and~$\Omega'\Subset\Omega$.
Let~$p\in(1,+\infty)$,~$f\in L^p(\Omega)$
and~$u\in L^p(\R^n)$ be a distributional solution of~$\sqrt{1-\Delta}\,u=f$ in~$\Omega$.

Then,~$u\in W^{1,p}(\Omega')$ and
\begin{equation*} \|u\|_{ W^{1,p}(\Omega')}\le C\Big(\| f\|_{L^p(\Omega)}+\|u\|_{L^p(\R^n)}\Big),\end{equation*}
for some positive constant~$C$ depending only on~$n$,~$p$,~$\Omega$ and~$\Omega'$.
\end{corollary}

\begin{proof} By Theorem~\ref{TH0-104-okn5KMD3}, used here with~$s:=\frac12$,
the norms in~$ {\mathcal{L}}^p_{1}(\R^n)$ and~$W^{1,p}(\R^n)$ are equivalent.

Thus, by~\eqref{GSHDJKUNXRTICONSDMUHNShasnd} and Theorem~\ref{09iuhgf8i1-Deltas023o4irtTH},
\begin{equation*}\begin{split}&
\| u\|_{W^{1,p}(\Omega)}=
\inf_{{\widetilde u\in W^{1,p}(\R^n)}\atop{\tiny{\mbox{$\widetilde u=u$ in~$\Omega'$}}}}\|\widetilde u\|_{W^{1,p}(\Omega)}\le
\inf_{{\widetilde u\in W^{1,p}(\R^n)}\atop{\tiny{\mbox{$\widetilde u=u$ in~$\Omega'$}}}}\|\widetilde u\|_{W^{1,p}(\R^n)}\le
C\inf_{{\widetilde u\in{\mathcal{L}}^p_{1}(\R^n)}\atop{\tiny{\mbox{$\widetilde u=u$ in~$\Omega'$}}}}\|\widetilde u\|_{{\mathcal{L}}^p_{1}(\R^n)}
\\&\qquad\qquad\qquad\qquad\qquad=
C\|u\|_{ {\mathcal{L}}^p_{1}(\Omega')}\le C\Big(\| f\|_{L^p(\Omega)}+\|u\|_{L^p(\R^n)}\Big).\qedhere\end{split}\end{equation*}
\end{proof}

\begin{corollary}\label{VRIVAJMSALCNMSDCOSMCUIJE5678}
Let~$\Omega$ be an open and bounded subset of~$\R^n$ and~$\Omega'\Subset\Omega$.
Let~$p\in(1,+\infty)$,~$f\in L^p(\Omega)$
and~$u\in L^p(\R^n)$ be a distributional solution of~$\sqrt{-\Delta}\,u=f$ in~$\Omega$.

Then,~$u\in W^{1,p}(\Omega')$ and
\begin{equation*} \|u\|_{ W^{1,p}(\Omega')}\le C\Big(\| f\|_{L^p(\Omega)}+\|u\|_{L^p(\R^n)}\Big),\end{equation*}
for some positive constant~$C$ depending only on~$n$,~$p$,~$\Omega$ and~$\Omega'$.
\end{corollary}

\begin{proof} The argument is identical to that presented in the proof of Corollary~\ref{ILCOLPSQRTPASKLX-k2eojr3lgrbbb9wiefh}, replacing the use of Theorem~\ref{09iuhgf8i1-Deltas023o4irtTH} there with
that of Theorem~\ref{09o2j3rhngrsdTh09iuhgf8i1-Deltas023o4irtTH} here.
\end{proof}

As in~\eqref{GSHDJKUNXRTICONSDMUHNShasnd}, we define a localized norm for Besov spaces by the relation
\begin{equation*} \|u\|_{B^{s,p,q}(\Omega)}:=\inf_{{\widetilde u\in B^{s,p,q}(\R^n)}\atop{\tiny{\mbox{$\widetilde u=u$ in~$\Omega$}}}}\|\widetilde u\|_{B^{s,p,q}(\R^n)}\end{equation*}
and we have a local regularity theory in Besov spaces as follows:

\begin{theorem}
Let~$\Omega$ be an open and bounded subset of~$\R^n$ and~$\Omega'\Subset\Omega$.

Let~$s\in(0,1)\setminus\left\{\frac12\right\}$ and~$p\in(1, 2]$.

Let~$f\in L^p(\Omega)$, and~$u\in L^p(\R^n)$ be a distributional solution of~$(-\Delta)^su=f$ in~$\Omega$.

Then,~$u\in B^{2s,p,2}(\Omega')$ and
\begin{equation*} \|u\|_{ B^{2s,p,2}(\Omega')}\le
C\,\Big(\|f\|_{L^p(\Omega)}+\|u\|_{L^p(\R^n)}\Big),\end{equation*}
for some positive constant~$C$ depending only on~$n$,~$p$,~$s$,~$\Omega$ and~$\Omega'$. 
\end{theorem}

\begin{proof} In this range of exponents~$p$, by Theorem~\ref{THPS2} we know that~$ \|\zeta\|_{B^{2s,p,2}(\R^n)}
\le C \|\zeta\|_{ {\mathcal{L}}^p_{2s}(\R^n)}$, for all functions~$\zeta$.

Consequently, by Theorem~\ref{09o2j3rhngrsdTh09iuhgf8i1-Deltas023o4irtTH},
\begin{equation*} \begin{split}&\|u\|_{B^{2s,p,2}(\Omega')}=
\inf_{{\widetilde u\in B^{2s,p,2}(\R^n)}\atop{\tiny{\mbox{$\widetilde u=u$ in~$\Omega'$}}}}
\|\widetilde u\|_{B^{2s,p,2}(\R^n)}\le
C\inf_{{\widetilde u\in{\mathcal{L}}^p_{2s}(\R^n)}\atop{\tiny{\mbox{$\widetilde u=u$ in~$\Omega'$}}}}\|\widetilde u\|_{{\mathcal{L}}^p_{2s}(\R^n)}
\\&\qquad\qquad=C\|u\|_{{\mathcal{L}}^p_{2s}(\Omega')}\le C\Big(\| f\|_{L^p(\Omega)}+\|u\|_{L^p(\R^n)}\Big).
\qedhere
\end{split}\end{equation*}
\end{proof}

\begin{theorem}
Let~$\Omega$ be an open and bounded subset of~$\R^n$ and~$\Omega'\Subset\Omega$.

Let~$s\in(0,1)$ and~$p\in[2,+\infty)$. Let~$f\in L^p(\Omega)$, and~$u\in L^p(\R^n)$ be a distributional solution of~$(-\Delta)^su=f$ in~$\Omega$.

Then,~$u \in W^{2s,p}(\R^n)$ and 
\begin{equation*} \|u\|_{ W^{2s,p}(\Omega')}\le
C\,\Big(\|u\|_{L^p(\R^n)}+\| f\|_{L^p(\Omega)}\Big)\end{equation*}
for some positive constant~$C$ depending only on~$n$,~$p$,~$s$,~$\Omega$ and~$\Omega'$. 
\end{theorem}

\begin{proof} We already know that this result holds true when~$s=\frac12$, owing to Corollary~\ref{VRIVAJMSALCNMSDCOSMCUIJE5678}, hence we can suppose that~$s\ne\frac12$.

Moreover, in the range of exponents~$p$ considered here,
we can apply~\eqref{MojfengTNSDSS9iekrfmgR-4g4eZXCHJKLRFG-544-NO1-I2L3C2O25R213t4TY}, which,
together with Theorem~\ref{09o2j3rhngrsdTh09iuhgf8i1-Deltas023o4irtTH}, entails that
\begin{equation*} \begin{split}&
\|u\|_{W^{2s,p}(\Omega')}=
\inf_{{\widetilde u\in W^{2s,p}(\R^n)}\atop{\tiny{\mbox{$\widetilde u=u$ in~$\Omega'$}}}}\|\widetilde u\|_{W^{2s,p}(\Omega')}
\le \inf_{{\widetilde u\in W^{2s,p}(\R^n)}\atop{\tiny{\mbox{$\widetilde u=u$ in~$\Omega'$}}}}\|\widetilde u\|_{W^{2s,p}(\R^n)}
\le
C\inf_{{\widetilde u\in{\mathcal{L}}^p_{2s}(\R^n)}\atop{\tiny{\mbox{$\widetilde u=u$ in~$\Omega'$}}}}\|\widetilde u\|_{{\mathcal{L}}^p_{2s}(\R^n)}
\\&\qquad\qquad\qquad\qquad=C\|u\|_{ {\mathcal{L}}^p_{2s}(\Omega')}\le C\Big(\| f\|_{L^p(\Omega)}+\|u\|_{L^p(\R^n)}\Big).\qedhere\end{split}\end{equation*}
\end{proof}

\begin{appendix}

\chapter{Minkowski's Integral Inequality}

We recall here a useful variation of the classical Minkowski's Inequality:

\begin{theorem}\label{MLAerSM:ijfKKSMdf02} Let~$p\ge1$. For~$j\in\{1,2\}$, let~$\mu_j$ be
a~$\sigma$-finite measure on the space~$S_j$.

Let~$F: S_1 \times S_2\to \R$ be a measurable function.

Then,
\begin{equation}\label{MLAerSM:ijfKKSMdf02:RHDN} \left(\int _{S_{2}}\left|\int _{S_{1}}F(x,y)\,d\mu _{1}(x)\right|^{p}\,d\mu _{2}(y)\right)^{\frac{1}{p}}\leq \int _{S_{1}}\left(\int _{S_{2}}|F(x,y)|^{p}\,d\mu _{2}(y)\right)^{\frac{1}{p}}d\mu _{1}(x).\end{equation}
\end{theorem}


\begin{proof} If~$p=1$ the claim follows from the Triangle Inequality and Tonelli's Theorem,
hence we can suppose that~$p>1$.
We use the dual approach to Lebesgue spaces, namely the fact that
$$ \|f\|_{L^p(S_j,\mu_j)}=\sup_{{g\in L^q(S_j,\mu_j)}\atop{\|g\|_{L^q(S_j,\mu_j)}=1}}\int_{S_j}f(\zeta)\,g(\zeta)\,d\mu_j(\zeta),\qquad{\mbox{ where }}\;q:=\frac{p}{p-1} .$$
Therefore,
\begin{equation}\label{MLAerSM:ijfKKSMdf02:RHDN-CSBN}
\begin{split}&
\left(\int _{S_{2}}\left|\int _{S_{1}}F(x,y)\,d\mu _{1}(x)\right|^{p}\,d\mu _{2}(y)\right)^{\frac{1}{p}}=
\left\| \int _{S_{1}}F(x,\cdot)\,d\mu _{1}(x)\right\|_{L^p(S_2,\mu_2)}\\&\qquad=
\sup_{{g\in L^q(S_2,\mu_2)}\atop{\|g\|_{L^q(S_2,\mu_2)}=1}}\int_{S_2}
\left(\int _{S_{1}}F(x,y)\,d\mu _{1}(x)\right)\,g(y)\,d\mu_2(y).
\end{split}\end{equation}

We can also assume that the right-hand side of~\eqref{MLAerSM:ijfKKSMdf02:RHDN} is finite, otherwise we are done. As a consequence, if~$g\in L^q(S_2,\mu_2)$, H\"older's Inequality yields that
\begin{eqnarray*}&&\int_{S_1}
\left(\int _{S_2}|F(x,y)|\,|g(y)|\,d\mu_2(y)\right)\,d\mu _{1}(x)\\&&\qquad\le\int_{S_1}
\left(
\left(\int _{S_2}|F(x,y)|^p\,d\mu_2(y)\right)^{\frac1p}
\left(\int _{S_2}|g(y)|^q\,d\mu_2(y)\right)^{\frac1q}
\right)\,d\mu _{1}(x)\\&&\qquad\le\|g\|_{L^q(S_2,\mu_2)}
\int _{S_{1}}\left(\int _{S_{2}}|F(x,y)|^{p}\,d\mu _{2}(y)\right)^{\frac{1}{p}}d\mu _{1}(x),
\end{eqnarray*}
which is finite.

We are thus allowed to use Tonelli's Theorem in the last term of~\eqref{MLAerSM:ijfKKSMdf02:RHDN-CSBN}
and obtain
\begin{eqnarray*}&&
\left(\int _{S_{2}}\left|\int _{S_{1}}F(x,y)\,d\mu _{1}(x)\right|^{p}\,d\mu _{2}(y)\right)^{\frac{1}{p}}=
\sup_{{g\in L^q(S_2,\mu_2)}\atop{\|g\|_{L^q(S_2,\mu_2)}=1}}\int_{S_1}
\left(\int _{S_2}F(x,y)\,g(y)\,d\mu_2(y)\right)\,d\mu _{1}(x)\\&&\qquad
\le\sup_{{g\in L^q(S_2,\mu_2)}\atop{\|g\|_{L^q(S_2,\mu_2)}=1}}\int_{S_1}
\left(\int _{S_2}|F(x,y)|^p\,d\mu_2(y)\right)^{\frac1p}
\left(\int _{S_2}|g(y)|^q\,d\mu_2(y)\right)^{\frac1q}
\,d\mu _{1}(x)\\&&\qquad=
\int _{S_{1}}\left(\int _{S_{2}}|F(x,y)|^{p}\,d\mu _{2}(y)\right)^{\frac{1}{p}}d\mu _{1}(x),
\end{eqnarray*}which establishes~\eqref{MLAerSM:ijfKKSMdf02:RHDN}.
\end{proof}

We observe that the~$\sigma$-finiteness assumption in Theorem~\ref{MLAerSM:ijfKKSMdf02} is
used in its proof to rely on Tonelli's Theorem
and cannot be, in general, removed: for instance, if~$S_1=S_2:=\R$,~$\mu_1$ is the counting measure and~$\mu_2$ is the Lebesgue measure, taking
$$ F(x,y):=\begin{dcases} 1&{\mbox{ if }}x=y,\\ 0&{\mbox{ otherwise,}}\end{dcases}$$
we have that, for every~$y\in\R$,
$$ \int _{S_{1}}F(x,y)\,d\mu _{1}(x)=\int _{\{y\}}\,d\mu _{1}(x)=1$$
and therefore
$$ \int _{S_{2}}\left|\int _{S_{1}}F(x,y)\,d\mu _{1}(x)\right|^{p}\,d\mu _{2}(y)=
\int _{\R} \,dy=+\infty.$$
However, for every~$x\in\R$,
$$ \int _{S_{2}}|F(x,y)|^{p}\,d\mu _{2}(y)=\int _{\{x\}}\,dy=0$$
and thus
$$ \int _{S_{1}}\left(\int _{S_{2}}|F(x,y)|^{p}\,d\mu _{2}(y)\right)^{\frac{1}{p}}d\mu _{1}(x)=0,$$
highlighting the importance of the~$\sigma$-finiteness assumption in Theorem~\ref{MLAerSM:ijfKKSMdf02}.

For further observations on the~$\sigma$-finiteness assumption in classical results of measure theory
and for different possible relaxations of such hypothesis, see~\cite{MR313470, MR328016}.

See also~\cite{MR1681462} for further information on the Minkowski's Integral Inequality.

\chapter{Combinatorial identities}

In this appendix we collect some useful combinatorial identities.

\begin{lemma}[Pascal's triangle]
For every~$j$, $N\in\N$ with~$j\leq N$ it holds that
\begin{align}\label{903ur2jipof}
\binom{N}{j}=\binom{N-1}{j-1}+\binom{N-1}{j}.
\end{align}
\end{lemma}

\begin{proof}
Multiplying both sides of the following identity
\begin{align*}
\frac{N}{(N-j)j}=\frac1{N-j}+\frac1j
\end{align*}
by~$(N-1)!/((N-j-1)!(j-1)!)$ gives the desired result.
\end{proof}

\begin{lemma}
For every~$m\in\N\cap[2,+\infty)$ and~$k\in\{-m+2,\ldots,m-2\}$ it holds that
\begin{align}\label{fihenklwc}
\binom{2m}{m-k}=
\binom{2m-2}{m-2-k}+2\binom{2m-2}{m-1-k}+\binom{2m-2}{m-k}.
\end{align}
\end{lemma}

\begin{proof}
Multiplying both sides of the following identity 
\begin{align*}
& \frac{2m(2m-1)}{(m+k+1)(m+k)(m-k)(m-k-1)}\\
& =\frac1{(m+k)(m+k-1)}+\frac2{(m+k-1)(m-k-1)}+\frac1{(m-k)(m-k-1)}
\end{align*}
by~$(2m-2)!/((m+k-2)!(m-k-2)!)$ gives the desired result.
\end{proof}

\begin{lemma}[Binomial formula]
For every~$a$, $b\in\R$ and~$N\in\N$ it holds that
\begin{align}\label{binom-form}
(a+b)^N=\sum_{j=0}^N\binom{N}{j}a^jb^{N-j}.
\end{align}
\end{lemma}

\begin{proof}
We will prove formula~\eqref{binom-form} by induction on~$N$.
For~$N=0,1$, the claim is obvious. 

For~$N\geq 2$, we suppose that the claim is true for~$N-1$ and we prove it for~$N$. The precise computations are as follows:
we use the inductive hypothesis to write that
\begin{align*}
&(a+b)^N =
(a+b)(a+b)^{N-1}=(a+b)\sum_{j=0}^{N-1}\binom{N-1}{j}a^jb^{N-1-j} \\
&\qquad=
\sum_{j=0}^{N-1}\binom{N-1}{j}a^{j+1}b^{N-1-j}+\sum_{j=0}^{N-1}\binom{N-1}{j}a^jb^{N-j} \\
&\qquad=
\sum_{j=1}^{N}\binom{N-1}{j-1}a^jb^{N-j}+\sum_{j=0}^{N-1}\binom{N-1}{j}a^jb^{N-j} \\
&\qquad=
a^N+b^N+
\sum_{j=1}^{N-1}\left[
\binom{N-1}{j-1}+\binom{N-1}{j}\right]a^jb^{N-j}.
\end{align*}
Then, it suffices to use~\eqref{903ur2jipof} to conclude the induction step.
\end{proof}

\begin{proposition}
For every~$m\in\N$ it holds that
\begin{align}\label{foot-combino}
\sum_{k=-m}^m(-1)^k\binom{2m}{m-k}=0.
\end{align}
\end{proposition}

\begin{proof}
This can be proved by simply using the binomial formula~\eqref{binom-form}
after a translation of indices. Precisely,
\begin{eqnarray*}
&&\sum_{k=-m}^m(-1)^k\binom{2m}{m-k}=\sum_{j=0}^{2m}\binom{2m}{j}(-1)^{m-j}
\\&&\qquad=(-1)^m\sum_{j=0}^{2m}\binom{2m}{j}(-1)^{2m-j}
=(-1)^m(1-1)^{2m}=0,
\end{eqnarray*} as desired.
\end{proof}

\begin{lemma}\label{lemma:B5peotrugb}
Let~$m$, $N\in\N$, with~$N$ even and such that~$N\leq 2m-1$.

Let
\begin{equation*}
L(m,N):=\sum_{k=-m}^m(-1)^k\binom{2m}{m-k}k^N.
\end{equation*}

Then,
$$ L(m,N)=-2\sum_{j=1}^{N/2}\binom{N}{2j} L(m-1, N-2j).$$
\end{lemma}

\begin{proof} We exploit~\eqref{fihenklwc} to see that
\begin{align*}
& \sum_{k=-m}^m(-1)^k\binom{2m}{m-k}k^N \\
&=
(-1)^{-m}(-m)^N+(-1)^{-m+1}2m(-m+1)^N+(-1)^{m-1}2m(m-1)^N+(-1)^m m^N \\
&\qquad
+\sum_{k=-m+2}^{m-2}(-1)^k\left[\binom{2m-2}{m-2-k}+2\binom{2m-2}{m-1-k}+\binom{2m-2}{m-k}\right]k^N\\
&=
\sum_{k=-m}^{m-2}(-1)^k\binom{2m-2}{m-2-k}k^N
+2\sum_{k=-m+1}^{m-1}(-1)^k\binom{2m-2}{m-1-k}k^N 
+\sum_{k=-m+2}^{m}(-1)^k\binom{2m-2}{m-k}k^N .
\end{align*}
Thus, a translation of indices gives that
\begin{equation}\label{ufewi8743qazxertgr4i3yt80}\begin{split}
& \sum_{k=-m}^m(-1)^k\binom{2m}{m-k}k^N \\
&=
-\sum_{k=-m+1}^{m-1}(-1)^k\binom{2m-2}{m-1-k}(k-1)^N
+2\sum_{k=-m+1}^{m-1}(-1)^k\binom{2m-2}{m-1-k}k^N \\
&\qquad
-\sum_{k=-m+1}^{m-1}(-1)^k\binom{2m-2}{m-1-k}(k+1)^N
.
\end{split}\end{equation}

We now use the Binomial Theorem to write that
\begin{eqnarray*}&& (k-1)^N= \sum_{j=0}^N \binom{N}{j}k^{N-j}(-1)^j 
\\{\mbox{and }}&& (k+1)^N= \sum_{j=0}^N \binom{N}{j}k^{N-j} .
\end{eqnarray*}
Consequently,
\begin{eqnarray*}
(k-1)^N+(k+1)^N&=&\sum_{j=0}^N \binom{N}{j}k^{N-j}\Big((-1)^j+1\Big)\\&=&2
\sum_{j=0}^{N/2} \binom{N}{2j}k^{N-2j}.
\end{eqnarray*}
As a result of this, we find that
$$ 2k^N -(k-1)^N-(k+1)^N= -2\sum_{j=1}^{N/2} \binom{N}{2j}k^{N-2j}.
$$

Plugging this information into~\eqref{ufewi8743qazxertgr4i3yt80}, we conclude that
\begin{equation*}\begin{split}
& \sum_{k=-m}^m(-1)^k\binom{2m}{m-k}k^N \\
&=-2 \sum_{j=1}^{N/2} \binom{N}{2j} \sum_{k=-m+1}^{m-1}(-1)^k\binom{2m-2}{m-1-k}k^{N-2j},
\end{split}\end{equation*}
which gives the desired result.
\end{proof}

\begin{theorem}
Let~$m$, $N\in\N$.

If either~$N$ is odds or~$N$ is even and~$N\leq 2m-1$, it holds that
\begin{align}\label{app-combinissima}
\sum_{k=-m}^m(-1)^k\binom{2m}{m-k}k^N=0.
\end{align}
\end{theorem}

\begin{proof}
If~$N\in\N$ is odd, we use the following observation
\begin{align*}
\binom{2m}{m+j}=\binom{2m}{m-j}
\qquad\text{for any }m\in\N \text{ and }j\in\{-m,\ldots,m\}
\end{align*}
to conclude that
\begin{eqnarray*}&&
-\sum_{k=-m}^m(-1)^k\binom{2m}{m-k}k^N
=
\sum_{k=-m}^m(-1)^k\binom{2m}{m-k}(-k)^N
\\&&\qquad=
\sum_{j=-m}^m(-1)^j\binom{2m}{m+j}j^N
=
\sum_{j=-m}^m(-1)^j\binom{2m}{m-j}j^N,
\end{eqnarray*}
which gives the desired result in this case.

If~$N\in\N$ is even and~$N\leq 2m-1$, 
we argue by induction on~$N$. If~$N=0$, the desired formula follows from~\eqref{foot-combino}.

If~$N\ge2$, we suppose that the formula holds true up to~$N-2$ and we prove it for~$N$. For this,
we exploit Lemma~\ref{lemma:B5peotrugb}
and we see that
\begin{eqnarray*}
\sum_{k=-m}^m(-1)^k\binom{2m}{m-k}k^N&=&L(m,N)
\\& =&-2\sum_{j=1}^{N/2}\binom{N}{2j} L(m-1, N-2j)\\&
=& -2\sum_{j=1}^{N/2}\binom{N}{2j} \sum_{k=-m+1}^{m-1}(-1)^k\binom{2m-2}{m-1-k}k^{N-2j},
\end{eqnarray*} which equals zero, thanks to the inductive hypothesis, as desired.
\end{proof}

\chapter{Special functions}

In this appendix we collect some useful results related to the Gamma function and the hypergeometric function, which belong to the class
of ``special functions''. We will limit ourselves to definitions and properties that are used throughout this monograph
and we refer the reader to~\cite{MR1688958} for historical insights, a thorough treatment and references on this topic.

\section{The Gamma function}\label{sec:gamma}

The Gamma function is the most common extension of the factorial function to complex numbers.
This problem has caught the attention of
several prominent mathematicians, including Bernoulli, Goldbach, Stirling, Euler, Gauss, Weierstrass and Legendre
(and it seems that the notation~$\Gamma$ is due to Legendre).
As such, the Gamma function has a long history and,
in the words of Philip J. Davis, ``each generation has found something of interest to say about the gamma function. Perhaps the next generation will also''.
See~\cite{MR106810} for a historical introduction to the Gamma function.
\medskip

For our purposes, we will only need the definition of the Gamma function on positive real numbers, therefore
we will not delve into the complicated aspects related to its extension to complex numbers.

\begin{definition}
The function~$\Gamma:(0,+\infty)\to \R$ defined as
\begin{equation}\label{gamma-def}
\Gamma(a):=\int_0^{+\infty}e^{-t}t^{a-1}\;dt
\end{equation}
is called the (Euler) Gamma function.
\end{definition}

One of the main features of the Gamma function is that it extends the factorial to positive real numbers, according to the
following observation:

\begin{lemma}
We have that
\begin{eqnarray}&&\Gamma(a+1) = a\Gamma(a) \qquad \text{for any }a>0\label{gamma-recursive}
 \\
{\mbox{and }}&&\Gamma(a) = (a-1)!\qquad \text{for any }a\in\N. \label{gamma-integer}
\end{eqnarray}
\end{lemma}

\begin{proof}
The recursive property expressed by~\eqref{gamma-recursive}
follows from an integration by parts:
\begin{align*}
\Gamma(a+1)=\int_0^{+\infty}e^{-t}t^a\;dt=-e^{-t}t^a\Big|_{t=0}^{+\infty}+a\int_0^{+\infty}e^{-t}t^{a-1}\;dt=0+a\Gamma(a).
\end{align*}

Also, we point out that
$$ \Gamma(1)=\int_0^{+\infty}e^{-t}\;dt=1. $$
This fact, together with~\eqref{gamma-recursive}, gives
the identity in~\eqref{gamma-integer}.
\end{proof}

Moreover,~\eqref{gamma-recursive} can also be used to extend the definition of~$\Gamma$ 
to negative non-integer numbers.

\begin{definition}
Given~$a>0$, with~$a\not\in\N$, we define
\begin{align}\label{gamma-recursive2}
\Gamma(-a):=\frac{\Gamma(k-a)}{(k-a-1)(k-a-2)\cdots(1-a)(-a)}
\qquad\text{for }k\in\N,\ k>a.
\end{align}
\end{definition}

The function~$\Gamma$ turns out to be useful to give a concise representation of some integrals:
we give in particular the following examples.

\begin{lemma}
For any~$a$, $b>0$ it holds that
\begin{align}\label{beta-identity}
\frac{\Gamma(a)\,\Gamma(b)}{\Gamma(a+b)}=\int_0^1 v^{a-1}{(1-v)}^{b-1}\;dv.
\end{align}
In particular,
\begin{align}\label{gamma-one-half}
\Gamma\left(\frac12\right)=\sqrt\pi.
\end{align}
\end{lemma}

\begin{proof}
Let us start from the definition of~$\Gamma$ by writing
\begin{align*}
\Gamma(a)\Gamma(b)=
\left(\int_0^{+\infty} e^{-t}t^{a-1}\;dt\right)
\left(\int_0^{+\infty}e^{-r}r^{b-1}\;dr\right)
=\int_0^{+\infty}\int_0^{+\infty}e^{-t-r}t^{a-1}r^{b-1}\;dt\;dr.
\end{align*}
We now apply the change of variables
\[
(t,r):=\Big(uv,u(1-v)\Big),\qquad {\mbox{with}}\quad
D(t,r)=\left(\begin{matrix}
v & u \\ 
1-v & -u
\end{matrix}\right)
\]
which yields
\begin{align*}
\Gamma(a)\Gamma(b)=\int_0^{+\infty}\int_0^1 e^{-u}u^{a+b-2}v^{a-1}{(1-v)}^{b-1} \;dv\;u\;du
=\Gamma(a+b)\int_0^1 v^{a-1}{(1-v)}^{b-1}\;dv,
\end{align*}
which establishes~\eqref{beta-identity}.

A consequence of~\eqref{beta-identity} is obtained by choosing~$a=b=\frac12$ and recalling that~$\Gamma(1)=1$.
In this way, we see that
\begin{align*}
\left(\Gamma\left(\frac12\right)\right)^2=\int_0^1\frac{dv}{\sqrt{v(1-v)}}=-2\arcsin\sqrt{1-v}\,\Big|_{v=0}^1=\pi,
\end{align*} thus proving~\eqref{gamma-one-half}.
\end{proof}

Another consequence of~\eqref{beta-identity} is the so-called Legendre duplication formula.
\begin{lemma}
For any~$a>0$, it holds that
\begin{align}\label{gamma-dupli}
\Gamma(a)\Gamma\left(a+\frac12\right)=2^{1-2a}\sqrt\pi\,\Gamma(2a).
\end{align}
\end{lemma}

\begin{proof}
We exploit~\eqref{beta-identity} twice, with~$a=b$ and~$b=\frac12$, to write that
\begin{eqnarray*}&&
\big(\Gamma(a)\big)^2=\Gamma(2a)\int_0^1 v^{a-1}(1-v)^{a-1}\;dv\\
\text{and}&&
\Gamma(a)\Gamma\left(\frac12\right)=\Gamma\left(a+\frac12\right)\int_0^1 v^{a-1}(1-v)^{-1/2}\;dv.
\end{eqnarray*}
In the first identity we perform the change of variable~$v:=\frac{1+u}2$ 
and we obtain that
\begin{align*}
\big(\Gamma(a)\big)^2=\Gamma(2a)2^{1-2a}\int_{-1}^1\big(1-u^2\big)^{a-1}\;du,
\end{align*}
while in the second identity we change variable~$v:=1-u^2$ to find that
$$ \Gamma(a)\Gamma\left(\frac12\right)=2\Gamma\left(a+\frac12\right)\int_0^1 \big(1-u^2\big)^{a-1}\;dv.$$
Therefore, we deduce that
\begin{align*}
\Gamma(a)\Gamma\left(\frac12\right)=\Gamma\left(a+\frac12\right)\frac{\big(\Gamma(a)\big)^2}{\Gamma(2a)2^{1-2a}}
\end{align*}
which, recalling also~\eqref{gamma-one-half}, entails~\eqref{gamma-dupli}.
\end{proof}

We also recall the Euler's reflection formula
\begin{equation}\label{euler-reflection}
\Gamma(a)\Gamma(1-a)=\frac{\pi}{\sin(\pi a)}, \qquad {\mbox{ for all }} a\in\R\setminus\Z.
\end{equation}
Several proofs of this formula are available in the literature and some of them have been put forth by Dirichlet, Dedekind and Gauss,
see e.g.~\cite{MR2281927, MR3012681}. See also~\cite[Exercise~2.10.11]{DVfourier}
for a proof based on complex analysis and contour integration.
\medskip

Finally, the~$\Gamma$ function also provides short representations 
for the expression of the measures of balls and spheres, as the next two lemmata highlight.

\begin{lemma}
The measure~$V_n$ of the~$n$-dimensional unitary ball~$B_1\subset\R^n$ satisfies the recurrence relation
\begin{align}\label{measure-n-ball-rec}
V_n=\frac{\sqrt{\pi}\,\Gamma(\frac{n+1}2)}{\Gamma(\frac{n}2+1)}\,V_{n-1}
\end{align}
and therefore, in a closed form,
\begin{align}\label{measure-n-ball}
V_n=\frac{2\pi^{n/2}}{n\Gamma(\frac{n}2)}.
\end{align}
\end{lemma}

\begin{proof}
For any~$r>0$, we denote by~$B'_r$ the~$(n-1)$-dimensional ball of~$\R^{n-1}$ of radius~$r$. With this notation,
we have that
\begin{align*}
V_n=\int_{B_1}dx=\int_{-1}^1\int_{B_{\sqrt{1-x_1^2}}'}dx'\;dx_1
=V_{n-1}\int_{-1}^1\big(1-x_1^2\big)^{(n-1)/2}\;dx_1.
\end{align*}
Thus, using the change of variables~$x_1:=2t-1$ and~\eqref{beta-identity} with~$a=b=(n+1)/2$,
\begin{align*}
V_n=V_{n-1}\,2^n\int_0^1 t^{(n-1)/2}(1-t)^{(n-1)/2}\;dt=V_{n-1}\,2^n\,\frac{\big( \Gamma(\frac{n+1}2)\big)^2}{\Gamma(n+1)}
.\end{align*}
Moreover, by applying~\eqref{gamma-recursive} and~\eqref{gamma-dupli}, we obtain that
\begin{eqnarray*}
&&\frac{\Gamma(\frac{n+1}2)}{\Gamma(n+1)}=
\frac{2^{1-n}\sqrt\pi\,\Gamma(n)}{\Gamma(\frac{n}2)\,\Gamma(n+1)}
=\frac{2^{1-n}\sqrt\pi}{n\,\Gamma(\frac{n}2)}.
\end{eqnarray*}
The last two displays give~\eqref{measure-n-ball-rec}.

Identity~\eqref{measure-n-ball} can now be verified by induction since, owing to~\eqref{gamma-one-half},
\begin{align*}
V_1=2=\frac{2\sqrt\pi}{\Gamma(\frac12)}
\end{align*}
and, owing to~\eqref{gamma-recursive},
\begin{align*}
V_n=\frac{\sqrt{\pi}\,\Gamma(\frac{n+1}2)}{\Gamma(\frac{n}2+1)}\,V_{n-1}
=\frac{\sqrt{\pi}\,\Gamma(\frac{n+1}2)}{\Gamma(\frac{n}2+1)}\,\frac{2\pi^{(n-1)/2}}{(n-1)\Gamma(\frac{n-1}2)}
=\frac{\pi^{n/2}}{\Gamma(\frac{n}2+1)}
=\frac{2\pi^{n/2}}{n\,\Gamma(\frac{n}2)},
\end{align*}as desired.
\end{proof}

\begin{lemma}
The surface measure~$S_{n-1}$ of the~$(n-1)$-dimensional sphere~$\mathbb{S}^{n-1}=\partial B_1\subset\R^n$ satisfies
\begin{align}\label{measure-n-sphere-vs-ball}
S_{n-1}=n\,V_n
\end{align}
and therefore, in a closed form,
\begin{align}\label{measure-n-sphere}
S_{n-1}=\frac{2\pi^{n/2}}{\Gamma(\frac{n}2)}.
\end{align}
\end{lemma}

\begin{proof}
By an application of polar coordinates, we have that
\begin{align*}
S_{n-1}=\int_{\mathbb{S}^{n-1}}d\theta=n\int_{\mathbb{S}^{n-1}}\int_0^1r^{n-1}\;dr\;d\theta
=n\int_{B_1}dx=n\,V_n
\end{align*}
which proves~\eqref{measure-n-sphere-vs-ball}.

Identity~\eqref{measure-n-sphere} is a direct consequence of~\eqref{measure-n-ball} and~\eqref{measure-n-sphere-vs-ball}.
\end{proof}

Finally, we prove here an integral identity which is needed in Section~\ref{fractsection:3}.

\begin{lemma}\label{lem:coseno}
For any~$s\in(0,1)$ it holds that
\[
\int_0^{+\infty}\frac{1-\cos y}{y^{1+2s}}\;dy=\frac{\Gamma(s)\Gamma(1-s)}{2\Gamma(1+2s)}.
\]
\end{lemma}
\begin{proof}
We use the definition of~$\Gamma$ in~\eqref{gamma-def} in order to write
\begin{align*}
\Gamma(1+2s)\int_0^{+\infty}\frac{1-\cos y}{y^{1+2s}}\;dy=
\int_0^{+\infty}\int_0^{+\infty}e^{-t}t^{2s}\frac{1-\cos y}{y^{1+2s}}\;dy\;dt.
\end{align*}
With the change of variables~$w:=y/t$ we obtain that
\begin{align*}
\Gamma(1+2s)\int_0^{+\infty}\frac{1-\cos y}{y^{1+2s}}\;dy &=
\int_0^{+\infty}\int_0^{+\infty}e^{-t}\frac{1-\cos(wt)}{w^{1+2s}}\;dw\;dt \\
&=
\int_0^{+\infty}\frac1{w^{1+2s}}\int_0^{+\infty}e^{-t}\left(1-\cos(wt)\right)\;dt\;dw \\
&=
\int_0^{+\infty}\frac1{w^{1+2s}}\left(1-\int_0^{+\infty}e^{-t}\cos(wt)\;dt\right)\;dw.
\end{align*}

Now, via a double integration by parts, we see that
\begin{align*}
\int_0^{+\infty}e^{-t}\cos(wt)\;dt=1-w\int_0^{+\infty}e^{-t}\sin(wt)\;dt
=1-w^2\int_0^{+\infty}e^{-t}\cos(wt)\;dt
\end{align*}
and thus
\[
\int_0^{+\infty}e^{-t}\cos(wt)\;dt=\frac1{1+w^2}.
\]
We therefore have
\begin{align*}
\Gamma(1+2s)\int_0^{+\infty}\frac{1-\cos y}{y^{1+2s}}\;dy=
\int_0^{+\infty}\frac1{w^{1+2s}}\left(1-\frac1{1+w^2}\right)\;dw=
\int_0^{+\infty}\frac{w^{1-2s}}{1+w^2}\;dw.
\end{align*}
We now perform one final change of variables, i.e.,~$v:=1/(1+w^2)$, and use~\eqref{beta-identity} to obtain that
\begin{align*}
\int_0^{+\infty}\frac{w^{1-2s}}{1+w^2}\;dw=
\frac12\int_0^1 v\left(\frac1v-1\right)^{-s}\frac{dv}{v^2}=
\frac12\int_0^1v^{s-1}(1-v)^{-s}\;dv=\frac{\Gamma(s)\Gamma(1-s)}{2}
,\end{align*}
which completes the proof of the desired formula. 
\end{proof}

\section{The hypergeometric function}\label{sec:hypergeometric}

The Gaussian hypergeometric function is a special function represented by hypergeometric series, as follows:

\begin{definition}\label{hyp:par}
Given~$a$, $b$, $c\in\R$ satisfying~$a+b<c$, the (Gaussian or ordinary) hypergeometric function~$\hf\big(a,b;c|\cdot\big):(-1,1) \to\R$
is defined, for all~$x\in(-1,1)$, as 
\begin{align}\label{def:hyp}
\hf\big(a,b;c|x\big):=\sum_{k=0}^{+\infty}\frac{(a)_k\,(b)_k}{(c)_k}\,\frac{x^k}{k!}
\end{align}
where, for~$q\in\R$, we have denoted by~$(q)_k$ the (rising) Pochhammer symbol
\begin{align}\label{pochhammer}
(q)_k:=q(q+1)\cdots(q+k-1)=\frac{\Gamma(q+k)}{\Gamma(q)}.
\end{align}
\end{definition}

Let us collect some features which are direct consequences of the definition:
\begin{enumerate}
\item For any~$a$, $b$, $c$ and~$x$ as in Definition~\ref{hyp:par}, we have that
\begin{align}\label{propsym75655535}
\hf\big(a,b;c|x\big)=\hf\big(b,a;c|x\big);
\end{align}
\item An alternative definition for~$\hf$, avoiding the use of Pochhammer symbols and in light of~\eqref{pochhammer}, is 
\begin{align}\label{def2:hyp}
\hf\big(a,b;c|x\big)=\frac{\Gamma(c)}{\Gamma(a)\Gamma(b)}\sum_{k=0}^\infty\frac{\Gamma(a+k)\Gamma(b+k)}{\Gamma(c+k)}\,\frac{x^k}{k!};
\end{align}
\item If either~$a$ or~$b$ are non-positive integers, the series defining~$\hf$ reduces to a finite sum
(therefore to a polynomial), as the associated Pochhammer symbol is eventually constant to~$0$.

For instance, if, say, $a:=-m$ with~$m\in\N$, then~$(-m)_k=0$ for any~$k\geq m+1$, and therefore
\begin{align}\label{hyp-poly}
\hf\big(-m,b;c|x\big)=\sum_{k=0}^m\frac{(-m)_k\,(b)_k}{(c)_k}\,\frac{x^k}{k!}
=\sum_{k=0}^m(-1)^k\binom{m}{k}\frac{(b)_k}{(c)_k}\,x^k.
\end{align}
Particular examples are: 
\begin{align}
\hf\big(0,b;c|x\big) &= 1, \label{hyp(0)} \\
\hf\big(-1,b;c|x\big) &= 1-\frac{b}{c}\,x, \notag \\
\hf\big(-2,b;c|x\big) &= 1-\frac{2b}{c}\,x+\frac{b(b+1)}{c(c+1)}\,x^2; \notag
\end{align}
\item The derivatives of~$\hf$ behave as follows:
\begin{align}
\frac\partial{\partial x}\hf\big(a,b;c|x\big) &=
\sum_{k=1}^{+\infty}\frac{(a)_k\,(b)_k}{(c)_k}\,\frac{x^{k-1}}{(k-1)!}\\&=
\sum_{k=0}^{+\infty}\frac{a(a+1)_k\,b(b+1)_k}{c(c+1)_k}\,\frac{x^k}{k!} \nonumber \\
&=\frac{ab}c\,\hf\big(a+1,b+1;c+1|x\big), \label{hyp'} \\
\frac{\partial^2}{\partial x^2}\hf\big(a,b;c|x\big) &=
\frac{a(a+1)b(b+1)}{c(c+1)}\,\hf\big(a+2,b+2;c+2|x\big), \label{hyp''} \\
\frac{\partial^j}{\partial x^j}\hf\big(a,b;c|x\big) &=
\frac{(a)_j\,(b)_j}{(c)_j}\,\hf\big(a+j,b+j;c+j|x\big), \label{hyp^j}
\end{align}
\end{enumerate}

Another equivalent representation for~$\hf$ is an integral one:

\begin{lemma}\label{yturidsyrue21wqsd48765960uglihfieworlemma}
For~$a$, $b$, $c$, $x$ as in Definition~\ref{hyp:par} with~$c>b>0$, it holds that
\begin{align}
\hf\big(a,b;c|x\big) &=
\frac{\Gamma(c)}{\Gamma(b)\,\Gamma(c-b)}\int_0^1
t^{b-1}{(1-t)}^{c-b-1}{(1-xt)}^{-a}\;dt \label{hyp:int1} \\
&=
\frac{\Gamma(c)}{\Gamma(b)\,\Gamma(c-b)}\int_0^{+\infty}
\tau^{b-1}{(1+\tau)}^{a-c}{(1+\tau-x\tau)}^{-a}\;d\tau \label{hyp:int2} \\
&=
\frac{\Gamma(c)}{\Gamma(b)\,\Gamma(c-b)}\int_0^{+\infty}
\eta^{c-b-1}{(1+\eta)}^{a-c}{(\eta+1-x)}^{-a}\;d\eta \label{hyp:int3}.
\end{align}
Moreover,
\begin{align}\label{hyp(1)}
\lim_{x\nearrow 1}\hf\big(a,b;c|x\big)=\frac{\Gamma(c)\,\Gamma(c-b-a)}{\Gamma(c-a)\,\Gamma(c-b)}.
\end{align}
\end{lemma}

\begin{proof} We have that
\begin{align*}
{(1-xt)}^{-a}=\sum_{k=0}^{+\infty}\binom{-a}{k}{(-xt)}^k
\end{align*}
where, by~\eqref{gamma-recursive},
\begin{align*}
\binom{-a}{k}=\frac{-a(-a-1)\cdots(-a-k+1)}{k!}=(-1)^k\frac{\Gamma(a+k)}{k!\,\Gamma(a)}.
\end{align*}
Thus, exploiting this fact, \eqref{beta-identity} and~\eqref{def2:hyp},
\begin{align*}
\int_0^1t^{b-1}{(1-t)}^{c-b-1}{(1-xt)}^{-a}\;dt
&=
\frac1{\Gamma(a)}\sum_{k=0}^{+\infty}\Gamma(a+k)\int_0^1t^{b+k-1}{(1-t)}^{c-b-1}\;dt\,\frac{x^k}{k!} \\
&=
\frac1{\Gamma(a)}\sum_{k=0}^{+\infty}\Gamma(a+k)\,\frac{\Gamma(b+k)\,\Gamma(c-b)}{\Gamma(c+k)}\,\frac{x^k}{k!} \\
&=
\frac{\Gamma(b)\,\Gamma(c-b)}{\Gamma(c)}\hf(a,b;c|x).
\end{align*}
This proves~\eqref{hyp:int1}.

Via the change of variables~$\tau:=t/(1-t)$, we obtain the representation in~\eqref{hyp:int2} from~\eqref{hyp:int1}. Also, via the change of variables~$\tau:=\eta^{-1}$, representation~\eqref{hyp:int2} is equivalent to~\eqref{hyp:int3}.

Finally, from~\eqref{hyp:int1}, we see that
\begin{align*}
\lim_{x\nearrow 1}\hf(a,b;c|x)=\frac{\Gamma(c)}{\Gamma(b)\,\Gamma(c-b)}\int_0^1t^{b-1}{(1-t)}^{c-b-a-1}\;dt=\frac{\Gamma(c)\,\Gamma(c-b-a)}{\Gamma(c-a)\,\Gamma(c-b)}
\end{align*}
in view of~\eqref{beta-identity}.
\end{proof}

The integral representations of~$\hf$ in the Lemma~\ref{yturidsyrue21wqsd48765960uglihfieworlemma} enable to give a couple of
transformation formulas.

\begin{lemma}\label{ytufrieds543769078-4w87u0-4w}
For~$a$, $b$, $c$, $x$ as in Definition~\ref{hyp:par} with~$c>b>0$, it holds that
\begin{align}
\hf\big(a,b;c|x\big) &= {(1-x)}^{c-a-b}\hf\big(c-a,c-b;c|x\big) \label{hyp-transf1} \\
&=
(1-x)^{-b}\hf\left(c-a,b;c\bigg|-\frac{x}{1-x}\right). \label{hyp-transf3}
\end{align}
\end{lemma}

The transformations leading to~\eqref{hyp-transf1} and~\eqref{hyp-transf3} are called
Euler Transformation and Pfaff Transformation, respectively.

\begin{proof}[Proof of Lemma~\ref{ytufrieds543769078-4w87u0-4w}]
Starting from~\eqref{hyp:int3} and applying the change of variables~$\eta:=(1-x)\theta$,
one obtains
\begin{align*}
\hf\big(a,b;c|x\big)=\frac{\Gamma(c)}{\Gamma(b)\,\Gamma(c-b)}{(1-x)}^{c-a-b}\int_0^{+\infty}
\theta^{c-b-1}{(1+\theta)}^{-a}{(1+\theta-x\theta)}^{a-c}\;d\theta
\end{align*}
which, using~\eqref{hyp:int2}, amounts to~\eqref{hyp-transf1}.

Starting from~\eqref{hyp:int2}, 
applying the change of variables~$\tau:=\theta/(1-x)$, we also find that
\begin{align*}
\hf\big(a,b;c|x\big) &=
(1-x)^{-b}\frac{\Gamma(c)}{\Gamma(b)\,\Gamma(c-b)}\int_0^{+\infty}
\theta^{b-1}{\left(1+\frac\theta{1-x}\right)}^{a-c}{(1+\theta)}^{-a}\;d\theta \\
&=
(1-x)^{-b}\frac{\Gamma(c)}{\Gamma(b)\,\Gamma(c-b)}\int_0^{+\infty}
\theta^{b-1}{(1+\theta)}^{-a}{\left(
1+\theta+\frac{x}{1-x}\theta\right)}^{a-c}\;d\theta,
\end{align*}which gives~\eqref{hyp-transf3}, as desired.
\end{proof}

\chapter{Fourier analysis, harmonic polynomials, fundamental solutions}

We collect here some auxiliary results about the Fourier Transform of harmonic polynomials, essentially borrowed from~\cite{MR0290095, MR3640641},
which come in handy for the construction of explicit examples showcased in Chapter~\ref{CH:EXAMPLES}.

A first useful result is a generalization of the fact that the Fourier Transform of a Gau{\ss}ian remains a Gau{\ss}ian. This property remains valid even when the Gau{\ss}ian is multiplied by a homogeneous harmonic polynomial, up to a multiplication factor that is a multiple of the imaginary unit:

\begin{theorem}\label{ILP:DS}
Let~$P$ be a homogeneous harmonic polynomial in~$\R^n$ of degree~$k$ and
$$f(x):=P(x) e^{-\pi|x|^2}.$$
Then,
$$ \widehat f(\xi)= (-i)^k\,f(\xi).$$
\end{theorem}

\begin{proof} For every~$z\in\C$, let
$$ Q(z):=\int_{\R^n} P(x) e^{-\pi|x-iz|^2}\,dx.$$
We observe that
\begin{eqnarray*}
Q(-iz)&=&\int_{\R^n} P(x) e^{-\pi|x-z|^2}\,dx\\
&=&\int_{\R^n} P(z+y) e^{-\pi|y|^2}\,dy\\
&=&\iint_{(0,+\infty)\times\partial B_1} P(z+\rho\omega) e^{-\pi \rho^2}\,\rho^{n-1}\,d\rho\,d{\mathcal{H}}^{n-1}_\omega.
\end{eqnarray*}
Also, since~$P$ is harmonic, the Mean Value Formula gives that
$$ \fint_{\partial B_1} P(z+\rho\omega)\,d{\mathcal{H}}^{n-1}_\omega
=P(z),$$
from which we infer that
\begin{eqnarray*}
Q(-iz)&=& {\mathcal{H}}^{n-1}(\partial B_1)
\int_{0}^{+\infty} P(z) e^{-\pi \rho^2}\,\rho^{n-1}\,d\rho\\
&=& P(z)\int_{\R^n} e^{-\pi|x|^2}\,dx\\
&=&P(z).
\end{eqnarray*}
For this reason,
\begin{eqnarray*} 
\widehat f(\xi)&=&\int_{\R^n} P(x) e^{-\pi|x|^2}\,e^{-2\pi i x\cdot\xi}\,dx
\\&=& e^{-\pi|\xi|^2}\,\int_{\R^n} P(x) e^{-\pi|x+i\xi|^2}\,dx
\\ &=& e^{-\pi|\xi|^2}\,Q(-\xi)\\
 &=& e^{-\pi|\xi|^2}\,P(-i\xi)
\\&=&(-i)^k e^{-\pi|\xi|^2}\,P(\xi)\\&=& 
 (-i)^k\,f(\xi),
\end{eqnarray*}
as desired.
\end{proof}

We consider the space~$C^\infty_{00}(\R)$ of all the functions~$ f\in C^{\infty}_0(\R)$ whose derivatives of any order
vanish at the origin. Then, we obtain the following result:

\begin{corollary}\label{COR:LO}
Let~$P$ be a homogeneous harmonic polynomial in~$\R^n$ of degree~$k$ and~$\phi_0\in C^\infty_{00}(\R)$.

Consider the radial functions
$$\R^n\ni x\mapsto \phi(x):=\phi_0(|x|)\qquad{\mbox{and}}\qquad
\R^{n+2k}\ni x_\star\mapsto \psi(x_\star):=\phi_0(|x_\star|).$$
Let also~$\Phi:=P\phi$.

Then, if~$\xi\in\R^n$ and~$\xi_\star\in\R^{n+2k}$ are such that~$|\xi|=|\xi_\star|$, we have that
$$ \widehat\Phi(\xi)=(-i)^k P(\xi)\,\widehat\psi(\xi_\star).$$
\end{corollary}

\begin{proof} For every~$t>0$, let~$\Upsilon_t(x):=P(x) e^{-\pi t|x|^2}$.
Thus, from Theorem~\ref{ILP:DS} we deduce that
$$ \widehat\Upsilon_1(\xi)=(-i)^k P(\xi) e^{-\pi|\xi|^2}$$
and consequently
\begin{equation}\label{COR:LO99}\begin{split}
\widehat\Upsilon_t(\xi)&=\int_{\R^n}P(x) e^{-\pi t|x|^2-2\pi i x\cdot\xi}\,dx\\
&=t^{-\frac{n}2}\int_{\R^n}P({t}^{-1/2} y) e^{-\pi|y|^2-2\pi i y\cdot{t}^{-1/2}\xi}\,dy\\
&=t^{-\frac{n+k}2}\int_{\R^n}P(y) e^{-\pi|y|^2-2\pi i y\cdot{t}^{-1/2}\xi}\,dy\\
&={t}^{-\frac{n+k}2} \,\widehat\Upsilon_1({t}^{-1/2}\xi)\\&=
(-i)^k {t}^{-\frac{n+k}2}\, P({t}^{-1/2}\xi) e^{-\frac{\pi|\xi|^2}t}\\&=
(-i)^k {t}^{-\frac{n+2k}2} \,P(\xi) e^{-\frac{\pi |\xi|^2}t}
.\end{split}\end{equation}

This is the desired result when~$\phi$ is a Gau{\ss}ian and we now deal with the general case by considering the Laplace Transform.
To this end, for every~$r>0$, we let~$\phi_1(r):=\phi_0(\sqrt{r})$ and~$\mu$ be the Inverse Laplace Transform of~$\phi_1$.

In this setting, we see that
$$\phi_0(\sqrt{r})=\phi_1(r)=\int_0^{+\infty} \mu(t) \, e^{-\pi tr}\,dt~$$
and accordingly
$$\phi(x)=\phi_0(|x|)=\int_0^{+\infty} \mu(t) \, e^{-\pi t |x|^2}\,dt .$$
Thus, we deduce from~\eqref{COR:LO99} that
\begin{equation}\label{LKTRDF}\begin{split}
\widehat\Phi(\xi)&= 
\int_{\R^n} \Phi(x) e^{-2\pi i x\cdot\xi}\,dx\\&=
\int_{\R^n}P(x) \phi(x) e^{-2\pi i x\cdot\xi}\,dx\\& =
\iint_{(0,+\infty)\times\R^n}P(x) \mu({t}) e^{-\pi{t}|x|^2-2\pi i x\cdot\xi}\,dt\,dx\\&=
(-i)^k P(\xi) \int_0^{+\infty} t^{-\frac{n+2k}2} \mu(t) e^{-\frac{\pi |\xi|^2}t} \,dt.\end{split}\end{equation}

Furthermore, if~$\xi_\star\in\R^{n+2k}$ and~$|\xi_\star|=|\xi|$, we see that
\begin{eqnarray*}
\widehat\psi(\xi_\star)&=&\int_{\R^{n+2k}} \phi_0(|x_\star|)\,e^{-2\pi ix_\star\cdot\xi_\star}\,dx_\star\\
&=&\iint_{(0,+\infty)\times\R^{n+2k}} \mu(t)\,e^{-2\pi ix_\star\cdot\xi_\star-\pi t|x_\star|^2}\,dt\,dx_\star.
\end{eqnarray*}
Exploiting the Fourier Transform of the Gau{\ss}ian in~$\R^{n+2k}$, this gives that
$$\widehat\psi(\xi_\star)=\int_{0}^{+\infty} t^{-\frac{n+2k}{2}}\mu(t)\,e^{-\frac{\pi |\xi_\star|^2}{t}}\,dt.
$$
{F}rom this and~\eqref{LKTRDF} we obtain the desired result.
\end{proof}

\begin{corollary}\label{KMDPL3D0IKD(YHNF)-YUK}
Let~$P$ be a homogeneous harmonic polynomial in~$\R^n$ of degree~$k$.
Let~$f_0$,~$\phi_0\in C^\infty_{00}(\R)$.

Consider the radial functions
\begin{eqnarray*}\R^n\ni x\mapsto f(x):=f_0(|x|),\qquad&&
\R^{n+2k}\ni x_\star\mapsto g(x_\star):=f_0(|x_\star|),\\
 \R^n\ni x\mapsto \phi(x):=\phi_0(|x|)\qquad{\mbox{and}}\qquad&&
\R^{n+2k}\ni x_\star\mapsto \psi(x_\star):=\phi_0(|x_\star|)
.\end{eqnarray*}

Then, if~$x\in\R^n$ and~$x_\star\in\R^{n+2k}$ are such that~$|x|=|x_\star|$, we have that
$$ \check\phi * (Pf)(x)=P(x)\,\check\psi*g(x_\star).$$
\end{corollary}

\begin{proof} For all~$\xi_\star\in\R^{n+2k}$ we let
$$h_\star(\xi_\star):=\psi(\xi_\star)\,\widehat{g}(\xi_\star)=
\phi_0(|\xi_\star|)\,\widehat{g}(\xi_\star).$$
Notice that~$\widehat{g}$ is a radial function, since so is~$g$,
therefore~$h_\star$ is also a radial function. As a result, there exists~$h_0:\R\to\R$ such that~$h_\star(\xi_\star)=h_0(|\xi_\star|)$.

For all~$\xi\in\R^{n}$ we let~$h(\xi):=h_0(|\xi|)$. Thus, by Corollary~\ref{COR:LO},
if~$\xi\in\R^n$ and~$\xi_\star\in\R^{n+2k}$ are such that~$|\xi|=|\xi_\star|$, then
\begin{equation*} \begin{split}&\phi(\xi)\widehat{Pf}(\xi)=(-i)^k \phi(\xi) P(\xi)\,\widehat{g}(\xi_\star)=(-i)^k \phi_0(|\xi|) P(\xi)\,\widehat{g}(\xi_\star)=(-i)^k \phi_0(|\xi_\star|) P(\xi)\,\widehat{g}(\xi_\star)\\&\qquad=(-i)^k P(\xi)\,h_\star(\xi_\star)=(-i)^k P(\xi)\,h_0(|\xi_\star|)=(-i)^k P(\xi)\,h(\xi)
.\end{split}\end{equation*}
We apply the Inverse Fourier Transform to this identity and we deduce that
\begin{equation*} \begin{split}&\check\phi*(Pf)(x)=(-i)^k {\mathcal{F}}^{-1}(Ph)(x)
.\end{split}\end{equation*}
Since the function~$Ph$ is real-valued, this identity can be written by using complex conjugation as
\begin{equation}\label{COR:LO78} \begin{split}&\check\phi*(Pf)(x)=(-i)^k\, \overline{ {\mathcal{F}}(Ph)(x)}
.\end{split}\end{equation}
Using again Corollary~\ref{COR:LO}, we see that
$${\mathcal{F}}(Ph)(\xi)=(-i)^k P(\xi)\,\widehat{h_\star}(\xi_\star)$$
and therefore
$$\overline{{\mathcal{F}}(Ph)(\xi)}=i^k P(\xi)\,\overline{\widehat{h_\star}(\xi_\star)}=
i^k P(\xi)\,{\mathcal{F}}^{-1}({ {h_\star}})(\xi_\star)=
i^k P(\xi)\,\check\psi* {g}(\xi_\star).
$$
{F}rom this and~\eqref{COR:LO78} we arrive at
$$\check\phi*(Pf)(x)=P(x)\,\check\psi* {g}(x_\star)
,$$
as desired.
\end{proof}

\begin{corollary}
Let~$s\in[0,1]$.
Let~$P$ be a homogeneous harmonic polynomial in~$\R^n$ of degree~$k$.
Let~$f_0\in C^\infty_{00}(\R)$.

Consider the radial functions
\begin{eqnarray*}&&\R^n\ni x\mapsto f(x):=f_0(|x|),\qquad{\mbox{and}}\qquad
\R^{n+2k}\ni x_\star\mapsto g(x_\star):=f_0(|x_\star|)
.\end{eqnarray*}

Then, if~$x\in\R^n$ and~$x_\star\in\R^{n+2k}$ are such that~$|x|=|x_\star|$, we have that
$$ (-\Delta)^s (Pf)(x)=P(x)\,(-\Delta)^s g(x_\star).$$
\end{corollary}

\begin{proof} If~$s=0$, we have that~$(-\Delta)^s$ is the identity operator and the result is obvious.
Also, if~$s=1$, we observe that
$$ \Delta (Pf)(x)=P(x)\Delta f(x)+2\nabla P(x)\cdot\nabla f(x)=
P(x)\left( f_0''(|x|)+\frac{n-1}{|x|}f_0'(|x|)\right)
+\frac{2\nabla P(x)\cdot x}{|x|} f'_0(x).$$
Moreover,
$$ k P(x)=\left.\frac{d}{dt}(t^k P(x))\right|_{t=1}=
\left.\frac{d}{dt} P(tx)\right|_{t=1}= \nabla P(x)\cdot x,$$
from which we infer that
\begin{equation}\label{COR:LO77} \Delta (Pf)(x)=
P(x)\left( f_0''(|x|)+\frac{n+2k-1}{|x|}f_0'(|x|)\right).\end{equation}
We also have that
$$\Delta g(x_\star)
=f_0''(|x_\star|)+\frac{n+2k-1}{|x_\star|}f_0'(|x_\star|).
$$
Combining this and~\eqref{COR:LO77} we obtain the desired result when~$s=1$.

Thus, in the rest of this argument, we may suppose that~$s\in(0,1)$. We take~$\e>0$ and define
\begin{eqnarray*}&& \R^n\ni x\mapsto\phi_\e(x):= (\e^2+|x|^2)^{s} e^{-\e|x|^2}\\{\mbox{and}}\quad &&
\R^{n+2k}\ni x_\star\mapsto\psi_\e(x_\star):= (\e^2+|x_\star|^2)^{s} e^{-\e|x_\star|^2}.\end{eqnarray*}
For all~$x\in\R^n$, let also
$$ G_\e(x):=\check\psi_\e*g(x,0),$$
where~$(x,0)\in\R^n\times\R^{2k}$.

Notice that
\begin{eqnarray*}
G_\e(x)&=&\iint_{\R^n\times\R^{2k}}\check\psi_\e(\eta',\eta'') g(x-\eta',-\eta'')\,d\eta'\,d\eta''\\&=&
\iiiint_{\R^n\times\R^{2k}\times\R^n\times\R^{2k}} \psi_\e(\beta',\beta'') g(x-\eta',-\eta'')
e^{2\pi i(\beta'\cdot\eta'+\beta''\cdot\eta'')}\,d\eta'\,d\eta''\,d\beta'\,d\beta''\\&=&
\iiiint_{\R^n\times\R^{2k}\times\R^n\times\R^{2k}} \psi_\e(\beta',\beta'') g(\alpha',\alpha'')
e^{2\pi i(\beta'\cdot(x-\alpha')-\beta''\cdot\alpha'')}\,d\alpha'\,d\alpha''\,d\beta'\,d\beta''\\&=&
\iint_{\R^{n+2k}\times\R^{n+2k}} \psi_\e(\beta) g(\alpha) e^{-2\pi i\beta\cdot\alpha}
e^{2\pi i \beta\cdot(x,0)}\,d\alpha\,d\beta\\&=&\int_{\R^{n+2k}} \psi_\e(\beta) \widehat g(\beta)
e^{2\pi i \beta\cdot(x,0)}\,d\beta
\end{eqnarray*}
and therefore
\begin{equation}\label{0-0-09iuwjhegvf-1}
\lim_{\e\searrow0}G_\e(x)=\int_{\R^{n+2k}}|\beta|^{2s} \widehat g(\beta)
e^{2\pi i \beta\cdot(x,0)}\,d\beta={\mathcal{F}}^{-1}\big( |\beta|^{2s} \widehat g(\beta)\big)(x,0)=\frac{
(-\Delta)^s g(x,0)}{(2\pi)^{2s}}.
\end{equation}

In addition, for all~$x\in\R^n$,
\begin{equation}\label{0-0-09iuwjhegvf-2}\begin{split}
\lim_{\e\searrow0}\check\phi_\e * (Pf)(x)=\;&\lim_{\e\searrow0}\int_{\R^n}\check\phi_\e(y)\,Pf(x-y)\,dy
\\=\;&\lim_{\e\searrow0}\iint_{\R^n\times\R^n} \phi_\e(\vartheta)\,Pf(x-y)\,e^{2\pi i\vartheta\cdot y}\,dy\,d\vartheta\\=\;&\lim_{\e\searrow0}\iint_{\R^n\times\R^n} \phi_\e(\vartheta)\,Pf(\alpha)\,e^{-2\pi i\vartheta\cdot \alpha}\,e^{2\pi i\vartheta\cdot x}\,d\alpha\,d\vartheta
\\=\;&\lim_{\e\searrow0}\int_{\R^n} \phi_\e(\vartheta)\,{\mathcal{F}}(Pf)(\vartheta)\,e^{2\pi i\vartheta\cdot x}\,d\vartheta\\ =\;&\int_{\R^n} |\vartheta|^{2s}\,{\mathcal{F}}(Pf)(\vartheta)\,e^{2\pi i\vartheta\cdot x}\,d\vartheta \\
=\;&{\mathcal{F}}^{-1}\big( |\vartheta|^{2s}\,{\mathcal{F}}(Pf)(\vartheta)\big)(x)\\=\;&\frac{(-\Delta)^s (Pf)(x)}{(2\pi)^{2s}}.
\end{split}
\end{equation}

Also, by Corollary~\ref{KMDPL3D0IKD(YHNF)-YUK} (applied with~$x_\star:=(x,0)$),
$$ \check\phi_\e * (Pf)(x)=P(x)\,\check\psi_\e*g(x,0)=P(x) G_\e(x).$$
{F}rom this,~\eqref{0-0-09iuwjhegvf-1} and~\eqref{0-0-09iuwjhegvf-2} we arrive at
$$ (-\Delta)^s (Pf)(x)=(2\pi)^{2s}
\lim_{\e\searrow0}\check\phi_\e * (Pf)(x)=(2\pi)^{2s}\lim_{\e\searrow0}P(x) G_\e(x)
=P(x)(-\Delta)^s g(x,0).$$
The desired result now follows, since~$g$ is rotationally invariant, thus so is~$(-\Delta)^s g$,
whence~$(-\Delta)^s g(x,0)=(-\Delta)^s g(x_\star)$ whenever~$|x|=|x_\star|$.
\end{proof}

For more information about harmonic polynomials, see e.g.~\cite[Sections~2.21--2.22]{2021arXiv210107941D}.
For additional results relating harmonic polynomials, Fourier analysis and nonlocal operators, see~\cite{MR0350027, MR0290095, MR2974318, MR3640641}.\medskip

We now deduce from Corollary~\ref{COR:LO} an additional result, which is useful to discuss the fundamental solution \index{fundamental solution}
of the fractional Laplacian:

\begin{corollary}\label{CAL:GA:099}
Let~$\alpha\in(0,n)$.
Let~$P$ be a homogeneous harmonic polynomial in~$\R^n$ of degree~$k$. Then, for every~$\eta\in C^\infty_c(\R^n)$,
$$ \int_{\R^n} |x|^{\alpha-n-k} P(x)\,\widehat\eta(x)\,dx=c_{n,k,\alpha}
\int_{\R^n} |x|^{-k-\alpha} P(x)\,\eta(x)\,dx,$$
where
$$ c_{n,k,\alpha}:=\frac{i^k \pi^{\frac{n-2\alpha}2}\,\Gamma\left(\displaystyle\frac{k+\alpha}{2} \right)}{\Gamma\left(\displaystyle\frac{n+k-\alpha}{2} \right)},$$
being~$\Gamma$ the Euler Gamma Function.
\end{corollary}

\begin{proof} Given~$\e>0$ we define~$\phi_{0,\e}(t):=e^{-\pi \e t^2}$,
and consider the radial functions
$$\R^n\ni x\mapsto \phi_\e(x):=\phi_{0,\e}(|x|)\qquad{\mbox{and}}\qquad
\R^{n+2k}\ni x_\star\mapsto \psi_\e(x_\star):=\phi_{0,\e}(|x_\star|).$$
Let also~$\Phi_\e:=P\phi_\e$.
Then, we apply Corollary~\ref{COR:LO} to infer that, for all~$\xi\in\R^n$,
\begin{equation}\label{CAL:GA:090}
\widehat\Phi_\e(\xi)=(-i)^k P(\xi)\,\widehat\psi_\e(\xi,0),\end{equation}
where~$(\xi,0)\in\R^n\times\R^{2k}$.

We point out that, given~$\beta\in\left(0,\frac{n+2k}2\right)$, to be conveniently chosen in what follows,
using the substitution~$\tau:=\frac{\pi |\xi|^2}{\e}$, we see that
\begin{equation*}
\begin{split}
&\int_0^{+\infty}\e^{\beta-1}P(\xi)\,\widehat\psi_\e(\xi,0)\,d\e=\int_0^{+\infty} \e^{\frac{2\beta-n-2k}{2}-1} P(\xi)\,e^{-\frac{\pi |\xi|^2}{\e}}\,d\e\\&\qquad=\pi^{\frac{2\beta-n-2k}2} |\xi|^{ 2\beta-n-2k}
P(\xi)\int_0^{+\infty} \tau^{\frac{n-2\beta+2k}{2}-1} e^{-\tau}\,d\tau=
\Gamma\left(\frac{n-2\beta+2k}{2} \right)\pi^{\frac{2\beta-n-2k}2} |\xi|^{ 2\beta-n-2k}
P(\xi).
\end{split}\end{equation*}
Therefore,
\begin{equation}\label{CQJJNS:02}
\begin{split}
\iint_{\R^n\times(0,+\infty)}\e^{\beta-1}P(\xi)\,\widehat\psi_\e(\xi,0)\,\eta(\xi)\,d\xi\,d\e=
\Gamma\left(\frac{n-2\beta+2k}{2} \right)\pi^{\frac{2\beta-n-2k}2} \int_{\R^n}|\xi|^{ 2\beta-n-2k}
P(\xi)\,\eta(\xi)\,d\xi.
\end{split}\end{equation}

Additionally, by Plancherel Theorem,
\begin{equation*}
\begin{split}
\int_{\R^n} \widehat\Phi_\e(\xi)\,\eta(\xi)\,d\xi=
\int_{\R^n} \Phi_\e(\xi)\,\widehat\eta(\xi)\,d\xi=\int_{\R^n} P(\xi)\,e^{-\pi\e|\xi|^2}\widehat\eta(\xi)\,d\xi
\end{split}\end{equation*}
and consequently, substituting for~$\tau:=\pi\e|\xi|^2$,
\begin{equation}\label{CQJJNS:03}
\begin{split}&
\iint_{\R^n\times(0,+\infty)} \e^{\beta-1}\widehat\Phi_\e(\xi)\,\eta(\xi)\,d\xi\,d\e=
\iint_{\R^n\times(0,+\infty)} \e^{\beta-1} P(\xi)\,e^{-\pi\e|\xi|^2}\widehat\eta(\xi)\,d\xi\,d\e\\
&\qquad=\iint_{\R^n\times(0,+\infty)} \tau^{\beta-1} \big(\pi|\xi|^2\big)^{-\beta} P(\xi)\,e^{-\tau}\widehat\eta(\xi)\,d\xi\,d\tau=
\frac{\Gamma(\beta)}{\pi^\beta}
\int_{\R^n} |\xi|^{-2\beta} P(\xi)\,\widehat\eta(\xi)\,d\xi.
\end{split}\end{equation}

Now we combine~\eqref{CAL:GA:090},~\eqref{CQJJNS:02}
and~\eqref{CQJJNS:03} and we conclude that
\begin{eqnarray*}&&\frac{\Gamma(\beta)}{\pi^\beta}
\int_{\R^n} |\xi|^{-2\beta} P(\xi)\,\widehat\eta(\xi)\,d\xi=
\iint_{\R^n\times(0,+\infty)}\e^{\beta-1}\widehat\Phi_\e(\xi)\,\eta(\xi)\,d\xi\,d\e\\&&\qquad
=(-i)^k \iint_{\R^n\times(0,+\infty)}\e^{\beta-1}P(\xi)\,\widehat\psi_\e(\xi,0)\,\eta(\xi)\,d\xi\,d\e\\&&\qquad=
(-i)^k\,\Gamma\left(\frac{n-2\beta+2k}{2} \right)\pi^{\frac{2\beta-n-2k}2}  \int_{\R^n}|\xi|^{ 2\beta-n-2k}
P(\xi)\,\eta(\xi)\,d\xi.
\end{eqnarray*}
The desired result follows by choosing~$\beta:=\frac{n-\alpha+k}{2}$.
\end{proof}

\begin{corollary}\label{FURIEZ}
The function
\begin{equation}\label{0oe2jdfc02uhcewubv89u2hnfivnionBNSM} \R^n\setminus\{0\}\ni x\mapsto {\mathcal{R}}(x):=\begin{dcases} \displaystyle\frac{\displaystyle\Gamma\left(\frac{n-2s}{2}\right)}{2^{2s}\pi^{\frac{n}2}\,\Gamma(s)} \,|x|^{2s-n} & {\mbox{ if }} n\ne2s,\\
-\frac{1}{\pi}\ln|x|
& {\mbox{ if~$ n=1$ and~$s=\displaystyle\frac12$}}
\end{dcases}\end{equation}
is the fundamental solution of~$(-\Delta)^s$, i.e.
$${\mbox{$(-\Delta)^s{\mathcal{R}}=\delta_0$ in the sense of distributions.}}$$
\end{corollary}

\begin{proof}
It is convenient to rewrite~\eqref{0oe2jdfc02uhcewubv89u2hnfivnionBNSM}
in a less compact, but more explicit way by separating the case~$n<2s$ from the case~$n>2s$.
Note that if~$n<2s$ then necessarily~$n=1$ and~$s\in\left(\displaystyle\frac12,1\right)$.
Furthermore,
$$ \frac{\Gamma\left(\displaystyle\frac{1-2s}{2}\right)}{2^{2s}\pi^{\frac{1}{2}} \Gamma(s)}=
\frac{1}{2(2s-1)\cos(\pi s)\Gamma(2s-1)},$$
therefore we can write~\eqref{0oe2jdfc02uhcewubv89u2hnfivnionBNSM} in the equivalent form
$$ \R^n\setminus\{0\}\ni x\mapsto {\mathcal{R}}(x)=\begin{dcases} \displaystyle\frac{\displaystyle\Gamma\left(\frac{n-2s}{2}\right)}{2^{2s}\pi^{\frac{n}2}\,\Gamma(s)} \,|x|^{2s-n} & {\mbox{ if }} n>2s,\\
\frac{1}{2(2s-1)\cos(\pi s)\Gamma(2s-1)}|x|^{2s-1} & {\mbox{ if~$ n=1$ and~$s\in\left(\displaystyle\frac12,1\right)$,}}\\ -\frac{1}{\pi}\ln|x|
& {\mbox{ if~$ n=1$ and~$s=\displaystyle\frac12$.}}
\end{dcases}$$

Now, to prove the desired result we need to check that, for every~$f\in C^\infty_c(\R^n)$,
$$ \int_{\R^n}(-\Delta)^s{\mathcal{R}}(x)\,f(x)\,dx= f(0).$$
This, in the distributional sense, is equivalent to
\begin{equation*}
\int_{\R^n}{\mathcal{R}}(x)\,(-\Delta)^sf(x)\,dx= f(0)\end{equation*}
and so, by Plancherel Theorem, to
\begin{equation}\label{CAL:GA:098} (2\pi)^{2s} \int_{\R^n}|\xi|^{2s}\widehat {\mathcal{R}}(\xi)\,\overline{\widehat f(\xi)}\,d\xi= f(0).\end{equation}
To check this, we distinguish three cases, namely~$n>2s$,~$n=2s$ and~$n<2s$.

When~$n>2s$,
we apply Corollary~\ref{CAL:GA:099} with~$P:=1$ (hence~$k:=0$) and~$\alpha:=2s$, obtaining that, for
every~$\eta\in C^\infty_c(\R^n)$,
$$ \int_{\R^n} |x|^{2s-n}\, \widehat\eta(x)\,dx=c_{n,0,2s}
\int_{\R^n} |x|^{-2s} \eta(x)\,dx.$$
Thus, by the Plancherel Theorem,
$$ (2\pi)^{2s}\int_{\R^n} \widehat {\mathcal{R}}(\xi)\,\eta(\xi)\,d\xi=\int_{\R^n} |x|^{-2s} \eta(x)\,dx.$$
Up to an approximation process, we can choose~$\eta(x):=|x|^{2s}\overline{\widehat f(x)}$ and thereby conclude that
\begin{eqnarray*}(2\pi)^{2s}
\int_{\R^n} \widehat {\mathcal{R}}(\xi)\,|\xi|^{2s}\overline{\widehat f(\xi)}\,d\xi=\int_{\R^n}\overline{\widehat f(x)}\,dx=
\int_{\R^n}\check f(x)\,dx=\int_{\R^n}\check f(x) e^{-2\pi x\cdot 0}\,dx=f(0),
\end{eqnarray*}
which proves~\eqref{CAL:GA:098} in this case.

Now we deal with the case~$n=1$ and~$s\in\left[\frac12,1\right)$. For this, we define
$$ \ell(x):=\begin{dcases}
\ln |x| & {\mbox{ if }}\displaystyle s=\frac12,\\
\\
\displaystyle\frac{|x|^{2s-1}}{2s-1} & {\mbox{ if }}\displaystyle s\in\left(\frac12,1\right).
\end{dcases}$$
Given~$R>0$ and~$g\in C^\infty_c(\R)$, we have that
\begin{equation}\label{M98ijsdLJSnae9A}\begin{split}&
\int_{-R}^R \ell(x)\check g(x)\,dx=\int_{0}^R \ell(x)\big(\check g(x)+\check g(-x)\big)\,dx=
\iint_{(0,R)\times\R} \ell(x) g(\xi) \big(e^{2\pi ix\xi}+e^{-2\pi ix\xi}\big)\,dx\,d\xi\\&\qquad=
2\iint_{(0,R)\times\R} \ell(x) g(\xi) \cos(2\pi x\xi)\,dx\,d\xi\\&\qquad=
\iint_{(0,R)\times\R} \frac{d}{dx}\left( \ell(x) \frac{g(\xi)}{\pi \xi} \sin(2\pi x\xi) \right) \,dx\,d\xi-
\iint_{(0,R)\times\R} \ell'(x) \frac{g(\xi)}{\pi \xi} \sin(2\pi x\xi) \,dx\,d\xi\\&\qquad=
\int_{\R} \ell(R) \,\frac{g(\xi)}{\pi \xi} \sin(2\pi R\xi)\,d\xi-
\iint_{(0,R)\times\R} x^{2s-2} \,\frac{g(\xi)}{\pi \xi} \sin(2\pi x\xi) \,dx\,d\xi.
\end{split}\end{equation}
Our goal is now to take the limit as~$R\to+\infty$. To this end, we pick~$g$ such that
\begin{equation}\label{M98ijsdLJSnae9B}
\int_{\R}\left|\frac{g'(\xi)}{\xi}-\frac{g(\xi)}{\xi^2} \right|\,d\xi<+\infty
\end{equation} and we
point out that
\begin{eqnarray*}&&\left|
\int_{\R} \ell(R) \,\frac{g(\xi)}{\pi \xi} \sin(2\pi R\xi)\,d\xi\right|\\&=&\left|
\int_{\R} \frac{d}{d\xi}\left(\ell(R) \,\frac{g(\xi)}{2\pi^2 R\xi} \cos(2\pi R\xi)\right)\,d\xi-
\int_{\R} \ell(R)\left(\frac{g'(\xi)}{2\pi^2 R\xi}-\frac{g(\xi)}{2\pi^2 R\xi^2} \right)\cos(2\pi R\xi)\,d\xi\right|\\
&=&\frac{\ell(R)}{2\pi^2 R}
\left|\int_{\R} \left(\frac{g'(\xi)}{\xi}-\frac{g(\xi)}{\xi^2} \right)\cos(2\pi R\xi)\,d\xi\right|\\&\le&
\frac{\ell(R)}{2\pi^2 R}
\int_{\R} \left|\frac{g'(\xi)}{\xi}-\frac{g(\xi)}{\xi^2} \right|\,d\xi,
\end{eqnarray*}
which is infinitesimal as~$R\to+\infty$.

Thus, using the substitution~$t:=2\pi x|\xi|$, we deduce from~\eqref{M98ijsdLJSnae9A} that, for all~$g\in C^\infty_c(\R)$ satisfying~\eqref{M98ijsdLJSnae9B},
\begin{equation}\label{M98ijsdLJSnae9C} \begin{split}&
\int_{\R} \ell(x)\check g(x)\,dx=-\lim_{R\to+\infty}
\iint_{(0,R)\times\R} x^{2s-2} \,\frac{g(\xi)}{\pi |\xi|} \sin(2\pi x|\xi|) \,dx\,d\xi\\&\qquad
=-\frac{2^{1-2s}}{\pi^{2s}}\lim_{R\to+\infty}\int_{\R}\left[
\int_0^{2\pi R|\xi|} t^{2s-2}\,\frac{g(\xi)}{|\xi|^{2s}} \sin t\,dt\right]\,d\xi\\&\qquad
=-\frac{2^{1-2s}}{\pi^{2s}}
\int_0^{+\infty} t^{2s-2}\sin t\,dt\,\int_{\R}
\frac{g(\xi)}{|\xi|^{2s}} \,d\xi\\&\qquad=
\frac{2^{1-2s}\cos(\pi s)\,\Gamma(2s-1)}{\pi^{2s}}
\int_{\R}\frac{g(\xi)}{|\xi|^{2s}} \,d\xi
,\end{split}\end{equation}
where the last identity follows from a classical formula involving the Gamma Function
(see e.g. Proposition~A.12 in~\cite{MR3461641}).

Now, given~$f\in C^\infty_c(\R)$, we take~$g(\xi):=|\xi|^{2s}\widehat f(\xi)$.
Notice that, in this case,
\begin{equation*}\begin{split}&
\int_{\R} \left|\frac{g'(\xi)}{\xi}-\frac{g(\xi)}{\xi^2} \right|\,d\xi=
\int_{\R} \left|(2s-1)|\xi|^{2s-2}\widehat f(\xi)-\frac{|\xi|^{2s}(\widehat f)'(\xi)}{\xi} \right|\,d\xi\\&\qquad
\le(2s-1)\int_{\R} |\xi|^{2s-2}|\widehat f(\xi)|\,d\xi+\int_{\R} |\xi|^{2s-1}|(\widehat f)'(\xi)|\,d\xi
<+\infty\end{split}
\end{equation*}
and therefore~\eqref{M98ijsdLJSnae9B} is fulfilled.

Though~$g$ is not necessarily in~$C^\infty_c(\R)$ in this case, we can apply~\eqref{M98ijsdLJSnae9C} after an approximation argument and we thereby conclude that
\begin{equation*} \begin{split}
\int_{\R} \ell(x)\,{\mathcal{F}}^{-1}\big( |\xi|^{2s}\widehat f(\xi)\big)(x)\,dx&=
\frac{2^{1-2s} \cos(\pi s)\,\Gamma(2s-1)}{\pi^{2s}}
\int_{\R}\widehat f(\xi)\,d\xi\\&=\frac{2^{1-2s} \cos(\pi s)\,\Gamma(2s-1)}{\pi^{2s}} f(0).
\end{split}\end{equation*}
That is, by Plancherel Theorem, we obtain that
$$ \int_{\R} \check\ell(\xi)\, |\xi|^{2s}\widehat f(\xi)\,d\xi
=\frac{2^{1-2s} \cos(\pi s)\,\Gamma(2s-1)}{\pi^{2s}} f(0).~$$
Since the quantity in the right-hand side is a real number, this identity can be written by using complex conjugation as
\begin{equation*} \begin{split}&
\int_{\R}  |\xi|^{2s}\widehat\ell(\xi)\,\overline{\widehat f(\xi)}\,d\xi=
\frac{2^{1-2s} \cos(\pi s)\,\Gamma(2s-1)}{\pi^{2s}}f(0).
\end{split}\end{equation*}
leading to~\eqref{CAL:GA:098}.
\end{proof}

\begin{figure}\begin{center}
  \includegraphics[width=.5\linewidth]{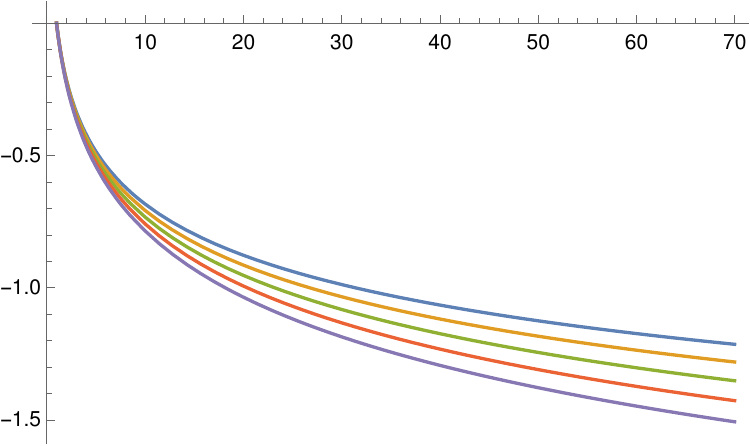}
  \caption{The function in~\eqref{RML52d} with~$s\in\left\{ 0.48,\,0.49,\,\frac12,\,0.51,\,0.52\right\}$.}
  \label{fig:sub2012orfuj}\end{center}
\end{figure}

Looking at~\eqref{0oe2jdfc02uhcewubv89u2hnfivnionBNSM} it is natural to wonder if the case~$n=1$ and~$s=\frac12$
can be recovered as a limit. This may be not immediately intuitive due to the presence of the logarithmic term.
Besides, to conveniently pass to the limit, it is advantageous to replace~\eqref{0oe2jdfc02uhcewubv89u2hnfivnionBNSM} with a different normalization, taking advantage
of the fact that fundamental solutions are actually defined ``up to an additive constant''.

To better understand the situation, we observe that~\eqref{0oe2jdfc02uhcewubv89u2hnfivnionBNSM} entails that, when~$n=1$,
$$ {\mathcal{R}}(1)=\begin{dcases} \displaystyle\frac{\displaystyle\Gamma\left(\frac{1-2s}{2}\right)}{2^{2s}\pi^{\frac{1}2}\,\Gamma(s)}  & {\mbox{ if~$s\ne\displaystyle\frac12$}},\\
0 & {\mbox{ if~$s=\displaystyle\frac12$,}}
\end{dcases}$$
and
$$ \lim_{s\to\frac12} \left|\frac{\displaystyle\Gamma\left(\frac{1-2s}{2}\right)}{2^{2s}\pi^{\frac{1}2}\,\Gamma(s)}\right|=+\infty\ne0,
$$
therefore to obtain the case~$s=\frac12$ through a limit procedure it is convenient to look at
an appropriate vertical translation of~\eqref{0oe2jdfc02uhcewubv89u2hnfivnionBNSM}.

For instance, we can consider the following gauge of the fundamental solution
\begin{equation}\label{RML52d} \R^n\setminus\{0\}\ni x\mapsto \widetilde{\mathcal{R}}(x):=\begin{dcases} \displaystyle\frac{\displaystyle\Gamma\left(\frac{n-2s}{2}\right)}{2^{2s}\pi^{\frac{n}2}\,\Gamma(s)} \,\Big(|x|^{2s-n}-1\Big) & {\mbox{ if }} n\ne2s,\\
-\frac{1}{\pi}\ln|x|
& {\mbox{ if~$ n=1$ and~$s=\displaystyle\frac12$,}}
\end{dcases}\end{equation}
which only differs from~\eqref{0oe2jdfc02uhcewubv89u2hnfivnionBNSM} by a vertical translation.

This allows for a convenient limit procedure as~$s\to\frac12$, because if~$n=1$, for all~$x\in\R\setminus\{0\}$,
\begin{eqnarray*}&&
\lim_{s\to\frac12}
\frac{\displaystyle\Gamma\left(\frac{1-2s}{2}\right)}{2^{2s}\pi^{\frac{1}2}\,\Gamma(s)} \,\Big(|x|^{2s-1}-1\Big)=
\lim_{\sigma\to0}
\frac{\displaystyle\Gamma(\sigma)}{2^{1-2\sigma}\;\pi^{\frac{1}2}\,\Gamma\left(\displaystyle\frac{1-2\sigma}{2}\right)} \,\Big(|x|^{-2\sigma}-1\Big)\\
&&\qquad=
\frac{1}{2\pi}
\lim_{\sigma\to0}\Gamma(\sigma)\,\Big(|x|^{-2\sigma}-1\Big)=\frac{1}{2\pi}
\lim_{\sigma\to0}\Gamma(\sigma)\,\Big( e^{-2\sigma\ln|x|}-1\Big)\\&&\qquad=\frac{1}{2\pi}
\lim_{\sigma\to0}\left(\frac1\sigma+O(1)\right)\,\Big( -2\sigma\ln|x|+O(\sigma^2)\Big)=\frac{1}{2\pi}\lim_{\sigma\to0}\Big( -2\ln|x|+O(\sigma)\Big)\\&&\qquad=-\frac{1}{\pi}\ln|x|,
\end{eqnarray*}
thus providing a coherent limit setting for the gauged fundamental solution~$\widetilde{\mathcal{R}}$.

See Figure~\ref{fig:sub2012orfuj} for a sketch of this limit configuration.
\medskip

See also~\cite{MR0290095, MR0350027, MR3461641, MR3916700, MR3965397}
for additional information on the fundamental solution of the fractional Laplacian.

\chapter{Some auxiliary results from the theory of distributions}\label{APPECC}

Here we recall some useful facts from the theory of distributions\footnote{For a detailed study of distribution theory see~\cite{MR0209834, MR2000535, MR2296978, MR2453959} and the references therein.} which come in handy to prove rigidity and classification results in the realm of partial differential equations (see e.g.~\cite{MR1996773} for further details on this topic).
To this end, we recall that if~$\Omega$ is an open subset of~$\R^n$ and~$T$ is a distribution
(i.e., a continuous linear form on~$C^\infty_c(\R^n)$), one says that~$T$ vanishes on~$\Omega$ if, for every~$\varphi\in C^\infty_c(\R^n)$ such that the support of~$\varphi$ is contained in~$\Omega$, it holds that~$T(\varphi )=0$.

In this notation, the \index{support of a distribution} support of a distribution~$T$ is defined as the complement of the largest open set on which~$T$ vanishes.

For example, one can check that the Dirac Delta Function at a point~$p$ vanishes on every open set not containing~$p$, from which one infers that its support is~$\{p\}$.

It is also useful to recall that, by the continuity property of a distribution~$T$ (see e.g. equation~(1.3.1) in~\cite{MR1996773}), one has that for every compact subset~$K$ of~$\R^n$ there exist~$N_K\in\N$ and~$C_K>0$ such that, for every~$\varphi\in C^\infty_c(K)$,
\begin{equation}\label{LSr-0Sor213De}
|T(\varphi)|\le C_K\sup_{{\alpha\in\N^n}\atop{|\alpha|\le N_K}}\| D^\alpha\varphi\|_{L^\infty(K)}.
\end{equation}
One says that the order of the distribution~$T$ is~$N$
if~$N$ is the smallest integer for which~\eqref{LSr-0Sor213De} holds true
independently of the compact set~$K$, i.e. the smallest~$N\in \N$ for which for every compact subset~$K$ of~$\R^n$ there exists~$C_K>0$ such that, for every~$\varphi\in C^\infty_c(K)$,
\begin{equation}\label{LSr-0Sor213De-U}
|T(\varphi)|\le C_K\sup_{{\alpha\in\N^n}\atop{|\alpha|\le N}}\| D^\alpha\varphi\|_{L^\infty(K)}.
\end{equation}
If there is no integer~$N$ for which~\eqref{LSr-0Sor213De-U} holds true, one says that the order of~$T$ is infinite.

For example, a derivative of order~$N$ of the Dirac Delta Function is a distribution of order~$N$, while the distribution
$$ T(\varphi):=\sum_{j=0}^{+\infty} \partial^j_{x_1} \varphi(je_1)$$
is of\footnote{To check that this distribution is of infinite order, one can argue for a contradiction. Suppose that~\eqref{LSr-0Sor213De-U} holds true for some~$N\in\N$, pick~$\e\in\left(0,\frac1{10}\right)$ and~$\varphi\in C^\infty_c(B_1)$ with~$\partial^{N+1}_{x_1}\varphi(0)=1$. Define~$\varphi_\e(x):=\varphi\left( \frac{x-(N+1)e_1}{\e}\right)$ and note that~$\varphi_\e\in C^\infty_c(K)$, with~$K:=B_{1/4}((N+1)e_1)$.

Then, one infers from~\eqref{LSr-0Sor213De-U} that
\begin{eqnarray*}&&\frac1{\e^{N+1}}=\frac1{\e^{N+1}}\left| \partial^{N+1}_{x_1}\varphi(0)\right|=
| \partial^{N+1}_{x_1}\varphi_\e((N+1)e_1)|=
\left|\sum_{j=0}^{+\infty} \partial^j_{x_1} \varphi_\e(je_1)\right|=
|T(\varphi_\e)|\\&&\qquad\le C_K\sup_{{\alpha\in\N^n}\atop{|\alpha|\le N}}\| D^\alpha\varphi_\e\|_{L^\infty(K)}
\le
C_K\sup_{{\alpha\in\N^n}\atop{|\alpha|\le N}} \frac1{\e^{|\alpha|}}
\| D^\alpha\varphi \|_{L^\infty(\R^n)}\le
\frac{C_K}{\e^N}\sup_{{\alpha\in\N^n}\atop{|\alpha|\le N}}
\| D^\alpha\varphi \|_{L^\infty(\R^n)}.
\end{eqnarray*}
Thus, multiplying by~$\e^{N+1}$ and taking the limit as~$\e\searrow0$ one obtains the desired contradiction.} infinite order.\medskip

A useful result in this setting goes as follows:

\begin{theorem}\label{DE:CALFA:TG}
Let~$p\in\R^n$.
Suppose that the support of a distribution~$T$ is contained in the singleton~$\{p\}$. 

Then,~$T$ is a finite combination of derivatives of the Dirac Delta Function at the point~$p$.

More explicitly, there exist~$N\in\N$ and~$\{c_\alpha\}_{{\alpha\in\N^n}\atop{|\alpha|\le N}}$ such that
$$ T=\sum_{{\alpha\in\N^n}\atop{|\alpha|\le N}} c_\alpha D^\alpha \delta_{p}.$$

Furthermore,~$T$ is of finite order, equal to~$N$, and the coefficients~$c_\alpha$ are determined by the relation
\begin{equation}\label{DE:CALFA} c_\alpha:=\frac{(-1)^{|\alpha|}}{\alpha!} T(x^\alpha).\end{equation}
\end{theorem}

We remark that the fact that~$T$ in Theorem~\ref{DE:CALFA:TG} is of finite order follows from the following general fact:

\begin{lemma} \label{DE:CALFA:TGle1}
If a distribution has compact support, then it has finite order.
\end{lemma}

Concerning the expression in~\eqref{DE:CALFA}, we stress that the function~$\R^n\ni x\mapsto x^\alpha$
is not compactly supported (not even decaying at infinity), hence the term~$T(x^\alpha)$ has to be understood
after a cutoff procedure. Interestingly, the result obtained is independent of the cutoff, as remarked by
the following observation:

\begin{lemma} \label{DE:CALFA:TGle2}
Let~$T$ be a distribution with compact support~$K$. Let~$f\in C^\infty(\R^n)$.
Let~$\tau_1$,~$\tau_2\in C^\infty_c(\R^n)$ with~$\tau_1=\tau_2=1$ in a neighborhood of~$K$.

Then,~$T(\tau_1 f)=T(\tau_2 f)$.
\end{lemma}

In the light of Lemma~\ref{DE:CALFA:TGle2}, one can thus define~$T(x^\alpha)$ in~\eqref{DE:CALFA} simply as~$T(x^\alpha\tau(x))$ for a given~$\tau\in C^\infty_c(\R^n)$
with~$\tau=1$ in a neighborhood of the origin (the definition being actually independent of the specific choice of~$\tau$).\medskip

An ancillary result helpful to prove Lemma~\ref{DE:CALFA:TGle2} and Theorem~\ref{DE:CALFA:TG} is the following:

\begin{proposition} \label{DE:CALFA:TGle3}
Let~$T$ be a distribution of finite order~$N\in\N$.

Then, 
$$ T(\varphi)=0$$
for every~$\varphi\in C^\infty_c(\R^n)$ such that~$D^\alpha\varphi(x)=0$ whenever~$\alpha\in\N^n$ with~$|\alpha|\le N$
and~$x$ belongs to the support of~$T$.
\end{proposition}

Now we provide the proofs of the results showcased above.

\begin{proof}[Proof of Lemma~\ref{DE:CALFA:TGle1}] Let~$K_\star$ denote the compact support of the distribution~$T$.
Let also~$K_\sharp$ be a compact set containing an open neighborhood of~$K_\star$.
We consider a partition of unity~$\{\rho_j\}_{j\in\N}$ such that~$\rho_0=1$ in~$K_\sharp$ and, for every~$x\in\R^n$,
$$\sum_{j=0}^{+\infty}\rho_j (x)=1.$$
Note that~$\rho_j=0$ in~$K_\sharp$ for all~$j\in\N\setminus\{0\}$, and therefore~$T(\rho_j\varphi)=0$ for all~$j\in\N\setminus\{0\}$ and~$\varphi\in C^\infty_c(\R^n)$.

In this way, given~$\varphi\in C^\infty_c(\R^n)$, if~$K$ is the support of~$K$ and~$K_0$ is the support of~$\rho_0$ we can use~\eqref{LSr-0Sor213De} and deduce that
\begin{eqnarray*}
|T(\varphi)|=\left|T\left(\sum_{j=0}^{+\infty}\rho_j\varphi\right)\right|=|T(\rho_0\varphi)|
\le C_{K_0} \sup_{{\alpha\in\N^n}\atop{|\alpha|\le N_{K_0}}}\| D^\alpha(\rho_0\varphi)\|_{L^\infty(K\cap K_0)}
\end{eqnarray*}
for some~$C_{K_0}>0$ and~$N_{K_0}\in\N$.

Consequently,
\begin{eqnarray*}
|T(\varphi)|
\le \widetilde C_{K_0} \sup_{{\alpha\in\N^n}\atop{|\alpha|\le N_{K_0}}}\| D^\alpha \varphi \|_{L^\infty(K)}
\end{eqnarray*}
for some~$\widetilde C_{K_0}>0$, from which we obtain that~$T$ has finite order, less than or equal to~$N_{K_0}$ (and we stress that~$K_0$ does not depend on~$\varphi$ and~$K$).
\end{proof}

\begin{proof}[Proof of Proposition~\ref{DE:CALFA:TGle3}] The proof is based on
a slight enlargement of supports, to avoid pathologies coming from boundary points of supports
(see e.g. the example in~(III, 7; 19) of~\cite{MR0209834} for related phenomena).

The technical details go as follows.
Let~$T$ and~$\varphi$ be as in the statement
of Proposition~\ref{DE:CALFA:TGle3} and~$K$ be the support of~$T$ (here,~$K$ is not necessarily a compact set). Pick
a mollifier~$\tau\in C^\infty_c(B_1,[0,+\infty))$ with unit mass. Given~$\e\in(0,1)$, to be taken as small as we wish, let
$$K_\e:=\big\{x\in\R^n {\mbox{ s.t. }}\exists y\in K{\mbox{ with }}|x-y|<\e\big\}.$$
Let also
$$ \tau_\e(x):=\frac1{\e^n}\int_{K_{4\e}}\tau\left(\frac{x-y}\e\right)\,dy.$$
We observe that~$\tau_\e\in C^\infty(\R^n)$. 

Furthermore, if~$x\in K_\e$ and~$y\in\R^n\setminus K_{4\e}$, then~$|x-y|\ge3\e$
and therefore~$\tau\left(\frac{x-y}\e\right)=0$. This entails that for all~$x\in K_\e$
$$ \tau_\e(x)=\frac1{\e^n}\int_{\R^n}\tau\left(\frac{x-y}\e\right)\,dy=\int_{\R^n}\tau(X)\,dX=1.$$
In particular,~$\varphi=\varphi\tau_\e$ in an open neighborhood of~$K$, giving that~$T(\varphi)=T(\varphi\tau_\e)$.

Additionally, if~$x\in \R^n\setminus K_{8\e}$ and~$y\in K_{4\e}$ then~$|x-y|\ge3\e$
and therefore~$\tau\left(\frac{x-y}\e\right)=0$. This yields that~$\tau_\e(x)=0$ for every~$x\in \R^n\setminus K_{8\e}$. As a result, the function~$\varphi\tau_\e$ is of compact support (since so is~$\varphi$) and its support is contained in~$K_{10\e}$ (which is not necessarily compact).

Besides, since the distribution~$T$ is supposed to have finite order~$N$, denoting by~$K_\star$ the support of~$\varphi$ we have that
\begin{equation}\label{eNalpha1}
|T(\varphi)|=|T(\varphi\tau_\e)|\le
C_{K_\star}\sup_{{\alpha\in\N^n}\atop{|\alpha|\le N}}\| D^\alpha(\varphi\tau_\e)\|_{L^\infty({K_\star})}=C_{K_\star}\sup_{{\alpha\in\N^n}\atop{|\alpha|\le N}}\| D^\alpha(\varphi\tau_\e)\|_{L^\infty(K_\star\cap {K_{10\e}})},
\end{equation}
for some~$C_{K_\star}>0$.

Now we observe that, if~$\alpha\in\N^n$ and~$|\alpha|\le N$, then, for all~$x\in K_{10\e}$,
\begin{equation}\label{0qodwhis-RSFVPDrgdpDD}
|D^\alpha \varphi(x)|\le C\e^{N-|\alpha|+1},
\end{equation}
for some~$C>0$.

To check this, we take~$x_0\in K$ with~$|x-x_0|\le10\e$ and we use the Taylor Formula to see that
$$ D^\alpha\varphi(x)=\sum_{{\beta\in\N^n}\atop{|\beta|\le N-|\alpha|}}\frac{1}{\beta!} D^{\alpha+\beta}\varphi(x_0)(x-x_0)^\beta+O\left(|x-x_0|^{N-|\alpha|+1}\right).
$$
The claim in~\eqref{0qodwhis-RSFVPDrgdpDD} thus follows since~$D^{\alpha+\beta}\varphi(x_0)$ here above vanishes, owing to our assumption on~$\varphi$.

Moreover, if~$\alpha\in\N^n$ and~$|\alpha|\le N$,
\begin{eqnarray*}
|D^\alpha\tau_\e(x)|\le \frac1{\e^{n+|\alpha|}}\int_{K_{4\e}}\left|D^\alpha\tau\left(\frac{x-y}\e\right)\right|\,dy\le\frac1{\e^{|\alpha|}}\int_{\R^n} |D^\alpha\tau(X)|\,dX\le C\e^{-|\alpha|},
\end{eqnarray*}
up to renaming~$C$.

Using this,~\eqref{0qodwhis-RSFVPDrgdpDD} and the Leibniz Product Rule, in~$K_\star\cap {K_{10\e}}$, for all~$\alpha\in\N^n$ with~$|\alpha|\le N$, we have that
\begin{eqnarray*}&& |D^\alpha(\varphi\tau_\e)|=\left|\sum_{{\beta,\gamma\in\N^n}\atop{\beta+\gamma=\alpha}}\left( {\alpha}\atop{\beta}\right)
D^\beta\varphi \,D^\gamma\tau_\e\right|\le
C \sum_{{\beta,\gamma\in\N^n}\atop{\beta+\gamma=\alpha}}\left( {\alpha}\atop{\beta}\right)
\e^{N-|\beta|+1} \e^{-|\gamma|}\\&&\qquad=C \sum_{{\beta,\gamma\in\N^n}\atop{\beta+\gamma=\alpha}}\left( {\alpha}\atop{\beta}\right)
\e^{N-|\alpha|+1}\le C\e^{N-|\alpha|+1}\le C\e,
\end{eqnarray*}
for some~$C>0$ possibly varying at each step of the calculation.

Combining this with~\eqref{eNalpha1}, we conclude that~$|T(\varphi)|\le C\e$ and therefore, taking~$\e$ arbitrarily small,~$T(\varphi)=0$, as desired.
\end{proof}

\begin{proof}[Proof of Lemma~\ref{DE:CALFA:TGle2}] We recall that~$T$ has finite order, due to Lemma~\ref{DE:CALFA:TGle1}. Given~$j\in\{1,2\}$, we observe that~$\tau_jf\in C^\infty_c(\R^n)$,
whence~$T(\tau_jf)$ is well defined. Let now~$\tau:=\tau_1-\tau_2$ and note that~$\tau$ vanishes identically in a neighborhood of~$K$.
Accordingly, by Proposition~\ref{DE:CALFA:TGle3}, we have that~$T(\tau f)=0$, from which we obtain the desired result.
\end{proof}

\begin{proof}[Proof of Theorem~\ref{DE:CALFA:TG}] Notice that~$T$ has finite order, say~$N\in\N$, thanks to Lemma~\ref{DE:CALFA:TGle1}.
Moreover, if the support of~$T$ is empty, then necessarily~$T(\varphi)=0$ for every~$\varphi\in C^\infty_c(\R^n)$ and we are done, therefore we can assume that the support of~$T$ is the
singleton~$\{p\}$.

Up to a translation, we can suppose that~$p=0$. Given~$f\in C^\infty_c(\R^n)$ and~$x\in B_1$, we use the Taylor Formula and write that
$$ f(x)=\sum_{{\alpha\in\N^n}\atop{|\alpha|\le N}}\frac1{\alpha!} D^\alpha f(0)x^\alpha+
\sum_{{\alpha\in\N^n}\atop{|\alpha|=N+1}}\frac1{\alpha!} D^\alpha f(\xi_\alpha)x^\alpha,$$
for some~$\xi_\alpha$ on the segment joining the origin to~$x$.

Thus, we take~$\tau\in C^\infty_c(B_3)$ with~$\tau=1$ in~$B_2$. In this way,
$$ T(f)=T(f\tau)=T\left(\sum_{{\alpha\in\N^n}\atop{|\alpha|\le N}}\frac1{\alpha!} D^\alpha f(0)x^\alpha\tau(x)\right)+T(g),$$
where
$$ g(x):=\tau(x)
\sum_{{\alpha\in\N^n}\atop{|\alpha|=N+1}}\frac1{\alpha!} D^\alpha f(\xi_\alpha)x^\alpha.$$
Notice that~$D^\beta g(0)$ for all~$\beta\in\N^n$ with~$|\beta|\le N$ and consequently, by
Proposition~\ref{DE:CALFA:TGle3}, we know that~$T(g)=0$.

{F}rom these observations we obtain that
\begin{eqnarray*}
T(f)=T\left(\sum_{{\alpha\in\N^n}\atop{|\alpha|\le N}}\frac1{\alpha!} D^\alpha f(0)x^\alpha\tau(x)\right)=
\sum_{{\alpha\in\N^n}\atop{|\alpha|\le N}}\frac1{\alpha!} D^\alpha f(0)T\left(x^\alpha\tau(x)\right)=
\sum_{{\alpha\in\N^n}\atop{|\alpha|\le N}}(-1)^{|\alpha|}c_\alpha D^\alpha f(0),
\end{eqnarray*}
which gives the desired result.
\end{proof}

\chapter{A version of a result by Norbert Wiener on Fourier analysis}

We recall a variant of a classical result by Norbert Wiener~\cite{MR1503035}
(the original version being about the fact that if the Fourier series of~$f$ is absolutely convergent and~$f$ never vanishes, then the Fourier series of~$1/f$ is absolutely convergent too;
the version here below was employed on page~\pageref{WI:TH:-0eirojeg:page}).
See~\cite{MR0107776, MR365002} for more information on the result by Wiener and related topics.

The case of the Fourier Transform needs several technical modifications with respect to the original argument and we follow here a proof based on complex analysis introduced in~\cite{MR1067573}.

\begin{theorem}\label{WI:TH:-0eirojeg}
Let~$\phi_1\in L^1(\R^n)$, with
$$\inf_{\R^n}|\widehat\phi_1+1|>0.$$

Then, there exists~$\phi_2\in L^1(\R^n)$ such that
$$ \widehat\phi_2=\frac{1}{\widehat\phi_1+1}-1.$$
\end{theorem}

\begin{proof} We let~$\iota\in\left(0,\frac14\right)$ be such that~$|\widehat\phi_1(\xi)+1|\ge\iota$ for all~$\xi\in\R^n$.
Since~$\phi_1\in L^1(\R^n)$, 
we can take~$\phi_\star\in C^\infty_c(\R^n)$ such that
\begin{equation}\label{UNAMSPDLPSUTIS1i2u45-L0} \|\phi_1-\phi_\star\|_{L^1(\R^n)}\le\frac\iota{16}.\end{equation}
{F}rom this, it follows that
\begin{equation}\label{UNAMSPDLPSUTIS1i2u45-89QL}
|\widehat\phi_\star(\xi)-\widehat\phi_1(\xi)|\le\int_{\R^n}|\phi_\star(x)-\phi_1(x)|\,dx\le\frac\iota{16}
\end{equation}
and accordingly
\begin{equation}\label{UNAMSPDLPSUTIS1i2u45}
|\widehat\phi_\star(\xi)+1|\ge|\widehat\phi_1(\xi)+1|-|\widehat\phi_\star(\xi)-\widehat\phi_1(\xi)|
\ge\iota-\frac\iota{16}=\frac{15\,\iota}{16}.
\end{equation}

Now, we use methods from complex analysis. For this, we denote the complex disk of radius~$r$ by
$$ D_r:=\{z\in\C{\mbox{ s.t. }}|z|<r\}.$$
Thus, given~$\xi\in\R^n$, we define the function
$$ D_{\iota/2}\ni z\longmapsto f(z,\xi):=-\frac{\widehat\phi_1(\xi)}{\widehat\phi_\star(\xi)+1+z}.$$
By~\eqref{UNAMSPDLPSUTIS1i2u45}, for all~$z\in D_{\iota/2}$, we know that
\begin{equation}\label{UNAMSPDLPSUTIS1i2u45-834}
|\widehat\phi_\star(\xi)+1+z|\ge|\widehat\phi_\star(\xi)+1|-|z|\ge\frac{15\,\iota}{16}-\frac\iota2=\frac{7\,\iota}{16},\end{equation}
therefore~$f(\cdot;\xi)$ is holomorphic in~$D_{\iota/2}$.

We also let~$a(\xi):=\widehat\phi_1(\xi)-\widehat\phi_\star(\xi)$ and we observe that \begin{equation}\label{UNAMSPDLPSUTIS1i2u45-834-98uh}
|a(\xi)|\le\frac\iota{16}, \end{equation}
thanks to~\eqref{UNAMSPDLPSUTIS1i2u45-89QL}, yielding that,
if~$\gamma$ is a circle of radius~$\frac{\iota}{4}$ centered at~$a(\xi)$ and oriented counterclockwise,
then~$\gamma$ lies in~$D_{\iota/2}$ and, as a result, using Cauchy's Integral Formula, 
$$f(a(\xi),\xi)= \frac{1}{2\pi i} \oint _{\gamma } \frac{f(z,\xi)}{z-a(\xi)}\,dz.$$

Now we define
$$ \Upsilon(\xi):=\frac{1}{2\pi i} \oint _{\gamma } \frac{f(z,\xi)}{z-a(\xi)}\,dz,$$
we use the notation
\begin{equation}\label{OJDAKS934-isolMilmCNLSUR} \| g\|_\star:=\|\check g\|_{L^1(\R^n)}\end{equation}
and we claim that
\begin{equation}\label{OJDAKS934-isolMilmCNLS}
\|\Upsilon\|_\star<+\infty.
\end{equation}
Once this is proven, the desired result in Theorem~\ref{WI:TH:-0eirojeg} is established, because
we would define~$\phi_2:=\check\Upsilon$, infer from~\eqref{OJDAKS934-isolMilmCNLSUR} and~\eqref{OJDAKS934-isolMilmCNLS} that~$\phi_2\in L^1(\R^n)$ and find that
$$\frac{1}{\widehat\phi_1(\xi)+1}-1=
-\frac{\widehat\phi_1(\xi)}{\widehat\phi_1(\xi)+1}=
-\frac{\widehat\phi_1(\xi)}{\widehat\phi_\star(\xi)+1+a(\xi)}=f(a(\xi),\xi)= \frac{1}{2\pi i} \oint _{\gamma } \frac{f(z,\xi)}{z-a(\xi)}\,dz=\Upsilon(\xi)=\widehat\phi_2(\xi),
$$
as desired.

Thus, to complete the proof of Theorem~\ref{WI:TH:-0eirojeg}, it suffices to check~\eqref{OJDAKS934-isolMilmCNLS}. To this end, we observe that, if~$\lambda$ is a smooth arc of finite length in a bounded region
of~$\C$, parameterized, with a slight abuse of notation,
by~$z=\lambda(t)$, with~$t\in[a,b]$, for some~$b>a$,
we have the general inequality
\begin{equation}\label{1324OJDAKS934-isolMilmCNLS}\begin{split}&
\left\|\int_\lambda \psi(z,\cdot)\,dz\right\|_\star
=\left\|\int_\lambda \check\psi(z,\cdot)\,dz\right\|_{L^1(\R^n)}=
\int_{\R^n} \left|\int_\lambda \check\psi(z,x)\,dz\right|\,dx\\&\qquad\le
\iint_{\R^n\times[a,b]} |\check\psi(\lambda(t),x)|\,|\lambda'(t)|\,dx\,dt
\le (b-a)\|\lambda'\|_{L^\infty([a,b])}\sup_{z\in\lambda}
\int_{\R^n} |\check\psi(z,x)|\,dx\\&\qquad=
(b-a)\|\lambda'\|_{L^\infty([a,b])}\sup_{z\in\lambda}
\|\check\psi(z,\cdot)\|_{L^1(\R^n)}
=(b-a)\|\lambda'\|_{L^\infty([a,b])}\sup_{z\in\lambda}
\|\psi(z,\cdot)\|_\star.\end{split}
\end{equation}

Moreover, for all~$g=g(\xi)$ and~$h=h(\xi)$,
\begin{equation}\label{LKMS-pdokfwjenB5A42N23t4Ddwfjenf2} \begin{split}&\| gh\|_\star=\|{\mathcal{F}}^{-1}( gh)\|_{L^1(\R^n)}
=\|\check g*\check h\|_{L^1(\R^n)}\le\iint_{\R^n\times\R^n} |\check g(x-y)|\,|\check h(y)|\,dx\,dy\\&\qquad\qquad\qquad=\|\check g\|_{L^1(\R^n)}\,\|\check h\|_{L^1(\R^n)}=
\|g\|_\star\,\|h\|_\star.\end{split}\end{equation}

Now, for all~$z\in \C$, we define
$$ \R^n\ni \xi\longmapsto H(z,\xi):=\frac1z\sum_{k=0}^{+\infty}\left(\frac{\widehat\phi_1(\xi)-\widehat\phi_\star(\xi)}{z}\right)^k.$$
We notice that, by~\eqref{UNAMSPDLPSUTIS1i2u45-L0},
$$ \|\widehat\phi_1-\widehat\phi_\star\|_\star=\|\phi_1-\phi_\star\|_{L^1(\R^n)}\le\frac\iota{16}$$
and therefore, by~\eqref{LKMS-pdokfwjenB5A42N23t4Ddwfjenf2},
$$ \left\|\frac1z\left(\frac{\widehat\phi_1-\widehat\phi_\star}{z}\right)^k\right\|_\star=\frac{1}{|z|^{k+1}}\,\big\|(\widehat\phi_1-\widehat\phi_\star)^k \big\|_\star\le
\frac{1}{|z|^{k+1}}\, \| \widehat\phi_1-\widehat\phi_\star\|_\star^k\le\frac{1}{|z|^{k+1}}\left(\frac\iota{16}\right)^k.
$$
Consequently, for every~$z\in\C\setminus D_{\iota/15}$,
\begin{equation}\label{LKMS-pdokfwjenB5A42N23t4Ddwfjenf2c} \|H(z,\cdot)\|_\star\le\sum_{k=0}^{+\infty}\left\|\frac1z\left(\frac{\widehat\phi_1-\widehat\phi_\star}{z}\right)^k\right\|_\star\le \frac{15}\iota\sum_{k=0}^{+\infty}
\left(\frac{15}{16}\right)^k\le C,\end{equation}
for some~$C>0$, which depends only on~$n$ and~$\phi_1$ and which we take the freedom of renaming in every step of the calculation.

Moreover, by~\eqref{UNAMSPDLPSUTIS1i2u45-89QL}, we have that, for every~$z\in\C\setminus D_{\iota/15}$,
$$ \left|\frac{\widehat\phi_1(\xi)-\widehat\phi_\star(\xi)}{z}\right|\le\frac{15}{16}$$
and therefore we can compute the sum of the geometric series and find that, for every~$z\in\C\setminus D_{\iota/15}$,
\begin{equation}\label{LKMS-pdokfwjenB5A42N23t4Ddwfjenf2b} H(z,\xi)=\frac{1}{z-\widehat\phi_1(\xi)+\widehat\phi_\star(\xi)}=\frac1{z-a(\xi)}.\end{equation}

Now we recall~\eqref{UNAMSPDLPSUTIS1i2u45-834} and let
$$ L(z,\xi):=\frac{\widehat\phi_\star(\xi)}{\widehat\phi_\star(\xi)+1+z}
.$$
We observe that, when~$z\in D_{\iota/2}$, the function~$\R^n\ni\xi \mapsto L(z,\xi)$ is smooth and rapidly decreasing.

As a consequence, for all~$z\in D_{\iota/2}$,
\begin{equation}\label{LKMS-pdokfwjenB5A42N23t4Ddwfjenf2A} \|L(z,\cdot)\|_\star\le C.\end{equation}
Thus, since
\begin{eqnarray*}
-\frac{1}{z+1}\Big( \widehat\phi_1(\xi)-\widehat\phi_1(\xi) L(z,\xi)\Big)=
-\frac{\widehat\phi_1(\xi)}{z+1}\left( 1-\frac{\widehat\phi_\star(\xi)}{\widehat\phi_\star(\xi)+1+z}\right)
=
-\frac{\widehat\phi_1(\xi)}{\widehat\phi_\star(\xi)+1+z}=f(z,\xi),
\end{eqnarray*}
we deduce from~\eqref{LKMS-pdokfwjenB5A42N23t4Ddwfjenf2}
and~\eqref{LKMS-pdokfwjenB5A42N23t4Ddwfjenf2A} that
\begin{equation}\label{LKMS-pdokfwjenB5A42N23t4Ddwfjenf2d}\begin{split}&
\|f(z,\cdot)\|_\star=\frac{1}{|z+1|}\left\| \widehat\phi_1-\widehat\phi_1L(z,\cdot)\right\|_\star\le
\frac{1}{|z+1|}\left(\| \widehat\phi_1\|_\star+\|\widehat\phi_1 L(z,\cdot)\|_\star\right)\\&\qquad\le
\frac{\| \widehat\phi_1\|_\star}{|z+1|}\left(1+\|L(z,\cdot)\|_\star\right)=
\frac{\|\phi_1\|_{L^1(\R^n)}}{|z+1|}\left(1+\|L(z,\cdot)\|_\star\right)\le C.\end{split}
\end{equation}

Consequently, for all~$z\in D_{\iota/2}\setminus D_{\iota/15}$,
$$ J(z,\xi):=\frac{f(z,\xi)}{z-a(\xi)}=f(z,\xi)\,H(z,\xi),
$$
thanks to~\eqref{LKMS-pdokfwjenB5A42N23t4Ddwfjenf2b},
giving that, for all~$z\in D_{\iota/2}\setminus D_{\iota/15}$,
\begin{equation}\label{LKMS-pdokfwjenB5A42N23t4Ddwfjenf2b2}
\|J(z,\cdot)\|_\star\le \|f(z,\cdot)\|_\star\,\|H(z,\cdot)\|_\star\le C,
\end{equation}
due to~\eqref{LKMS-pdokfwjenB5A42N23t4Ddwfjenf2},~\eqref{LKMS-pdokfwjenB5A42N23t4Ddwfjenf2c} and~\eqref{LKMS-pdokfwjenB5A42N23t4Ddwfjenf2d}.

We now observe that if~$p\in\gamma$ then, in view of~\eqref{UNAMSPDLPSUTIS1i2u45-834-98uh},
$$ |p|\ge |p-a(\xi)|-|a(\xi)|\ge\frac\iota4-\frac\iota{16}=\frac{3\iota}{16}>\frac{\iota}{15},$$
which makes it possible to employ~\eqref{LKMS-pdokfwjenB5A42N23t4Ddwfjenf2b2}
along~$\gamma$.

For this reason, and recalling~\eqref{1324OJDAKS934-isolMilmCNLS},
\begin{eqnarray*}
\|\Upsilon\|_\star=\frac{1}{2\pi} \left\|\oint _{\gamma } \frac{f(z,\cdot)}{z-a(\cdot)}\,dz\right\|_\star
=\frac{1}{2\pi} \left\|\oint _{\gamma }J(z,\cdot)\,dz\right\|_\star\le
\frac{1}{2\pi} \oint _{\gamma }\|J(z,\cdot)\|_\star\,dz\le C.
\end{eqnarray*}
The proof of~\eqref{OJDAKS934-isolMilmCNLS} is thereby complete.
\end{proof}

\chapter{An interpolation inequality}\label{APPE:INTERPLEMM}

Here we present an interpolation inequality in Bessel potential spaces.
The proof that we present is based on dyadic partitions of unity in frequency space, according to
the general estimate presented in Theorem~\ref{0pirj09365-4g4eZXCHJKLRFG-544-NO3-COR-0oerkCC}.

For related inequalities, see e.g.~\cite{MR2570437, MR3813967, MR3990737} and the references therein;
see also~\cite[page~3]{MR2250142} and~\cite[Lemma~3.1]{MR1877265}
for further information and~\cite[page~59, formula~(3), and page~186, Remark 6]{MR1328645} for an approach based on complex interpolation.

\begin{proposition}
Let~$p\ge1$ and~$b>a>0$. Let also~$\mu\in(0,1)$.
Then, there exists~$C>0$, depending only on~$n$,~$p$,~$a$,~$b$ and~$\mu$, such that 
$$ \|u\|_{{\mathcal{L}}^p_{a} (\R^n)} \le C\, \|u\|^{\frac{b-\mu a}b}_{L^p(\R^n)}\,
\|u\|_{{\mathcal{L}}^p_{b} (\R^n)}^{\frac{\mu a}b}
$$
\end{proposition}

\begin{proof} Let~$\varphi_j$ be given 
by Lemma~\ref{0pirj09365-4g4eZXCHJKLRFG-544-NO1-LL}. We use H\"older's Inequality with exponents~$\frac{b}{\mu a}$ and~$\frac{b}{b-\mu a}$ to see that
\begin{eqnarray*}&&
\sum_{j=0}^{+\infty} 2^{2ja}|\check\varphi_j*u|^2=
\sum_{j=0}^{+\infty} 2^{2j\mu a}|\check\varphi_j*u|^{\frac{2\mu a}b}\,2^{2j(1-\mu)a}|\check\varphi_j*u|^{\frac{2(b-\mu a)}{b}}\\&&\qquad\le
\left(\sum_{j=0}^{+\infty} 2^{2jb}|\check\varphi_j*u|^2\right)^{\frac{\mu a}{b}}
\left(\sum_{j=0}^{+\infty} 2^{\frac{2j(1-\mu)ab}{b-\mu a}}|\check\varphi_j*u|^2\right)^{\frac{b-\mu a}b}.
\end{eqnarray*}

On this account, recalling Theorem~\ref{0pirj09365-4g4eZXCHJKLRFG-544-NO3-COR-0oerkCC} and applying again H\"older's Inequality with exponents~$\frac{b}{\mu a}$ and~$\frac{b}{b-\mu a}$ we conclude that
\begin{eqnarray*}
\frac{ \|u\|_{{\mathcal{L}}^p_{a} (\R^n)}}{C}&\le&\left\|\,\left(
\sum_{j=0}^{+\infty} 2^{2ja}|\check\varphi_j*u|^2\right)^{\frac12} \,\right\|_{L^p(\R^n)}\\&=&
\left[ \int_{\R^n}\left(
\sum_{j=0}^{+\infty} 2^{2ja}|\check\varphi_j*u(x)|^2\right)^{\frac{p}2}\,dx\right]^{\frac1p}\\&\le&
\left[ \int_{\R^n}\left(\sum_{j=0}^{+\infty} 2^{2jb}|\check\varphi_j*u(x)|^2\right)^{\frac{p\mu a}{2b}}
\left(\sum_{j=0}^{+\infty} 2^{\frac{2j(1-\mu)ab}{b-\mu a}}|\check\varphi_j*u(x)|^2\right)^{\frac{p(b-\mu a)}{2b}}\,dx\right]^{\frac1p}\\&\le&\left[ \int_{\R^n}\left(\sum_{j=0}^{+\infty} 2^{2jb}|\check\varphi_j*u(x)|^2\right)^{\frac{p}{2}}
\,dx\right]^{\frac{\mu a}{bp}}
\left[ \int_{\R^n}
\left(\sum_{j=0}^{+\infty} 2^{\frac{2j(1-\mu)ab}{b-\mu a}}|\check\varphi_j*u(x)|^2\right)^{\frac{p}{2}}\,dx\right]^{\frac{b-\mu a}{bp}}\\&=&
\left\|\,\left(\sum_{j=0}^{+\infty} 2^{2jb}|\check\varphi_j*u(x)|^2\right)^{\frac12} \,\right\|_{L^p(\R^n)}^{\frac{\mu a}b}
\left\|\,
\left(\sum_{j=0}^{+\infty} 2^{\frac{2j(1-\mu)ab}{b-\mu a}}|\check\varphi_j*u(x)|^2\right)^{\frac12} \,\right\|_{L^p(\R^n)}^{\frac{b-\mu a}{b}}.
\end{eqnarray*}

Consequently, using again Theorem~\ref{0pirj09365-4g4eZXCHJKLRFG-544-NO3-COR-0oerkCC} and renaming constants,
\begin{equation}\label{MLAerSM:ijfKKSMdf0202o3tg-293rtijMMS} \frac{ \|u\|_{{\mathcal{L}}^p_{a} (\R^n)}}{C}\le
\|u\|_{{\mathcal{L}}^p_{b} (\R^n)}^{\frac{\mu a}b}\left\|\,
\left(\sum_{j=0}^{+\infty} 2^{\frac{2j(1-\mu)ab}{b-\mu a}}|\check\varphi_j*u(x)|^2\right)^{\frac12} \,\right\|_{L^p(\R^n)}^{\frac{b-\mu a}{b}}.\end{equation}

Besides, using the Minkowski's Integral Inequality in Theorem~\ref{MLAerSM:ijfKKSMdf02},
\begin{eqnarray*}
&&\left\|\,
\left(\sum_{j=0}^{+\infty} 2^{\frac{2j(1-\mu)ab}{b-\mu a}}|\check\varphi_j*u(x)|^2\right)^{\frac12} \,\right\|_{L^p(\R^n)}^2=\left[
\int_{\R^n}
\left(\sum_{j=0}^{+\infty} 2^{\frac{2j(1-\mu)ab}{b-\mu a}}|\check\varphi_j*u(x)|^2\right)^{\frac{p}{2}}\,dx\right]^{\frac2p}\\&&\qquad\le\sum_{j=0}^{+\infty} \left(\int_{\R^n}\Big(2^{\frac{2j(1-\mu)ab}{b-\mu a}}|\check\varphi_j*u(x)|^2\Big)^{\frac{p}2}
\,dx\right)^{\frac2p}=\sum_{j=0}^{+\infty} 2^{\frac{2j(1-\mu)ab}{b-\mu a}}\left(\int_{\R^n}
|\check\varphi_j*u(x)|^p
\,dx\right)^{\frac2p}\\&&\qquad=
\sum_{j=0}^{+\infty} 2^{\frac{2j(1-\mu)ab}{b-\mu a}}\|\check\varphi_j*u\|^2_{L^p(\R^n)}.
\end{eqnarray*}
Thus, by Young's Convolution Inequality and~\eqref{0pirj09365-4g4eZXCHJKLRFG-544-NO1-LL-eq4-09-BIS},
\begin{eqnarray*}
&&\left\|\,
\left(\sum_{j=0}^{+\infty} 2^{\frac{2j(1-\mu)ab}{b-\mu a}}|\check\varphi_j*u(x)|^2\right)^{\frac12} \,\right\|_{L^p(\R^n)}^2\le
\sum_{j=0}^{+\infty} 2^{\frac{2j(1-\mu)ab}{b-\mu a}}\|\check\varphi_j\|_{L^1(\R^n)}^2\|u\|^2_{L^p(\R^n)}\\&&\qquad\le C
\sum_{j=0}^{+\infty} 2^{\frac{2j(1-\mu)ab}{b-\mu a}}\|u\|^2_{L^p(\R^n)}=C\|u\|^2_{L^p(\R^n)}.
\end{eqnarray*}

We combine this information and~\eqref{MLAerSM:ijfKKSMdf0202o3tg-293rtijMMS} and, up to renaming constants, we obtain the desired result.
\end{proof}

\chapter{A short stroll around pseudodifferential operators}\label{ojdlfwenSTRFGbdollDeltafRn}

We provide a very succinct exposition of some basic facts related to the noble, and very deep, theory of 
pseudodifferential operators. Once again, we do not aim at being exhaustive and we refer to~\cite{MR0259335, MR388463, MR597144, MR597145, MR618463, MR1385196, MR1852334, MR2453959, MR2567604, MR2884718} and the references therein for thorough information.

To start with, we recall the notion of \index{pseudodifferential symbol} pseudodifferential symbol, which, given~$n$,~$N\in\N$ and~$m\in\R$, is any function~$a\in C^\infty(\R^N\times\R^n,\,\C)$ such that, for all~$\alpha\in\N^N$ and~$\beta\in\N^n$,
\begin{equation}\label{1e12q3ua02it} \sup_{{x\in\R^N}\atop{\xi\in\R^n}}(1+|\xi|)^{|\beta|-m}|D^\alpha_x D^\beta_\xi a(x,\xi)|<+\infty.\end{equation}

The space of such pseudodifferential symbols is sometimes denoted~$S^m_{1,0}(\R^N\times\R^n)$,
or~$S^m$ for short, and~$m$ is the order\footnote{By~\eqref{1e12q3ua02it}, we see that if a pseudodifferential symbol
has order~$m$, then it has also order~$m'$ for all~$m'\in[m,+\infty)$. Thus, sometimes one can pay additional
attention in order to select the ``best possible'' order of a symbol, according to the situation. This observation \label{COM78REDU-912ieyrhf-001-IMANSWJh}
will play a role when dealing with commutators in Section~\ref{COM78REDU-912ieyrhf-001}, since interesting cancellations occur in that framework to improve the order.}
of the pseudodifferential symbol.\medskip

We also say that a pseudodifferential symbol is of order~$-\infty$ if it is of order~$m$ for every~$m\in\R$. Pseudodifferential symbols of order~$-\infty$ are often denoted by~$S^{-\infty}$ and are also called ``smoothing operators'' (for a reason which will become apparent in the forthcoming Corollary~\ref{COMPOTHPS-q0woeirjUJSDNIie00}). 
If~$a\in C^\infty_c(\R^N\times\R^n)$ we have that~$a$ is a pseudodifferential symbol is of order~$-\infty$,
since if the support of~$a$ is contained in~$B_R\times B_R$ (where, with a mild abuse of notation,
we have denoted by~$B_R$ the balls in both~$\R^N$ and~$\R^n$),
then
$$ \sup_{{x\in\R^N}\atop{\xi\in\R^n}}(1+|\xi|)^{|\beta|-m}|D^\alpha_x D^\beta_\xi a(x,\xi)|
=\sup_{{x\in B_R}\atop{\xi\in B_R}}(1+R)^{|\beta|+|m|}|D^\alpha_x D^\beta_\xi a(x,\xi)|
<+\infty.$$
See also Corollary~\ref{COMPOTHPS-q0woeirjUJSDNIie00-cr5tgd} for an interesting example of
a pseudodifferential symbol of order~$-\infty$.

\begin{lemma}\label{1e12q3ua02it012oierjohfgbFJJasxLALEM}
Let~$j\in\N$. For each~$\ell\in\{0,\dots,j\}$, let~$c_\ell\in\C$ and~$a_\ell$ be a pseudodifferential symbol of order~$m_\ell$.

Then,~$c_0a_0+\dots+c_ja_j$ is a pseudodifferential symbol of order~$\max\{m_0,\dots,m_\ell\}$.
\end{lemma}

\begin{proof} The function~$c_0a_0+\dots+c_ja_j$ belongs to~$C^\infty(\R^N\times\R^n,\,\C)$ since so does each addendum.
In addition, by~\eqref{1e12q3ua02it}, for each~$\ell\in\{0,\dots,j\}$,
$$ C_{\alpha,\beta,\ell}:=\sup_{{x\in\R^N}\atop{\xi\in\R^n}}(1+|\xi|)^{|\beta|-m_\ell}|D^\alpha_x D^\beta_\xi a_\ell(x,\xi)|<+\infty.$$
As a result,
\begin{eqnarray*}&&
\Big|D^\alpha_x D^\beta_\xi \big(c_0a_0(x,\xi)+\dots+c_ja_j(x,\xi)\big)\Big|\le
\sum_{\ell=0}^j |c_\ell|\, |D^\alpha_x D^\beta_\xi a_\ell(x,\xi)|\le
\sum_{\ell=0}^j |c_\ell |\,C_{\alpha,\beta,\ell} (1+|\xi|)^{m_\ell-|\beta|}\\&&\qquad\qquad\le
\left(\sum_{\ell=0}^j |c_\ell|\,C_{\alpha,\beta,\ell} \right)\,(1+|\xi|)^{\max\{m_0,\dots,m_j\}-|\beta|},
\end{eqnarray*}
from which the desired result plainly follows.
\end{proof}

Each pseudodifferential symbol induces a pseudodifferential operator \index{pseudodifferential operator}
acting on a function~$u$ belonging to the Schwartz space of smooth and rapidly decreasing functions on~$\R^n$ via the relation
\begin{equation}\label{eqTA6DE4D3DF} T_a u(x):=\int_{\R^n} a(x,\xi)\,\widehat u(\xi)\,e^{2\pi ix\cdot\xi}\,d\xi.\end{equation}
Alternative notations for~$T_a$ are~$a(x,D)$ and~${\rm Op}(a)$.
One can also say that the order~$m$ of the symbol~$a$ is the order of the pseudodifferential operator~$T_a$.\medskip

Multiplication operators are useful examples of pseudodifferential operators of order zero:
\begin{lemma}\label{98iujrt4grLEetaom90iuy65rf-1kd-1}
Let~$\psi\in C^\infty_c(\R^n)$ and, for every function~$\zeta$ belonging to the Schwartz space of smooth and rapidly decreasing functions, set
$$ T\zeta(x):=\psi(x)\,\zeta(x).$$
Then, in the notation of~\eqref{eqTA6DE4D3DF}, we have that~$T=T_\psi$.

Also,~$T_\psi$ is a pseudodifferential operator of order zero.
\end{lemma}

\begin{proof} For every function~$\zeta$ belonging to the Schwartz space of smooth and rapidly decreasing functions on~$\R^n$,
\begin{equation*} T_{\psi} \zeta(x)=\int_{\R^n} {\psi}(x)\,\widehat \zeta(\xi)\,e^{2\pi ix\cdot\xi}\,d\xi={\psi}(x)\zeta(x)=T\zeta(x).\end{equation*}
Moreover, for all~$\alpha$, $\beta\in\N^n$,
$$ D^\alpha_x D^\beta_\xi {\psi}=\begin{cases}
D^\alpha_x {\psi}&{\mbox{ if }}\beta=0,\\
0&{\mbox{otherwise,}}
\end{cases}$$
and consequently
$$ \sup_{{x\in\R^N}\atop{\xi\in\R^n}}(1+|\xi|)^{|\beta|}|D^\alpha_x D^\beta_\xi {\psi}(x)|=
\sup_{{x\in\R^N}}|D^\alpha_x {\psi}(x)|<+\infty.
$$
Comparing with~\eqref{1e12q3ua02it}, we conclude that~${\psi}$ is a
pseudodifferential symbol of order zero.\end{proof}

We will be specifically interested in the case~$N=n$, so we restrict to this situation from now on for the sake of simplicity (however, for interesting cases in which it is convenient to consider~$N\ne n$ see the ``amplitude functions''
discussed in~\cite[page~168]{MR2453959} and the references therein).
\medskip

We observe that pseudodifferential operators contain as a particular case differential operators of a given order~$K$,
and in such a case the order of the pseudodifferential operator coincides with that of the differential operator, since, if
$$ a(x,\xi):=\sum_{{\gamma\in\N^n}\atop{|\gamma|\le K}} (-2\pi i)^\gamma c_\gamma(x)\,\xi^\gamma,$$
we have that, under suitable smoothness assumptions on~$c_\gamma$,
\begin{eqnarray*}&& \sup_{{x\in\R^n}\atop{\xi\in\R^n}}(1+|\xi|)^{|\beta|-K}|D^\alpha_x D^\beta_\xi a(x,\xi)|
\le C\sup_{\xi\in\R^n} (1+|\xi|)^{|\beta|-K}
\sum_{{\gamma\in\N^n}\atop{{\beta\le\gamma}\atop{|\gamma|\le K}}} \|c_\gamma\|_{C^{|\alpha|}(\R^n)}\,|\xi^{\gamma-\beta}|\\ &&\qquad\qquad\le C\sup_{\xi\in\R^n} \sum_{{\gamma\in\N^n}\atop{{\beta\le\gamma}\atop{|\gamma|\le K}}}(1+|\xi|)^{|\gamma|-K}\le C
\end{eqnarray*}
and in this situation
\begin{eqnarray*} T_a u(x)&=&\sum_{{\gamma\in\N^n}\atop{|\gamma|\le K}} \int_{\R^n} (-2\pi i)^\gamma c_\gamma(x)\,\xi^\gamma\,\widehat u(\xi)\,e^{2\pi ix\cdot\xi}\,d\xi\\&=&
\sum_{{\gamma\in\N^n}\atop{|\gamma|\le K}} c_\gamma(x)\int_{\R^n} {\mathcal{F}}(\partial^\gamma u)(\xi)\,e^{2\pi ix\cdot\xi}\,d\xi\\&=&\sum_{{\gamma\in\N^n}\atop{|\gamma|\le K}} c_\gamma(x) \,\partial^\gamma u(x).
\end{eqnarray*}

The operator~$(1-\Delta)^s$ is also a pseudodifferential operator with symbol~$\big(1+4\pi^2|\xi|^2\big)^{s}$,
due to~\eqref{HSZUONDCCONT-FF}, \label{0pirj09365-4g4eZXCHJKLRFG-544-NO1-LL-eq49-09i23w-PAG-t2e32r4s}
which has order~$2s$, owing to~\eqref{0pirj09365-4g4eZXCHJKLRFG-544-NO1-LL-eq49-09i23w}.

Strictly speaking, in this notation, the operator~$(-\Delta)^s$ is not quite a pseudodifferential operator (according to the rather ``strict'' definition given in~\eqref{1e12q3ua02it} and~\eqref{eqTA6DE4D3DF})
since its
symbol would be~$(2\pi|\xi|)^{2s}$, which is not smooth at the origin. \label{SQUEZ-PP}
Though other more relaxed definitions are possible to comprise~$(-\Delta)^s$ in the theory
of pseudodifferential operators,
in the setting that we have presented here a way out to reduce~$(-\Delta)^s$ to a pseudodifferential operator up to a minor remainder will be presented
in Theorem~\ref{SQUEZ}.\medskip

It is useful to observe that ``translation invariant'' pseudodifferential operators reduce to Fourier multipliers, according to the following observation:

\begin{lemma}\label{TRAINVA}
In the notation of~\eqref{TAUNOTA}, given a pseudodifferential symbol~$a$, suppose that, for every~$y\in\R^n$,
$$ \tau_y T_a=T_a \tau_y.$$
Then, the symbol~$a$ does not depend on~$x$.

In this case, writing~$a=a(\xi)$, we have that~$T_a$ acts as the Fourier multiplier~$a$, namely
$$ T_a u={\mathcal{F}}^{-1}\Big( a\widehat u\Big).$$ 
\end{lemma}

\begin{proof} We have that, for every~$u$ belonging to the Schwartz space of smooth and rapidly decreasing functions
and every~$y\in\R^n$,
\begin{eqnarray*}
0&=&\tau_y (T_au)(x)-T_a (\tau_yu)(x)\\
&=&\int_{\R^n} a(x+y,\xi)\,\widehat u(\xi)\,e^{2\pi i(x+y)\cdot\xi}\,d\xi-
\int_{\R^n} a(x,\xi)\,{\mathcal{F}}(\tau_y u)(\xi)\,e^{2\pi ix\cdot\xi}\,d\xi\\&=&
\int_{\R^n} \Big(a(x+y,\xi)-a(x,\xi)\Big)\,\widehat u(\xi)\,e^{2\pi i(x+y)\cdot\xi}\,d\xi,
\end{eqnarray*}
whence~$a(x+y,\xi)-a(x,\xi)=0$ and the desired result follows by choosing~$x:=0$.
\end{proof}

\section{Approximation}

For handy calculations in~\eqref{eqTA6DE4D3DF}, it would be desirable to write explicitly~$\widehat u$ in terms of~$u$
and swap arbitrarily the order of integration. This is possible when~$a$ is compactly supported (say, with~$u$
belonging to the Schwartz space of smooth and rapidly decreasing functions), but not in general. Therefore,
a useful auxiliary tool consists in the possibility of approximating arbitrary pseudodifferential symbols with compactly supported ones. This is indeed possible, according to the following result:

\begin{lemma} \label{LEM-012o3lrtgpseudodifferentimbol}
Let~$m\in\R$ and~$a$ be a pseudodifferential symbol of order~$m$.

For every~$\e\in(0,1]$ there exists~$a_\e\in C^\infty_c(\R^n\times\R^n)$ which is a pseudodifferential symbol of order~$m$ and for which there is a bound in~\eqref{1e12q3ua02it} that is uniform in~$\e$, i.e.
\begin{equation}\label{IKJYDMAMS1efcYS2YIS-1} \sup_{{{x\in\R^n}\atop{\xi\in\R^n}}\atop{\e\in(0,1]}}(1+|\xi|)^{|\beta|-m}|D^\alpha_x D^\beta_\xi a_\e(x,\xi)|<+\infty.\end{equation}

Additionally,~$a_\e$ converges to~$a$ as~$\e\searrow0$, together with all its derivatives,
and uniformly in~$\e$, i.e. for every~$\alpha$,~$\beta\in\N^n$,
\begin{equation}\label{IKJYDMAMS1efcYS2YIS-2} \lim_{\delta\searrow0}\sup_{\e\in(0,\delta]}\Big|D^\alpha_x D^\beta_\xi \Big(a_\e(x,\xi)-a(x,\xi)\Big)\Big|=0.\end{equation}

Finally, for all~$u$ in the Schwartz space of smooth and rapidly decreasing functions, for every~$x\in\R^n$,
\begin{equation}\label{IKJYDMAMS1efcYS2YIS-3} \lim_{\e\searrow0} T_{a_\e}u(x)=T_au(x).\end{equation}
\end{lemma}

\begin{proof} Let~$\tau_0\in C^\infty_c((-2,2),\,[0,1])$ with~$\tau_0=1$ in~$[-1,1]$. Let also
$$\tau(x,\xi):=\tau_0(x_1)\dots\tau_0( x_n)\,\tau_0(\xi_1)\dots\tau_0(\xi_n).$$
Given~$\e\in(0,1]$, let
$$ a_\e(x,\xi):= \tau(\e x,\e\xi)\,a(x,\xi).$$
By the Leibniz Product Rule,
\begin{eqnarray*}
D^{(\alpha,\beta)}_{x,\xi} a_\e(x,\xi)
&=&\sum_{{(\eta,\sigma)\in\N^n\times\N^n}\atop{(\eta,\sigma) \leq( \alpha,\beta)} }
{{(\alpha,\beta)}\choose{( \eta,\sigma)} } \,D^{(\eta,\sigma)}_{x,\xi}\big(\tau(\e x,\e\xi)\big)\;D^{(\alpha-\eta,\beta-\sigma)}_{x,\xi}a(x,\xi)\\&=&\sum_{{(\eta,\sigma)\in\N^n\times\N^n}\atop{(\eta,\sigma) \leq( \alpha,\beta)} }
{{(\alpha,\beta)}\choose{( \eta,\sigma)} } \,\e^{|\eta|+|\sigma|}D^{(\eta,\sigma)}_{x,\xi}\tau(\e x,\e\xi)\;D^{(\alpha-\eta,\beta-\sigma)}_{x,\xi}a(x,\xi).
\end{eqnarray*}

{F}rom this and~\eqref{1e12q3ua02it} for the symbol~$a$ it follows that
\begin{eqnarray*}
|D^\alpha_x D^\beta_\xi a_\e(x,\xi)|&\le&
C_{\alpha,\beta} \sum_{{\sigma\in\N^n}\atop{\sigma \leq\beta} }
\e^{|\sigma|}\,\Big|D^\sigma_\xi \big(\tau_0(\e \xi_1)\dots\tau_0(\e\xi_n)\big)\Big|\,(1+|\xi|)^{m-|\beta|+|\sigma|}\\&\le&
C_{\alpha,\beta}\left((1+|\xi|)^{m-|\beta|}+ \sum_{{\sigma\in\N^n\setminus\{0\}}\atop{\sigma \leq\beta} }
\e^{|\sigma|}\,\chi_{B_{2\sqrt{n}}\setminus B_1}(\e\xi)\,(1+|\xi|)^{m-|\beta|+|\sigma|}\right)\\&\le&
C_{\alpha,\beta}(1+|\xi|)^{m-|\beta|}\left(1+ \sum_{{\sigma\in\N^n\setminus\{0\}}\atop{\sigma \leq\beta} }
\e^{|\sigma|}\,\left(1+\frac1\e\right)^{|\sigma|}\right)
\\&\le&C_{\alpha,\beta}
(1+|\xi|)^{m-|\beta|},
\end{eqnarray*}
which is~\eqref{IKJYDMAMS1efcYS2YIS-1}.

Moreover, if~$\e\in(0,\delta]$,
\begin{eqnarray*}
&&\Big|D^\alpha_x D^\beta_\xi \Big(a_\e(x,\xi)-a(x,\xi)\Big)\Big|
\\&\le&\big|\tau(\e x,\e\xi)-1\big|\;\big| D^{(\alpha,\beta)}_{x,\xi}a(x,\xi)\big|\\&&\qquad+C_{\alpha,\beta}\sum_{{(\eta,\sigma)\in\N^{2n}\setminus\{(0,0)\}}\atop{(\eta,\sigma) \leq( \alpha,\beta)} }
\e^{|\eta|+|\sigma|}\big|D^{(\eta,\sigma)}_{x,\xi}\tau(\e x,\e\xi)\big|\;\big|D^{(\alpha-\eta,\beta-\sigma)}_{x,\xi}a(x,\xi)\big|\\&\le&\chi_{\R^n\setminus B_{1/\delta}}(x)\;\big| D^{(\alpha,\beta)}_{x,\xi}a(x,\xi)\big|+C_{\alpha,\beta}\sum_{{(\eta,\sigma)\in\N^{2n}\setminus\{(0,0)\}}\atop{(\eta,\sigma) \leq( \alpha,\beta)} }
\e^{|\eta|+|\sigma|}(1+|\xi|)^{m}\\&\le&\chi_{\R^n\setminus B_{1/\delta}}(x)\;\big| D^{(\alpha,\beta)}_{x,\xi}a(x,\xi)\big|+C_{\alpha,\beta}\,\delta\,(1+|\xi|)^{m}\,
\end{eqnarray*}
from which we deduce~\eqref{IKJYDMAMS1efcYS2YIS-2}.

Finally, if~$u$ lies in the Schwartz space of smooth and rapidly decreasing functions,
\begin{eqnarray*}
|T_{a_\e}u(x)-T_au(x)|&\le&\int_{\R^n} |a_\e(x,\xi)-a(x,\xi)|\,|\widehat u(\xi)|\,d\xi.
\end{eqnarray*}
We also notice that
$$ |a_\e(x,\xi)-a(x,\xi)|\,|\widehat u(\xi)|\le\big(|a_\e(x,\xi)|+|a(x,\xi)|\big)\,|\widehat u(\xi)|
\le C(1+|\xi|)^m\,|\widehat u(\xi)|\in L^1(\R^n),$$
owing to~\eqref{1e12q3ua02it} and~\eqref{IKJYDMAMS1efcYS2YIS-1}, used with~$\alpha:=0$ and~$\beta:=0$.

Thus, using~\eqref{IKJYDMAMS1efcYS2YIS-2} with~$\alpha:=0$ and~$\beta:=0$ and the Dominated Convergence Theorem, one obtains~\eqref{IKJYDMAMS1efcYS2YIS-3}, as desired.
\end{proof}

\section{Composition rule}

A pivotal step is understanding how pseudodifferential operators combine together.

\begin{theorem}\label{COMPOTHPSDO}
Let~$a$ and~$b$ be pseudodifferential symbols.
Assume that~$a$ is of order~$m_1\in\R$ and~$b$ is of order~$m_2\in\R$.

Then, there exists a pseudodifferential symbol~$c$ such that
$$ T_a\circ T_b=T_c.$$

Moreover,~$c$ is of order~$m_1+m_2$ and we have that, for every~$K\in\N$,
\begin{equation} \label{COMPOTHPSDO1}
c(x,\xi)-\sum_{{\alpha\in\N^n}\atop{|\alpha|\le K}}\frac{1}{(2\pi i)^{|\alpha|}\,\alpha!}\,\partial^\alpha_\xi a(x,\xi)\,\partial^\alpha_x b(x,\xi)\end{equation}
is a pseudodifferential symbol of order~$m_1+m_2-K-1$.
\end{theorem}

We observe that for translation invariant operators (i.e., Fourier multipliers, according Lemma~\ref{TRAINVA}),
the result in Theorem~\ref{COMPOTHPSDO} is straightforward: indeed, if~$a=a(\xi)$ and~$b=b(\xi)$, then
$$ T_a (T_b u)={\mathcal{F}}^{-1}\Big( a{\mathcal{F}} (T_b u)\Big)={\mathcal{F}}^{-1}\Big( ab \widehat u\Big)$$
and therefore~$c=ab$ in this case. This is in agreement with~\eqref{COMPOTHPSDO1} because, in this situation,
$$ \sum_{{\alpha\in\N^n}\atop{|\alpha|\le K}}\frac{1}{(2\pi i)^{|\alpha|}\,\alpha!}\,\partial^\alpha_\xi a\,\partial^\alpha_x b=ab,$$
since the derivatives in~$x$ vanish.

In this sense, Theorem~\ref{COMPOTHPSDO} provides a generalization of this multiplicative rule using an expansion,
plus a remainder having an arbitrarily small order.
\medskip

To prove Theorem~\ref{COMPOTHPSDO}, we make some preliminary observations.

\begin{lemma}\label{1e12q3ua02it012oierjohfgbFJJasxLA-lemma}
Let~$k$,~$K\in\N$, with~$k\le K$.
Let~$a$ and~$b$ be pseudodifferential symbols.
Assume that~$a$ is of order~$m_1\in\R$ and~$b$ is of order~$m_2\in\R$.

Then, for every set of coefficients~$c_\gamma\in\C$,
\begin{equation}\label{1e12q3ua02it012oierjohfgbFJJasxLA}
\sum_{{\gamma\in\N^n}\atop{k\le |\gamma|\le K}}c_\gamma\,\partial^\gamma_\xi a(x,\xi)\,\partial^\gamma_x b(x,\xi)\end{equation}
is a pseudodifferential symbol of order~$m_1+m_2-k$.
\end{lemma}

\begin{proof} By~\eqref{1e12q3ua02it},
$$ C_{\alpha,\beta}:=
\sup_{{x\in\R^n}\atop{\xi\in\R^n}}(1+|\xi|)^{|\beta|-m_1}|D^\alpha_x D^\beta_\xi a(x,\xi)|
+\sup_{{x\in\R^n}\atop{\xi\in\R^n}}(1+|\xi|)^{|\beta|-m_2}|D^\alpha_x D^\beta_\xi b(x,\xi)|<+\infty$$
and therefore, by the Leibniz Product Rule,
\begin{eqnarray*}&&
\left|D^\alpha_x D^\beta_\xi\left(\partial^\gamma_\xi a(x,\xi)\,\partial^\gamma_x b(x,\xi)\right)\right|\\&
\le& \sum_{{(\sigma,\theta)\in\N^n\times\N^n}\atop{(\sigma,\theta) \leq (\alpha,\beta)}}{{(\alpha,\beta)} \choose{( \sigma,\theta)} }
\big|D^{(\sigma,\theta+\gamma)}_{x,\xi} a(x,\xi) \big|\,\big|D^{(\alpha-\sigma+\gamma,\beta-\theta)}_{x,\xi} b(x,\xi)\big|\\&
\le& C\sum_{{(\sigma,\theta)\in\N^n\times\N^n}\atop{(\sigma,\theta) \leq (\alpha,\beta)}}
(1+|\xi|)^{m_1-|\theta+\gamma|} (1+|\xi|)^{m_2-|\beta-\theta|}
\\&\le&C(1+|\xi|)^{m_1+m_2-|\beta|-|\gamma|}.
\end{eqnarray*}
This gives that~$\partial^\gamma_\xi a(x,\xi)\,\partial^\gamma_x b(x,\xi)$
is a pseudodifferential symbol of order~$m_1+m_2-|\gamma|$.
Thus, owing to Lemma~\ref{1e12q3ua02it012oierjohfgbFJJasxLALEM}, the function in~\eqref{1e12q3ua02it012oierjohfgbFJJasxLA}
is a pseudodifferential symbol of order
\begin{equation*}
\max_{{\gamma\in\N^n}\atop{k\le |\gamma|\le K}}m_1+m_2-|\gamma|=m_1+m_2-k.
\qedhere\end{equation*}\end{proof}

\begin{corollary}\label{0ojwf2iyfuYHBhbv6c2nb-76cy4FSGBy3u654t}
Let~$k$,~$K\in\N$, with~$k< K$.
Let~$a$ and~$b$ be pseudodifferential symbols.
Assume that~$a$ is of order~$m_1\in\R$ and~$b$ is of order~$m_2\in\R$.

If
\begin{equation}\label{COMPOT0pw3oefjrkvfNNieNTE}
c(x,\xi)-\sum_{{\gamma\in\N^n}\atop{k\le |\gamma|\le K}}\frac{1}{(2\pi i)^{|\gamma|}\,\gamma!}\,\partial^\gamma_\xi a(x,\xi)\,\partial^\gamma_x b(x,\xi)\end{equation}
is a pseudodifferential symbol of order~$m_1+m_2-K-1$, then~$c$ is a pseudodifferential symbol of order~$m_1+m_2-k$.
\end{corollary}

\begin{proof} Let~$d(x,\xi)$ be as in~\eqref{COMPOT0pw3oefjrkvfNNieNTE}.
Let also~$s(x,\xi)$ denote the sum in~\eqref{COMPOT0pw3oefjrkvfNNieNTE}.
In this way, we know that~$d$ is of order~$m_1+m_2-K-1$
and, by Lemma~\ref{1e12q3ua02it012oierjohfgbFJJasxLA-lemma},~$s$ is of order~$m_1+m_2-k$.
Thus, using Lemma~\ref{1e12q3ua02it012oierjohfgbFJJasxLALEM}
and the fact that~$c=d+s$, we infer that~$c$ has order~$\max\{m_1+m_2-k,m_1+m_2-K-1\}=m_1+m_2-k$.
\end{proof}

We now dive into the proof of Theorem~\ref{COMPOTHPSDO}.

\begin{proof}[Proof of Theorem~\ref{COMPOTHPSDO}] First of all, we observe that
if the pseudodifferential symbol in~\eqref{COMPOTHPSDO1} is of order~$m_1+m_2-K-1$,
then automatically the order of~$c$ is~$m_1+m_2$, thanks to Corollary~\ref{0ojwf2iyfuYHBhbv6c2nb-76cy4FSGBy3u654t},
used here with~$k:=0$.

We will perform the proof under the additional assumption that~$a$ and~$b$ are compactly supported
(the general case can be obtained by approximation, using Lemma~\ref{LEM-012o3lrtgpseudodifferentimbol}
and some refined version of Fubini's Theorem which is needed to justify formal interchanges of integrals in the presence of complex exponentials;
here, we will sloppily surf over these important details about  oscillatory integrals and we refer
to~\cite[Sections~2.2 and~2.3]{MR1211419}, \cite[pages~168--169]{MR2453959}, \cite[Section~1.4]{MR1314815},
\cite[Section~2.5.1]{MR2567604} and~\cite[Section~3.3]{MR2884718} for a thorough theory).

Thus, using the compact support additional assumption to deal with integrable functions in the sense of Lebesgue,
we can write that, if~$u$ belongs to the Schwartz space of smooth and rapidly decreasing functions,
\begin{eqnarray*}
T_{b} u(z)&=&\int_{\R^n} b(z,\eta)\,\widehat u(\eta)\,e^{2\pi iz\cdot\eta}\,d\eta
\end{eqnarray*}
yielding that
\begin{eqnarray*}(
T_{a}\circ T_b) u(x)&=&\int_{\R^n} a(x,\xi)\,{\mathcal{F}} \big( T_bu\big)(\xi)\,e^{2\pi ix\cdot\xi}\,d\xi\\&=&\iint_{\R^n\times\R^n} a(x,\xi)\, T_bu(z)\,e^{2\pi i(x-z)\cdot\xi}\,d\xi\,dz\\
&=&\iiint_{\R^n\times\R^n\times\R^n} a(x,\xi)\, b(z,\eta)\,
\widehat u(\eta)\,e^{2\pi iz\cdot\eta}
\,e^{2\pi i(x-z)\cdot\xi}\,d\eta\,d\xi\,dz\\
&=&\int_{\R^n} c(x,\eta)\,\widehat u(\eta)\,e^{2\pi ix\cdot\eta}\,d\eta\\&=&T_cu(x),
\end{eqnarray*}
where
$$ c(x,\eta):=
\iint_{\R^n\times\R^n} a(x,\xi)\, b(z,\eta)
\,e^{2\pi i(x-z)\cdot(\xi-\eta)}\,d\xi\,dz
.$$
We stress that the above integrand is Lebesgue integrable, due to the compact support additional assumption on~$a$ and~$b$.

Thus, using the change of variable~$\sigma:=\xi-\eta$
and a Taylor expansion, and considering Fourier Transforms with respect to the first variable,
\begin{eqnarray*}
c(x,\eta)&=&
\int_{\R^n} a(x,\xi)\, \widehat b(\xi-\eta,\eta)
\,e^{2\pi i x\cdot(\xi-\eta)}\,d\xi\\&=&
\int_{\R^n} a(x,\sigma+\eta)\, \widehat b(\sigma,\eta)
\,e^{2\pi i x\cdot\sigma}\,d\sigma\\&=&
\int_{\R^n}\left(
\sum_{{\gamma\in\N^n}\atop{0\le |\gamma|\le K}}\frac1{\gamma!}\partial^\gamma_\eta a(x,\eta)\sigma^\gamma+{\mathcal{R}}_K(x,\eta,\sigma)
\right)\, \widehat b(\sigma,\eta)
\,e^{2\pi i x\cdot\sigma}\,d\sigma\\&=&
\int_{\R^n}
\sum_{{\gamma\in\N^n}\atop{0\le |\gamma|\le K}}\frac1{\gamma!}\partial^\gamma_\eta a(x,\eta)\sigma^\gamma\, \widehat b(\sigma,\eta)
\,e^{2\pi i x\cdot\sigma}\,d\sigma+{\mathcal{J}}(x,\eta),\end{eqnarray*}
for a suitable remainder~${\mathcal{R}}_K$ of the form
$$ {\mathcal{R}}_K(x,\eta,\sigma)=\sum_{{\gamma\in\N^n}\atop{ |\gamma|= K+1}}
\int _{0}^{1}\frac{(K+1)\sigma^\gamma}{\gamma!}\,
\partial^\gamma_\eta a(x,t\sigma+\eta)\,(1-t)^K\,dt
,$$
where we have used the short notation
$$ {\mathcal{J}}(x,\eta):=\int_{\R^n}{\mathcal{R}}_K(x,\eta,\sigma)\, \widehat b(\sigma,\eta)
\,e^{2\pi i x\cdot\sigma}\,d\sigma.$$

In particular, for some~$C_{K,\alpha,\beta}>0$,
\begin{equation}\label{45tyhT2AUpiuygfLOSKMD6782u3tg} |D^\alpha_x D^\beta_\eta\,
{\mathcal{R}}_K(x,\eta,\sigma)|\le C_{K,\alpha,\beta}\,|\sigma|^{K+1}\sup_{{{\gamma\in\N^n}\atop{|\gamma|= K+1}}\atop{t\in[0,1]}}
|\partial^\alpha_x\partial^{\gamma+\beta}_\eta a(x,t\sigma+\eta)|.\end{equation}

Therefore, since, for every given~$\eta\in\R^n$, if~$\phi(x):=\partial^\gamma_x b(x,\eta)$,
\begin{eqnarray*} &&
\int_{\R^n}(2\pi i\sigma)^\gamma\,\widehat b(\sigma,\eta)\,e^{2\pi i x\cdot\sigma}\,d\sigma=
\iint_{\R^n\times\R^n}(2\pi i\sigma)^\gamma b(y,\eta)\,e^{2\pi i( x-y)\cdot\sigma}\,d\sigma\,dy\\&&\qquad=(-1)^{|\gamma|}
\iint_{\R^n\times\R^n}b(y,\eta)\,\partial^\gamma_y \big(e^{2\pi i( x-y)\cdot\sigma}\big)\,d\sigma\,dy=
\iint_{\R^n\times\R^n}\partial^\gamma_y b(y,\eta)\,e^{2\pi i( x-y)\cdot\sigma}\,d\sigma\,dy\\
&&\qquad={\mathcal{F}}^{-1}\big({\mathcal{F}} (\phi)\big)(x)=\phi(x)=\partial^\gamma_x b(x,\eta),
\end{eqnarray*}
we conclude that
\begin{eqnarray*}
c(x,\eta)&=&
\sum_{{\gamma\in\N^n}\atop{0\le |\gamma|\le K}}\frac1{(2\pi i)^{|\gamma|}\gamma!}\partial^\gamma_\eta a(x,\eta)\partial_xb(x,\eta)+{\mathcal{J}}(x,\eta).\end{eqnarray*}

Consequently, comparing with~\eqref{COMPOT0pw3oefjrkvfNNieNTE}, to complete the proof of the desired result,
it suffices to show that
\begin{equation}\label{COMPOT0pw3oefjrkvfNNieNTEamsdcVSQujasn71}
{\mbox{${\mathcal{J}}$ is a pseudodifferential symbol of order~$m_1+m_2-K-1$.}}\end{equation} To this end,
we recall~\eqref{45tyhT2AUpiuygfLOSKMD6782u3tg} and the fact that~$a$ has order~$m_1$ to deduce that
$$ |D^\alpha_x D^\beta_\eta\,{\mathcal{R}}_K(x,\eta,\sigma)|\le C_{K,\alpha,\beta}\,|\sigma|^{K+1}\sup_{t\in[0,1]}(1+|t\sigma+\eta|)^{m_1-|\beta|-K-1}.$$

Furthermore, since we assumed that~$\R^n\ni x\mapsto b(x,\xi)$ is smooth and compactly supported,
using the fact that, for every~$M\in\N$,
$$ (1-\Delta_y)^M (e^{-2\pi iy\cdot\sigma})=(1+4\pi^2|\sigma|^2)^M e^{-2\pi iy\cdot\sigma},$$
we see that, if the support of~$b$ in its first variable is contained in some ball~$B$,
\begin{eqnarray*}
\widehat b(\sigma,\eta)&=& \int_{\R^n} b(y,\eta)\,e^{-2\pi iy\cdot\sigma}\,dy\\
&=& \frac{1}{(1+4\pi^2|\sigma|^2)^M}\int_{\R^n} b(y,\eta)\,(1-\Delta_y)^M(e^{-2\pi iy\cdot\sigma})\,dy\\
&=& \frac{1}{(1+4\pi^2|\sigma|^2)^M}\int_{\R^n} (1-\Delta_y)^M b(y,\eta)\,e^{-2\pi iy\cdot\sigma}\,dy.
\end{eqnarray*}

In addition, since~$b$ is of order~$m_2$,
$$ |D^\alpha_y D^\beta_\eta b(y,\eta)| \le C_{\alpha,\beta} (1+|\eta|)^{m_2-|\beta|}$$
for some~$C_{\alpha,\beta}>0$ and therefore
$$ |D^\beta_\eta\widehat b(\sigma,\eta)|\le\frac{C_{M,\beta}\,|B|\,(1+|\eta|)^{m_2-|\beta|}}{(1+4\pi^2|\sigma|^2)^M}.$$

{F}rom these observations and the Leibniz Product Rule we infer that
\begin{eqnarray*}&&|D^\alpha_x D^\beta_\eta
{\mathcal{J}}(x,\eta)|\\&\le&C_{\alpha,\beta}\sup_{{\gamma\in\N^n}\atop{\gamma\le\beta}}
\int_{\R^n} |D^\alpha_x D^\gamma_\eta {\mathcal{R}}_K(x,\eta,\sigma)|\, |D^{\beta-\gamma}_\eta\widehat b(\sigma,\eta)|\,d\sigma\\&\leq&C_{K,M,\alpha,\beta}\,|B|\,\sup_{{\gamma\in\N^n}\atop{\gamma\le\beta}}
\int_{\R^n} |\sigma|^{K+1}\sup_{t\in[0,1]}(1+|t\sigma+\eta|)^{m_1-|\gamma|-K-1}
(1+|\eta|)^{m_2-|\beta|+|\gamma|}\,\frac{d\sigma}{(1+4\pi^2|\sigma|^2)^M}.
\end{eqnarray*}

Now, if~$\sigma\in B_{|\eta|/2}$, for all~$t\in[0,1]$ we have that
$$ |t\sigma+\eta|\in \Big[ |\eta|-|\sigma|,\,|\eta|+|\sigma|\Big]\subseteq \left[ \frac{|\eta|}2,\frac{3|\eta|}2\right]$$
and thus, using~$C$ for short to rename constants and taking~$M$ conveniently large,
\begin{eqnarray*}&&
\int_{B_{|\eta|/2}} |\sigma|^{K+1}\sup_{t\in[0,1]}(1+|t\sigma+\eta|)^{m_1-|\gamma|-K-1}
(1+|\eta|)^{m_2-|\beta|+|\gamma|}\,\frac{d\sigma}{(1+4\pi^2|\sigma|^2)^M}\\&\le& C
\int_{B_{|\eta|/2}} |\sigma|^{K+1}(1+|\eta|)^{m_1-|\gamma|-K-1}
(1+|\eta|)^{m_2-|\beta|+|\gamma|}\,\frac{d\sigma}{(1+4\pi^2|\sigma|^2)^M}\\&\le&C(1+|\eta|)^{m_1+m_2-K-1-|\beta|}
\int_{\R^n} \frac{|\sigma|^{K+1}\,d\sigma}{(1+4\pi^2|\sigma|^2)^M}\\&\le&C(1+|\eta|)^{m_1+m_2-K-1-|\beta|},
\end{eqnarray*}yielding that
\begin{equation}\label{SCMEIEPDALPCGEARAD0234it}\begin{split}&|D^\alpha_x D^\beta_\eta
{\mathcal{J}}(x,\eta)|\\&\le
C|B|
\Bigg((1+|\eta|)^{m_1+m_2-K-1-|\beta|}\\&\qquad+\sup_{{\gamma\in\N^n}\atop{\gamma\le\beta}}
\int_{\R^n\setminus B_{|\eta|/2}} |\sigma|^{K+1}\sup_{t\in[0,1]}(1+|t\sigma+\eta|)^{m_1-|\gamma|-K-1}
(1+|\eta|)^{m_2-|\beta|+|\gamma|}\,\frac{d\sigma}{(1+4\pi^2|\sigma|^2)^M}\Bigg).
\end{split}\end{equation}

We also point out that, choosing~$M$ conveniently large,
\begin{equation}\label{SCMEIEPDALPCGEARAD0234it2}\begin{split}&
\int_{\R^n\setminus B_{|\eta|/2}} |\sigma|^{K+1}\sup_{t\in[0,1]}(1+|t\sigma+\eta|)^{m_1-|\gamma|-K-1}
(1+|\eta|)^{m_2-|\beta|+|\gamma|}\,\frac{d\sigma}{(1+4\pi^2|\sigma|^2)^M}\\&\qquad\leq
C(1+|\eta|)^{m_1+m_2-K-1-|\beta|}.\end{split}
\end{equation}
To check this, we distinguish two cases. If~$m_1-|\gamma|-K-1\ge0$ then, when~$\sigma\in\R^n\setminus B_{|\eta|/2}$,
$$ \sup_{t\in[0,1]}(1+|t\sigma+\eta|)^{m_1-|\gamma|-K-1}\le (1+|\sigma|+|\eta|)^{m_1-|\gamma|-K-1}\le C(1+|\sigma|)^{m_1-K-1}$$
and thus in this case, if~$M$ is large enough,
\begin{eqnarray*}&&
\int_{\R^n\setminus B_{|\eta|/2}} |\sigma|^{K+1}\sup_{t\in[0,1]}(1+|t\sigma+\eta|)^{m_1-|\gamma|-K-1}
(1+|\eta|)^{m_2-|\beta|+|\gamma|}\,\frac{d\sigma}{(1+4\pi^2|\sigma|^2)^M}\\&\le&C
\int_{\R^n\setminus B_{|\eta|/2}} |\sigma|^{K+1}(1+|\sigma|)^{m_1-K-1}
(1+|\eta|)^{m_1+m_2-K-1-|\beta|}\,\frac{d\sigma}{(1+4\pi^2|\sigma|^2)^M}\\&\le&
C(1+|\eta|)^{m_1+m_2-K-1-|\beta|}.
\end{eqnarray*}

If instead~$m_1-|\gamma|-K-1<0$ then, for~$\gamma\le\beta$,
\begin{eqnarray*}&&
\int_{\R^n\setminus B_{|\eta|/2}} |\sigma|^{K+1}\sup_{t\in[0,1]}(1+|t\sigma+\eta|)^{m_1-|\gamma|-K-1}
(1+|\eta|)^{m_2-|\beta|+|\gamma|}\,\frac{d\sigma}{(1+4\pi^2|\sigma|^2)^M}\\&\leq&
\int_{\R^n\setminus B_{|\eta|/2}} |\sigma|^{K+1}
(1+|\eta|)^{m_2-|\beta|+|\gamma|}\,\frac{d\sigma}{(1+4\pi^2|\sigma|^2)^M}\\&\leq&
(1+|\eta|)^{m_2}\int_{\R^n\setminus B_{|\eta|/2}}\frac{|\sigma|^{K+1}\,d\sigma}{(1+4\pi^2|\sigma|^2)^M}\\&\leq&C
(1+|\eta|)^{m_2}\Big( \chi_{B_1}(\eta)+|\eta|^{K+1+n-2M}\chi_{\R^n\setminus B_1}(\eta)\Big)\\&\le&C
\Big( \chi_{B_1}(\eta)+(1+|\eta|)^{m_2+K+1+n-2M}\chi_{\R^n\setminus B_1}(\eta)\Big)\\&\le&C(1+|\eta|)^{m_1+m_2-K-1-|\beta|},
\end{eqnarray*}as long as~$M$ is large enough.
This completes the proof of~\eqref{SCMEIEPDALPCGEARAD0234it2}.

Hence, by~\eqref{SCMEIEPDALPCGEARAD0234it} and~\eqref{SCMEIEPDALPCGEARAD0234it2},
\[ |D^\alpha_x D^\beta_\eta{\mathcal{J}}(x,\eta)|\le
C|B|(1+|\eta|)^{m_1+m_2-K-1-|\beta|}\]
and we have thereby established~\eqref{COMPOT0pw3oefjrkvfNNieNTEamsdcVSQujasn71}, as desired.
\end{proof}

An interesting application of Theorem~\ref{COMPOTHPSDO} is that
a pseudodifferential operator of order~$-\infty$ arises by using functions with disjoint supports:

\begin{corollary}\label{COMPOTHPS-q0woeirjUJSDNIie00-cr5tgd}
Let~$a$,~$\phi$,~$\psi$ be pseudodifferential symbols. Assume that~$\phi$ and~$\psi$ have disjoint supports.

Then, the pseudodifferential operator~$T_\phi\circ T_a\circ T_\psi$ has order~$-\infty$.
\end{corollary}

\begin{proof} Let~$m_a$,~$m_\phi$ and~$m_\psi$ denote the orders of~$a$,~$\phi$ and~$\psi$ respectively.
Let also~$K\in\N$. By Theorem~\ref{COMPOTHPSDO}, we know that
$$ T_a\circ T_\psi=T_c,$$
for some pseudodifferential symbol~$c$ of order~$m_c:=m_a+m_\psi$ such that, setting
$$\omega(x,\xi):=\sum_{{\alpha\in\N^n}\atop{|\alpha|\le K}}\frac{1}{(2\pi i)^{|\alpha|}\,\alpha!}\,\partial^\alpha_\xi a(x,\xi)\,\partial^\alpha_x \psi(x,\xi),$$
we have that
\begin{equation*}
c(x,\xi)-\omega(x,\xi)\end{equation*}
is a pseudodifferential symbol of order~$m_a+m_\psi-K-1$.

This and Theorem~\ref{COMPOTHPSDO} yield that~$R:=T_\phi\circ T_{c-\omega}$ is a pseudodifferential symbol of order~$\mu_1:=m_a+m_\phi+m_\psi-K-1$.

We also remark that the support of~$\omega$ is contained in the support of~$\psi$ and it is therefore disjoint from
that of~$\phi$. As a result, for all~$x$,~$\xi\in\R^n$,
$$ \sum_{{\alpha\in\N^n}\atop{|\alpha|\le K}}\frac{1}{(2\pi i)^{|\alpha|}\,\alpha!}\,\partial^\alpha_\xi \phi(x,\xi)\,\partial^\alpha_x \omega(x,\xi)=0.$$
{F}rom this and Theorem~\ref{COMPOTHPSDO}, we infer that~$S:=T_\phi\circ T_\omega$ is a pseudodifferential operator of order~$\mu_2:=m_\phi+m_\omega-K-1$, where~$m_\omega$ is the order of~$\omega$.

Since, by Lemma~\ref{1e12q3ua02it012oierjohfgbFJJasxLALEM},~$$m_\omega=\max\{m_{c-\omega},m_c\}=
\max\{m_a+m_\psi-K-1,m_a+m_\psi\}=m_a+m_\psi,$$
we conclude that
$$ \mu_2=m_a+m_\phi+m_\psi-K-1=\mu_1.$$

All in all,
\begin{eqnarray*}
T_\phi\circ T_a\circ T_\psi=T_\phi\circ T_c=T_\phi\circ (T_{c-\omega}+T_\omega)=
R+S,
\end{eqnarray*}
which, by Lemma~\ref{1e12q3ua02it012oierjohfgbFJJasxLALEM}, has order
$\max\{\mu_1,\mu_2\}=\mu_1=m_a+m_\phi+m_\psi-K-1$, which can be made arbitrarily small by taking~$K$ as large as we wish.
\end{proof}

\section{Commutator rule}\label{COM78REDU-912ieyrhf-001}

We now deal with the commutator of two pseudodifferential operators.
In the translation invariant case (see Lemma~\ref{TRAINVA}) the corresponding symbols are associated with multipliers independent of~$x$, thus, since multiplication is commutative, the commutator vanishes.

In the general case, the commutator is nontrivial. Interestingly, it can be sharply investigated in view of
the composition result in Theorem~\ref{COMPOTHPSDO}. Actually, at a first glance, this result gives
that if~$T_a$ is of order~$m_1$ and~$T_b$ is of order~$m_2$, then both~$ T_a\circ T_b$ and~$ T_b\circ T_a$
are of order~$m_1+m_2$. But this can be sharpened in the light of the explicit representation
in~\eqref{COMPOTHPSDO1}, which allows one to cancel one term and find that the commutator is
in fact of order~$m_1+m_2-1$.

The details go as follows:

\begin{theorem}\label{COMPOTHPSDO-COMMTU215ER}
Let~$a$ and~$b$ be pseudodifferential symbols.
Assume that~$a$ is of order~$m_1\in\R$ and~$b$ is of order~$m_2\in\R$.

Then, the commutator~$ T_a\circ T_b- T_b\circ T_a$ is a pseudodifferential operator of order~$m_1+m_2-1$.
\end{theorem}

\begin{proof} We use Theorem~\ref{COMPOTHPSDO} to see that
$$ T_a\circ T_b=T_c\qquad{\mbox{and}}\qquad
T_b\circ T_a=T_d,$$
for suitable pseudodifferential symbols such that, for every~$K\in\N$,
\begin{equation*} \begin{split}&
c(x,\xi)-\sum_{{\alpha\in\N^n}\atop{|\alpha|\le K}}\frac{1}{(2\pi i)^{|\alpha|}\,\alpha!}\,\partial^\alpha_\xi a(x,\xi)\,\partial^\alpha_x b(x,\xi)
\\ {\mbox{and }}\quad&
d(x,\xi)-\sum_{{\alpha\in\N^n}\atop{|\alpha|\le K}}\frac{1}{(2\pi i)^{|\alpha|}\,\alpha!}\,\partial^\alpha_\xi b(x,\xi)\,\partial^\alpha_x a(x,\xi)\end{split}
\end{equation*}
are pseudodifferential symbols of order~$m_1+m_2-K-1$.

As a result,~$ T_a\circ T_b- T_b\circ T_a=T_c-T_d=T_{c-d}$ and, by Lemma~\ref{1e12q3ua02it012oierjohfgbFJJasxLALEM},
\begin{eqnarray*} &&f(x,\xi)\\&:=&
c(x,\xi)-d(x,\xi)-\sum_{{\alpha\in\N^n}\atop{|\alpha|\le K}}\frac{1}{(2\pi i)^{|\alpha|}\,\alpha!}\,\partial^\alpha_\xi a(x,\xi)\,\partial^\alpha_x b(x,\xi)-\sum_{{\alpha\in\N^n}\atop{|\alpha|\le K}}\frac{1}{(2\pi i)^{|\alpha|}\,\alpha!}\,\partial^\alpha_\xi b(x,\xi)\,\partial^\alpha_x a(x,\xi)\\&=&
c(x,\xi)-d(x,\xi)-\sum_{{\alpha\in\N^n}\atop{1\le|\alpha|\le K}}\frac{1}{(2\pi i)^{|\alpha|}\,\alpha!}\,\partial^\alpha_\xi a(x,\xi)\,\partial^\alpha_x b(x,\xi)-\sum_{{\alpha\in\N^n}\atop{1\le|\alpha|\le K}}\frac{1}{(2\pi i)^{|\alpha|}\,\alpha!}\,\partial^\alpha_\xi b(x,\xi)\,\partial^\alpha_x a(x,\xi)
\end{eqnarray*}
is a pseudodifferential symbol of order~$m_1+m_2-K-1$.

By Lemma~\ref{1e12q3ua02it012oierjohfgbFJJasxLA-lemma}, used here with~$k:=1$, we also know that the two latter sums are pseudodifferential symbols of order~$m_1+m_2-1$.
That is,~$c-d-f$ is a pseudodifferential symbol of order~$m_1+m_2-1$ and therefore~$c-d$ has order~$\max\{
m_1+m_2-K-1,m_1+m_2-1\}=m_1+m_2-1$.
\end{proof}

\section{Continuity in Lebesgue spaces}

A delicate, and very useful, result from the theory of pseudodifferential operators is the continuity in Lebesgue spaces
for operators of order zero:

\begin{theorem}\label{TRAINVA-TH}
Let~$p\in(1,+\infty)$ and let~$a$ be a pseudodifferential symbol of order zero.
Then,~$T_a$ extends to a bounded linear operator from~$L^p(\R^n)$ to~$L^p(\R^n)$.
\end{theorem}

For a proof of this result, see e.g.~\cite[Theorem~5.19]{MR2884718}. Here, we just observe that
in the case of translation invariant pseudodifferential operators Theorem~\ref{TRAINVA-TH}
is a consequence of Mikhlin Multiplier Theorem (see Theorem~\ref{Mikhlin Multiplier Theorem}, used in this case with~$X:=Y:=\C$,
and Lemma~\ref{TRAINVA}; note indeed that condition~\eqref{Mikhlin Multiplier Theorem-ASSUN}
would be guaranteed in this framework by the assumption that the symbol in Theorem~\ref{TRAINVA-TH} has order zero).

\begin{corollary}\label{COMPOTHPSDO-CORCON}
Let~$p\in(1,+\infty)$ and let~$a$ be a pseudodifferential symbol of order~$\sigma\ge0$.

Then,~$T_a$ extends to a bounded linear operator from~${\mathcal{L}}^p_{\sigma}(\R^n)$ to~$L^p(\R^n)$.

More generally, for every~$\theta\ge0$,
$T_a$ extends to a bounded linear operator from~${\mathcal{L}}^p_{\sigma+\theta}(\R^n)$ to~${\mathcal{L}}^p_{\theta}(\R^n)$.
\end{corollary}

\begin{proof}The first claim follows from the second one by picking~$\theta:=0$, hence we focus on the proof of the second claim.

The operator~$L:=(1-\Delta)^{-\frac{\sigma+\theta}2}$ (that is, by
Lemma~\ref{PRODpoikjhr3-2:le-7}, the convolution against~${\mathcal{B}}^{((\sigma+\theta)/2)}$)
is a pseudodifferential operator of order~$-(\sigma+\theta)$ (see the discussion on page~\pageref{0pirj09365-4g4eZXCHJKLRFG-544-NO1-LL-eq49-09i23w-PAG-t2e32r4s}).
Similarly, the operator~$L_0:=(1-\Delta)^{\frac{\theta}2}$ has order~$\theta$.

Consequently, Theorem~\ref{COMPOTHPSDO} entails that the composition~$S:=L_0\circ T_a\circ L$ is a pseudodifferential operator of order zero.

This observation and Theorem~\ref{TRAINVA-TH} yield that
$$ \|S f\|_{L^p(\R^n)}\le C\|f\|_{L^p(\R^n)}.$$
Thus, choosing~$f:=(1-\Delta)^{\frac{\sigma+\theta}2}u$ (so that~$u={\mathcal{B}}^{((\sigma+\theta)/2)}*f$)
and recalling the Bessel potential norm in~\eqref{ERFGHJN6789-09tftd90u8yhgiug8erhISB},
\begin{equation*} 
\|T_au\|_{{\mathcal{L}}^p_{\theta}(\R^n)}=
\|L_0(T_au)\|_{L^p(\R^n)}=\|S ((1-\Delta)^{\frac{\sigma+\theta}2}u)\|_{L^p(\R^n)}\le C\| f\|_{L^p(\R^n)}=C\|u\|_{ {\mathcal{L}}^p_{\sigma+\theta}(\R^n)}
.\qedhere\end{equation*}
\end{proof}

\begin{corollary}\label{COMPOTHPS-q0woeirjUJSDNIie00}
Let~$p\in(1,+\infty)$ and let~$a$ be a pseudodifferential symbol of order~$\sigma\le0$.

Then,~$T_a$ extends to a bounded linear operator from~$L^p(\R^n)$ to~${\mathcal{L}}^p_{|\sigma|}(\R^n)$.

In particular, if~$a$ has order~$-\infty$, then, for all~$\theta\ge0$,~$T_a$ extends to a bounded linear operator from~$L^p(\R^n)$ to~${\mathcal{L}}^p_{\theta}(\R^n)$.
\end{corollary}

\begin{proof} The second statement follows from the first one, so we focus on the proof of the first claim.

The operator~$L:=(1-\Delta)^{\frac{|\sigma|}2}$
is a pseudodifferential operator of order~$|\sigma|=-\sigma$. Hence, by Theorem~\ref{COMPOTHPSDO},
we have that~$S:=L\circ T_a$ is a pseudodifferential operator of order zero.

{F}rom this observation and Theorem~\ref{TRAINVA-TH} we arrive at
\begin{equation*}
\|T_af\|_{ {\mathcal{L}}^p_{|\sigma|}(\R^n)}=
\|(1-\Delta)^{\frac{|\sigma|}2}(T_af)\|_{L^p(\R^n)}=\|S f\|_{L^p(\R^n)}\le C\|f\|_{L^p(\R^n)}
.\qedhere\end{equation*}
\end{proof}

Notice that Corollary~\ref{COMPOTHPS-q0woeirjUJSDNIie00} entails that Bessel potential spaces and pseudodifferential operators nicely combine in terms of continuity theory,
which in turn has a positive impact on the regularity theory as well. 
Namely, Corollary~\ref{COMPOTHPS-q0woeirjUJSDNIie00} can be seen as 
a continuity result in scales of Bessel potential spaces of the forward mapping~$f\mapsto T_a f$
when the symbol of~$a$ is non-positive.
This, in many situations of interest, provides essential information about the question on whether the solutions to a considered
equation have a regularity property in terms of the data (e.g., thanks to the ``ellipticity'' of the operator):
in our setting, if the equation is driven by an operator such as~$(1-\Delta)^{\frac{s}2}$, with~$s\ge0$,
one can consider an inverse operator~$(1-\Delta)^{\frac{\sigma}2}$, with~$\sigma:=-s\le0$,
which allows one to employ Corollary~\ref{COMPOTHPS-q0woeirjUJSDNIie00} to deduce a regularity statement.

It is also interesting that Corollary~\ref{COMPOTHPS-q0woeirjUJSDNIie00} highlights a
``universal'' feature of Bessel potential spaces with respect to pseudodifferential operators,
in the sense that the continuity property of the forward mapping associated to such an operator
does not depend much on the specific case, but mostly only on the (non-positive) order of the operator
(that is, for Corollary~\ref{COMPOTHPS-q0woeirjUJSDNIie00}
every pseudodifferential symbol of order~$\sigma\le0$ produces  in the Bessel potential spaces framework the same
continuity effect of the particular case of~$(1-\Delta)^{\frac\sigma2}$).
This is useful, because, while, in view of
their definition in~\eqref{HSZUONDCCONT-FF0iejdfujhnHSNDJNdik9238475tyhfj}, Bessel potential spaces are designed on the particular model
provided by~$(1-\Delta)^{\frac\sigma2}$,
Corollary~\ref{COMPOTHPS-q0woeirjUJSDNIie00} ensures that
all pseudodifferential operators with the same (non-positive) order behave the same, thus ``freeing'' on the one hand
Bessel potential spaces from the specific choice of a single pseudodifferential operator and allowing on the other hand a general continuity theory for all pseudodifferential operators. 

\section[The fractional Laplacian and pseudodifferential operators]{Squeezing the fractional Laplacian into the theory of pseudodifferential operators}

As noticed already on page~\pageref{SQUEZ-PP}, some further work is needed to
cover the fractional Laplacian with the standard theory of pseudodifferential operators
(at least, in the framework presented here in which the pseudodifferential symbols are supposed to be everywhere smooth).
Several options are possible for this, here we pursue the following strategy:

\begin{theorem}\label{SQUEZ} 
There exists a pseudodifferential symbol~$a=a(\xi)$ of order~$2s$ such that
\begin{equation}\label{OUAsjhkdnwGrfGttgr5CEinde9i3rnefijf-1} \inf_{\R^n}a>0\end{equation}
and a continuous function~$\eta=\eta(\xi)$ supported in~$B_1$ such that the following statement holds true.

We have that
\begin{equation}\label{OUAsjhkdnwGrfGttgr5CEinde9i3rnefijf-2}(-\Delta)^s = T_a+S,\end{equation}
where, for every~$u$ belonging to the Schwartz space of smooth and rapidly decreasing functions,
$$ Su :={\mathcal{F}}^{-1}\Big( \eta\widehat u\Big).$$

In addition, the function~$\frac1a$ is a pseudodifferential symbol of order~$-2s$ and~$T_{1/a}$ is the inverse of~$T_a$ in the sense that both~$T_a\circ T_{1/a}$ and~$T_{1/a}\circ T_a$ are the identity.

Moreover, for any~$p\ge1$, any~$\alpha\in\N^n$,
and any~$\zeta$ in the Schwartz space of smooth and rapidly decreasing functions,
\begin{equation}\label{OUAsjhkdnwGrfGttgr5CEinde9i3rnefijf-3}
\|D^\alpha(S \zeta)\|_{L^\infty(\R^n)}\le C\|\zeta\|_{L^1(\R^n)},
\end{equation}
with~$C>0$ depending only on~$n$,~$s$,~$p$ and~$\alpha$.

Furthermore, 
\begin{equation}\label{0p09i723erfghbnm5678uytrfdesxdfghjiuytfrdsextcyuygihYTGFCGYUIJN} (-\Delta)^s (T_{1/a}u) = T_{1/a}((-\Delta)^s u) = u + Ru,\end{equation}
where
$$ Ru :={\mathcal{F}}^{-1}\left( \frac{\eta}{a}\,\widehat u\right).$$
 
Finally, for any~$p\ge1$, any~$\alpha\in\N^n$,
and any~$\zeta$ in the Schwartz space of smooth and rapidly decreasing functions,
\begin{equation}\label{OUAsjhkdnwGrfGttgr5CEinde9i3rnefijf-4}
\|D^\alpha(R \zeta)\|_{L^\infty(\R^n)}\le C\|\zeta\|_{L^1(\R^n)},
\end{equation}
with~$C>0$ depending only on~$n$,~$s$,~$p$ and~$\alpha$.
\end{theorem}

Statements such as the one in~\eqref{0p09i723erfghbnm5678uytrfdesxdfghjiuytfrdsextcyuygihYTGFCGYUIJN} are quite common in the theory of pseudodifferential operators. Roughly speaking, they involve the construction of an approximate inverse operator, called in jargon a ``parametrix'' \index{parametrix}
(see e.g.~\cite[Chapter~7]{MR2453959}): the gist of this procedure is that sometimes one can invert a pseudodifferential operator up to operators which possess a smoothing effect
(when this inversion is only possible from one side, one speaks about
a right parametrix or a left parametrix).

\begin{proof}[Proof of Theorem~\ref{SQUEZ}] Let~$\phi\in C^\infty_c(B_1,[0,1])$ with~$\phi=1$ in~$B_{1/2}$.
Let
\begin{eqnarray*}
&&\eta(\xi):=\phi(\xi)\Big((2\pi|\xi|)^{2s}-(1+4\pi^2|\xi|^2)^s\Big)\\ {\mbox{and }}&&
a(\xi):=\big(1-\phi(\xi)\big)(2\pi|\xi|)^{2s}+\phi(\xi)(1+4\pi^2|\xi|^2)^s
.\end{eqnarray*}

We write~$a=a_1+a_2$, with~$a_1:=\big(1-\phi(\xi)\big)(2\pi|\xi|)^{2s}$
and~$a_2:=\phi(\xi)(1+4\pi^2|\xi|^2)^s$.
Notice that~$T_{a_2}$ is a pseudodifferential operator of order~$2s$,
due to Lemma~\ref{98iujrt4grLEetaom90iuy65rf-1kd-1} and Theorem~\ref{COMPOTHPSDO}.

Also,~$T_{a_1}$
is a pseudodifferential operator of order~$2s$, since, by the Leibniz Product Rule,
\begin{eqnarray*}&&
\sup_{{x\in\R^n}\atop{\xi\in\R^n}}(1+|\xi|)^{|\beta|-2s}|D^\alpha_x D^\beta_\xi a_1(\xi)|=
\chi_{\{0\}}(\alpha)\sup_{{\xi\in\R^n}}(1+|\xi|)^{|\beta|-2s}|D^\beta_\xi a_1(\xi)|
\\&&\qquad=(2\pi)^{2s}\chi_{\{0\}}(\alpha)\sup_{{\xi\in\R^n}}(1+|\xi|)^{|\beta|-2s}\left|
\sum_{{\eta\in\N^n}\atop{\eta\le\beta} }
{{\beta}\choose{\eta} } D^\eta\big(1-\phi(\xi)\big) D^{\beta-\eta}|\xi|^{2s}
\right|\\ &&\qquad\le
C\sup_{{\xi\in\R^n}}(1+|\xi|)^{|\beta|-2s}\left(\chi_{\R^n\setminus B_{1/2}}(\xi)\,|\xi|^{2s-|\beta|}
+
\sum_{{\eta\in\N^n}\atop{0\ne\eta\le\beta}} \chi_{B_1}(\xi) |\xi|^{2s-|\beta|+|\eta|}
\right)\\ &&\qquad\le
C\sup_{{\xi\in\R^n}}(1+|\xi|)^{|\beta|-2s}\left( \chi_{\R^n\setminus B_{1/2}}(\xi)\,|\xi|^{2s-|\beta|}
+ \chi_{B_1}(\xi)\,|\xi|^{2s-|\beta|+1}\right)
\\&&\qquad\le C,
\end{eqnarray*}
up to renaming~$C$.

Consequently,~$T_a=T_{a_1}+T_{a_2}$ is a pseudodifferential operator of order~$2s$, owing to
Lemma~\ref{1e12q3ua02it012oierjohfgbFJJasxLALEM}.
Notice also that~\eqref{OUAsjhkdnwGrfGttgr5CEinde9i3rnefijf-1} and~\eqref{OUAsjhkdnwGrfGttgr5CEinde9i3rnefijf-2} follow from the construction of~$\eta$ and~$a$.

Moreover,
\begin{eqnarray*}&&
\sup_{{x\in\R^n}\atop{\xi\in\R^n\setminus B_1}}(1+|\xi|)^{|\beta|+2s}\left|D^\alpha_x D^\beta_\xi \frac1{a(\xi)}\right|=
(2\pi)^{-2s}\chi_{\{0\}}(\alpha)\sup_{{\xi\in\R^n\setminus B_1}}(1+|\xi|)^{|\beta|+2s}\big|D^\beta_\xi |\xi|^{-2s}\big|
\\&&\qquad\le C\sup_{{\xi\in\R^n\setminus B_1}}(1+|\xi|)^{|\beta|+2s}|\xi|^{-2s-|\beta|} \le C.
\end{eqnarray*}
This information, combined with the smoothness of~$\frac1{a}$, yields that~$\frac1a$ is a pseudodifferential symbol of order~$-2s$.

Also, since~$a$ and~$\frac1a$ act as Fourier multipliers,~$T_a$ and~$T_{1/a}$ are the inverse of each other.
Additionally,
$$ {\mathcal{F}}\Big( (-\Delta)^s (T_{1/a}u)\Big) = {\mathcal{F}}\Big(T_{1/a}((-\Delta)^s u)\Big)=
\frac{(2\pi|\xi|)^{2s}}{a(\xi)}\,\widehat u(\xi)=
\frac{a(\xi)+\eta(\xi)}{a(\xi)}\,\widehat u(\xi)=
\widehat u(\xi)+\frac{\eta(\xi)}{a(\xi)}\,\widehat u(\xi),$$
which establishes~\eqref{0p09i723erfghbnm5678uytrfdesxdfghjiuytfrdsextcyuygihYTGFCGYUIJN}.

Now we prove~\eqref{OUAsjhkdnwGrfGttgr5CEinde9i3rnefijf-3} and~\eqref{OUAsjhkdnwGrfGttgr5CEinde9i3rnefijf-4}.
We observe that these claims are established once we prove that if~$\varphi=\varphi(\xi)$ 
is a continuous and compactly supported function, then
\begin{equation}\label{09o2j3rhngrsdTh09iuhgf6ytgvhyuujnbfdssdfghoqwieuyhgfweifu78i1-Deltas0023rTH}
\|D^\alpha( T_\varphi \zeta)\|_{L^\infty(\R^n)}\le C\|\zeta\|_{L^1(\R^n)},\end{equation}
with~$C>0$ depending only on~$n$,~$\varphi$,~$p$ and~$\alpha$.

To prove this, we observe that, if the support of~$\varphi$ is contained in some ball~$B_R$,
\begin{eqnarray*}
D^\alpha( T_\varphi \zeta)=D^\alpha \left( {\mathcal{F}}^{-1} \Big( \varphi\widehat\zeta\Big)\right)= D^\alpha \big( \check\varphi*\zeta\big)=
\big(D^\alpha \check\varphi\big)*\zeta.
\end{eqnarray*}
We also recall that, for every~$x\in\R^n$,
\begin{eqnarray*}
|D^\alpha \check\varphi(x)|&=&\left| D^\alpha \int_{\R^n} \varphi(\xi)\,e^{2\pi ix\cdot\xi}\,d\xi\right|\\
&=&\left|\int_{B_R} (2\pi i\xi)^\alpha\varphi(\xi)\,e^{2\pi ix\cdot\xi}\,d\xi\right|\\
&\le& CR^{n+|\alpha|}\|\varphi\|_{L^\infty(\R^n)}\\&=&C\|\varphi\|_{L^\infty(\R^n)},
\end{eqnarray*}
up to renaming~$C$.

As a consequence, for every~$x\in\R^n$,
\begin{eqnarray*}
\big|\big(D^\alpha \check\varphi\big)*\zeta(x)\big|&\le&\int_{\R^n}|D^\alpha \check\varphi (x-y)| \,|\zeta(y)|\,dy\le
C\int_{\R^n}|\zeta(y)|\,dy,
\end{eqnarray*}
which gives~\eqref{09o2j3rhngrsdTh09iuhgf6ytgvhyuujnbfdssdfghoqwieuyhgfweifu78i1-Deltas0023rTH}.
\end{proof}

The approach sketched here can be generalized to a broad class of
``elliptic'' operators, namely to the operators whose symbol
satisfies
$$|p(x, \xi)| \geq c|\xi|$$ for all~$\xi\in\R^n\setminus B_R$
and~$x\in\R^n$ (for some~$c$, $R>0$). 

This condition gives that there
exists a symbol~$q(x,\xi)$ such that, if~$P$ is the operator associated
to~$p$ and~$Q$ is the operator associated to~$q$, one has that $$ P Q -
I = S_1 \qquad{\mbox{and}}\qquad QP - I = S_2,$$ where~$I$ is identity, and~$ S_1$ and~$S_2$ are
of order~$-\infty$. The operator~$Q$ is called a
parametrix (or almost-inverse) of~$ P$.
For modern, detailed accounts of this method see~\cite{MR2453959, MR2884718}. 

\end{appendix}

\begin{appendix}

\chapter*{Acknowledgments}
This work was supported by the Australian Laureate Fellowship FL190100081,
by the Australian Future Fellowship FT230100333,
by the GNAMPA-INdAM project ``Equazioni nonlocali di tipo misto e geometrico'' (Italy) and the PRIN project 2022R537CS ``$NO^3$ - Nodal Optimization, NOnlinear elliptic equations, NOnlocal geometric problems, with a focus on regularity'' (Italy). It is a pleasure
to thank Gerd Grubb for very helpful discussions that introduced us to the elegant theory of pseudodifferential operators. We also thank 
Neils Jacob for useful and informative discussions about differential operators,
Yoshihiro Sawano for interesting discussions about Besov spaces, and Josh Troy for stimulating comments.

\end{appendix}

\end{document}